\documentclass[11pt,reqno]{amsart}
\usepackage[left=1.3in, right=1.2in, top=1in, bottom=1in, includefoot]{geometry}
\usepackage{anyfontsize}
\usepackage{graphicx}
\usepackage{mathrsfs} 
\usepackage{amssymb}   
\usepackage{amsthm}  
\usepackage{bbm}
\usepackage[shortlabels]{enumitem}
\usepackage{xcolor}
\usepackage{soul}
\usepackage{scalerel}
\newtheorem{thm}{Theorem}[section]

\newtheorem{prop}[thm]{Proposition}
\newtheorem{cor}[thm]{Corollary}
\newtheorem{Def}[thm]{Definition}
\newtheorem{lemma}[thm]{Lemma}
\newtheorem{remark}[thm]{Remark}

\newcommand{\ve}{{\varepsilon}}
\newcommand{\N}{{\mathbb{N}}}
\newcommand{\R}{{\mathbb{R}}}
\numberwithin{equation}{section}
\title{Justification of the Benjamin-Ono equation as an internal water waves model} 
\author[M. Oen Paulsen]{Martin Oen Paulsen}

\address{Department of Mathematics\\ University of Bergen\\ Postbox 7800\\ 5020 Bergen\\ Norway}
\email{Martin.Paulsen@UiB.no}

\date{\today}
\begin{document}

	\begin{abstract}  In this paper, we give the first rigorous justification of the Benjamin-Ono  equation:
		\begin{equation*}
			\hspace{3cm}
			\partial_t \zeta
			+
			(1 - \frac{\gamma}{2}\sqrt{\mu}|\mathrm{D}|)\partial_x \zeta 
			+
			\frac{3\ve }{2}\zeta \partial_x\zeta =0, 
			\hspace{3cm}  
			\text{(BO)}
		\end{equation*}
		as an internal water wave model on the physical time scale. Here, $\ve$ is a small parameter measuring the weak nonlinearity of the waves, $\mu$ is the shallowness parameter, and $\gamma\in (0,1)$ is the ratio between the densities of the two fluids. To be precise, we first prove the existence of a solution to the internal water wave equations for a two-layer fluid with surface tension, where one layer is of shallow depth and the other is of infinite depth. The existence time is of order $\mathcal{O}(\frac{1}{\ve})$ for a small amount of surface tension such that $\ve^2 \leq\mathrm{bo}^{-1} $ where  $\mathrm{bo}$ is the Bond number. Then, we show that these solutions are close, on the same time scale,  to the solutions of the BO equation with a precision of order $\mathcal{O}(\mu + \mathrm{bo}^{-1})$. In addition, we provide the justification of new equations with improved dispersive properties, the Benjamin equation, and the Intermediate Long Wave (ILW) equation in the deep-water limit.

		The long-time well-posedness of the two-layer fluid problem was first studied by Lannes [Arch. Ration. Mech. Anal., 208(2):481-567, 2013] in the case where both fluids have finite depth. Here, we adapt this work to the case where one of the fluid domains is of finite depth, and the other one is of infinite depth. The novelties of the proof are related to the geometry of the problem, where the difference in domains alters the functional setting for the Dirichlet-Neumann operators involved. In particular, we study the various compositions of these operators that require a refined symbolic analysis of the Dirichlet-Neumann operator on infinite depth and derive new pseudo-differential estimates that might be of independent interest.

	\end{abstract}
    \maketitle
     \section{Introduction}
    \subsection{The Benjamin-Ono equation}
    The Benjamin-Ono (BO) equation is a nonlocal asymptotic model for the unidirectional propagation of weakly nonlinear, long internal waves in a two-layer fluid. The equation is given by 
    \begin{equation}\label{Bo equation}
    	\partial_t \zeta
    	+
    	(1 - \frac{\gamma}{2}\sqrt{\mu}|\mathrm{D}|)\partial_x \zeta 
    	+
    	\frac{3\ve }{2}\zeta \partial_x\zeta =0, 
    \end{equation}
    where $x\in \R$, $t>0$ and $\zeta = \zeta(x,t)$ denote the free surface, which is a real-valued function. Here, $\ve$ is a small parameter measuring the weak nonlinearity of the waves, $\mu$ is the shallowness parameter, and $\gamma\in (0,1)$ is the ratio between the densities of the two fluids.  The operator $|\mathrm{D}|$ is a Fourier multiplier defined by
    \begin{equation*}
    	|\mathrm{D}|f(x) = \mathcal{F}^{-1}\big{(}|\xi| \mathcal{F}(f)(\xi)\big{)}(x).
    \end{equation*}
    The BO equation was introduced formally by Benjamin \cite{Benjamin_67} in 1967 and at the same time independently by Davis and Acrivos \cite{Davis_Acrivos_67}.  We also refer the reader to the book by Klein and Saut \cite{Klein_Saut_21}, Chapter 3, for a detailed state-of-the-art. The studies in \cite{Benjamin_67,Davis_Acrivos_67} showed that the BO equation admits solitary waves with mere algebraic decay, as opposed to the exponential decay exhibited for the solitary waves of the KdV equation. Davis and Acrivos also gave experimental results. The experiments were carried out in a wave tank with a stratified solution of salt and water, where almost any disturbance to the surface
    would, after a short time, produce a wave with a fixed shape that propagates stably. It was later noted by Ono \cite{Ono_75}  that the ease with which they could generate solitary waves indicates that they are solitons. 

    The paper by Ono sparked much interest in studying the dynamics of the BO equation. It was proved that the solitary waves are unique (up to translation) \cite{Amick_Toland_91}, and the stability of these objects is studied in \cite{Bennett_etal_83,Weinstein_87,Kenig_Martel_09}  (see the references for a precise definition). Moreover, the stability of these waves is strong enough to preserve its own identity upon nonlinear interactions. The strong interaction between several solitary waves is studied in \cite{Yoshimasa_06,Neves_Lopes_06} and relies on explicit formulas 
    (see also \cite{Kenig_Martel_09} for the asymptotic stability of one soliton and $N-$solitons).
    %
    %
    %
    %
    %
    %

    The fact that explicit solutions like the soliton (or multi-solitons) exist is a consequence of the complete integrability of the BO equation. Nakamura \cite{Nakamura_79} proved the existence of an infinite number of conserved quantities and discovered a Lax pair structure (see also \cite{Bock_79,Fokas_Ablowitz_83,Gerard_Kappeler_21}). This insight is proven to be crucial for the study of the dynamics of the BO equation and was further developed by Gérard and Kapeller \cite{Gerard_Kappeler_21}. They constructed a nonlinear Fourier transform for the BO equation on the torus, which has several applications to low regularity well-posedness of the initial value problem, the long-time behavior of solutions, and stability of traveling waves (see \cite{Gerard_20} for a survey on this topic). More recently, Gérard  \cite{Gerard_22} derived an explicit formula for the BO equation based on the Lax pair structure with remarkable consequences, for instance, the zero-dispersion limit problem \cite{Gerard_23} (see also \cite{Gassot_23_zero, Gassot_23}).

    The Cauchy problem for BO has been extensively studied in the last 40 years. It was first proved to be globally well-posed in $H^s(\R)$ for $s>\frac{3}{2}$ using an energy method, see \cite{Abdelouhab_Bona_Felland_Saut_89,Iorio_86}. We also refer to the results \cite{Ponce_91,Koch_Tzvetkov_03,Kenig_Koenig_03} for an improvement by including the dispersive smoothing effects in the energy estimates.
    %
    %
    %
    %
    %
    %
    One of the main difficulties in improving the result further is that the flow map fails to be $C^2$ in any Sobolev space $H^s(\R)$ \cite{Molinet_Saut_Tzvetkov_01} (see also \cite{Koch_Tzvetkov_05}). 
    %
    %
     A breakthrough was achieved by Tao \cite{Tao_04}, where he introduced a clever change of variables (the gauge transform) to improve the structure of the nonlinearity. As a consequence, he obtained a global well-posedness result for data in $H^1(\R)$. Several papers expanded on these ideas. We refer the interested reader to  \cite{Burq_Planchon_08,Ionescu_Kenig_07,Molinet_Pilod_12,Ifrim_Tataru_19} for results on the line and \cite{Molinet_07,Molinet_08,Molinet_Pilod_12} in the periodic case culminating in the global well-posedness in $L^2$. 
    So far, the theory is based on PDE methods. However, by actively using the integrable structure, Gérard, Kappeler, and Topalov \cite{Gerard_Kappeler_Topalov_23}  proved the sharp global well-posedness result in $H^s(\mathbb{T})$ for $s>-\frac{1}{2}$ on the torus. Also, still relying on the integrability,  Killip, Laurens, and Vi\c san \cite{Killip_Laurens_Visan_23} recently proved the global well-posedness in $H^s(\mathbb{R})$ for $s>-\frac{1}{2}$ on the line. 
    \begin{figure}[ht]\label{Fig set up}
    	\vstretch{0.8}{\includegraphics[scale=0.38]{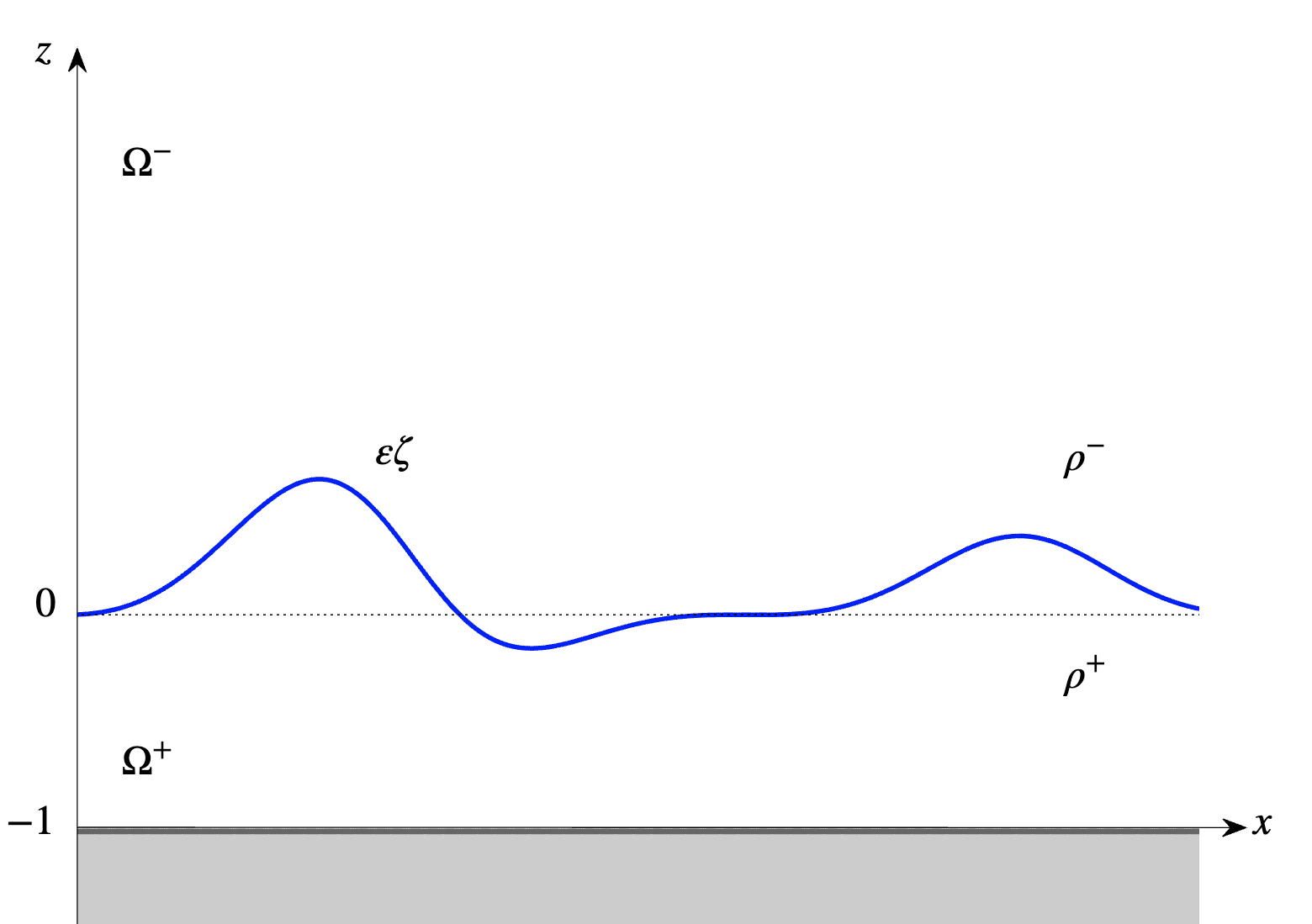}}
    	\caption{\small The blue line denotes the surface elevation $z=\ve \zeta$ and separates two fluids with density $0<\rho^-<\rho^+$. }
    \end{figure} 
	\subsubsection{Full justification}
	%
	%
	%
	Despite the rich well-posedness theory for the BO equation, it is still an open question whether its solutions are close to the ones of the original physical system. In the rigorous derivation of any asymptotic model, it is fundamental to know whether its solutions converge to the solutions of the reference model from which it is derived. The BO equation's reference model is a coupled system of Euler equations for two fluids that are joined with an interface, as in Figure  \ref{Fig set up}. Under the irrotationality condition, we will call the reference model the \lq\lq internal water waves system\rq\rq. To prove that BO is a valid approximation, we shall compare their solutions to the physical parameters:
	\begin{equation*}
		\ve = \frac{a}{H} \quad, \quad \mu = \frac{H^2}{\lambda^2} \quad \text{and} \quad  \mathrm{bo} = \frac{\rho^+ g\lambda^2(1-\gamma)}{\sigma},
	\end{equation*}
	where $a$ is the typical amplitude of the waves, $H$ is the still water depth in the lower fluid,  $\lambda$ is the typical wavelength, $\rho^+$ is the density of the lower fluid, $g$ is the acceleration of gravity, $\sigma$ is the surface tension parameter. Since  surface tension is only relevant for short waves, we will suppose that $\mathrm{bo}^{-1}$ is small.  To be precise, we will let  $\ve^2 \leq \mathrm{bo}^{-1} $, where we answer the following three questions:
	\begin{enumerate}[1.]
		
	   	\item  The solutions of the internal water wave equations exist on the relevant time scale $\mathcal{O}(\frac{1}{\ve})$.
	   	
	   	\item The solutions to the BO equation are uniformly controlled on the same time scale.
	   	
	   	
	   	\item 

	   The \textit{consistency} between the BO equation and the internal water wave equations, i.e., solutions of the internal water wave equations solves the BO equation up to a small reminder. Then  use it to show that the error is of order $\mathcal{O}( (\mu +\mathrm{bo}^{-1})t)$ when comparing the two solutions.
   \end{enumerate}

   The first point is the most challenging step of this paper, where we need to prove that the internal water waves equations are long-time well-posed for regular initial data. Moreover, we are confined to the specific geometry where one fluid layer is of shallow depth and the other of infinite depth. The main obstacle to constructing such solutions is the tendency of internal waves to break down due to Kelvin-Helmholtz instabilities. This issue was resolved in the case of a single fluid (i.e., $\rho^- = 0$), where stable solutions are deduced by imposing the \textit{Rayleigh-Taylor criterion}:
   \begin{equation}\label{R-T crit}
   		-\partial_z P^+|_{z = \ve \zeta} >0,
   \end{equation}
   where $z$ is the vertical coordinate, and $P^+|_{z = \ve \zeta}$  is the pressure at the surface. The physical relevance of this criterion can be seen by considering the Euler equations for a trivial flow $-\partial_z P^+|_{z = \ve \zeta} = \rho^+ g$, where gravity $g>0$ is the restoring force\color{black}. In this sense, the criterion is a natural condition to ensure that the pressure force is restoring and does not amplify the waves \cite{Sir_Taylor_50}. From a mathematical perspective, the criterion ensures the hyperbolicity of the water waves equation. Moreover, it is proved that under this condition, the water waves system in finite depth is locally well-posed by Lannes \cite{FirstWPWW_Lannes_05} and then long-time well-posed by Alvarez-Samaniego and Lannes \cite{Alvarez-SamaniegoLannes08a}. We also refer the reader to the pioneering work of Wu \cite{Wu_97,Wu_99} in the case of infinite depth where one of the key observations was the use of the Rayleigh-Taylor criterion \eqref{R-T crit} to remove a smallness condition on the data (see also the more recent work on extended life span and improved regularity results \cite{Wu_09,Wu_11,Wu_19,AlazardMetivier09,AlazardBurqZuily11,AlazardBurqZuily14,Ai_Ifrim_Tataru_22,Ai_23}).

   In the case of two fluid systems, the internal water waves system produces Kelvin-Helmholtz instabilities and becomes ill-posed unless there is surface tension $\sigma>0 $ \cite{Ebin_88,Iguchi_Tanaka_Tani_97,Kamotski_Lebeau_05}. There are several results on the well-posedness of the internal water wave systems in different configurations of the fluid domain where the time of existence depends on $\sigma$ \cite{Ambrose_03,Ambrose_Masmoudi_07,Cheng_Coutland_Shkoller_08,Shatah_Zeng_08,Shatah_Zeng_11}. On the other hand, the strength of surface tension is only relevant for very small values in water waves theory. Therefore, we must add surface tension to exploit its regularizing effect, but obtain a uniform local well-posedness result with respect to $\sigma>0$, allowing it to be taken small. Lannes solved this problem in  \cite{LannesTwoFluid13}  for two fluid layers of finite depth, where he derived a new stability criterion depending on $\sigma$. A crucial point is that surface tension could be taken small enough in the criterion such that it does not affect the main dynamics of the equation. It is also noted in the paper that the criterion depends strongly on the geometry of the problem. Many of the technical difficulties in this paper are related to this observation. In this work, we consider one of the layers to be of infinite depth. This is the key point of the paper, which will require a symbolic analysis of the Dirichlet-Neumann operator in a new functional setting. To achieve this goal, we derive several pseudo-differential estimates for symbols with limited smoothness.
   
   


  The second point is well-known since the BO equation is globally well-posed for regular data (see the discussion above). However, we will consider three intermediate models to derive the BO equation.   We will first derive a new weakly dispersive Benjamin-type system from the internal water waves system that is consistent with a precision of order $\mathcal{O}(\ve \sqrt{\mu})$. The novelty of this system is that it is exact at the linear level and models exactly the dispersion relation of the internal water waves systems when $\ve = 0$. These types of systems are commonly referred to as fully dispersive, and the analysis is motivated by the work of Emerald \cite{Emerald21}  (see also \cite{DucheneMMWW21}), where he derived fully dispersive shallow water models.  Then, we will consider unidirectional solutions of this system to deduce a weakly dispersive Benjamin equation, which is exact at the linear level, and from it, we give the first rigorous derivation of the Benjamin equation \cite{Benjamin_92,Benjamin_96}:
  \begin{equation}\label{Benjamin}
		\partial_t \zeta
            +
		\big{(} 1- \frac{\gamma}{2} \sqrt{\mu}|\mathrm{D}| - \frac{\partial_x^2}{2\mathrm{bo}}\big{)}\partial_x \zeta
		+
		\frac{3\ve }{2}\zeta \partial_x\zeta =  0.
	\end{equation}  
  We refer the reader to \cite{klein2023benjaminrelatedequations} for a detailed survey on the Benjamin equation and \cite{Linares_99, Chen_Guo_Xiao_11, Kozono_Ogawa_Tanisaka_01,Li_Wu_10} for global well-posedness results. Then from \eqref{Benjamin} we may use it to derive the BO equation by neglecting the effect of surface tension. Finally, we also use this result to derive the ILW equation rigorously. The ILW was derived formally in \cite{Joseph_77,Kubota_Ko_Dobbs_78} as an interface model where one of the fluid layers is of great depth, and its rigorous derivation remains an open problem. However, in the deep-water limit, it is well-known that its solutions converge to the ones of BO \cite{chapouto2024deepwaterlimitintermediatelong, Li_24} (see also  \cite{Ifrim_Saut_23,Albert_Bona_Saut_97,Molinet_Vento_25,Molinet_Pilod_Vento_18,Chapouto_Forlano_Oh_Pilod_2024} for other results on the Cauchy problem).  In this case, we can prove that the ILW approximates the internal water waves equations when one layer has infinite depth.  To rigorously derive these models, we closely follow the work of Bona, Lannes, and Saut in \cite{Bona_Saut_Lannes_08}. In this paper, they derive several shallow water models for internal fluids and comment on the formal derivation of the BO equation.



    

    Finally, we comment on several works that are closely related to the derivation of the BO equation. In \cite{Craig_Guyenne_Kalisch_05}, Craig, Guyenne, and Kalisch used a Hamiltonian perturbation approach to formally derive asymptotic models from the two-layer system. Among the models is the BO equation. The benefit of this approach is that the systems inherit the Hamiltonian structure, but as noted in \cite{Klein_Saut_21}, the process could lead to ill-posed systems. In particular, the BO system they derive, which can be used to derive the BO equation, is linearly ill-posed. In \cite{Ifrim_Rowan_Tataru_Wan_22}, Ifrim, Rowan, Tataru, and Wan show that the BO equation can also be viewed as an asymptotic model from the water waves equations in infinite depth in the case of constant vorticity. The approximation they obtain is rigorously justified but, of course, not related to the asymptotic description of internal waves. Lastly, in \cite{Ohi_Iguchi_09}, Ohi and Iguchi proved the well-posedness of the internal water waves for one fluid of infinite depth to derive the BO equation. However, in their paper, the existence time is of order $\mathcal{O}(\mathrm{bo}^{-\frac{1}{2}})$, which is too short to justify the BO equation on the physical time scale. The technique is based on the one of Wu \cite{Wu_97}, where the reference model is given in holomorphic coordinates. We will instead use a version of the Zakharov-Craig-Sulem formulation and closely follow the work of Lannes \cite{LannesTwoFluid13}. \textit{In particular, the goal of this paper is to prove the long-time existence of the internal water waves equations with one fluid of infinite depth and positive surface tension. Then we show that the difference between two regular solutions of the internal water waves equations and the BO equation, which evolves from the same initial datum, is bounded by $\mathcal{O}(\mu+ \mathrm{bo}^{-1})$ for all $0 \leq t \lesssim \ve^{-1}$ for any $\ve,\mu \in (0,1)$ such that $\ve^2 \leq \mathrm{bo}^{-1}$}.


  	\subsection{The governing equations} The basis of this study is the Euler equations for an irrotational two-layer fluid written in the Zakharov-Craig-Sulem formulation \cite{Zakharov68,Craig_Sulem_93,Craig_Sulem_Sulem_92}. For the upper layer, we consider the following set of equations
  	\begin{equation}\label{WW - upper fluid}
  		\begin{cases}
  			\partial_t \zeta - \mathcal{G}^-[\zeta]\psi^{-} = 0
  			\\
  			\rho^-\Big{(}\partial_t \psi^{-} +g \zeta + \frac{1}{2}(\partial_x \psi^{-})^2 - \frac{1}{2} \frac{(\mathcal{G}^{-}[\zeta] \psi^{-} +  \partial_x \zeta  \partial_x \psi^{-})^2}{1+   (\partial_x \zeta)^2}\Big{)} = -P^{-}|_{z= \zeta}.
  		\end{cases}
  	\end{equation}
  	Here, the free surface elevation is the graph of $\zeta(t,x) \in \R$, the function $P^{-}|_{z=\epsilon \zeta}$ is the pressure force at the free surface. The function $\psi^{-}(t,x)\in\R$ is the trace at the surface of the velocity potential solving the elliptic problem
  	\begin{equation}\label{Phi-}
  		\begin{cases}
  			(\partial_{x}^2 + \partial_{z}^2)\Phi^- = 0 \quad \text{for} \quad  \Omega^-=\{(x,z) \: : \: z>\zeta\}
  			\\ 
  			\Phi^-|_{z= \zeta} = \psi^-,
  		\end{cases}
  	\end{equation}
  	and $\mathcal{G}^{-}$ is the negative Dirichlet-Neumann operator defined by
  	 \begin{equation*}
  		\mathcal{G}^{-}[\zeta] \psi^{-}  =  (\partial_{z} \phi^{-} - \partial_{x} \zeta \partial_{x} \phi^{-})|_{z =  \zeta}.
  	\end{equation*}
  	 For the fluid in the  lower layer, the governing equations are given in terms of $(\zeta,\psi^+)$ and read
  	\begin{equation}\label{WW - lower fluid}
  		\begin{cases}
  			\partial_t \zeta - \mathcal{G}^+[\zeta]\psi^{+} = 0
  			\\
  			\rho^+\Big{(}\partial_t \psi^{+} + g\zeta + \frac{1}{2}(\partial_x \psi^{+})^2 - \frac{1}{2} \frac{(\mathcal{G}^{+}[\zeta] \psi^{+} +  \partial_x \zeta  \partial_x \psi^{+})^2}{1+    (\partial_x \zeta)^2}\Big{)} = -P^{+}|_{z=\zeta},
  		\end{cases}
  	\end{equation}
  	where the elliptic problem in the lower fluid is given by
	\begin{equation}\label{Phi+}
	  	\begin{cases}
	  		(\partial_x^2 + \partial_z^2)\Phi^+ = 0 \quad \text{for} \quad \Omega^+= \{(x,z) \: : \: -H<z<\zeta \}
	  		\\ 
	  		\Phi^+|_{z=  \zeta} = \psi^+ \quad \partial_z \Phi^+ |_{z=-H} = 0,
	  	\end{cases}
	 \end{equation}
  	 and the positive Dirichlet-Neumann operator is defined by
  	  \begin{equation*}
  			\mathcal{G}^{+}[\zeta] \psi^{+}  =  (\partial_z \phi^{+} - \partial_x \zeta \partial_x \phi^{+})|_{z =  \zeta}.
  	  \end{equation*} 
  	 To ease the notation, we make the following simplifications
  	 \begin{equation*}
  	 	\gamma = \frac{\rho^-}{\rho^+}, \quad \rho^-<\rho^+ = 1, \quad g =1.
  	 \end{equation*}
   	 Moreover, we recall that  the difference in pressure at the interface is proportional to the mean curvature of the interface:
   	 \begin{equation*}
   	 	P^+|_{z= \zeta}-P^-|_{z = \zeta} = \sigma\kappa(\zeta),
   	 \end{equation*}
   	 where $\sigma\in (0,1)$ is the surface tension parameter and $\kappa(\zeta)$ is defined by
   	 \begin{equation}\label{kappa}
   	 	\kappa(\zeta) = - \partial_x\Big{(} \frac{ \partial_x\zeta}{\sqrt{1+ (\partial_x\zeta)^2}}\Big{)}.
   	 \end{equation}

  	 We will now collect all these equations into one system, where we reduce the number of unknowns by using the first equation in \eqref{WW - lower fluid} and \eqref{WW - upper fluid} to see that
  	 \begin{equation}\label{Relation Gpm}
  	 	\mathcal{G}^-[\zeta]\psi^{-} = \mathcal{G}^+[\zeta]\psi^+.
  	 \end{equation}
  	 In particular, we will prove later that we can write $\psi^-$ as a function of $\psi^+$ through the inverse relation
  	  \begin{equation}\label{psi-}
  	 	\psi^{-} = (\mathcal{G}^-[\zeta])^{-1}\mathcal{G}^+[\zeta]\psi^+.
  	 \end{equation}
	Then, we can define the new variable $\psi$ by the formula
   	\begin{align*}
   		\psi 
   		& = \psi^+ - \gamma \psi^{-} 
   		\\
   		& = \big{(}1-\gamma (\mathcal{G}^-[\zeta])^{-1}\mathcal{G}^+[\zeta]\big{)}\psi^+
   		\\ 
   		& = 
   		\mathcal{J}[\zeta]\psi^+.
   	\end{align*}
  	The unknowns $\zeta$ and $\psi$ are the primary variables. We follow the work of Lannes \cite{LannesTwoFluid13} to show that we can use them to recover the velocity potentials $\Phi^{\pm}$ through the transmission problem:
  	\begin{equation}\label{transition Phi}
  		\begin{cases}
  			\Delta_{x,z}\Phi^{\pm} = 0 \hspace{0.5cm}\qquad \text{in} \quad \Omega^{\pm}
  			\\
  			\Phi^{+}|_{z=\zeta}  - \gamma \Phi^-|_{z=\zeta} = \psi
  			\\
  			\partial_n\Phi^-|_{{z=\zeta}}= \partial_n\Phi^+|_{{z=\zeta}},
  			\quad
  			\partial_z \Phi^+|_{z=-H}  = 0,
  		\end{cases}
  	\end{equation}
  	with $\psi^{\pm} = \Phi^{\pm}|_{z=\zeta}$ and the normal condition on $z=\zeta$ is the same as \eqref{Relation Gpm} where $\partial_n$ stands for the upwards normal derivative. From these relations, it will be possible to reduce the two-fluid equations into a set of equations defined by $\zeta$ and $\psi$ where we formally define a new Dirichlet-Neumann operator that links the two fluids through the relation,
  	  \begin{equation}\label{Main G}
  	  	\mathcal{G}[\zeta] = \mathcal{G}^+[\zeta](\mathcal{J}[\zeta])^{-1}.
  	  \end{equation}
  	 From the above expressions, we have the main governing equations (in dimensional form) that we will study throughout this paper:
  	 \begin{equation}\label{IWW dim}
  	 	\begin{cases}
  	 		\partial_t \zeta - \mathcal{G}[\zeta] \psi = 0
  	 		\\
  	 		\partial_t \psi + (1-\gamma)\zeta 
  	 		+
  	 		\frac{1}{2}  
  	 		\big{(}
  	 		 (\partial_x \psi^{+})^2 
  	 		-
  	 		\gamma (\partial_x \psi^{-})^2 
  	 		\big{)}
  	 		+
  	 		\mathcal{N}[\zeta,\psi^{\pm}]
  	 		= - \sigma \kappa( \zeta),
  	 	\end{cases}
  	 \end{equation}
  	 where
  	 \begin{equation*}
  	 	\mathcal{N}[\zeta,\psi^{\pm}] = \frac{\gamma (\mathcal{G}^{-}[\zeta] \psi^{-} +  \partial_x \zeta  \partial_x \psi^{-})^2 -(\mathcal{G}^{+}[\zeta] \psi^{+} +  \partial_x \zeta  \partial_x \psi^{+})^2 }{2(1+    (\partial_x \zeta)^2)}.
  	 \end{equation*}

  	 \subsubsection{Nondimensionalization of the equations}\label{SubSec nondim} To derive an asymptotic model from \eqref{IWW dim}, we will compare every variable and function with physical characteristic parameters of the same dimension $H, a$, or $\lambda$. Since the BO equation describes long waves, it is natural to consider the scaling:
  	 \begin{equation*}
  	 	x = \lambda x', \quad \zeta = a \zeta',
  	 \end{equation*}
    where the prime notation denotes a nondimensional quantity. To identify the remaining variables $\psi'$ and $t = (1-\gamma)^{\frac{1}{2}} \frac{\lambda}{c_{\mathrm{ref}}}t'$ one needs information on the reference velocity $c_{\mathrm{ref}}$. To do so, we look at the linearized equations (with $\sigma=0$):
   	\begin{equation}\label{Linear eq.}
   			\begin{cases}
   			\partial_t \zeta - \mathcal{G}[0]\psi = 0
   			\\
   			\partial_t \psi + (1-\gamma)\zeta = 0,
   		\end{cases}
   	\end{equation}
   	where $\mathcal{G}[0]$ is a Fourier multiplier given by\footnote{See Remark \ref{Rmrk G[0]}.}
   	\begin{equation*}
   		\mathcal{G}[0]\psi(x) = \mathcal{F}^{-1}\Big{(} |\xi| \frac{\mathrm{tanh}(H |\xi|)}{1+\gamma \mathrm{tanh}(H|\xi|)} \hat{\psi}(\xi)\Big{)}(x).
   	\end{equation*}
   	For a wave with typical wavelength $\lambda$, the frequencies are concentrated around $|\xi| = \frac{2\pi}{\lambda}$. Therefore, if we suppose that the depth of lower fluid is small compared to the wavelength, then we have by a Taylor expansion that
   	 $$\mathcal{G}[0]\psi  = -H\partial_x^2 \psi,$$
  	 up to higher order terms in $\mu$. From this simplification we can reduce \eqref{Linear eq.} to a wave equation where we make the identification $c_{\mathrm{ref}}^2= H(1-\gamma)$, and from the second equation we find the dimensions of $\psi$:\footnote{The purpose of $(1-\gamma)^{\frac{1}{2}}$ is to replace $(1-\gamma)\zeta$ by $\zeta$ in \eqref{Linear eq.}.}
  	 \begin{equation*}
  	 	\psi = \frac{a \lambda(1-\gamma)^{\frac{1}{2}}}{\sqrt{H}}\psi'. 
  	 \end{equation*}
   	 Lastly, we also choose to scale the transverse variable with $H$ (i.e., $z = H z'$) to have a reference domain in the lower fluid of unitary depth. Then applying these changes of variables and dropping the prime notation, we find that the nondimensional internal water waves system \eqref{IWW dim} is given by:
  	  \begin{equation}\label{IWW}
  	 	\begin{cases}
  	 		\partial_t \zeta - \frac{1}{\mu}\mathcal{G}_{\mu}[\ve \zeta]\psi = 0
  	 		\\
  	 		\partial_t \psi + \zeta 
  	 		+
  	 		\frac{1}{2}  
  	 		\big{(}
  	 		\ve (\partial_x \psi^{+})^2 
  	 		-
  	 		\gamma \ve (\partial_x \psi^{-})^2 
  	 		\big{)}
  	 		+
  	 		\ve \mathcal{N}[\ve\zeta,\psi^{\pm}]
  	 		= - \frac{1}{\mathrm{bo}}\frac{1}{\ve \sqrt{\mu}} \kappa(\ve \sqrt{\mu} \zeta),
  	 	\end{cases}
  	 \end{equation}
  	 where
  	 \begin{equation*}
  	 	\mathcal{N}[\ve \zeta,\psi^{\pm}] =  \frac{1}{2\mu} \frac{\gamma (\mathcal{G}^{-}_{\mu}[\ve \zeta] \psi^{-} + \ve \mu \partial_x \zeta  \partial_x \psi^{-})^2 - (\mathcal{G}^{+}_{\mu}[\ve \zeta] \psi^{+} + \ve \mu \partial_x \zeta  \partial_x \psi^{+})^2 }{(1+   \varepsilon^2 \mu (\partial_x \zeta)^2)}.
  	 \end{equation*}
   	The operators $\mathcal{G}_{\mu}^{\pm}[\ve \zeta]$ are defined by 
   	\begin{equation*}
   		\mathcal{G}_{\mu}^{\pm}[\ve \zeta] \psi^{\pm} =  \sqrt{1+\ve^2 (\partial_x \zeta)^2} \partial_n \Phi^{\pm}|_{z = \ve \zeta}, 
   	\end{equation*}
   	through the solutions of the scaled Laplace equations:
   	\begin{equation*}
   		\begin{cases}
   			(\mu\partial_x^2 + \partial_z^2)\Phi^{+} = 0 \quad \text{for} \quad \Omega^+=\{(x,z) \: : \: -1<z<\ve \zeta\}
   			\\ 
   			\Phi^+|_{z=  \ve \zeta} = \psi^{+} \quad \partial_z \Phi^+ |_{z=-1} = 0,
   		\end{cases}
   	\end{equation*}
  	and
  	\begin{equation*} 
  		\begin{cases}
  			(\mu \partial_{x}^2 + \partial_{z}^2)\Phi^- = 0 \quad \text{for} \quad   \Omega^-=\{(x,z) \: : \: z>\ve \zeta\}
  			\\ 
  			\Phi^-|_{z= \ve \zeta} = \psi^-.
  		\end{cases}
  	\end{equation*}

  	\subsection{Main results} In this paper we  will first prove the well-posedness of \eqref{IWW} on a time scale $\mathcal{O}(\frac{1}{\ve})$ for $\ve \lesssim \sqrt{\mu}$ and $\mathrm{bo}^{-1} = \ve \sqrt{\mu}$. To state this result, there are two fundamental assumptions that we need to make. 
  	\begin{Def}[Non-cavitation condition] Let $\ve \in (0,1)$, $s >\frac{1}{2}$ and take $\zeta_0 \in H^{s}(\R)$.  We say $\zeta_0$ satisfies the \lq\lq non-cavitation condition\rq\rq\: if there exist $h_{\min}\in(0,1)$ such that 
  		\begin{equation}\label{non-cavitation}
  			h = 1+\ve\zeta_0(x)  \geq h_{\min}, \quad \text{for all} \: \: \: x\in \mathbb{R}.
  		\end{equation}  
  	\end{Def}
  	
  	The second condition is to ensure the solutions do not break down due to Kelvin-Helmholtz instabilities and is key to showing the long-time existence for solutions of \eqref{IWW}. The criterion is enforced for data in the energy space, which we will now define.
  	\begin{Def}[Energy space] Let $\ve, \mu, \mathrm{bo}^{-1} \in (0,1)$ and $N \in \N$. Then we define the function space $H^{N+1}_{\mathrm{bo}}(\R)$ by
  		\begin{equation*}
  			H^{N+1}_{\mathrm{bo} }(\R) = H^{N+1}(\R),
  		\end{equation*}
  		with norm
  		\begin{equation*}
  			|u|_{H^{N+1}_{\mathrm{bo} }}^2  = |u|_{H^N}^2 + \mathrm{bo}^{-1}  |\partial_xu|^2_{H^{N}}.
  		\end{equation*}
  		We define $\dot{H}_{\mu}^{s+\frac{1}{2}}(\R)$ as a Beppo-Levi  space 
  		\begin{equation*}
  			\dot{H}_{\mu}^{s+\frac{1}{2}}(\R) = \dot{H}^{s+\frac{1}{2}}(\R) =\{u \in L^2_{loc}(\R) \: : \: \partial_x u \in H^{s-\frac{1}{2}}(\R) \},
  		\end{equation*}
  		endowed with
  		\begin{equation*}
  			|u|_{\dot{H}_{\mu}^{s+\frac{1}{2}}}  = \Big{|}\frac{|\mathrm{D}|}{(1+\sqrt{\mu}|\mathrm{D}|)^{\frac{1}{2}}}u \Big{|}_{H^{s}}.
  		\end{equation*}
  		Moreover, let $\alpha \in \N^2$ and define the \lq\lq good unknowns\rq\rq \: by 
  		\begin{equation*}
  			\zeta_{(\alpha)} = \partial_{x,t}^{\alpha} \zeta,
  			\quad
  			\psi_{(\alpha)} = \partial_{x,t}^{\alpha}\psi - \ve \underline{w} \partial_{x,t}^{\alpha}\zeta.
  		\end{equation*}
  		Then the natural energy space $\mathscr{E}_{\mathrm{bo},T}^N$ associated to \eqref{IWW}  is defined for functions $\mathbf{U}= (\zeta_{(\alpha)}, \psi_{(\alpha)})$ in 
  		\begin{equation}\label{Energy space}
  			\mathscr{E}_{\mathrm{bo},T}^N = \{\mathbf{U} \in C([0,T];H^N(\R)\times\dot{H}^{t_0+3}(\R), \quad \sup \limits_{t\in [0,T]} 	\mathcal{E}^{N,t_0}_{\mathrm{bo},\mu}(\mathbf{U}(t))<\infty\},
  		\end{equation}
  		where
  		%
  		%
  		%
  		%
  		%
  		%
  		%
  		%
  		%
  		\begin{equation}\label{Energy functional N}
  			\mathcal{E}^{N,t_0}_{\mathrm{bo},\mu}(\mathbf{U}) = |\partial_x\psi|_{H^{t_0+2}}^2 + \sum \limits_{\alpha\in\N^2, |\alpha|\leq N} |\zeta_{(\alpha)}|_{H^1_{\mathrm{bo}}}^2 + |\psi_{(\alpha)}|^2_{\dot{H}^{\frac{1}{2}}_{\mu}}.
  		\end{equation}
  	\end{Def}
  	\begin{remark}
  		Here, the energy depends on both time derivatives and spatial derivatives. This is the method put forward by \cite{RoussetTzvetkoz10,RoussetTzvetkoz11,MingRoissetTzvetkov15} to control the surface tension term for the water wave equations (see Remark \ref{Remark on the energy} for the specifics on this point). This method was later used for the internal water waves equations with surface tension in the case of two fluids of finite depth \cite{LannesTwoFluid13}, which is one of the primary references of this paper.
  	\end{remark}

	  \begin{Def}[Stability criterion]\label{Def stability crit} Let $\mathbf{U}_0 = (\zeta_0,\psi_0) \in L^2(\R)\times \dot{H}_{\mu}^{\frac{1}{2}}(\R)$ and $ \mathcal{E}^{N,t_0}_{\mathrm{bo},\mu}(\mathbf{U}_0)<\infty$. Then we define the \lq\lq stability criterion\rq\rq \:  by 
	  	\begin{equation}\label{Stability criterion}
	  		\qquad 0 <  \mathfrak{d}(\mathbf{U}) : = \inf \limits_{\R} \mathfrak{a}- \Upsilon \mathfrak{c} (\zeta)   |[\![ \underline{V}^{\pm}]\!]|_{L^{\infty}}^4, \quad \text{at} \:\: \: t=0,
	  	\end{equation}
	  	where  
	  	\begin{equation*}
	  		\Upsilon = \frac{\mathrm{bo}}{2} (1-\gamma)^2\gamma^2 \ve^4\mu,
	  	\end{equation*}
	  	and
	  	\begin{align*}
	  		\mathfrak{a}
	  		& =  \Big{(}1
	  		+\ve \big{(} (\partial_t + \ve \underline{V}^{+}\partial_x)\underline{w}^{+} 
	  		-
	  		\gamma  (\partial_t + \ve \underline{V}^{-}\partial_x)\underline{w}^{-}
	  		\big{)}\Big{)}
	  		\\
	  		\mathfrak{e} (\zeta)
	  		& = \sup\limits_{f \in H^{\frac{1}{2}}(\R), f\neq 0} \mu \frac{\big{(}(\mathcal{J}_{\mu}[\ve \zeta])^{-1}(\mathcal{G}^-_{\mu}[\ve \zeta])^{-1} \partial_x f, \partial_x f \big{)}_{L^2}}{|1+\sqrt{\mu}|\mathrm{D}|^{\frac{1}{2}}f|_{L^2}^2}
	  		\\
	  		\mathfrak{c}(\zeta) & = 
	  		\mathfrak{e}(\zeta)^2(1+\ve^2 \mu |\partial_x \zeta|_{L^{\infty}}^2)^{\frac{3}{2}}.
	  	\end{align*}
	  	The quantities $\underline{V}^{\pm}$, $\underline{w}^{\pm}$ describe the horizontal and vertical velocity field in the fluids and are given in Definition \ref{Def op}. See also Corollary \ref{cor definitions} in the Appendix, where they are given in terms of $\zeta$ and $\psi$. While $\mathfrak{e} (\zeta)$ is defined in \eqref{def e} such that it is uniform with respect to $\mu$, see Proposition \ref{Prop e}.
	  \end{Def}


\begin{remark}\label{Remark on dynamics}
	The stability criterion \eqref{Stability criterion} can be seen as a two-layer generalization of the Rayleigh-Taylor criterion where
	$$\mathfrak{a} =
		-\big{(}\partial_z P^+ - \gamma \partial_z P^- \big{)}|_{z = \ve \zeta}>\Upsilon \mathfrak{c} (\zeta)  |[\![ \underline{V}^{\pm}]\!]|_{L^{\infty}}^4.
	$$ 
	We will let $\ve^2\mathrm{bo}$ to be of order one. In this case, the size of the quantity $\Upsilon$ is of order $\mathcal{O}(\ve^2\mu)$ and is neglected in the BO regime and the models considered in this paper. 
	
\end{remark}
  
	\begin{remark}
		One key difference with the work of Lannes \cite{LannesTwoFluid13} for the internal water waves on finite depth is in the symbolic analysis of $\mathcal{G}^-_{\mu}[\ve \zeta]$. The operator depends on the solution of an elliptic problem on a domain with infinite depth. This alters the functional setting, where we also need precise estimates depending on the parameters $\ve, \mu \in (0,1)$.
	\end{remark} 

  	\begin{thm}\label{Thm 1}
  		Let $t_0 = 1$, $N \geq 5$, $\ve, \mu, \gamma, \mathrm{bo}^{-1} \in (0,1)$ such that $\ve^2 \leq \mathrm{bo}^{-1}$. Assume that  $\mathbf{U}_0 = (\zeta_0,\psi_0)^T \in L^2(\R)\times \dot{H}_{\mu}^{\frac{1}{2}}(\R)$ such that $\mathcal{E}^{N,t_0}_{\mathrm{bo},\mu}(\mathbf{U}_0)< \infty$. Suppose further that $\mathbf{U}_0$  satisfies the non-cavitation condition \eqref{non-cavitation} and the stability criterion \eqref{Stability criterion}. Then for nondecreasing function of its argument $C = C(\mathcal{E}^N(\mathbf{U}_0), {h^{-1}_{\min}}, \mathfrak{d}(\mathbf{U}_0)^{-1})>0$  there exist a time
  		$$T = C^{-1},$$ 
  		and a unique solution $\mathbf{U}= (\zeta, \psi)^T \in \mathscr{E}_{\mathrm{bo},\ve^{-1}T}^N$ of \eqref{IWW}. Moreover, the solution satisfies
  		\begin{equation}\label{bound on sol E}
  			\sup \limits_{t \in [0, \ve^{-1}T]} \mathcal{E}^{N,t_0}_{\mathrm{bo},\mu}(\mathbf{U}) \leq C	\mathcal{E}^{N,t_0}_{\mathrm{bo},\mu}(\mathbf{U}_0).
  		\end{equation}
  	\end{thm}
  	
  	\begin{remark}
  		For notational convenience  we shall write $\mathcal{E}^{N}(\mathbf{U})$ instead of $\mathcal{E}^{N,t_0}_{\mathrm{bo},\mu}(\mathbf{U})$. We also consider the case of one horizontal dimension since our primary goal is to justify the BO equation, which is a model that does not include transverse effects. We will deal with the higher dimensional case in a forthcoming paper.
  	\end{remark}
  	
  	\begin{remark}
  		The local well-posedness of \eqref{IWW} was first proved by Ohi and Iguchi \cite{Ohi_Iguchi_09}. However, in their paper, the existence time is of order $\mathcal{O}(\mathrm{bo}^{-\frac{1}{2}})$, which is far too short to justify the BO equation on the physical time scale. In fact, Theorem \ref{Thm 1} is the first proof of the long-time well-posedness of the internal water waves in the case where one layer is of infinite depth. 
  	\end{remark}
  
  \begin{remark}
  	The surface tension term is regularizing and plays a fundamental role in the well-posedness of \eqref{IWW}. However, as noted in Remark \ref{Remark on dynamics}, it does not affect the dynamics of the BO equation.
  \end{remark}
	\begin{remark}
		It is also possible to let the top fluid layer be of finite depth and let the lower fluid layer be of infinite depth. The crucial point is to have a lower density in the top layer and the choice of the order of composition for Dirichlet-Neumann operators appearing in \eqref{psi-}. In particular, if $\mathcal{G}^{+}_{\mu}[\ve \zeta]$ corresponds to the operator for infinite depth, then one would need to simply change the order of composition in the proof (see Remark \ref{order of comp}).
	\end{remark}

  Having defined a solution of the reference model \eqref{IWW} on a long time scale, the next step is to derive the asymptotic models. Here, we follow the road map in \cite{Bona_Saut_Lannes_08}, where they derived several internal water wave models in finite depth and gave comments on the formal derivation of the BO equation. In particular, it is convenient to write \eqref{IWW} in terms of $\psi^+\in \dot{H}^{\frac{3}{2}}_{\mu}(\R)$ through the interface operator:
  \begin{equation}\label{interface op}
  	\mathbf{H}_{\mu}[\ve \zeta] \psi^+ = \partial_x \psi^-  \in H^{\frac{1}{2}}(\R),
  \end{equation}
  where $\psi^- = \Phi^-|_{z= \ve\zeta} \in \mathring{H}^{\frac{3}{2}}(\R)$ and $\Phi^-\in \dot{H}^2(\Omega^-)$ is the unique solution\footnote{See Proposition \ref{Inverse 2}.} (up to a constant) of
  \begin{equation}\label{For exp phi-}
  	\begin{cases}
  		(\mu \partial_x^2 + \partial_z^2) \Phi^- = 0 \hspace{3.35cm}\qquad \text{in} \quad \Omega^-
  		\\
  		\partial_{n} \Phi^- =  (1+\ve^2(\partial_x \zeta)^2)^{-\frac{1}{2}}\mathcal{G}^+_{\mu}[\ve \zeta]\psi^+ \hspace{1.25cm} \text{on} \quad  z = \ve \zeta.
  	\end{cases}
  \end{equation}
  Then we may define the the velocity variable 
  \begin{align}\label{def V}
  	v
  	& =
  	\partial_x\psi 
  	\\\notag
  	& = \partial_x \psi^+ -\gamma \mathbf{H}_{\mu}[\ve\zeta]\psi^+,
  \end{align}
  and apply a derivative to the second equation of \eqref{IWW} to find that
  \begin{equation}\label{IWW - v}
  	\begin{cases}
  		\partial_t \zeta - \frac{1}{\mu}\mathcal{G}^+_{\mu}[\ve \zeta]\psi^+ = 0
  		\\
  		\partial_t v + \partial_x \zeta 
  		+
  		\frac{\ve}{2}  \partial_x
  		\big{(}
  		(\partial_x \psi^{+})^2 
  		-
  		\gamma (\mathbf{H}_{\mu}[\ve\zeta]\psi^+)^2 
  		\big{)}
  		+
  		\ve \partial_x  \mathcal{N}[\ve\zeta,\psi^{\pm}]
  		= - \frac{1}{\mathrm{bo}}\frac{1}{\ve \sqrt{\mu}} \partial_x \kappa(\ve \sqrt{\mu} \zeta).
  	\end{cases}
  \end{equation}
  %
  %
  %
  %
  %
  %
  %
  %
  %
  	%


  As noted in the introduction, we will first derive a system from \eqref{IWW - v}, 
 where we will show that a solution of the internal water waves equations, with regular data $(\zeta_0,v_0)$, solves a weakly dispersive Benjamin system. In particular, we define the Fourier multipliers
  \begin{equation}\label{t, k}
  	\mathrm{t}(\mathrm{D}) = \frac{1}{1+\gamma \tanh(\sqrt{\mu}|\mathrm{D}|)}\frac{\tanh(\sqrt{\mu}|\mathrm{D}|)}{\sqrt{\mu}|\mathrm{D}|}, \quad \mathrm{k}(\mathrm{D}) = \mathrm{t}(\mathrm{D})\big{(}1 - \mathrm{bo}^{-1} \partial_x^2\big{)},
  \end{equation}
  and the function
  $$u = \mathrm{t}(\mathrm{D})v.$$
  Then we will show that the solution of the internal water waves equations satisfies
  \begin{equation}\label{System Full disp B}
  	\begin{cases}
  		\partial_t \zeta
  		+
  		\partial_x u +\ve \mathrm{t}(\mathrm{D})\partial_x(\zeta u) = 0
  		\\
  		\partial_t u
  		+
  		\mathrm{k}(\mathrm{D})\partial_x \zeta 
  		+
  		\frac{\ve}{2} \mathrm{t}(\mathrm{D})\partial_x (u)^2 =   0,
  	\end{cases}
  \end{equation}
  up to an error of order $\mathcal{O}( \ve \sqrt{\mu})$, where give precise Sobolev bounds on the error in terms of the initial data. Then, under an additional assumption on the data (for right-moving waves), we will show that we can approximate this  system with the solutions of a weakly dispersive version of the Benjamin equation
  \begin{equation}\label{Full disp Benjamin}
  	\partial_t \zeta
  	+
 	 \sqrt{k}(\mathrm{D})\partial_x \zeta
  	+
  	\frac{3\ve }{2}\zeta \partial_x\zeta =  0,
  \end{equation}
   up to an error of order $\mathcal{O}( \ve^2 + \ve \mathrm{bo}^{-1}+\ve \sqrt{\mu})$. Then if we simplify the multiplier $\sqrt{k}(\mathrm{D})$ we will show that one get  the Benjamin equation
   \begin{equation*}
	   	\partial_t \zeta
	   	+
	   	\big{(} 1- \frac{\gamma}{2} \sqrt{\mu}|\mathrm{D}| - \frac{\partial_x^2}{2\mathrm{bo}}\big{)}\partial_x \zeta
	   	+
	   	\frac{3\ve }{2}\zeta \partial_x\zeta =  0,
   \end{equation*}
    up to an error $\mathcal{O}( \ve^2 + \ve \mathrm{bo}^{-1}  + \sqrt{\mu} \mathrm{bo}^{-1}+\mu )$. Moreover,  if we neglect the effect of surface tension we obtain the Benjamin-Ono equation.

   \begin{remark}
   	Since we need to prove a consistency result  to link the internal water waves equations with the BO equation, we also use it to derive some new models. These models have the added benefit of being fully dispersive; meaning they are exact at the linear level. Several extensions could be made, for instance, derivation of Green-Naghdi type models \cite{Choi_Camassa_96}, counterpropagating BO equations \cite{Bambusi_Paleari_24}, or improving the precision of \eqref{Full disp Benjamin} by following the analysis of \cite{Emerald21b}. 
   \end{remark}
  \begin{remark}
  	An alternative approach would be to rigorously derive the regularized BO system given by:
  	\begin{equation}\label{reg Bo}
  		\begin{cases}
  			\big{(}1 + \alpha\sqrt{\mu} \gamma|\mathrm{D}| \big{)}\partial_t \zeta
  			+
  			\big{(}
  			1
  			+
  			(\alpha - 1) \gamma\sqrt{\mu}|\mathrm{D}|\big{)}\partial_x v  +\ve  \partial_x(\zeta v) = 0
  			\\
  			\partial_t v 
  			+
  			\partial_x \zeta 
  			+
  			\ve v\partial_xv = 0,
  		\end{cases}
  	\end{equation}
  	for $\alpha \geq 0$. This is the model that was formally derived in \cite{Bona_Saut_Lannes_08}, and moreover we can use it to derive the \lq\lq regularized Benjamin-Ono equation\rq\rq \:  \cite{Craig_Sulem_Sulem_92}: 
  	\begin{equation*}
  		\big{(}1 + \alpha\sqrt{\mu} \gamma|\mathrm{D}| \big{)}\partial_t \zeta
  		+
  		\partial_x \zeta +
  		(2\alpha - 1)\frac{\gamma}{2}\sqrt{\mu}|\mathrm{D}|\partial_x \zeta 
  		+
  		\frac{3\ve }{2}\zeta \partial_x\zeta = 0.
  	\end{equation*}
  	The rigorous derivation of these models is straightforward when having Theorem \ref{Thm 1} and a consistency result, if we combine it with the long-time well-posedness of \eqref{reg Bo} provided by Xu in \cite{Xu_12} for  $\alpha > 1$. 
  \end{remark}

  Before we proceed, we   establish the long-time well-posedness of the weakly dispersive system introduced above. To do so, we must define the energy space associated with \eqref{System Full disp B}.
  \begin{Def}\label{Energy space X} Let $s\geq 0$ and $\mu, \mathrm{bo}^{-1} \in (0,1)$. Then we define the norm on the function space $X^s_{\mu, \mathrm{bo}}(\R)$ to be
  	\begin{equation*}
  		|(\zeta,u)|_{X^s_{\mu,\mathrm{bo}}}^2 : = |\zeta|_{H^{s+1}_{\mathrm{bo}}}^2 +|u|_{H^s}^2+\sqrt{\mu}||\mathrm{D}|^{\frac{1}{2}}u |_{H^s}^2.
  	\end{equation*}
  \end{Def}

  \begin{thm}\label{W-P System BO} Let $\ve, \mu,\gamma, \mathrm{bo}^{-1} \in (0,1)$ and $s> \frac{3}{2}$. Assume that $(\zeta_0,u_0) \in X^s_{\mu, \mathrm{bo}}(\R)$. Then there exists a  $C = C(h_{\min}^{-1},|(\zeta_0,u_0)|_{X^s_{\mu,\mathrm{bo}}})>0$  nodecreasing function of its argument and a time
  $$T = C^{-1},$$ 
  such that there exist a unique solution $(\zeta^{\text{\tiny wBs}}, u^{\text{\tiny wBs}}) \in C([0,\ve^{-1}T] :  X^s_{\mu, \mathrm{bo}}(\R)\big{)}$ to the weakly dispersive Benjamin system \eqref{System Full disp B}  that satisfies
  \begin{equation}\label{Bound sol Full disp B system}
  	\sup\limits_{t \in [0,\ve^{-1}T]} |(\zeta^{\text{\tiny wBs}},u^{\text{\tiny wBs}})|_{X^s_{\mu, \mathrm{bo}}} \leq C | (\zeta_0,u_0)|_{X^s_{\mu, \mathrm{bo}}}.
  \end{equation}
  \end{thm}

 \begin{remark}
 	System \eqref{System Full disp B} and \eqref{Full disp Benjamin} are new and are chosen such that it is easy to deduce the long-time existence. The choice was based on the observations made in \cite{Paulsen22}, where weakly dispersive shallow water models are considered and can, in some cases, give rise to well-posed systems while their strongly dispersive versions are not. 
 \end{remark}
  
    \begin{remark}\label{Remark wp}
  	For the data $\zeta_0 \in H^s(\R)$ with $s>\frac{3}{2}$ the long-time (global)  well-posedness of the Benjamin equation and the Benjamin-Ono equation is classical. The long time well-posedness can easily be extended for \eqref{Full disp Benjamin} with $\ve \in (0,1)$. See, for instance, \cite{LinaresPonce15} in the case of the Benjamin-Ono equation on a fixed time with $\ve = 1$. 
  \end{remark}

  With this result in hand, we can use it as a link to prove the consistency between the BO equation and the internal water waves equations.

 \begin{thm}\label{Consistency}\label{Thm: Consistency} Let $\ve, \mu, \gamma, \mathrm{bo}^{-1}  \in (0,1)$, such that $\ve^2 \leq \mathrm{bo}^{-1}$. Assume that   $\mathbf{U}_0 = (\zeta_0,\psi_0)^T$ satisfies the assumptions of Theorem \ref{Thm 1} and define $v_0 = \partial_x \psi_0$.  Then for $v = \partial_x \psi$ there exists $C=C(\mathcal{E}^N(\mathbf{U}_0), {h^{-1}_{\min}}, \mathfrak{d}(\mathbf{U}_0)^{-1})>0$ nondecreasing function of its argument and  a time $T=C^{-1}$ such that
 	\begin{equation}\label{Est on Sol}
 		\sup_{t\in [0,\ve^{-1}T] }\big{(}|\zeta|_{H^{N+1}_{\mathrm{bo}}}^2 + |v|_{H^{N-\frac{1}{2}}}^2\big{)} \leq  C.
 	\end{equation}	
 	Moreover, there exists a generic function $R$ satisfying 
 	$$|R|_{H^{N-5}}\leq C,$$ 
 	for all $t\in [0,\varepsilon^{-1}T]$ such that we have the following results:
 	\begin{itemize}
 		\item [1.] There exists a unique solution $(\zeta, v)^T \in C\big{(}[0,\ve^{-1}T] \: : \: H^{N+1}_{\mathrm{bo}}(\R) \times H^{N-\frac{1}{2}}(\R) \big{)}$ to \eqref{IWW - v}, where $v = \partial_x \psi$. Moreover, on the same time interval, for 
 		\begin{equation*}
 			\mathrm{t}(\mathrm{D}) = \frac{1}{1+\gamma \tanh(\sqrt{\mu}|\mathrm{D}|)}\frac{\tanh(\sqrt{\mu}|\mathrm{D}|)}{\sqrt{\mu}|\mathrm{D}|}, \quad \mathrm{k}(\mathrm{D}) = \mathrm{t}(\mathrm{D})\big{(}1 - \mathrm{bo}^{-1} \partial_x^2\big{)},
 		\end{equation*}
 		and $u = \mathrm{t}(\mathrm{D})v$ we have that the solution also satisfies
 		\begin{equation*}\label{System Full disp BO}
	 		\begin{cases}
	 			\partial_t \zeta
	 			+
	 			\partial_x u +\ve \mathrm{t}(\mathrm{D})\partial_x(\zeta u) = \ve \sqrt{\mu}R
	 			\\
	 			\partial_t u
	 			+
	 			\mathrm{k}(\mathrm{D})\partial_x \zeta 
	 			+
	 			\frac{\ve}{2} \mathrm{t}(\mathrm{D})\partial_x (u)^2 =   \ve \sqrt{\mu}R.
	 		\end{cases}
 		\end{equation*} \\

 		\item [2.] 

 		There exist a unique solution $\zeta^{\text{\tiny wB}}\in C\big{(}[0,\ve^{-1}T] \: : \:  H^{N+1}_{\mathrm{bo}}(\R) \big{)}$ that the solves weakly dispersive Benjamin equation 
 		\begin{equation*} 
 			\partial_t \zeta^{\text{\tiny wB}}
 			+
 			 \sqrt{k}(\mathrm{D})\partial_x \zeta^{\text{\tiny wB}}
 			+
 			\frac{3\ve }{2}\zeta^{\text{\tiny wB}}  \partial_x\zeta^{\text{\tiny wB}} =  0.
 		\end{equation*}
 		Suppose further that the data $v_0$ is given by
 		\begin{equation}\label{right moving wave}
 			v_0 = \mathrm{t}^{-1}(\mathrm{D})
 			\big{(}\sqrt{	\mathrm{k}}(\mathrm{D})\zeta_0^{\text{\tiny wB}}  - \frac{\ve}{4} (\zeta_0^{\text{\tiny wB}})^2\big{)},
 		\end{equation}
 		and define $v^{\text{\tiny wB}},u^{\text{\tiny wB}}  \in C\big{(}[0,\ve^{-1}T] \: : \:  H^{N-\frac{1}{2}}(\R) \big{)}$ by
 		\begin{equation*}
 			u^{\text{\tiny wB}}=
 			\mathrm{t}(\mathrm{D})v^{\text{\tiny wB}} =  
 			\big{(}\sqrt{	\mathrm{k}}(\mathrm{D})\zeta^{\text{\tiny wB}}  - \frac{\ve}{4} (\zeta^{\text{\tiny wB}})^2\big{)}.
 		\end{equation*}
 		Then the solution $(\zeta^{\text{\tiny wB}},u^{\text{\tiny wB}})$ also satisfies
 		\begin{equation*} 
 			\begin{cases}
 				\partial_t \zeta^{\text{\tiny wB}}
 				+
 				\partial_{ x} u^{\text{\tiny wB}}  +\ve\mathrm{t}(\mathrm{D}) \partial_{x}(\zeta^{\text{\tiny wB}}u^{\text{\tiny wB}}) = \varepsilon( \varepsilon + \sqrt{\mu} + \mathrm{bo}^{-1} )R
 				\\
 				\partial_t u^{\text{\tiny wB}}
 				+
 				\sqrt{	\mathrm{k}}(\mathrm{D})\partial_x  \zeta^{\text{\tiny wB}}
 				+
 				\frac{\ve}{2} \mathrm{t}(\mathrm{D}) \partial_{{x}}\big{(}u^{\text{\tiny wB}}\big{)}^2 = \varepsilon( \varepsilon + \sqrt{\mu} + \mathrm{bo}^{-1} )R.
 			\end{cases}
 		\end{equation*}\\ 
 		\item [3.] There exist a unique solution $\zeta^{\text{\tiny B}} \in C\big{(}[0,\ve^{-1}T] \: : \: H^{N+1}_{\mathrm{bo}}(\R)\big{)}$ that solves the Benjamin equation
 		\begin{equation*}
 			\partial_t \zeta^{\text{\tiny B}} 
 			+
 			\big{(} 1- \frac{\gamma}{2} \sqrt{\mu}|\mathrm{D}| - \frac{\partial_x^2}{2\mathrm{bo}}\big{)}\partial_x \zeta^{\text{\tiny B}} 
 			+
 			\frac{3\ve }{2}\zeta^{\text{\tiny B}}  \partial_x\zeta^{\text{\tiny B}} =  0,
 		\end{equation*}
 		and on the same time interval it satisfies
 		\begin{equation*} 
 			\partial_t \zeta^{\text{\tiny B}}
 			+
 			\sqrt{k}(\mathrm{D})\partial_x \zeta^{\text{\tiny B}}
 			+
 			\frac{3\ve }{2}\zeta^{\text{\tiny B}}  \partial_x\zeta^{\text{\tiny B}} =  (\mu + \sqrt{\mu}\mathrm{bo}^{-1}) R.
 		\end{equation*}\\

 		\item [4.] There exist a unique solution $\zeta^{\text{\tiny BO}} \in C\big{(}[0,\ve^{-1}T] \: : \: H^{N+1}_{\mathrm{bo}}(\R)\big{)}$ that solves
 		\begin{equation*}
 			\partial_t \zeta^{\text{\tiny BO}}
 			+
 			\big{(}1 - \frac{\gamma}{2}\sqrt{\mu}|\mathrm{D}|\big{)}\partial_x \zeta^{\text{\tiny BO}}
 			+
 			\frac{3\ve }{2}\zeta^{\text{\tiny BO}} \partial_x\zeta^{\text{\tiny BO}} =  0,
 		\end{equation*}
 		and on the same time interval it satisfies
 		\begin{equation*} 
 			\partial_t \zeta^{\text{\tiny BO}} 
 			+
 			\big{(} 1- \frac{\gamma}{2} \sqrt{\mu}|\mathrm{D}| - \frac{\partial_x^2}{2\mathrm{bo}}\big{)}\partial_x \zeta^{\text{\tiny BO}} 
 			+
 			\frac{3\ve }{2}\zeta^{\text{\tiny BO}}  \partial_x\zeta^{\text{\tiny BO}} =  \mathrm{bo}^{-1}R.
 		\end{equation*}
 	\end{itemize}
 \end{thm}

\begin{remark}
	The results in Theorem \ref{Thm: Consistency} are stated in the one-dimensional case, but the derivation of \eqref{System Full disp B} can be extended to higher dimensions (see Remark \ref{Rmrk higher d}).
\end{remark}

 %
 %
  %
  %
%
	A consequence of the above results is the full justification of BO and related equations. 


  	\begin{thm}\label{Justification} Let $\ve, \mu, \gamma, \mathrm{bo}^{-1} \in (0,1)$, such that $\ve^2 \leq \mathrm{bo}^{-1}$.  Assume that   $\mathbf{U}_0 = (\zeta_0,\psi_0)^T$ satisfies the assumptions of Theorem \ref{Thm 1} with $N\geq 8$ and define $v = \partial_x \psi$ with  $v_0 = \partial_x \psi_0$. Suppose further that $v_0$ satisfies \eqref{right moving wave}. Then there exists $C=C(\mathcal{E}^N(\mathbf{U}_0), {h^{-1}_{\min}}, \mathfrak{d}(\mathbf{U}_0)^{-1})>0$ nondecreasing function of its argument and  a time $T=C^{-1}$ such that:	\\ 
  		%
  		%
  		%
  		\begin{itemize}
  			\item [1.] There exist a unique solution $(\zeta, v)  \in C\big{(}[0,\ve^{-1}T] \: : \: H^{N+1}_{\mathrm{bo}}(\R) \times H^{N-\frac{1}{2}}(\R) \big{)}$ to \eqref{IWW - v}.\\

  			\item [2.] From the same initial data:
  			\begin{itemize}
  				\item [2.1.] There exists a unique solution $(\zeta^{\text{\tiny wBs}}, u^{\text{\tiny wBs}}) \in C([0,\ve^{-1}T] : X^{N-\frac{1}{2}}_{\mu, \mathrm{bo}}(\R) \big{)}$, where $u^{\text{\tiny wB}} = \mathrm{t}(\mathrm{D})v^{\text{\tiny wB}}$,  to the weakly dispersive Benjamin system
  				%
  				%
  				\begin{equation*}
  					\begin{cases}
  						\partial_t \zeta
  						+
  						 \partial_x u^{\text{\tiny wBs}} +\ve \mathrm{t}(\mathrm{D})\partial_x(\zeta^{\text{\tiny wBs}} u^{\text{\tiny wBs}}) = 0
  						\\
  						\partial_t u^{\text{\tiny wBs}}
  						+
  						\mathrm{k}(\mathrm{D})\partial_x \zeta^{\text{\tiny wBs}} 
  						+
  						\frac{\ve}{2} \mathrm{t}(\mathrm{D})\partial_x( u^{\text{\tiny wBs}})^2 =   0,
  					\end{cases} 
  				\end{equation*} 
  				and for any $t\in [0,\ve^{-1}T]$ there holds,
  				\begin{equation}\label{CV BOs}
  					| (\zeta-\zeta^{\text{\tiny wBs}}, v-v^{\text{\tiny wBs}}) |_{L^{\infty}([0,t]: H^{N-7}\times H^{N-7})} 
  				\leq   \varepsilon  \sqrt{\mu} t C.
  				\end{equation}\\

  				\item [2.2.] There exists a unique solution $\zeta^{\text{\tiny wB}} \in C\big{(}[0,\ve^{-1}T] \: : \: H^{N+1}_{\mathrm{bo}}(\R)  \big{)}$ to the weakly dispersive Benjamin equation
  				\begin{equation*}
  					\partial_t \zeta^{\text{\tiny wB}}
  					+
  					\sqrt{	\mathrm{k}}(\mathrm{D})\partial_x \zeta^{\text{\tiny wB}}
  					+
  					\frac{3\ve }{2}\zeta^{\text{\tiny wB}} \partial_x\zeta^{\text{\tiny wB}} =0
  				\end{equation*}
  				and for any $t\in [0,\ve^{-1}T]$ there holds,
  				\begin{equation}\label{CV wB}
  					|\zeta-\zeta^{\text{\tiny wB}}|_{L^{\infty}([0,t] : H^{N-7})} \leq \varepsilon( \ve+ \sqrt{\mu} +  \mathrm{bo}^{-1}) tC.
  				\end{equation}\\

  				\item [2.3.] There exists a unique solution $\zeta^{\text{\tiny B}} \in C\big{(}[0,\ve^{-1}T] \: : \: H^{N+1}_{\mathrm{bo}}(\R) \big{)}$ to the Benjamin equation
  				\begin{equation*} 
	  				\partial_t \zeta^{\text{\tiny B}} 
	  				+
	  				\big{(} 1- \frac{\gamma}{2} \sqrt{\mu}|\mathrm{D}| - \frac{\partial_x^2}{2\mathrm{bo}}\big{)}\partial_x \zeta^{\text{\tiny B}} 
	  				+
	  				\frac{3\ve }{2}\zeta^{\text{\tiny B}}  \partial_x\zeta^{\text{\tiny B}} = 0,
  				\end{equation*}
  				and for any $t\in [0,\ve^{-1}T]$ there holds,
  				\begin{equation}\label{CV B}
  					|\zeta-\zeta^{\text{\tiny B}}|_{L^{\infty}([0,t] : H^{N-7})} \leq \big{(} \ve^2 + \ve \mathrm{bo}^{-1}  + \sqrt{\mu} \mathrm{bo}^{-1}+\mu \big{)} tC.
  				\end{equation}\\

  				\item [2.4.] There exists a unique solution $\zeta^{\text{\tiny BO}} \in C\big{(}[0,\ve^{-1}T] \: : \:H^{N+1}_{\mathrm{bo}}(\R) \big{)}$ to the BO equation
  				\begin{equation*} 
  					\partial_t \zeta^{\text{\tiny BO}}
  					+
  					c(1 - \frac{\gamma}{2}\sqrt{\mu}|\mathrm{D}|)\partial_x \zeta^{\text{\tiny BO}}
  					+
  					c\frac{3\ve }{2}\zeta^{\text{\tiny BO}} \partial_x\zeta^{\text{\tiny BO}} = 0,
  				\end{equation*}
  				and for any $t\in [0,\ve^{-1}T]$ there holds,
  				\begin{equation}\label{CV BO}
  					|\zeta-\zeta^{\text{\tiny BO}}|_{L^{\infty}([0,t] : H^{N-7})} \leq \big{(} \mu +\ \mathrm{bo}^{-1})tC.
  				\end{equation}
  			\end{itemize}
  			 
  		\end{itemize}

  	\end{thm}

   \color{black}
  	
  	\begin{remark}  \color{white} space \color{black}
  		
  		\begin{itemize}
  			\item Estimates \eqref{CV BOs}, \eqref{CV wB}, and \eqref{CV B} together with the well-posedness theory imply the full justification of their respective systems as internal water waves equations. These are new results, but their primary purpose is to serve as an intermediate step for the derivation of the BO equation and are added here for the sake of completeness.  \\ 
  			
  			\item Several asymptotic regimes are not captured by the result of this paper and the one of Lannes \cite{LannesTwoFluid13}. 
  			An interesting example would be the intermediate long wave regime where one layer is allowed to be larger than the other (see  \cite{Klein_Saut_21,Bona_Saut_Lannes_08,klein2023benjaminrelatedequations}). This will require a new well-posedness result for the internal water waves equations where one would have to track the dependencies in the depths carefully. \\ 
  			
  			\item The Intermediate Long Wave equation (ILW) is a particular deep water model closely related to the BO equation. To be precise, the ILW equation is given by
  			\begin{equation}\label{ilw}
  				\partial_t \zeta
  				+
  				(1 - \frac{\gamma}{2}\sqrt{\mu}\mathcal{L}(\mathrm{D}))\partial_x \zeta 
  				+
  				\frac{3\varepsilon }{2}\zeta \partial_x\zeta =0,
  			\end{equation}
  			where $\mathcal{L}(\mathrm{D}) = |\mathrm{D}|\coth(\sqrt{\mu^-}|\mathrm{D}|)-\frac{1}{\sqrt{\mu^-}}$
  			 and $\sqrt{\mu^-} = \frac{H^-}{\lambda}$, and $H^-$ is the depth of upper fluid domain and $\lambda$ is the typical wavelength. Then as $H^-\rightarrow\infty$, we observe that the symbol associated with $\mathcal{L}(\mathrm{D})$ converges to the Hilbert transform $\mathcal{H}$, whose symbol in frequency reads $-i \text{sign}(\xi)$. In fact, it has been proved that the solutions of \eqref{ilw} converge to the ones of BO, even on $L^2(M)$ for $M = \R$ or $\mathbb{T}$ \cite{chapouto2024deepwaterlimitintermediatelong}. 
  			
  		\end{itemize}
  		
  	\end{remark}

  	From Theorem \ref{Justification}, one can prove as a Corollary, that the ILW equation provides a good approximation of the internal water waves equations if $(\frac{\mu}{\mu^-})^{\frac{1}{2}} + \mu +\mathrm{bo}^{-1}$ is small.

  	\begin{cor}\label{Cor ilw}
		Let $N\geq 8$, and $\ve, \mu, \gamma, (\mu^-)^{-1}, \mathrm{bo}^{-1} \in (0,1)$, such that $\ve^2 \leq \mathrm{bo}^{-1}$. Assume that  $\mathbf{U}_0 = (\zeta_0,\psi_0)^T$ satisfies the assumptions of Theorem \ref{Justification}.  Then the solutions $\zeta, \zeta^{\text{\tiny ILW}} \in C\big{(}[0,\ve^{-1}T] \: : \:H^{N}(\R) \big{)}$ of \eqref{IWW - v} and \eqref{ilw} satisfies the convergence estimate
  		\begin{equation}\label{CV ILW}
  			|\zeta-\zeta^{\text{\tiny ILW}}|_{L^{\infty}([0,t] : H^{N-7})} \leq \big{(} (\frac{\mu}{\mu^-})^{\frac{1}{2}} +\mu + \mathrm{bo}^{-1})tC.
  		\end{equation}\\ 
  	\end{cor}
  
  	\subsubsection{Strategy and outline of the proofs}  The main body of the paper is devoted to the proof of Theorem \ref{Thm 1}, which relies on energy estimates similar to the ones provided by Lannes \cite{LannesTwoFluid13}. To do so, we first need to prove that the operators involved in the main system \eqref{IWW} are well-defined and can be formulated solely in terms of $\zeta$ and $\psi$. We start by studying the operator $\mathcal{G}_{\mu}$ in Section \ref{Properties of G}, which will be given by the expression:
  	\begin{equation}\label{Outline G}
  		\mathcal{G}_{\mu}[\ve \zeta]\psi =  \mathcal{G}^+_{\mu}[\ve \zeta] \Big{(} 1 -   \gamma   (\mathcal{G}^-_{\mu}[\ve \zeta])^{-1} \color{black}\mathcal{G}^+_{\mu}[\ve \zeta] \Big{)}^{-1}\psi.
  	\end{equation}
  	The main difference with the work of Lannes is that we have the composition of two operators $\mathcal{G}^+_{\mu}$ and $\mathcal{G}^-_{\mu}$ that act on different spaces. This is a consequence of having one fluid of finite depth and the other of infinite depth. In particular, to define the composition $(\mathcal{G}_{\mu}^-)^{-1}$  with $\mathcal{G}_{\mu}^+$ we need to work on a homogeneous\footnote{The function space with $|f|_{\mathring{H}^{s+\frac{1}{2}}} = | |\mathrm{D}|^{\frac{1}{2}} f|_{H^s}$.} type  target space. 
  	
  	The next step is to give a symbolic description of the operators involved in the expression of $\mathcal{G}_{\mu}$. This involves some of the key estimates that will be used to close the energy estimates. To do so, we need a symbolic description of each operator involved in \eqref{Outline G}. The key estimate, which is proved in Section \ref{Sec Symbolic} reads,
  	\begin{equation*}
  		|\mathcal{G}^-_{\mu}[\ve \zeta]\psi^- - \big{(}-\sqrt{\mu} |\mathrm{D}|\psi^- \big{)}|_{H^{s+\frac{1}{2}}} \leq \ve \mu C\big{(}|\zeta|_{H^{t_0+3}}\big{)}|\psi^-|_{\mathring{H}^{s+\frac{1}{2}}}.
  	\end{equation*}
	The symbolic approximation of $\mathcal{G}^-_{\mu}$ is well-known; see the collection of work by Lannes, Alazard, Metivier, Burq, and Zuily \cite{FirstWPWW_Lannes_05,AlazardMetivier09,AlazardBurqZuily11, AlazardBurqZuily14} for similar estimates. Their results are sharper in the sense that they require less regularity on $\zeta$. However, the symbolic description they provide is without the parameters $\ve, \mu$, and more importantly, is given on the space $\psi^{-} \in H^{s+\frac{1}{2}}(\R)$. More precisely, the main difference is that we need to account for the small parameters and arrive at the space $\mathring{H}^{s+\frac{1}{2}}(\R)$ to close the estimates later. This result seems new and could be of independent interest. We refer the reader to Remark \ref{Remark dim} for the novelties of the proof and possible extensions.

  	The last step before providing the a priori estimates is the quasilinearization of the internal water wave system. This is done in Section \ref{Quasilinearization}, where we only give detailed proofs of the steps unique to the current setting. In fact, we can reduce some of the proofs to the estimates performed in \cite{LannesTwoFluid13}. We also provide details on the derivation of the stability criterion in Proposition \ref{Prop stability}, where we obtain coercivity type estimates. 	From these results, we  establish a priori bounds on the solution in Section \ref{Energy est IWW}. Then, we use them to prove Theorem \ref{Thm 1} in Section \ref{Pf main thm}.

  	For the remainder of the paper, we give the details on the derivation of the BO equation. We also provide a complete justification of the models that link the BO equation with the internal water waves equations. In Section \ref{Sec Consistency}, we follow the work of Bona, Lannes, and Saut \cite{Bona_Saut_Lannes_08} to derive the intermediate systems provided in Theorem \ref{Thm: Consistency} and the BO equation. Then, in Section \ref{Sec Wp system BO}, we give a short proof of Theorem \ref{W-P System BO}, which gives the long-time well-posedness of a weakly dispersive Benjamin system. Lastly, in Section \ref{Sec just}, we conclude the paper by proving Theorem \ref{Justification}, i.e., the full justification of each of the systems derived in Theorem \ref{Thm: Consistency}, and then a short proof of Corollary \ref{Cor ilw}.

  	\subsection{Definition and notations}
  	\begin{itemize}

		\item We define the gradient by
		\begin{equation*}
			\nabla^{\mu}_{x,z} 
			= 	(\sqrt{\mu} \partial_x, \partial_z)^T
		\end{equation*}
		and we introduce the scaled Laplace operator
		\begin{equation*}
			\Delta^{\mu}_{x,z} = \nabla^{\mu}_{x,z}\cdot \nabla^{\mu}_{x,z} = \mu \partial_x^2 + \partial_z^2.
		\end{equation*}  		
		\item We  let $c$ denote a positive constant independent of $\mu, \ve$, and $\mathrm{bo}$ that may change from line to line. Also, as a shorthand, we use the notation $a \lesssim b$ to mean $a \leq c\: b$.
		\item Let $L^2(\R)$ be the usual space of square integrable functions with norm $|f|_{L^2} = \sqrt{\int_{\R} |f(x)|^2 \: dx}$. Also, for any $f,g \in L^2(\R)$ we denote the scalar product by $\big{(} f,g \big{)}_{L^2} = \int_{\R} f(x) \overline{g(x)} \: dx$.
		\item Let $f:\mathbb{R} \to \mathbb{R}$ be a tempered distribution, let $\hat{f}$ or $\mathcal{F}f$ be its Fourier transform. Let $F:\mathbb{R}\to\mathbb{R}$ be a smooth function. Then the Fourier multiplier associated with $F(\xi)$ is denoted $\mathrm{F}$ and defined by the formula:
		\begin{align*}
			\mathcal{F}\big{(}\mathrm{F}(\mathrm{D})f(x)\big{)}(\xi) = F(\xi)\hat{f}(\xi).
		\end{align*}
		\item For any $s \in \mathbb{R}$ we call the multiplier $ \widehat{|\mathrm{D}|^{s}f} (\xi)= |\xi|^s \hat{f}(\xi)$ the Riesz potential of order $-s$. 
		\item For any $s \in \mathbb{R}$ we call the multiplier $\Lambda^s = (1+\mathrm{D}^2)^{\frac{s}{2}} = \langle \mathrm{D} \rangle^s$ the Bessel potential of order $-s$. 
		\item The Sobolev space $H^s(\mathbb{R})$  is equivalent to the weighted $L^2-$space with $|f|_{H^s} = |\Lambda^s f|_{L^2}$. 
		\item For any $s \geq 0$ we will denote  $\mathring{H}^{s+\frac{1}{2}}(\mathbb{R})$ a homogeneous type Sobolev space of degree $\frac{1}{2}$ with $|f|_{\mathring{H}^{s+\frac{1}{2}}} = | |\mathrm{D}|^{\frac{1}{2}} f|_{H^s}$. One should note that $|\mathrm{D}| = \mathcal{H} \partial_x$, where $\widehat{\mathcal{H}f}(\xi) = -i\: \text{sgn}(\xi)\hat{f}(\xi)$  is the Hilbert transform.
		\item For any $s \geq 0$ we will denote $\dot{H}^{s+1}(\mathbb{R})$ the Beppo-Levi space  with $|f|_{\dot{H}^{s+1}} = |\Lambda^{s}\partial_x f|_{L^2}$. 
		\item For any $s \geq 0$ we will denote $\dot{H}_{\mu}^{s+\frac{1}{2}}(\mathbb{R}) = \dot{H}^{s+\frac{1}{2}}(\mathbb{R})$ with  $|f|_{\dot{H}_{\mu}^{s+\frac{1}{2}}} = |\mathfrak{P}f|_{L^2}$ and where $\mathfrak{P}$ is a Fourier multiplier defined in frequency by:
		\begin{equation*}
			\mathcal{F}(\mathfrak{P}f)(\xi) = \frac{|\xi|}{(1+\sqrt{\mu}|\xi|)^{\frac{1}{2}}} \hat{f}(\xi).
		\end{equation*} 
		\item For any $s \geq 0$, $\mathrm{bo}^{-1} \in (0,1)$ we denote  $H^{s+1}_{\mathrm{bo} }(\R) = H^{s+1}(\R)$ 
		with norm
		\begin{equation*}
			|u|_{H^{s+1}_{ \mathrm{bo} }}^2  = |u|_{H^s}^2 + \mathrm{bo}^{-1}  |\partial_xu|^2_{H^{s}}.
		\end{equation*}
	
		\item  For any $s \geq 0$, $\mu, \mathrm{bo}^{-1} \in (0,1)$ we denote $X^s_{\mu, \mathrm{bo}}(\R) = H^{s+1}_{\mathrm{bo} }(\R)\times H^{s+\frac{1}{2}}(\R)$ with norm
		\begin{equation*}
		|(\zeta,u)|_{X^s_{\mu,\mathrm{bo}}}^2  = |\zeta|_{H^{s+1}_{\mathrm{bo}}}^2 +|u|_{H^s}^2+\sqrt{\mu}||\mathrm{D}|^{\frac{1}{2}}u |_{H^s}^2.
		\end{equation*}

		\color{black}
		%
		%
		%
		%
		%
		%
		\item We say $f$ is a  Schwartz function $\mathscr{S}(\mathbb{R})$, if $f \in C^{\infty}(\mathbb{R})$ and satisfies for all $j,k  \in \mathbb{N}$,
		\begin{equation*}
			\sup \limits_{x} |x^{j} \partial_x^{k} f | < \infty.
		\end{equation*}
		\item  Let $a<b$ be real numbers and consider the domain $\mathcal{S} = (a,b) \times \R$. Then the space $ \dot{H}^{s}(\mathcal{S})$ is endowed with the seminorm
		 $$\|f\|^2_{\dot{H}^s(\mathcal{S}) } = \int_{a}^{b}|\nabla_{x,z}f(\cdot,z)|_{H^{s-1}}^2\: \mathrm{d}z.$$
		%
		%
		%
		%
		%
		%
		\item If  $A$ and $B$ are two operators, then we denote the commutator between them to be $[A,B] = AB - BA$.		
  		\item Let $t_0 > \frac{1}{2}$, $s\geq 0$, and $h_{\min}\in (0,1)$. Then for  $C(\cdot)$ being a positive non-decreasing function of its argument, we define the constants
  		\begin{equation*}
  			M =  C(\frac{1}{h_{\min}},  |\zeta|_{H^{t_0+2}}),
  		\end{equation*}
  		and
  		\begin{equation*}
  			M(s) = C(M, |\zeta|_{H^{s}}).
  		\end{equation*}
  	\end{itemize}

  	\subsubsection{Diffeomorphisms} In many instances it is convenient to \lq\lq straighten\rq\rq\: the fluid domain. In particular, instead of working on the fluid domain $\Omega^{\pm}_t$ we introduce the two strips:
  	\begin{equation*}
  		\mathcal{S}^+ = \{ (x,z) \in \R^2 \: : \: -1<z<0 \} 
  		\quad
  		\text{and} 
  		\quad
  		\mathcal{S}^- = \{ (x,z)  \in \R^2 \: : \: 0<z \}.               
  	\end{equation*}
	  The mapping between $\mathcal{S}^{\pm}$ and $\Omega_t^{\pm}$ will be given by the trivial diffeomorphisms defined below.
	  \begin{Def}\label{diffeomorphism +-}
		  	Let $t_0 > \frac{1}{2}$ and $\zeta \in H^{t_0+2}(\mathbb{R})$  such that the non-cavitation assumptions \eqref{non-cavitation} is satisfied. For the lower domain, we have that:
		  	\begin{itemize}
		  		\item [1.] 	We define the time-dependent diffeomorphism mapping the lower domain $\mathcal{S}^+$ onto the water domain $\Omega_t^+$ through
		  		\begin{equation*}\label{TrivialDiffeomorphismEq+}
		  			\Sigma^+ 
		  			: \: 
		  			\begin{cases}
		  				\mathcal{S}^+ \hspace{0.5cm} \longrightarrow & \Omega_t^+ \\
		  				(x,z) \: \:  \mapsto &  (x,z(1+\ve  \zeta)+  \ve \zeta).
		  			\end{cases}
		  		\end{equation*}
	  			\item [2.] The  Jacobi matrix $J_{\Sigma^+}$  is given by
	  			\begin{equation*}
	  				J_{\Sigma^+}
	  				=
	  				\begin{pmatrix}
	  					1& 0 \\
	  					\varepsilon(1+z)\partial_x \zeta & (1 + \varepsilon \zeta)  
	  				\end{pmatrix},
	  			\end{equation*}
	  			and is bounded on $\mathcal{S}^+$. Moreover the determinant is given by $1+\ve \zeta$ and is bounded below due to the non-cavitation condition \eqref{non-cavitation}.\\
	  			\item [3.] The matrix associated with the change of variable for the Laplace problem is given by
	  			\begin{align*}
	  				P(\Sigma^+)
	  				= 
	  				 \begin{pmatrix}
	  					(1+\ve \zeta)  & -\ve \sqrt{\mu}(z+1)\partial_x \zeta \\
	  					-\ve \sqrt{\mu}(z+1)\partial_x\zeta & \frac{1 + \ve^2 \mu (z+1)^2 |\partial_x \zeta|^2}{1 + \ve \zeta}
	  				\end{pmatrix},
	  			\end{align*}
  				and is uniformly coercive. In fact, one can verify that it satisfies for all $\theta \in \R^{d+1}$ and any $(X,z) \in \mathcal{S}^+$ that 
  				\begin{align}\label{Coercivity+}
  					P(\Sigma^+) \theta \cdot \theta \geq  \frac{1}{1+M} |\theta|^2 \quad \text{and} \quad \|P^+(\Sigma^+)\|_{L^{\infty}} \leq M.
  				\end{align}
  			\end{itemize}
  				Moreover, we define $\partial_{n}^{P^+} =  \mathbf{e}_{z} \cdot  P(\Sigma^+) \nabla^{\mu}_{x,z}$.\\ 

  				Similarly, for the upper domain, we have that:

  				\begin{itemize}
  					\item [1.] 	We define the time-dependent diffeomorphism mapping the upper-half plane $\mathcal{S}^-$ onto the water domain $\Omega_t^-$ through the transformation
  					\begin{equation*}\label{TrivialDiffeomorphismEq-}
  						\Sigma^- 
  						: \: 
  						\begin{cases}
  							\mathcal{S}^- \hspace{0.5cm} \longrightarrow & \Omega_t^- \\
  							(x,z) \: \:  \mapsto & (x, z + \ve \zeta ).
  						\end{cases}
  					\end{equation*}
  					\item [2.] The  Jacobi matrix $J_{\Sigma^-}$  is given by
  					\begin{equation*}
  						J_{\Sigma^-}
  						=
  						\begin{pmatrix}
  							1 & 0 \\
  							\ve  \partial_x \zeta  & 1 
  						\end{pmatrix},
  					\end{equation*}
  					and is bounded on $\mathcal{S}^-$. Moreover, the determinant is given by $1$.
  					\item [3.] The matrix associated with the change of variable for the Laplace problem is given by
  					\begin{align*}
  						P(\Sigma^-)
  						= 
  						\begin{pmatrix}
  							1 & - \ve \sqrt{\mu} \partial_x \zeta \\
  							-\ve \sqrt{\mu}\partial_x \zeta & 1 + \ve^2 \mu (\partial_x\zeta)^2
  						\end{pmatrix},
  					\end{align*}
  					and is uniformly coercive. In fact, one can verify that it satisfies  for all $\theta \in \R^{2}$ and any $(x,z) \in \mathcal{S}^-$ that
  					\begin{align}\label{Coercivity-}
  						P(\Sigma^-) \theta \cdot \theta \geq \frac{1}{1+M} |\theta|^2 \quad \text{and} \quad \|P^-(\Sigma^-)\|_{L^{\infty}} \leq M.
  					\end{align}
  					Moreover, we define $\partial_{n}^{P^-} =  \mathbf{e}_{z} \cdot  P(\Sigma^-) \nabla^{\mu}_{x,z}$. \\ 
		  	\end{itemize} 
	  \end{Def}

  	\section{Properties of $\mathcal{G}_{\mu}$}\label{Properties of G} In this section, we aim to give a rigorous meaning to the operator $\mathcal{G}_{\mu}$ introduced formally in equation \eqref{Main G} and study its properties. The main results in this section will now be stated.
  	\begin{prop}\label{Prop G}
  		Let $t_0 \geq 1$, $s\in [0,t_0+1]$,  and $\zeta \in H^{t_0+2}(\R)$ be such that \eqref{non-cavitation} is satisfied.  Then the mapping
  		\begin{equation}\label{Def G}
  			\mathcal{G}_{\mu}[\ve \zeta]
  			\:
  			:
  			\:
  			\begin{cases}
  				\dot{H}_{\mu}^{s + \frac{1}{2}}(\R) & \rightarrow  H^{s - \frac{1}{2}}(\R)
  				\\
  				\psi & \mapsto   \mathcal{G}^+_{\mu}[\ve \zeta] \Big{(} 1 -   \gamma   (\mathcal{G}^-_{\mu}[\ve \zeta])^{-1} \color{black}\mathcal{G}^+_{\mu}[\ve \zeta] \Big{)}^{-1}\psi,
  			\end{cases} 
  		\end{equation}
  		is well-defined and satisfies the following properties:
  		\begin{itemize}
  			\item [1.] For any $0 \leq s \leq t_0 +1$ ,there holds,
  			\begin{align}\label{1 Est G} 
  					|\mathcal{G}_{\mu}[\ve \zeta]\psi|_{H^{s-\frac{1}{2}}} \leq \mu^{\frac{3}{4}}M |\psi|_{\dot{H}_{\mu}^{s+\frac{1}{2}}},
  			\end{align}
  			and
  			\begin{align}\label{Est G} 
  				|\mathcal{G}_{\mu}[\ve \zeta]\psi|_{H^{s-\frac{1}{2}}} \leq \mu^{\frac{1}{2}}M |\partial_x\psi|_{H^{s-\frac{1}{2}}}.
  			\end{align}
  			\item [2.] For any $0 \leq s \leq t_0 + \frac{1}{2}$ there holds,
  			\begin{align}\label{2 Est G}
  				|\mathcal{G}_{\mu}[\ve \zeta]\psi|_{H^{s-\frac{1}{2}}} \leq \mu M |\psi|_{\dot{H}_{\mu}^{s+1}},
  			\end{align}
  			and
  			\begin{align}\label{Est G 2}
  				|\mathcal{G}_{\mu}[\ve \zeta]\psi|_{H^{s-\frac{1}{2}}} \leq \mu^{\frac{3}{4}} M |\partial_x \psi|_{H^{s}}.
  			\end{align}
  			\item [3.] The operator is uniformly coercive on on  $\psi \in \dot{H}_{\mu}^{\frac{1}{2}}(\R)$,
  			\begin{equation}\label{G Coercive}
  				|\psi|_{\dot{H}_{\mu}^{\frac{1}{2}}}^2\leq M \big{(}\psi, \frac{1}{\mu}\mathcal{G}_{\mu}[\ve \zeta]\psi\big{)}_{L^2}.
  			\end{equation}
  			\item [4.] The bilinear form is symmetric on $\dot{H}_{\mu}^{\frac{1}{2}}(\R) \times \dot{H}_{\mu}^{\frac{1}{2}}(\R)$,
  			\begin{equation}\label{G Symmetic}
  				\big{(}\mathcal{G}_{\mu}[\ve \zeta]\psi, \psi\big{)}_{L^2}
  				=
  				\big{(}\psi,\mathcal{G}_{\mu}[\ve \zeta]\psi\big{)}_{L^2}.
  			\end{equation}
  			\item[5.] For all $s\in [0,t_0+1]$ and $f,g \in \dot{H}_{\mu}^{s+\frac{1}{2}}(\R)$ there holds,
  			\begin{equation}\label{G continuous form}
  				|\big{(}\Lambda^s\mathcal{G}_{\mu}[\ve \zeta]f, \Lambda^s g \big{)}_{L^2}| \leq
  				 \mu M |f|_{\dot{H}_{\mu}^{s+\frac{1}{2}}} |g|_{\dot{H}_{\mu}^{s+\frac{1}{2}}}.
  			\end{equation}
  		\end{itemize}
  	\end{prop}
  	
  	\begin{remark}
  		Here we take $t_0\geq 1$ since we will apply the result at high regularity. If one would take $t_0 >  \frac{1}{2}$ then we would need $s\in [\max\{0, 1-t_0\},t_0+1]$ to enforce Remark \ref{Remark extension of G-}, which specifies the needed regularity to define $\mathcal{G}^-_{\mu}$.
  	\end{remark}
  
  \begin{remark}\label{Rmrk G[0]}
  	For the undisturbed case, we can use formulas \eqref{G+[0]} and \eqref{G-[0]} to find that
  	\begin{equation}\label{G[0]}
  		\mathcal{G}_{\mu}[0]  = \sqrt{\mu} |\mathrm{D}| \frac{\mathrm{tanh}(\sqrt{\mu}|\mathrm{D}|)}{1+\gamma\mathrm{tanh}(\sqrt{\mu}|\mathrm{D}|)}. 
  	\end{equation}
  	
  \end{remark}
  	
  	We will follow the same strategy as in \cite{LannesTwoFluid13}, where we study the operators involved in the expression for $\mathcal{G}_{\mu}$ and prove that we can define them by $\psi$ through the transition problem \eqref{transition Phi}. We study each part individually in separate subsections. The main difference from the previous work is that we have to carefully track the dependence on the parameters for the current regime. Moreover, the functional setting of the upper fluid is fundamentally different from the one in the lower fluid domain.

  	\subsection{Properties of $(\mathcal{G}_{\mu}^-)^{-1} \mathcal{G}_{\mu}^+$. } For the description of the operator $(\mathcal{G}_{\mu}^-)^{-1} \mathcal{G}_{\mu}^+$ we will follow the proof of Proposition $1$ in \cite{LannesTwoFluid13}. The main difference is that we have an interaction of two operators that act on different scales, and this is seen in the estimates given below: 

  	\begin{prop}\label{Inverse 2}
   	Let $t_0\geq 1$, $s\in [0, t_0+1]$ and $\zeta \in H^{t_0+2}(\R)$ be such that \eqref{non-cavitation} is satisfied.  Then the mapping
   	\begin{equation}\label{inverse G 2}
   		(\mathcal{G}_{\mu}[\ve \zeta]^-)^{-1}  \mathcal{G}^+_{\mu}[\ve \zeta]
   		\:
   		:
   		\:
   		\begin{cases}
   			\dot{H}_{\mu}^{s + \frac{1}{2}}(\R) & \rightarrow   \mathring{H}^{s + \frac{1}{2}}(\R)\\
   			\psi^+ &  \mapsto (\mathcal{G}^-_{\mu}[\ve \zeta])^{-1}  \mathcal{G}^+_{\mu}[\ve \zeta]\psi^+
   		\end{cases} 
   	\end{equation}
    is well-defined and satisfies
    \begin{equation}\label{Inverse est 2}
    	  | (\mathcal{G}^-_{\mu}[\ve \zeta])^{-1} \mathcal{G}^+_{\mu}[\ve \zeta]\psi^{+} |_{\mathring{H}^{s+ \frac{1}{2}}}\leq \mu^{\frac{1}{4}} M |\psi^{+}|_{\dot{H}_{\mu}^{s+\frac{1}{2}}}, 
    \end{equation}
	and 
	\begin{equation}\label{Inverse est 2.1}
		| \partial_x \big{(}(\mathcal{G}^-_{\mu}[\ve \zeta])^{-1} \mathcal{G}^+_{\mu}[\ve \zeta]\psi^{+} \big{)}|_{H^{s}}
		\leq
		M |\partial_x\psi^{+}|_{H^{s}}.
	\end{equation}
    \end{prop}
	\begin{remark}\label{order of comp}
		If we change the role of $\pm$ the result is not true. Indeed, consider the case $\ve \zeta = 0$, then we have by direct computations that
		\begin{align*}
			(\mathcal{G}_{\mu}^+[0])^{-1} \mathcal{G}_{\mu}^-[0] = ( \tanh(\sqrt{\mu}|\mathrm{D}|))^{-1} ,
		\end{align*}
		and we obtain the estimate
		\begin{equation*}
			| (\mathcal{G}_{\mu}^+[0])^{-1} \mathcal{G}_{\mu}^-[0]\psi^- |_{\dot{H}_{\mu}^{ \frac{1}{2}}}\lesssim \frac{1}{\sqrt{\mu}} |\psi^-|_{H^{\frac{1}{2}}},
		\end{equation*}
		which is incompatible with having $\psi^- \in \mathring{H}^{s+\frac{1}{2}}(\R)$.
		
	\end{remark}

	Before we turn to the proof, we need a Lemma that will be used to justify some of the computations that will be made.

	\begin{lemma}\label{Lemma trace} Suppose the provisions of Proposition \ref{Inverse 2}, and further that $\phi^- = \Phi^-\circ \Sigma^- \in  \dot{H}^{s+1}(\mathcal{S}^-)$ is a variational solution to   
		\begin{equation}\label{Claim phi-}
			\begin{cases}
				\nabla_{x,z}^{\mu} \cdot P(\Sigma^{-})\nabla_{x,z}^{\mu} \phi^- = 0 \hspace{0.5cm}\qquad \text{in} \quad \mathcal{S}^-
				\\
				\partial_{n}^{P^-} \phi^- =  \mathcal{G}^+_{\mu}[\ve \zeta]\psi^+ \hspace{2.1cm} \text{on} \quad  z = 0.
			\end{cases}
		\end{equation}
		Then there is a number $R>0$ such that for $\mathcal{S}^{-}_R = \{(x,z)\in\mathcal{S}^- \: : \: z>R \}$ and any $\alpha+\beta$  integer $>0$ there holds, 
		\begin{equation}\label{Decay est}
			\partial_z^{\alpha} \partial_x^{\beta} \phi^-\in L^2(\mathcal{S}^{-}_R)\quad \text{and} \quad
			\lim \limits_{z \rightarrow \infty} \sup \limits_{x \in \R}
			\big{|} 
			\partial_z^{\alpha} \partial_x^{\beta} \phi^-(x,z) \big{|}
			= 0.
		\end{equation}
		Moreover, for any $\varphi \in \dot{H}^1(\mathcal{S}^-)$ have the following trace inequalities:
		\begin{equation}\label{trace 1}
			 |\varphi^-|_{z = 0}|_{\mathring{H}^{s+\frac{1}{2}}} \leq \mu^{-\frac{1}{4}}  \| \Lambda^s \nabla_{x,z}^{\mu} \varphi \|_{L^2(\mathcal{S}^-)},\color{black}
		\end{equation}
		and 
		\begin{equation}\label{trace 2}
		 	 |\varphi|_{z = 0}|_{\dot{H}_{\mu}^{s+\frac{1}{2}}} \leq \mu^{-\frac{1}{2}} \| \Lambda^s \nabla^{\mu}_{x,z} \varphi \|_{L^2(\mathcal{S}^-)}.\color{black}
		\end{equation}

	\end{lemma}

	\begin{remark}
		For the proof of \eqref{trace 1}, we need to work on the upper-half plane to prove the estimate, and this is the technical reason why we do not study the reverse composition: $(\mathcal{G}_{\mu}^+)^{-1}\mathcal{G}_{\mu}^-$.
	\end{remark}

	\begin{remark}\label{Alazard}
		The proof of estimate \eqref{Decay est} is given in \cite{Alazard21} for $\phi|_{z=0} \in H^{s+\frac{1}{2}}(\R)$ and we will briefly explain the changes for $\phi|_{z=0} \in \mathring{H}^{s+\frac{1}{2}}(\R)$.
	\end{remark}

	Assuming for a moment that Lemma \ref{Lemma trace} holds true, we can give the proof of Proposition \ref{Inverse 2}.

	\begin{proof}[Proof of Proposition \ref{Inverse 2}]
		We divide the proof into four main steps.\\

		\noindent
		\underline{Step 1.} \textit{$(\mathcal{G}_{\mu}^-)^{-1} \circ \mathcal{G}_{\mu}^+$ is well-defined on $\mathring{H}^{s+\frac{1}{2}}(\R)$}. It is sufficient to prove that there exists a unique variational solution $\phi^- \in  \dot{H}^{s+1}(\mathcal{S}^-)$ to the system \eqref{Claim phi-}  for any $\psi^+ \in \dot{H}^{s+\frac{1}{2}}_{\mu}(\R)$. Indeed, assuming there is such a solution, then by \eqref{trace 1}, we can define
		\begin{equation*}
			\psi^-  = \phi^-|_{z=0}\in  \mathring{H}^{s+\frac{1}{2}}(\R),
		\end{equation*}
		so that	for any $\psi^+ \in \dot{H}^{s+\frac{1}{2}}_{\mu}(\R)$ there is a $\psi^-$ where Proposition \ref{Prop G- dual est} implies
		\begin{equation*}
			\mathcal{G}^-_{\mu}[\ve \zeta]\psi^-
			=
			\mathcal{G}^+_{\mu}[\ve \zeta]\psi^+.
		\end{equation*}

		To prove the claim we first let $s=0$, and use \eqref{Decay est} to define the variational problem associated to \eqref{Claim phi-} by
		\begin{align}\label{var phi-}
			a(\phi^-, \varphi) 
			&  =
			\int_{\mathcal{S}^-} P(\Sigma^-)\nabla_{x,z}^{\mu} \phi^{-} \cdot \nabla_{x,z}^{\mu} \varphi \: \mathrm{d}x \mathrm{d}z
			\\
			\notag 
			& = 
			-\int_{\{z=0\}}( \mathcal{G}^+_{\mu}[\ve \zeta]\psi^+) \varphi \: \mathrm{d} x
			\\
			& 
			=
			\notag
			L(\varphi),
		\end{align}
		for any $\varphi \in C^{\infty}(\overline{\mathcal{S}^-})\cap\dot{H}^1(\mathcal{S}^-)$. We will now verify the assumptions of Lax-Milgram's Theorem to deduce a variational solution in $\dot{H}^1(\mathcal{S}^-)$ in two steps (extending the result to $\dot{H}^{s+1}(\mathcal{S}^-)$ is classical).

		The first step is to show that the application $\varphi \mapsto L(\varphi)$ is continuous on $\dot{H}^1(\mathcal{S}^-)$. To do so, we note by estimate \eqref{dual est G+ s}  that
		\begin{align*}
			|L(\varphi)| 
			& \leq
			|\big{(}\mathcal{G}^+_{\mu}[\ve \zeta] \psi^+,\varphi|_{z=0} \big{)}_{L^2}|
			\\
			& \leq 
			\mu| \psi^+|_{\dot{H}^{\frac{1}{2}}_{\mu}} |\varphi|_{z=0}|_{\dot{H}^{\frac{1}{2}}_{\mu}}.
		\end{align*}
		Then use \eqref{trace 2} to obtain the needed bound
		\begin{equation}\label{Claim 3 ineq}
			|\varphi|_{z = 0}|_{\dot{H}^{\frac{1}{2}}_{\mu}} \leq \mu^{-\frac{1}{2}}| \|\nabla_{x,z}^{\mu} \varphi \|_{L^2(\mathcal{S}^-)}.
			\color{black}
		\end{equation}
		Lastly, The bilinear form $a(\cdot,\cdot)$ is continuous and coercive on $\dot{H}^1(\mathcal{S}^-)$ by using estimate  \eqref{Coercivity-}.  We may therefore conclude that there is a unique solution $\phi^- \in \dot{H}^1(\mathcal{S}^-)$. \\

		\noindent
		\underline{Step 2.} \textit{Estimate \eqref{Inverse est 2} holds true.} Let $\psi^- = \phi^-|_{z=0} \in  \mathring{H}^{s+\frac{1}{2}}$ defined by the solution of \eqref{Claim phi-}. Then we can use estimate \eqref{est psi -} to get that
		\begin{align*}
			| (\mathcal{G}^-_{\mu}[\ve \zeta])^{-1}  \mathcal{G}^+_{\mu}[\ve \zeta]\psi^{+} |_{\mathring{H}^{s+ \frac{1}{2}}}
			& =
		 	| \psi^{-} |_{\mathring{H}^{s+ \frac{1}{2}}}   
			\\
			& 
			\leq \mu^{-\frac{1}{4}} M\| \Lambda^s \nabla_{x,z}^{\mu} \phi^{-}\|_{L^2(\mathcal{S}^+)}.\color{black}
		\end{align*}
		Thus, we need to verify
		\begin{equation}\label{est on phi- in prop 2.4}
			\| \Lambda^s \nabla^{\mu} \phi^{-}\|_{L^2(\mathcal{S}^-)} \leq \mu^{\frac{1}{2}}M|\psi^+|_{\dot{H}_{\mu}^{s+\frac{1}{2}}}.
		\end{equation}
		To this end, we apply $\Lambda^s$ to \eqref{Claim phi-} and study:
		\begin{equation*}
			\begin{cases}
				\nabla_{x,z}^{\mu} \cdot P(\Sigma^{-})\nabla_{x,z}^{\mu} \Lambda^s \phi^- = -\nabla_{x,z}[P(\Sigma^-),\Lambda^s]\nabla_{x,z}^{\mu} \phi^{-}
				\hspace{1cm} \text{in} \quad \mathcal{S}^-
				\\
				\partial_{n}^{P^-}\Lambda^s \phi^-  =-\mathbf{e}_z \cdot[P(\Sigma^-), \Lambda^s]\nabla_{x,z}^{\mu} \phi^- +  \Lambda^s\mathcal{G}^{+}_{\mu}[\ve \zeta]\psi^+ \hspace{0.75cm} \text{on} \quad  z = 0.
			\end{cases}
		\end{equation*}
		We may therefore multiply the equation with $\Lambda^s \phi^-$ and integrate by parts on a finite domain $\mathcal{S}^{-}_R = \{(x,z)\in \mathcal{S}^- \: : \: 0<z<R\}$ and use the Dominated Convergence Theorem with \eqref{Decay est} to obtain
		\begin{align*}
			\int_{\mathcal{S}^-}P(\Sigma^-) \nabla_{x,z}^{\mu}\Lambda^s \phi^- \cdot \nabla_{x,z}^{\mu} \Lambda^s \phi^- \: \mathrm{d}x \mathrm{d}z
			& =
			-
			\int_{\{z=0\}} \Lambda^s(\mathcal{G}^{+}_{\mu}[\ve \zeta]\psi^+) \Lambda^s \phi^- \: \mathrm{d}x
			\\
			& 
			\hspace{0.5cm}
			+
			\int_{\mathcal{S}^-} [P(\Sigma^-), \Lambda^s]\nabla_{x,z}^{\mu} \phi^- \cdot \nabla_{x,z}^{\mu} \Lambda^s \phi^- \: \mathrm{d}x \mathrm{d}z.
		\end{align*}
		Now, by the coercivity estimate \eqref{Coercivity-}, Cauchy-Schwarz, \eqref{dual est G+ s}, \eqref{trace 2}, and the commutator estimate \eqref{Commutator est} we get that
		\begin{align}\label{induction relation 2}
			\| \Lambda^s  \nabla_{x,z}^{\mu} \phi^- \|_{L^2(\mathcal{S}^-)}
			\leq  M \big{(}
			\mu^{\frac{1}{2}}|\psi^+|_{\dot{H}^{s+\frac{1}{2}}_{\mu}} + \| \Lambda^{s-1}\nabla_{x,z}^{\mu} \phi^-\|_{L^2(\mathcal{S}^-)}\big{)}.\color{black}
		\end{align}
		Also, using the estimates in Step $1$, we obtain the base case
		\begin{align}\label{Base case 2}
			\| \nabla_{x,z}^{\mu} \phi^-\|_{L^2(\mathcal{S}^-)} \leq \mu^{\frac{1}{2}} M	|\psi^+|_{\dot{H}^{\frac{1}{2}}_{\mu}}.\color{black}
		\end{align}
		We may now conclude the proof of this step by continuous induction. \\
		
		\noindent
		\underline{Step 3.} To prove estimate \eqref{Inverse est 2.1} we first observe that
		\begin{equation*}
			\langle \xi \rangle^s |\xi| \lesssim 	\langle \xi \rangle^{s+\frac{1}{2}} |\xi|^{\frac{1}{2}} \quad \text{and} \quad  \langle \xi \rangle^{s+\frac{1}{2}} \frac{|\xi|}{(1+ \sqrt{\mu}|\xi|)^{\frac{1}{2}}} \lesssim \mu^{-\frac{1}{4}} \langle \xi \rangle^{s} |\xi|.
		\end{equation*}
		Then use Plancherel's identity and \eqref{Inverse est 2} 
		to find that
		\begin{align*}
			| \partial_x \big{(}(\mathcal{G}^-_{\mu}[\ve \zeta])^{-1} \circ \mathcal{G}^+_{\mu}[\ve \zeta]\psi^{+} \big{)}|_{H^{s}}
			& \lesssim 
			|  (\mathcal{G}^-_{\mu}[\ve \zeta])^{-1} \circ \mathcal{G}^+_{\mu}[\ve \zeta]\psi^{+} |_{\mathring{H}^{(s+\frac{1}{2})+\frac{1}{2}}}
			\\ 
			& \leq 
			\mu^{\frac{1}{4}} M |\psi^{+}|_{\dot{H}_{\mu}^{(s+\frac{1}{2})+\frac{1}{2}}}
			\\
			& \leq 
		 	M |\partial_x \psi^{+}|_{H^{s}}
			.
		\end{align*}
		\end{proof}

		To close this subsection, we give the proof of Lemma \ref{Lemma trace}.
		
		\begin{proof}[Proof of Lemma \ref{Lemma trace}] We first give a proof of the trace inequalities on $C^{\infty}(\overline{\mathcal{S}^-})\cap\dot{H}^1(\mathcal{S}^-)$.  
			For the proof of \eqref{trace 1}, we define a multiplier being a smooth cut-off function in frequency $\chi: [0,\infty) \rightarrow [0,1]$ such that $\chi(0) = 1$, $\chi(\xi) = 0$ for $\xi>1$, and $\chi, \chi'\in L^{\infty}(\R)$.  Then by the Fundamental Theorem of Calculus we obtain that
			\begin{align*}
				||\mathrm{D}|^{\frac{1}{2}} \varphi|_{z=0}|_{L^2}^2 
				& = \Big{|}
				\int_{\R} \int_0^{\frac{2}{\sqrt{\mu}|\xi|}} \partial_z\big{(}\chi(z\sqrt{\mu}|\xi|) \big{|}|\xi|^{\frac{1}{2}} \hat{\varphi}\big{|}^2 \big{)}\: \mathrm{d}z \mathrm{d}\xi \Big{|}
				\\ 
				& 
				\lesssim \sqrt{\mu}  |\chi'|_{L^{\infty}}
				\int_{\R} \int_0^{\infty} \big{|}|\xi| \hat{\varphi}\big{|}^2 \: \mathrm{d}z \mathrm{d}\xi
				+
				|\chi|_{L^{\infty}}\int_{\R} \int_0^{\infty}  \big{|} |\xi| \hat{\varphi}\big{|}   \big{|} \partial_z\hat{\varphi}\big{|} \: \mathrm{d}z \mathrm{d}\xi.
			\end{align*}
			For the last term, we use Young's inequality to find that
			\begin{align*}
				\int_{\R} \int_0^{\infty}  \big{|} |\xi| \hat{\varphi}\big{|}   \big{|} \partial_z\hat{\varphi}\big{|} \: \mathrm{d}z \mathrm{d}\xi
				\lesssim 
				\sqrt{\mu}\int_{\R} \int_0^{\infty}  \big{|} |\xi| \hat{\varphi}\big{|}^2\: \mathrm{d}z \mathrm{d}\xi + \frac{1}{\sqrt{\mu}}\int_{\R} \int_0^{\infty}  \big{|} \partial_z\hat{\varphi}\big{|}^2 \: \mathrm{d}z \mathrm{d}\xi
			\end{align*}
			Then to conclude, we use Plancherel's identity to obtain \eqref{trace 1} for $s=0$:
			\begin{align*}
				||\mathrm{D}|^{\frac{1}{2}} \varphi|_{z=0}|_{L^2}^2
				&  \lesssim \mu^{-\frac{1}{2}}
				\int_{\R} \int_0^{\infty}  \Big{(}\big{|} \sqrt{\mu}|\xi| \hat{\varphi}\big{|}^2 +  \big{|} \partial_z\hat{\varphi}\big{|}^2\Big{)} \: \mathrm{d}z \mathrm{d}\xi
				\\
				& 
				\lesssim 
				\mu^{-\frac{1}{2}}
				\| \nabla_{x,z}^{\mu} \varphi \|_{L^2(\mathcal{S}^-)}^2.
			\end{align*}
			The general case is proved similarly by applying the same estimate to $\tilde{\varphi} = \Lambda^s \varphi$.
			
			For the proof of \eqref{trace 2} we simply use that $|\varphi|_{z = 0}|_{\dot{H}_{\mu}^{s+\frac{1}{2}}} \leq \mu^{-\frac{1}{4}}	| \varphi|_{z=0}|_{\mathring{H}^{s+\frac{1}{2}}}$ and then apply \eqref{trace 1} to conclude.  \color{black}

			For the proof of \eqref{Decay est}, we note from Definition \ref{diffeomorphism +-} that $\Phi^- = \phi^-\circ{\Sigma}^{-1} \in  \dot{H}^{s+1}(\Omega^-)$ is harmonic, and by Sobolev embedding we have that $\zeta$ is bounded above by some constant $R>0$. From these observations, we can use \eqref{trace 1} to see that $\Phi^-|_{z=R} \in \mathring{H}^{s+\frac{1}{2}}(\R)$. Moreover, we consider the harmonic extension on $\mathcal{S}^-_R$  given by the Poisson formula:
			\begin{equation*}
				\Phi^{-}_R (x,z) = \mathrm{e}^{-(z-R)\sqrt{\mu}|\mathrm{D}|}\Phi^-(x,R).
			\end{equation*}			
				From this formula, we can verify \eqref{Decay est} for $\Phi^{-}_R$ by direct computations as in \cite{Alazard21}. To conclude, we use that both functions are harmonic and agree on the line $z=R$.

		\end{proof}
	
		We should note that there are several results that follow from Proposition \ref{Inverse 2}, and that will be used throughout the paper. However, to ease the presentation, we postponed these results for the Appendix in Section \ref{Appendix Corollaries} since the proofs are technical and not needed in this section.

	\subsection{Properties of $(1-\gamma (\mathcal{G}_{\mu}^-)^{-1} \mathcal{G}_{\mu}^+)^{-1}$. } Following the road map provided in \cite{LannesTwoFluid13}, we can recover the velocity potentials $\phi^{\pm}$ from the knowledge of $\zeta$ and a trace $\psi$ defined though the transmission problem:
	\begin{equation}\label{Transmission pb}
		\begin{cases}
			\nabla_{x,z}^{\mu^{\pm}} \cdot P(\Sigma^{\pm})\nabla_{x,z}^{\mu^{\pm}} \phi^{\pm} = 0 \hspace{0.5cm}\qquad \text{in} \quad \mathcal{S}^{\pm}
			\\
			\phi^{+}|_{z=0}  - \gamma \phi^-|_{z=0} = \psi
			\\
			\partial_{n}^{P^{-}} \phi^-|_{z=0} = \partial_{n}^{P^{+}} \phi^+|_{z=0},
			\quad
			\partial_{n}^{P^{+}} \phi^+|_{z=-1}  = 0,
		\end{cases}
	\end{equation}
	where the solvability of this problem is ensured in the next result:
	\begin{prop}\label{Prop Transition}
		Let $t_0 \geq 1$, $s \in [0,t_0 +1]$ and $\zeta \in H^{t_0 +2}(\R)$ be such that \eqref{non-cavitation} is satisfied. Then for all $\psi \in \dot{H}_{\mu}^{s+\frac{1}{2}}(\R)$, there exist a unique solution $\phi^{\pm}\in \dot{H}^{s+1}(\mathcal{S}^{\pm})$ to \eqref{Transmission pb} and that satisfies
		\begin{equation*}
			\| \Lambda^s \nabla^{\mu}_{x,z} \phi^{\pm} \|_{L^2(\mathcal{S}^{+})} \leq \sqrt{\mu} M |\psi|_{\dot{H}^{s+\frac{1}{2}}_{\mu}}.
		\end{equation*}
		%
		%
		%
		%
		%
		%
	\end{prop}
	The proof of this result is a consequence of the following Lemma:
	\begin{lemma}\label{Prop op J}
		Let $t_0\geq 1$, $s \in [0,t_0 +1]$. Then the mapping
		\begin{equation}\label{J}
			\mathcal{J}_{\mu}[\ve \zeta]
			\:
			:
			\:
			\begin{cases}
				\dot{H}_{\mu}^{s + \frac{1}{2}}(\R) & \rightarrow  \dot{H}_{\mu}^{s + \frac{1}{2}}(\R)\\
				\psi^+ & \mapsto ( 1-  \gamma (\mathcal{G}^-_{\mu}[\ve \zeta])^{-1} \mathcal{G}^+_{\mu}[\ve \zeta] )\psi^+
			\end{cases} 
		\end{equation}
		is one-to-one and onto. Moreover, it satisfies the estimate
		\begin{equation}\label{Est J}
			|(\mathcal{J}_{\mu}[\ve \zeta])^{-1}\psi^+|_{\dot{H}_{\mu}^{s+\frac{1}{2}}} \leq M |\psi^+|_{\dot{H}_{\mu}^{s+\frac{1}{2}}}.			
		\end{equation}
		%
		%
		%
	\end{lemma}	   
	  	\begin{remark}\label{rmk id psi- psi+}
		For any $\psi \in \dot{H}^{s+\frac{1}{2}}_{\mu}(\R)$ we can define $\psi^+$ by
		\begin{equation*}
			\psi^+ = (\mathcal{J}_{\mu}[\ve \zeta])^{-1} \psi \in \dot{H}^{s+\frac{1}{2}}_{\mu}(\R),
		\end{equation*}
		from which we define $\psi^-$ by
		\begin{align*}
			\psi^{-} =  (\mathcal{G}^-_{\mu}[\ve \zeta])^{-1} \mathcal{G}^+_{\mu}[\ve \zeta] \psi^+ \in \mathring{H}^{s+ \frac{1}{2}}(\R).
		\end{align*}
		Also, note that from these identities and \eqref{Inverse est 2} that there holds,
		\begin{align}\label{Est J with psi}
			|\psi|_{\dot{H}_{\mu}^{s+\frac{1}{2}}}
			& = 
			| \mathcal{J}_{\mu}[\ve \zeta]\psi^+|_{\dot{H}_{\mu}^{s+\frac{1}{2}}} \\ 
			& \leq \notag 
			|\psi^+|_{\dot{H}_{\mu}^{s+\frac{1}{2}}} +  \gamma |(\mathcal{G}^-_{\mu}[\ve \zeta])^{-1} \mathcal{G}^+_{\mu}[\ve \zeta] \psi^+|_{\dot{H}_{\mu}^{s+\frac{1}{2}}}
			\\ 
			& \leq 
			\notag 
			|\psi^+|_{\dot{H}_{\mu}^{s+\frac{1}{2}}} + \gamma \mu^{-\frac{1}{4}}|(\mathcal{G}^-_{\mu}[\ve \zeta])^{-1} \mathcal{G}^+_{\mu}[\ve \zeta] \psi^+|_{\mathring{H}^{s+\frac{1}{2}}}
			\\
			& \leq 
			\notag M
			|\psi^+|_{\dot{H}_{\mu}^{s+\frac{1}{2}}}
			.			
		\end{align}
		
	\end{remark}
	\begin{proof}[Proof of Lemma \ref{Prop op J}]
	
		To prove estimate \eqref{Est J}, we first consider the case $s=0$. Then  we use the definition of $\mathcal{J}_{\mu}[\ve \zeta]$, the construction $\psi^- =  (\mathcal{G}^-_{\mu}[\ve \zeta])^{-1}\mathcal{G}^+_{\mu}[\ve \zeta]\psi^+$ 
		 to get that
		\begin{align*}
			\int_{\R} \mathcal{J}_{\mu}[\ve \zeta] \psi^+ \: \mathcal{G}^+_{\mu}[\ve \zeta] \psi^+ \: \mathrm{d}x 
			& = 
			\int_{\R}\psi^+\: \mathcal{G}^+_{\mu}[\ve \zeta] \psi^+ \: \mathrm{d}x  
			-
			\gamma 
		 	\int_{\R} (\mathcal{G}^-_{\mu}[\ve \zeta])^{-1} \mathcal{G}^+_{\mu}[\ve \zeta]\psi^+ \: \mathcal{G}^+_{\mu}[\ve \zeta] \psi^+ \: \mathrm{d}x
			\\
			& = 
			\int_{\R}\psi^+ \: \mathcal{G}^+_{\mu}[\ve \zeta] \psi^+ \: \mathrm{d}x
			-
			\gamma \int_{\R}\psi^- \: \mathcal{G}^-_{\mu}[\ve \zeta] \psi^- \: \mathrm{d}x.
		\end{align*}
		Then apply Proposition \ref{Prop G+ dual est} and \ref{Prop G- dual est} combined with the coercivity of $P(\Sigma^{\pm})$ to obtain the estimate
		\begin{align*}
			\int_{\R} \mathcal{J}_{\mu}[\ve \zeta] \psi^+ \: \mathcal{G}^+_{\mu}[\ve \zeta] \psi^+ \: \mathrm{d}x 
			& =
			\int_{\mathcal{S}^+} P(\Sigma^+)\nabla^{\mu}_{x,z}\phi^{+} \cdot \nabla^{\mu}_{x,z}\phi^+  \: \mathrm{d}x \mathrm{d}z
			\\ 
			& 
			\hspace{0.5cm}
			+
			\gamma \int_{\mathcal{S}^-} P(\Sigma^-)\nabla_{x,z}^{\mu}\phi^{-} \cdot \nabla_{x,z}^{\mu}\phi^-\: \mathrm{d}x \mathrm{d}z 
			\\
			& 
			\geq 
			\frac{1}{1+M}
			\|\nabla_{x,z}^{\mu} \phi^+\|_{L^2(\mathcal{S}^+)}^2.
		\end{align*}
		Moreover, since $\mathcal{J}_{\mu}[\ve \zeta] \psi^+ \in \dot{H}_{\mu}^{s+\frac{1}{2}}(\R) \subset \dot{H}_{\mu}^{\frac{1}{2}}(\R)$ we have \eqref{dual est G+ s} at hand. In particular, we obtain from the above estimates and \eqref{est psi +} that
		\begin{align}\label{lower bound}
			|\psi^+|_{\dot{H}_{\mu}^{\frac{1}{2}}} \leq M|\mathcal{J}_{\mu}[\ve \zeta]\psi^+|_{\dot{H}_{\mu}^{\frac{1}{2}}}.
		\end{align}
		Equivalently, estimate \eqref{Est J} holds in the case $s=0$. We may also use this estimate to prove the invertibility as in \cite{LannesTwoFluid13}. In particular, we have from the lower bound \eqref{lower bound} and Proposition \ref{Inverse 2} with estimate \eqref{Inverse est 2} that $\mathcal{J}_{\mu}$ is an injective and closed operator since,
		\begin{equation*}
			\gamma |(\mathcal{G}^-_{\mu}[\ve \zeta])^{-1} \mathcal{G}^+_{\mu}[\ve \zeta]\psi^+|_{\dot{H}^{s+ \frac{1}{2}}_{\mu}}
			\leq 
			\mu^{-\frac{1}{4}}\gamma 	| (\mathcal{G}^-_{\mu}[\ve \zeta])^{-1}  \mathcal{G}^+_{\mu}[\ve \zeta]\psi^{+} |_{\mathring{H}^{s+ \frac{1}{2}}}\leq \gamma  M |\psi^{+}|_{\dot{H}_{\mu}^{s+\frac{1}{2}}}.
		\end{equation*}
		Moreover, from the lower bound, \eqref{lower bound}, on $\mathcal{J}_{\mu}$ we know that it is also semi-Fredholm. Then since for small enough values of $\gamma \in (0,1)$ it is invertible by a Neumann series expansion, we have from the homotopic invariance of the index that the operator is in fact Fredholm of index zero \cite{Kato_66}. Consequently, the operator is also surjective and therefore invertible. 
		
		For the general case of $s \in [0,t_0+1]$, the proof is the same as Lemma $2$ in \cite{LannesTwoFluid13}.
	\end{proof}
   
   We may now use Lemma \ref{Prop op J} to prove Proposition \ref{Prop Transition}.
   \begin{proof} [Proof of Proposition \ref{Prop Transition}]
   		We first consider the existence of a unique solution $\phi^+$ in the lower domain. Since $\psi \in \dot{H}_{\mu}^{s+\frac{1}{2}}(\R)$ we can use Lemma \ref{Prop op J} to make the definition:
   		\begin{equation*}
   			\psi^+ = (\mathcal{J}_{\mu}[\ve \zeta])^{-1} \psi \in \dot{H}_{\mu}^{s+\frac{1}{2}}(\R)
   			\quad 	
   			\text{and}
   			\quad   			
   			\psi^{-} =  (\mathcal{G}^-_{\mu}[\ve \zeta])^{-1}\mathcal{G}^+_{\mu}[\ve \zeta] \psi^+ \in \mathring{H}^{s+ \frac{1}{2}}(\R)\subset  \dot{H}_{\mu}^{s+\frac{1}{2}}(\R),
   		\end{equation*}
   		where we let $\phi^{\pm}|_{z=0} = \psi^{\pm}$. Then we can use the first point of Proposition \ref{Prop G+ dual est} to deduce a unique solution $\phi^+$ in the lower domain, where the estimate follows from \eqref{est phi +} and \eqref{Est J}:
   		\begin{align*}
   			\| \Lambda^s \nabla^{\mu}_{x,z} \phi^{+} \|_{L^2(\mathcal{S}^{\pm})} 
   			& \leq  \sqrt{\mu}M |\psi^+|_{\dot{H}^{s+\frac{1}{2}}_{\mu}}
   			\\ 
   			& =
   			\sqrt{\mu }M |(\mathcal{J}_{\mu}[\ve \zeta])^{-1}\psi|_{\dot{H}^{s+\frac{1}{2}}_{\mu}}
   			\\ 
   			& 
   			\leq 
   			\sqrt{\mu }M |\psi|_{\dot{H}^{s+\frac{1}{2}}_{\mu}}.
   		\end{align*}

   		For the upper half plane, we use Proposition \ref{Inverse 2}  we use Remark \ref{Remark extension of G-} together with the first point of Proposition \ref{Prop G- dual est} to deduce a unique solution $\phi^-$. The estimate is a consequence of estimates \eqref{est phi -}, \eqref{Inverse est 2}, and then \eqref{Est J}:
   		\begin{align*}
   			\| \Lambda^s\nabla_{x,z}^{\mu}\phi^-\|_{L^2(\mathcal{S}^-)} 
   			& \leq
   			\mu^{\frac{1}{4}} M 	|\psi^-|_{\mathring{H}^{s+\frac{1}{2}}}
   			\\
   			& = 
   		 	\mu^{\frac{1}{4}} \gamma M 	|(\mathcal{G}^-_{\mu}[\ve \zeta])^{-1}\mathcal{G}^+_{\mu}[\ve \zeta] \psi^+|_{\mathring{H}^{s+\frac{1}{2}}}
   			\\
   			& 
   			\leq 
   			\sqrt{\mu} M |\psi^{+}|_{\dot{H}_{\mu}^{s+\frac{1}{2}}}
   			\\ 
   			& 
   			\leq 
   			\sqrt{\mu} M |\psi|_{\dot{H}^{s+\frac{1}{2}}_{\mu}}.
   		\end{align*} 
   \end{proof}

   We may now give the proof of the main result of the section.

   \begin{proof}[Proof of Proposition \ref{Prop G}] We prove each point individually in four separate steps. \\
   	
   	\noindent
   	\underline{Step 1.} The proof of estimate \eqref{1 Est G} and \eqref{Est G} follows by \eqref{Est G+}, \eqref{Est J}, and Plancherel's identity:    
   	\begin{align*}
   		|\mathcal{G}_{\mu}[\ve \zeta]\psi|_{H^{s-\frac{1}{2}}} 
   		 & =
   		|\mathcal{G}^+_{\mu}[\ve \zeta](\mathcal{J}_{\mu}[\ve \zeta])^{-1}\psi|_{H^{s-\frac{1}{2}}}
   		\\
   		&
   		\leq 
   		\mu^{\frac{3}{4}} |(\mathcal{J}_{\mu}[\ve \zeta])^{-1}\psi|_{\dot{H}^{s+\frac{1}{2}}_{\mu}}
   		\\ 
   		&
   		\leq 
   		\mu^{\frac{3}{4}} |\psi|_{\dot{H}^{s+\frac{1}{2}}_{\mu}} 
   		\\
   		&
   		\leq \sqrt{\mu}|\partial_x \psi|_{H^{s-\frac{1}{2}}}.
   	\end{align*}

   \noindent
   \underline{Step 2.} The proof of \eqref{2 Est G} and \eqref{Est G 2}, is proved the same way as in Step $1$, but we instead use \eqref{Est G+ mu}.\\

   \noindent
   \underline{Step 3.} The coercivity estimate \eqref{G Coercive} follows by construction where use the identities in  Remark \ref{rmk id psi- psi+} to get that
   \begin{align*}
   		\big{(} \psi, \mathcal{G}_{\mu}[\ve \zeta] \psi \big{)}_{L^2}
   		= 
   		\big{(} \psi, \mathcal{G}^+_{\mu}[\ve \zeta] (\mathcal{J}_{\mu}[\ve \zeta])^{-1}\psi \big{)}_{L^2}
   		 = 
   		\big{(} \mathcal{J}_{\mu}[\ve \zeta]\psi^+, \mathcal{G}^+_{\mu}[\ve \zeta]\psi^+ \big{)}_{L^2}.
   \end{align*}
   Now, argue as in the proof Lemma  \ref{Prop op J} where estimates \eqref{est psi +} and \eqref{Est J with psi} implies
   \begin{align*}
   		\big{(} \psi, \mathcal{G}_{\mu}[\ve \zeta] \psi \big{)}_{L^2}
   		\geq
   		\frac{\mu}{1+M}|\psi^+|^2_{\dot{H}^{\frac{1}{2}}_{\mu}}
   		\geq 
   		\frac{\mu}{1+M}|\psi|^2_{\dot{H}^{\frac{1}{2}}_{\mu}}.
   \end{align*}

	\noindent
	\underline{Step 4.} The symmetry follows by the second point in Propostions \ref{Prop G+ dual est} and \ref{Prop G- dual est}. \\
	
	\noindent
	\underline{Step 5.} Finally, for the proof estimate \eqref{G continuous form} we use \eqref{dual est G+ s} and Lemma \ref{Prop op J}.\\
   \end{proof}

	\section{Symbolic analysis of the Dirichlet-Neumann operator}\label{Sec Symbolic}
	In this section, we will give a symbolic description of the operator $\mathcal{G}_{\mu}$ defined given by \eqref{Def G}. The estimates need to be precise in terms of the parameters $\mu, \ve$, and where we carefully track the Sobolev regularity with respect to $\psi$ and $\zeta$. One reason for these expressions is that we need to have an estimate of the type:
	\begin{equation}\label{Motivation symbolic}
		|\big{(}\mathcal{G}_{\mu}[\ve \zeta]\circ  (\mathcal{G}^-_{\mu}[\ve \zeta])^{-1}\big{(} \partial_x(f g)\big{)},f
		\big{)}_{L^2(\R)}| \leq  M(t_0+3)|g|_{H^{t_0+1}}|f|^2_{L^2},
	\end{equation}
	which appears naturally in the energy estimates and the quasilinearisation of the main equations. As we can see from \eqref{Motivation symbolic}, one needs to absorb a derivative and be uniform with respect to the small parameters. Also, recall that	
	\begin{equation*}
		\mathcal{G}_{\mu}[\ve \zeta]\psi =   \mathcal{G}^+_{\mu}[\ve \zeta] \Big{(} 1 -   \gamma   (\mathcal{G}^-_{\mu}[\ve \zeta])^{-1} \color{black}\mathcal{G}^+_{\mu}[\ve \zeta] \Big{)}^{-1}\psi,
	\end{equation*}
	which means we need a good symbolic description of each of the operators involved in the expression. This is the strategy that was implemented in \cite{FirstWPWW_Lannes_05,LannesTwoFluid13}, where we first will consider the symbolic description of $\mathcal{G}^{\pm}_{\mu}$, then we treat the inverse operators that are involved.

	\subsection{Symbolic analysis of $\mathcal{G}_{\mu}^{\pm}$} For the symbolic description of $\mathcal{G}^+_{\mu}$, we know that the operator coincides with the ones studied in \cite{LannesTwoFluid13}. In particular, we can use Theorem $4$, in dimension one, which is one of the key estimates of the paper: 
	\begin{prop}[Theorem 4 in \cite{LannesTwoFluid13}]\label{Symbolic thm 4 Lannes}
		Let $t_0>\frac{1}{2}$ and $\zeta \in H^{t_0+3}(\R)$ be such that \eqref{non-cavitation} is satisfied. Then for all $0\leq s \leq t_0$ and $\psi \in \dot{H}_{\mu}^{s+\frac{1}{2}}(\R)$, one can approximate the positive Dirichlet-Neumann operator by
		\begin{equation*}
			\mathrm{Op}(S^+)\psi(x) = \sqrt{\mu}  \mathcal{F}^{-1}\Big{(}|\xi| \tanh(\sqrt{\mu} t(x, \xi)\hat{\psi}(\xi)\Big{)}(x),
		\end{equation*}
		where we define the \lq\lq tail\rq\rq\: by the symbol
		\begin{equation}\label{tail operator}
			t(x,\xi) = (1+\ve \zeta)\frac{\arctan(\ve \sqrt{\mu}\partial_x \zeta)}{\ve \sqrt{\mu}\partial_x \zeta}|\xi|.
		\end{equation}
		Moreover, for $k=0,1$, the operator satisfy
		\begin{equation}\label{Symbolic G+}
			|\mathcal{G}^+_{\mu}[\ve \zeta]\psi - \mathrm{Op}(S^+)\psi|_{H^{s+\frac{k}{2}}}
			\leq
			\ve \mu^{1-\frac{k}{4}}M(t_0 + 3)|\psi|_{\dot{H}_{\mu}^{s+\frac{1}{2}}}.
		\end{equation}
	\end{prop}

	In the case of infinite depth, it is pointed out in Remark $17$ in \cite{LannesTwoFluid13} that the tail effects vanish (formally) since the hyperbolic tangent is a bottom effect. In particular, we have the following result that proves this fact.

	\begin{prop}\label{Prop G- S-}
		Let $t_0\geq 1$ and $\zeta \in H^{t_0+3}(\R)$. Then for all $0\leq s \leq t_0$ and $\psi \in \mathring{H}^{s+\frac{1}{2}}(\R)$, one can approximate the negative Dirichlet-Neumann operator by
		\begin{equation*}
			S^-(\mathrm{D}) = -\sqrt{\mu} |\mathrm{D}|,
		\end{equation*}
		where we have the following estimate
		\begin{equation}\label{G- symbolic}
			|\mathcal{G}^-_{\mu}[\ve \zeta]\psi - \mathrm{Op}(S^-)\psi|_{H^{s+\frac{1}{2}}} \leq \ve \mu M(t_0 + 3)|\psi|_{\mathring{H}^{s+\frac{1}{2}}}.
		\end{equation}
	\end{prop}

	%
	%
	%

	\begin{remark}\label{Remark dim}
		The proof can be seen as a modified version of the proof presented in  \cite{LannesTwoFluid13}, but where we need to make two approximations of an elliptic problem depending on the frequencies and weighted estimates to deal with some integrability issues in $\mathcal{S}^- = \R \times [0,\infty)$. The weights and the cut-off functions are adapted to the approximation and deal with separate issues. We also note that the proof in \cite{LannesTwoFluid13} also extends to two dimensions. In fact, Proposition \ref{Prop G- S-} is mainly the result that restricts Theorem \ref{Thm 1}  to one horizontal dimension in this paper. However, we expect the result to be true for $X \in \R^2$ by modifying the ansatz in the proof and letting $\mathrm{Op}(S^-)$ be given by
		\begin{equation*}
			\mathrm{Op}(S^-)\psi(X) = -\sqrt{\mu}\mathcal{F}^{-1}\Big{(}\sqrt{|\xi|^2 +  \ve^2\mu\big{(} (|\nabla_X \zeta|^2 |\xi|^2)-(\nabla_X \zeta\cdot \xi)^2\big{)}}\hat{\psi}(\xi)\Big{)}(X),
		\end{equation*}
		for $X,\xi \in \R^2$. 
	\end{remark}
	\begin{proof}
		The proof will be given in several steps, where we first decompose the estimate into two main parts depending on the frequencies. In particular, let $\chi_1$ be a Fourier multiplier with a smooth symbol and equal to one around zero. Also, let $\chi_2 = 1- \chi_1$ be the high-frequency part. Then we have by duality
		\begin{align}\notag
			|\mathcal{G}^-_{\mu}[\ve \zeta]\psi + \sqrt{\mu} |\mathrm{D}|\psi|_{H^{s+\frac{1}{2}}} 
			& = 
			\sup\limits_{ |\varphi|_{L^2}=1}
			\Big{|}
			\int_{\R} \Lambda^{s+\frac{1}{2}}\big{(} \mathcal{G}^-_{\mu}[\ve \zeta]\psi + \sqrt{\mu}|\mathrm{D}|\psi\big{)} \chi_1(\mathrm{D}) \varphi \: \mathrm{d}x
			\\ 
			& \hspace{2cm}\notag
			+
			\int_{\R} \Lambda^{s+\frac{1}{2}}\big{(} \mathcal{G}^-_{\mu}[\ve \zeta]\psi +\sqrt{\mu} |\mathrm{D}|\psi\big{)} \chi_2(\mathrm{D})\varphi \: \mathrm{d}x\Big{|}
			\\ 
			& \leq  
			\sup\limits_{ |\varphi|_{L^2}=1} \Big{(} |I_1(\varphi)| + |I_2(\varphi)|\Big{)}.
		\end{align}
		We will now make two approximations of the elliptic problem \eqref{Elliptic pb G-}, where the estimate on $I_1$ can be made using the Poisson kernel. On the other hand, the estimate on $I_2$ requires an approximate solution of \eqref{Elliptic pb G-} that accounts for the principal part of the elliptic operator and is \lq\lq well-behaved\rq\rq\: in high frequency on weighted Sobolev spaces. \\

		\noindent
		\underline{Step 1.} \textit{Approximate solutions for $I_1$}. We know that $\phi_{\mathrm{app}}^1 = e^{-z\sqrt{\mu}|\mathrm{D}|}\psi$ solves
		\begin{equation}\label{Phi app 1}
			\begin{cases}
				(\mu \partial_x^2 + \partial_z^2)\phi_{\text{app}}^1 = 0 \quad \text{in} \quad  \mathcal{S}^-
				\\
				\phi_{\text{app}}^1|_{z=0} = \psi,
			\end{cases}
		\end{equation}
		and
		\begin{align*}
			\partial_z \phi_{\mathrm{app}}^1 |_{z=0} = - \sqrt{\mu}|\mathrm{D}|\psi.
		\end{align*}
		Now, let $\phi^-$ be the solution of \eqref{Elliptic pb G-} and define $u_1 = \phi^- - \phi_{\mathrm{app}}^1$. Then $u_1$ solves
		\begin{equation}\label{Pb u_1}
			\begin{cases}
				\nabla_{x,z}^{\mu} \cdot P(\Sigma^-)\nabla_{x,z}^{\mu} u_1 = \ve \mu r_1 \quad \text{in} \quad  \mathcal{S}^-
				\\
				u_1|_{z=0} = 0, \quad \lim \limits_{z\rightarrow \infty} |\nabla^{\mu}_{x,z} u_1| = 0,
			\end{cases}
		\end{equation}
		where $r_1$ reads
		\begin{align*}
			r_1 
			& = 
			\partial_x\big{(}(\partial_x \zeta)  \partial_z \phi_{\mathrm{app}}^1\big{)} + (\partial_x \zeta)\partial_x\partial_z \phi_{\mathrm{app}}^1 - \ve (\partial_x \zeta)^2 \partial_z^2 \phi_{\mathrm{app}}^1.
		\end{align*}
		Moreover, we have by construction that
		\begin{align*}
			\partial_n^{P^-}u_1 |_{z=0} = \mathcal{G}^-_{\mu}[\ve \zeta]\psi + \sqrt{\mu} |\mathrm{D}|\psi + \ve \mu r_2,
		\end{align*}
		where the second rest is given by
		\begin{equation*}
			r_2  = (\partial_x\zeta) \partial_x \psi + \ve \sqrt{ \mu} (\partial_x \zeta)^2 |\mathrm{D}| \psi.
		\end{equation*}
		\noindent
		\underline{Step 1.1.} \textit{Estimate on $I_1$}. From the previous step, we deduce that
		\begin{align*}
			|I_1(\varphi)| 
			& \leq
			\Big{|} \int_{\R} (\Lambda^{s+\frac{1}{2}}\partial_n^{P^-}u_1 |_{z=0}) \chi_1(\mathrm{D})\varphi \: \mathrm{d}x \Big{|} 
			+
			\ve \mu \Big{|} \int_{\R} (\Lambda^{s+\frac{1}{2}}r_2)  \chi_1(\mathrm{D})\varphi  \: \mathrm{d}x\Big{|}
			\\
			& = 
			|I_1^1(\varphi)| + |I_1^2(\varphi)|.
		\end{align*}
		We will now estimate each piece separately. For the estimate on $|I_1^1(\varphi)|$, we use the Divergence Theorem to deduce that
		\begin{align*}
			I_1^1(\varphi) 
			& = 
			\int_{\mathcal{S}^-} P(\Sigma^-)\nabla_{x,z}^{\mu}u_1 \cdot (\Lambda^{s+\frac{1}{2}}\chi_1(\mathrm{D})\nabla_{x,z}^{\mu}\varphi^{\mathrm{ext_1}})  \: \mathrm{d}x\mathrm{d}z
			+
			\ve \mu
			\int_{\mathcal{S}^-} r_1 (\Lambda^{s+\frac{1}{2}}\chi_1(\mathrm{D}) \varphi^{\mathrm{ext_1}}) \: \mathrm{d}x\mathrm{d}z
			\\
			& =
			I_1^{1,1}(\varphi) + I_1^{1,2}(\varphi).
		\end{align*}
		Here we let $\varphi^{\mathrm{ext_1}}$ be the extension of $\varphi$ onto the upper half-plane defined by
		\begin{equation*}
			\varphi^{\mathrm{ext_1}}(x,z) = \eta(\sqrt{\mu} z) \varphi(x),
		\end{equation*}
		where $\eta(s)$ is a positive smooth function supported on $[0,1]$. Then we estimate each term separately. For the estimate on $I_1^{1,1}$, we first use Hölder inequality and Sobolev embedding to deduce that
		\begin{align}\label{I_1^1,1 to be est}
				|I_1^{1,1}(\varphi)|
				& \leq \sqrt{\mu}(|\eta|_{L^{\infty}}+|\eta'|_{L^{\infty}})
				M \|\nabla_{x,z}^{\mu} u_1 \|_{L^2(\mathcal{S}^-)} |\varphi|_{L^2}.
		\end{align}
		To conclude this estimate, we use \eqref{Coercivity-} and the Divergence Theorem to make the following observation:
		\begin{align*}
			\|\nabla_{x,z}^{\mu} u_1 \|_{L^2(\mathcal{S}^-)}^2 
			& \leq 
			M\int_{\mathcal{S}^-} P(\Sigma^-)\nabla_{x,z}^{\mu} u_1 \cdot \nabla_{x,z}^{\mu} u_1 \: \mathrm{d}x \mathrm{d}z
			\\ 
			& \leq 	\ve \mu
			M  \Big{|}
			\int_{\mathcal{S}^-} r_1 u_1 \: \mathrm{d}x \mathrm{d}z\Big{|}.
		\end{align*}
		Then use the definition of $r_1$ and integration by parts to distribute the derivatives in a convenient way:
		\begin{align*}
			\|\nabla_{x,z}^{\mu} u_1 \|_{L^2(\mathcal{S}^-)}^2 
			& \leq 
			\ve \mu \Big{|}\int_{\mathcal{S}^-} (\partial_x \zeta)  (\partial_z \phi_{\mathrm{app}}^1)\partial_xu_1\: \mathrm{d}x \mathrm{d}z\Big{|}
			+
			\ve \mu \Big{|}\int_{\mathcal{S}^-}(\partial_x \zeta)(\partial_x \phi_{\mathrm{app}}^1)\partial_zu_1\: \mathrm{d}x \mathrm{d}z\Big{|}
			\\ 
			& 
			\hspace{0.5cm}
			+
			\ve^2 \mu \Big{|}\int_{\mathcal{S}^-}(\partial_x \zeta)^2 \partial_z \phi_{\mathrm{app}}^1 \partial_z u_1\: \mathrm{d}x \mathrm{d}z\Big{|}.
		\end{align*}
		 In particular, if we apply Hölder's inequality, Sobolev embedding, and integrate the definition of $ \phi_{\mathrm{app}}^1$ in $z$, we find that
		\begin{align*}
			\|\nabla_{x,z}^{\mu} u_1 \|_{L^2(\mathcal{S}^-)} 
			& \leq
			\ve \mu M  \| e^{-z\sqrt{\mu}|\mathrm{D}|}|\mathrm{D}|\psi\|_{L^2(\mathcal{S}^-)}
			\\
			& \leq
			\ve \mu^{\frac{3}{4}} M |\psi|_{\mathring{H}^{\frac{1}{2}}}
			,
		\end{align*}
		and we use it to can contol \eqref{I_1^1,1 to be est} by
		\begin{align*}
			|I_1^{1,1}(\varphi)|
			& \leq
			\ve \mu M |\psi|_{\mathring{H}^{\frac{1}{2}}} |\varphi|_{L^2}.
		\end{align*}
		Moreover, the estimate on $I_1^{1,2}(\varphi)$ contains $r_1$, and can be estimated similarly by integrating only in $x$ and using the cut-off $\chi_{1}$ to obtain that
			\begin{align*}
			|I_1^{1,2}(\varphi)|
			& \leq
			\ve \mu\Big{(} \Big{|}\int_{\mathcal{S}^-} (\partial_x \zeta)  (\partial_z \phi_{\mathrm{app}}^1)\partial_x(\Lambda^{s+\frac{1}{2}}\chi_1(\mathrm{D}) \varphi^{\mathrm{ext_1}}) \: \mathrm{d}x \mathrm{d}z\Big{|}
			\\ 
			& 
			\hspace{0.5cm}
			+
			\Big{|}\int_{\mathcal{S}^-}(\partial_z \phi_{\mathrm{app}}^1)\partial_x\big((\partial_x \zeta)\Lambda^{s+\frac{1}{2}}\chi_1(\mathrm{D}) \varphi^{\mathrm{ext_1}}\big) \: \mathrm{d}x \mathrm{d}z\Big{|}
			\\ 
			& 
			\hspace{0.5cm}
			+
			\ve \mu\Big{|}\int_{\mathcal{S}^-} (\partial_x e^{-z\sqrt{\mu}|\mathrm{D}|}\psi) \partial_x \big((\partial_x \zeta)^2\Lambda^{s+\frac{1}{2}}\chi_1(\mathrm{D}) \varphi^{\mathrm{ext_1}} \big) \: \mathrm{d}x \mathrm{d}z\Big{|}\Big{)}
			\\ 
			& 
			\leq \ve \mu
			M  |\psi|_{\mathring{H}^{\frac{1}{2}}} |\varphi|_{L^2}.
		\end{align*}
		Finally, we arrive at the estimate of $I_1^2(\varphi)$. Here we use the definition of $r_2$ to find that 
		\begin{align*}
			|I_1^2(\varphi)|
			& \leq 
			\ve \mu \Big{|} \int_{\R}  (\partial_x\zeta) (\partial_x \psi)  (\Lambda^{s+\frac{1}{2}} \chi_1(\mathrm{D})\varphi)  \: \mathrm{d}x\Big{|}
			+
			\ve^2\mu^{\frac{3}{2}} \Big{|} \int_{\R}   (\partial_x \zeta)^2 (|\mathrm{D}| \psi) (\Lambda^{s+\frac{1}{2}} \chi_1(\mathrm{D})\varphi)  \: \mathrm{d}x\Big{|}
			\\ 
			& = 
			|I_1^{2,1}(\varphi)| + |I_1^{2,2}(\varphi)|.
		\end{align*}
		Each term is estimated similarly, and for the estimate of the first term, we recall $\partial_x = - \mathcal{H}|\mathrm{D}|$, where $\mathcal{H}$ is the Hilbert transform. Consequently, we have from Hölder's inequality, Sobolev embedding, and commutator estimate \eqref{Commutator Dhalf} that
		\begin{align*}
			|I_1^{2,1}(\varphi)|
			& 
			\leq 
			\ve \mu \Big{|} \int_{\R}  [\partial_x\zeta, |\mathrm{D}|^{\frac{1}{2}}] \mathcal{H}|\mathrm{D}|^{\frac{1}{2}} \psi)  (\Lambda^{s+\frac{1}{2}} \chi_1(\mathrm{D})\varphi)  \: \mathrm{d}x\Big{|}
			\\ 
			& 
			\hspace{0.5cm}
			+
			\ve \mu \Big{|} \int_{\R}  (\partial_x\zeta) (\mathcal{H}|\mathrm{D}|^{\frac{1}{2}} \psi)  (\Lambda^{s+\frac{1}{2}} |\mathrm{D}|^{\frac{1}{2}}\chi_1(\mathrm{D})\varphi)  \: \mathrm{d}x\Big{|}
			\\ 
			& \leq 
			\ve \mu
			M |\psi|_{\mathring{H}^{\frac{1}{2}}}|\varphi|_{L^2}.
		\end{align*}
		We may now collect all these estimates to conclude this step:
		\begin{align*}
			|I_1(\varphi)| \leq \ve \mu M |\psi|_{\mathring{H}^{\frac{1}{2}}}|\varphi|_{L^2}.
		\end{align*}

		\noindent
		\underline{Step 2.} \textit{Approximate solutions for $I_2$}. By using the approach in \cite{FirstWPWW_Lannes_05,LannesTwoFluid13} we can show that the function:
		\begin{equation}\label{phi app 2}
			\phi_{\text{app}}^2(x,z)
			= 
			(\text{Op}(L) \psi) (x,z)
			= 
			\mathcal{F}^{-1}\Big{(} L(x,\xi,z)\hat{\psi}(\xi)\Big{)}(x),
		\end{equation} 
		where $L$ is a symbol given by
		\begin{equation*}
			L(x,\xi, z) = e^{-z\big( \frac{	\sqrt{\mu}|\xi|}{1+\ve^2\mu(\partial_x \zeta)^2} - i \frac{\ve \mu \partial_x \zeta \xi}{1+\ve^2\mu(\partial_x \zeta)^2}\big)},
		\end{equation*}
		provides a good approximation of the elliptic problem \eqref{Elliptic pb G-} and, by extension, an approximation of the Neumann condition at the boundary. To see this, we use Definition \ref{diffeomorphism +-} for $P(\Sigma^-)$, and \eqref{phi app 2} to compute $(\partial_n^{P^-} \phi^2_{\text{app}})|_{z=0}$:
		\begin{align*}
			(\partial_n^{P^-} \phi^2_{\text{app}})|_{z=0} 
			& 
			=
			\mathbf{e}_z \cdot P(\Sigma^-)\nabla_{x,z}^{\mu}\phi^2_{\text{app}}|_{z=0} 
			\\ 
			& 
			= 
			- 
			\varepsilon \mu (\partial_x\zeta )\partial_x \phi^2_{\text{app}}|_{z=0}  
			+
			\big{(}1+\varepsilon^2 \mu (\partial_x\zeta)^2\big{)}\partial_z \phi^2_{\text{app}}|_{z=0} 
			\\ 
			& 
			= 
			- 
			\varepsilon \mu (\partial_x\zeta )\mathrm{Op}(L)\partial_x\psi|_{z=0}
			+
			\big{(}1+\varepsilon^2 \mu (\partial_x\zeta)^2\big{)}\mathrm{Op}(\partial_zL)\psi|_{z=0}
			\\ 
			&
			 =
			 -\sqrt{\mu} |\mathrm{D}|\psi.
		\end{align*}		
		where used that $\partial_x L |_{z = 0} = 0$. \color{black} 	In fact, the approximation $\phi_{\text{app}}^2$ is constructed from the solution of the following ODE:
		\begin{equation*}
			(1+\ve^2\mu (\partial_x \zeta)^2) \partial_z^2 L - 2 i \ve \mu (\partial_x \zeta)\xi \partial_z L - \mu \xi^2 L = 0,
		\end{equation*}
		which corresponds to the principal part of  \eqref{Elliptic pb G-} with \lq\lq frozen coefficients\rq\rq. As a result, we have that $u_2 = \phi^- - \phi_{\mathrm{app}}^2$ solves
		\begin{equation}\label{Elliptic phi app 2}
			\begin{cases}
				\nabla_{x,z}^{\mu} \cdot P(\Sigma^-)\nabla_{x,z}^{\mu} u_2 = \ve \mu r_3 \quad \text{in} \quad  \mathcal{S}^-
				\\
				u_2|_{z=0} = 0, \quad \lim \limits_{z\rightarrow \infty} \omega(z)|\nabla^{\mu}_{x,z} u_2| = 0,
			\end{cases}
		\end{equation}
		where we simply take $\omega(z) = e^{-\frac{z}{2}}$ and the rest term can be computed explicitly:
		\begin{align*}
			 r_3 
			= 
			2  \text{Op}((\partial_x \zeta) \partial_x\partial_z L) \psi
			- 
			2i \frac{1}{\ve}\text{Op}(\partial_xL\xi) \psi
			-
			\frac{1}{\ve}\text{Op}(\partial_x^2L) \psi
			+
			\mathrm{Op}\big{(}(\partial_x^2\zeta)\partial_zL\big{)}\psi.
		\end{align*}
		One should note that derivatives in $x$ on $L$ give rise to  $\ve \mu$ and is polynomial in $z$. The polynomial dependence in $z$ will require weighted estimates in high frequency to get an estimate for $\psi \in \mathring{H}^{s+\frac{1}{2}}(\R)$.\\
	
		\noindent
		\underline{Step 2.2.} \textit{Estimate on $I_2$}. By construction and the Divergence Theorem, we have that
		\begin{align*}
			I_2(\varphi) 
			& = 
			\int_{\R} (\Lambda^{s+\frac{1}{2}}\partial_n^{P^-}u_2|_{z=0})
			\chi_2(\mathrm{D})\varphi \: \mathrm{d}x
			\\ 
			& = 
			\int_{\mathcal{S}^-} \Lambda^{s+1}P(\Sigma^-)\nabla_{x,z}^{\mu}u_2 \cdot (\Lambda^{-\frac{1}{2}}\chi_2(\mathrm{D})\nabla_{x,z}^{\mu}\varphi^{\mathrm{ext_2}})  \: \mathrm{d}x\mathrm{d}z
			\\ 
			& 
			\hspace{0.5cm}
			+
			\ve \mu
			\int_{\mathcal{S}^-} (\Lambda^{s}r_3) (\Lambda^{\frac{1}{2}}\chi_2(\mathrm{D}) \varphi^{\mathrm{ext_2}}) \: \mathrm{d}x\mathrm{d}z
			\\
			& =
			I_2^{1}(\varphi) + I_2^{2}(\varphi).
		\end{align*}
		We now need to incorporate a weight in the estimates. To do so, we let $\varphi^{\mathrm{ext_2}}$ be the extension of $\varphi$ onto the upper half-plane defined by
		\begin{equation*}
			\varphi^{\mathrm{ext_2}}(x,z) = (\omega_{\mu}(z))^2 e^{-z\sqrt{\mu}|\mathrm{D}|}\varphi(x).
		\end{equation*}
		Here we let $\omega_{\mu}(z) = e^{-\frac{\sqrt{\mu} z}{2M}}$. 
		For the estimate on $I_2^{1}(\varphi)$ we first make the decomposition
		\begin{align*}
			I_2^{1}(\varphi)
			& = 
			\int_{\mathcal{S}^-} \omega_{\mu}\chi_2(\mathrm{D})\Lambda^{s+1}P(\Sigma^-)\nabla_{x,z}^{\mu}u_2 \cdot (\Lambda^{-\frac{1}{2}} (\omega_{\mu})^{-1}\nabla_{x,z}^{\mu}\varphi^{\mathrm{ext_2}})  \: \mathrm{d}x\mathrm{d}z
			\\
			& = 
			\int_{\mathcal{S}^-} [\chi_2(\mathrm{D})\Lambda^{s+1},P(\Sigma^-)]\omega_{\mu}\nabla_{x,z}^{\mu}u_2 \cdot (\Lambda^{-\frac{1}{2}}(\omega_{\mu})^{-1}\nabla_{x,z}^{\mu}\varphi^{\mathrm{ext_2}})  \: \mathrm{d}x\mathrm{d}z
			\\ 
			& 
			\hspace{0.5cm}
			+
			\int_{\mathcal{S}^-} P(\Sigma^-) \Lambda^{s+1}\chi_2(\mathrm{D})\omega_{\mu}\nabla_{x,z}^{\mu}u_2 \cdot (\Lambda^{-\frac{1}{2}}(\omega_{\mu})^{-1}\nabla_{x,z}^{\mu}\varphi^{\mathrm{ext_2}})  \: \mathrm{d}x\mathrm{d}z
			\\ 
			& =
			I_2^{1,1}(\varphi) + I_2^{1,2}(\varphi).
		\end{align*}
		For the control of $I_2^{1,1}(\varphi)$, we use Cauchy-Schwarz, commutator estimate \eqref{Commutator est}, and the half-derivative smoothing from the Poisson kernel:
		\begin{equation}\label{smoothing}
			 \|  |\mathrm{D}|^{\frac{1}{2}} e^{-z\sqrt{\mu}|\mathrm{D}|}\varphi \|_{L^2(\mathcal{S}^-)} \leq \mu^{-\frac{1}{4}}|\varphi|_{L^2},
		\end{equation}
		and the weight estimate:
		\begin{equation*}
			\| \omega_{\mu}(z) \varphi \|_{L^2(\mathcal{S}^-)} \leq \mu^{-\frac{1}{4}}M|\varphi|_{L^2},
		\end{equation*}
		to get that
		\begin{align*}
			|I_2^{1,1}(\varphi)|
			& \leq
			\ve \sqrt{\mu} M \| \omega_{\mu}\Lambda^{s} \nabla_{x,z}^{\mu} u_2\|_{L^2(\mathcal{S}^-)} \| \Lambda^{-\frac{1}{2}}(\omega_{\mu})^{-1}\nabla_{x,z}^{\mu}\varphi^{\mathrm{ext_2}}) \|_{L^2(\mathcal{S}^-)} 
			\\
			& \leq 
			\ve \sqrt{\mu} 
			M \| \omega_{\mu}\Lambda^{s} \nabla_{x,z}^{\mu} u_2\|_{L^2(\mathcal{S}^-)} \mu^{\frac{1}{4}} |\varphi |_{L^2}.
		\end{align*}
		We also use the definition $u_2 = \phi^- - \phi_{\mathrm{app}}^2$, the decay of the weight,  estimates \eqref{est phi -} and the ones provided in Lemma \ref{est Op L} to obtain	
		\begin{align*}
			 \| \omega_{\mu} \Lambda^{s} \nabla_{x,z}^{\mu}u_2\|_{L^2(\mathcal{S}^-)} & \leq  \| \Lambda^{s} \nabla_{x,z}^{\mu} \phi \|_{L^2(\mathcal{S}^-)} 
			 +
			 \|\frac{1}{z} \Lambda^{s} \mathrm{Op}(\partial_x L )\psi\|_{L^2(\mathcal{S}^-)} 
			 \\ 
			 & 
			 \hspace{0.5cm}
			 +
			 \sqrt{\mu}\| \Lambda^{s} \mathrm{Op}( L)\partial_x\psi\|_{L^2(\mathcal{S}^-)} 
			 +
			 \| \Lambda^{s} \mathrm{Op}( \partial_z L)\psi\|_{L^2(\mathcal{S}^-)}
			 \\
			 & \leq \mu^{\frac{1}{4}} M |\psi|_{\mathring{H}^{s+\frac{1}{2}}}.
		\end{align*}
		We therefore have
		\begin{align*}
				|I_2^{1,1}(\varphi)| \leq \ve \mu M |\psi|_{\mathring{H}^{s+\frac{1}{2}}} |\varphi |_{L^2}.
		\end{align*}
		For the estimate on $ I_2^{1,2}(\varphi)$, we use Hölder's inequality and Sobolev embedding to find that
		\begin{align*}
			|I_2^{1,2}(\varphi)| 
			& 
			\leq \mu^{\frac{1}{4}}
			M \|\Lambda^{s+1}\chi_2(\mathrm{D})\omega_{\mu}\nabla_{x,z}^{\mu} u_2\|_{L^2(S^-)}
			| \varphi|_{L^2}.
		\end{align*}
		So we have again reduced the problem to an elliptic estimate. Using the coercivity of $P(\Sigma^-)$ and integration by parts we find that
		\begin{align}
			 A \label{A norm}
			 & = \|\Lambda^{s+1}\chi_2(\mathrm{D})\omega_{\mu}\nabla_{x,z}^{\mu} u_2 \|_{L^2(S^-)}^2
			 \\ 
			 & \leq  \notag
			M  \Big{|}\int_{\mathcal{S}^-}
			 P(\Sigma^-)\Lambda^{s+1}\chi_2(\mathrm{D})\omega_{\mu}\nabla_{x,z}^{\mu} u_2 \cdot \Lambda^{s+1}\chi_2(\mathrm{D})\omega_{\mu}\nabla_{x,z}^{\mu} u_2 \: \mathrm{d}x \mathrm{d}z \Big{|}
			 \\ 
			 & \leq  \notag
			 M \Big{|}\int_{\mathcal{S}^-}
			 (\nabla_{x,z}^{\mu}\cdot  P(\Sigma^-)\Lambda^{s+1}\chi_2(\mathrm{D})\omega_{\mu}\nabla_{x,z}^{\mu} u_2)  (\Lambda^{s+1}\chi_2(\mathrm{D})\omega_{\mu}u_2) \: \mathrm{d}x \mathrm{d}z \Big{|}
			 \\ 
			 &  \notag
			 \hspace{0.5cm}
			 +
			 M \Big{|}\int_{\mathcal{S}^-}
			 (\mathbf{e}_z \cdot P(\Sigma^-)\Lambda^{s+1}\chi_2(\mathrm{D})\omega_{\mu}\nabla_{x,z}^{\mu} u_2)  (\Lambda^{s+1}\chi_2(\mathrm{D})\omega'_{\mu}u_2) \: \mathrm{d}x \mathrm{d}z \Big{|}
			 \\  \notag
		 	 & = M \cdot (A_1 + A_2). 
		\end{align}
		Before estimating $A_1$, we will decompose it into three pieces:
		\begin{align*}
			A_1 
			& = 
			\Big{|}\int_{\mathcal{S}^-}
			(\nabla_{x,z}^{\mu}\cdot  [P(\Sigma^-),\Lambda^{s+1}\chi_2(\mathrm{D})]\omega_{\mu}\nabla_{x,z}^{\mu}u_2)  (\Lambda^{s+1}\chi_2(\mathrm{D})\omega_{\mu}u_2) \: \mathrm{d}x \mathrm{d}z \Big{|}
			\\ 
			& \hspace{0.5cm}	
			+
			\ve \mu
			\Big{|}\int_{\mathcal{S}^-}
			(\Lambda^{s+1}\chi_2(\mathrm{D})\omega_{\mu}r_3)  (\Lambda^{s+1}\chi_2(\mathrm{D})\omega_{\mu}u_2) \: \mathrm{d}x \mathrm{d}z \Big{|}
			\\ 
			& \hspace{0.5cm}
			+
			\Big{|}\int_{\mathcal{S}^-}
			(\Lambda^{s+1}\chi_2(\mathrm{D})(\omega_{\mu})' \mathbf{e}_z \cdot  P(\Sigma^-) \nabla_{x,z}^{\mu} u_2)  (\Lambda^{s+1}\chi_2(\mathrm{D})\omega_{\mu}u_2) \: \mathrm{d}x \mathrm{d}z \Big{|}
			\\ 
			& = 
			A_1^1 + A_1^2 + A_1^3.
		\end{align*}
		For the estimate on $A_1^1$, we write out each term explicitly to clearly see the terms we need to treat. In particular, we have
		\begin{align*}
			\nabla_{x,z}^{\mu} \cdot  [P(\Sigma^-), \Lambda^{s+1}\chi_2&(\mathrm{D})]\omega_{\mu}\nabla_{x,z}^{\mu} u_2 
			\\ 
			& = 
			-\ve \mu \partial_x[\partial_x \zeta, \Lambda^{s+1}\chi_2(\mathrm{D})]\omega_{\mu}\partial_z u_2
			-
			\ve \mu \partial_z[\partial_x \zeta, \Lambda^{s+1}\chi_2(\mathrm{D})]\omega_{\mu}\partial_x u_2
			\\ 
			& \hspace{0.5cm}
			+
			\ve^2\mu \partial_z[(\partial_x \zeta)^2, \Lambda^{s+1}\chi_2(\mathrm{D})]\omega_{\mu}\partial_z u_2.
		\end{align*}
		Then using this decomposition together with Plancherel's identity, integration by parts, commutator estimate \eqref{Commutator est}, and the support of $\chi_2$,  we find that
		\begin{align*}
			A_1^1 
			& \leq 
			\ve \mu  \Big{|}\int_{\mathcal{S}^-}
			(\Lambda^{-1}\partial_x[\partial_x \zeta, \Lambda^{s+1}\chi_2(\mathrm{D})]\omega_{\mu}\partial_z u_2)  (\Lambda^{s+2}\chi_2(\mathrm{D})\omega_{\mu}u_2) \: \mathrm{d}x \mathrm{d}z \Big{|}
			\\ 
			&
			\hspace{0.5cm}
			+
			\ve \mu  \Big{|}\int_{\mathcal{S}^-}
			([\partial_x^2 \zeta, \Lambda^{s+1}\chi_2(\mathrm{D})]\omega_{\mu}\partial_z u_2)  (\Lambda^{s+1}\chi_2(\mathrm{D})\omega_{\mu}u_2) \: \mathrm{d}x \mathrm{d}z \Big{|}
			\\ 
			&
			\hspace{0.5cm}
			+
			\ve \mu  \Big{|}\int_{\mathcal{S}^-}
			([\partial_x \zeta, \Lambda^{s+1}\chi_2(\mathrm{D})]\omega_{\mu}\partial_z u_2)  (\Lambda^{s+1}\chi_2(\mathrm{D})\omega_{\mu}\partial_x u_2) \: \mathrm{d}x \mathrm{d}z \Big{|}
			\\ 
			&
			\hspace{0.5cm}
			+
			\ve^2\mu \Big{|}\int_{\mathcal{S}^-}
			( [\partial_x \zeta, \Lambda^{s+1}\chi_2(\mathrm{D})]\omega_{\mu}\partial_z u_2)
			(\Lambda^{s+1}\chi_2(\mathrm{D})\partial_z(\omega_{\mu} u_2)) \: \mathrm{d}x \mathrm{d}z \Big{|}
			\\
			& 
			\leq \ve \mu  M 
			\| \Lambda^s\partial_z u_2\|_{L^2(\mathcal{S}^-)}
			\|\Lambda^{s+1}\chi_2(\mathrm{D})\omega_{\mu}\nabla_{x,z} u_2 \|_{L^2(S^-)}.
		\end{align*}
		Then to conclude this part, we simply use that $u_2 = \phi^- - \phi_{\mathrm{app}}^2$, together with the estimates \eqref{est phi -} and Lemma \ref{est Op L}:
		\begin{align*}
			A_1^1 
			& \leq \ve \mu^{1+\frac{1}{4}}  M |\psi|_{\mathring{H}^{s+\frac{1}{2}}}
			\|\Lambda^{s+1}\chi_2(\mathrm{D})\omega_{\mu}\nabla_{x,z} u_2 \|_{L^2(S^-)}
			\\
			& \leq \ve \mu^{\frac{3}{4}}  M |\psi|_{\mathring{H}^{s+\frac{1}{2}}}
			\|\Lambda^{s+1}\chi_2(\mathrm{D})\omega_{\mu}\nabla_{x,z}^{\mu} u_2 \|_{L^2(S^-)}.
		\end{align*}
		Next, we make an estimate on $A_1^2$ using the definition of $r_3$ in the previous step:
		\begin{align*}
			A_1^2 
			& \leq 
			2\ve \mu
			\Big{|}\int_{\mathcal{S}^-}
			(\Lambda^{s}\chi_2(\mathrm{D})\omega_{\mu} \text{Op}((\partial_x \zeta) \partial_x\partial_z L) \psi)  (\Lambda^{s+2}\chi_2(\mathrm{D})\omega_{\mu}u_2) \: \mathrm{d}x \mathrm{d}z \Big{|}
			\\ 
			& 
			\hspace{0.5cm}
			+
			2\mu 
			\Big{|}\int_{\mathcal{S}^-}
			(\Lambda^{s}\chi_2(\mathrm{D})\omega_{\mu}\text{Op}(\partial_xL) \partial_x \psi)  (\Lambda^{s+2}\chi_2(\mathrm{D})\omega_{\mu}u_2) \: \mathrm{d}x \mathrm{d}z \Big{|}
			\\ 
			& 
			\hspace{0.5cm}
			+\ve \mu 
			\Big{|}\int_{\mathcal{S}^-}
			(\Lambda^{s}\chi_2(\mathrm{D})\omega_{\mu} \mathrm{Op}\big{(}(\partial_x^2\zeta)\partial_zL\big{)}\psi)  (\Lambda^{s+2}\chi_2(\mathrm{D})\omega_{\mu}u_2) \: \mathrm{d}x \mathrm{d}z \Big{|}
			\\ 
			& 
			\hspace{0.5cm}
			+\mu
			\Big{|}\int_{\mathcal{S}^-}
			(\Lambda^{s}\chi_2(\mathrm{D})\omega_{\mu}\text{Op}(\partial_x^2L) \psi)  (\Lambda^{s+2}\chi_2(\mathrm{D})\omega_{\mu}u_2) \: \mathrm{d}x \mathrm{d}z \Big{|}
			\\ 
			& = 
			A_{1}^{2,1}
			+
			A_{1}^{2,2}
			+
			A_{1}^{2,3}
			+
			A_{1}^{2,4}.
		\end{align*}
		The first three terms are easily estimated by Lemma \ref{est Op L}, with Remark \ref{remark op L}, and the support of $\chi_2$ where we find that
		\begin{align*}
			A_{1}^{2,1}
			+
			A_{1}^{2,2}
			+
			A_{1}^{2,3} 
			& \leq  
			\ve \mu^{1+\frac{1}{4}} M |\psi|_{\mathring{H}^{s+\frac{1}{2}}}
			\|\Lambda^{s+1}\chi_2(\mathrm{D})\omega_{\mu}\nabla_{x,z} u_2 \|_{L^2(S^-)}
			\\
			& \leq
			\ve \mu^{\frac{3}{4}} M |\psi|_{\mathring{H}^{s+\frac{1}{2}}}
			\|\Lambda^{s+1}\chi_2(\mathrm{D})\omega_{\mu}\nabla_{x,z}^{\mu} u_2 \|_{L^2(S^-)}.
		\end{align*}
		For the last term, we use the weight to gain decay in $z$. Then apply the same estimates as above to find that
		\begin{align*}
			A_{1}^{2,4} 
			& \leq  
			\sqrt{\mu} \|\Lambda^{s}\chi_2(\mathrm{D})\text{Op}(\frac{1}{z}\partial_x^2L) \psi\|_{L^2(\mathcal{S}^-)} 
			\|\Lambda^{s+2}\chi_2(\mathrm{D})\omega_{\mu}u_2\|_{L^2(\mathcal{S}^-)}
			\\
			& \leq  
			\ve \mu^{\frac{3}{4}} M(t_0+3) |\psi|_{\mathring{H}^{s+\frac{1}{2}}} 
			\|\Lambda^{s+1}\chi_2(\mathrm{D})\omega_{\mu}\nabla_{x,z}^{\mu} u_2\|_{L^2(\mathcal{S}^-)}.
		\end{align*}
		As a result, we have the same estimate on $A_1^2$, and so we proceed with the estimate on $A_1^3$. We observe by definition of $P(\Sigma^-)$ that we have one term without $\ve$ that need special attention:
		\begin{align*}
			A_1^3
			& \leq 
			 \frac{\ve \mu^{\frac{3}{2}}}{2M} \Big{|}\int_{\mathcal{S}^-}
			(\Lambda^{s}\chi_2(\mathrm{D})\omega_{\mu} (\partial_x \zeta) \partial_x u_2)  (\Lambda^{s+2}\chi_2(\mathrm{D})\omega_{\mu}u_2) \: \mathrm{d}x \mathrm{d}z \Big{|}
			\\ 
			& \hspace{0.5cm}
			+
			\frac{\ve^2 \mu^{\frac{3}{2}}}{2M} \Big{|}\int_{\mathcal{S}^-}
			(\Lambda^{s}\chi_2(\mathrm{D})\omega_{\mu} (\partial_x \zeta)^2 \partial_z u_2)  (\Lambda^{s+2}\chi_2(\mathrm{D})\omega_{\mu}u_2) \: \mathrm{d}x \mathrm{d}z \Big{|}
			\\ 
			&
			\hspace{0.5cm}
			+
			\frac{ \sqrt{\mu}}{2M}\Big{|}\int_{\mathcal{S}^-}
			(\Lambda^{s+1}\chi_2(\mathrm{D})\omega_{\mu}  \partial_z u_2)  (\Lambda^{s+1}\chi_2(\mathrm{D})\omega_{\mu}u_2) \: \mathrm{d}x \mathrm{d}z \Big{|}
			\\ 
			& = 
			A_1^{3,1} + A_1^{3,2} + A_1^{3,3}.
		\end{align*}
		Here we see that the first two terms are easily estimated by using the weight to gain decay and then combining it with the estimates in Lemma \ref{est Op L},
		\begin{align*}
			A_1^{3,1} + A_1^{3,2} \leq \ve \mu^{\frac{3}{4}} M(t_0+1)|\psi|_{\mathring{H}^{s+\frac{1}{2}}} \|\Lambda^{s+1}\chi_2(\mathrm{D})\omega_{\mu}\nabla_{x,z}^{\mu}u_2\|_{L^2(\mathcal{S}^-)}.
		\end{align*}
		Lastly, for $A_1^{3,3}$ we use integration by parts and the support of $\chi_2$ to find that
		\begin{align*}
		 	A_1^{3,3} 
		 	& \leq 
		 	\frac{\mu}{4M} 
		 	\Big{|}\int_{\mathcal{S}^-}
		 	(\Lambda^{s+1}\chi_2(\mathrm{D})\omega_{\mu}  u_2)  (\Lambda^{s+1}\chi_2(\mathrm{D})\omega_{\mu}u_2) \: \mathrm{d}x \mathrm{d}z \Big{|}
		 	\\ 
		 	& 
		 	\leq
		 	\frac{\mu}{4M}  \|\Lambda^{s+1}\chi_2(\mathrm{D})\omega_{\mu}\nabla_{x,z}u_2\|_{L^2(\mathcal{S}^-)}^2,
		\end{align*}
		and therefore can be reabsorbed into $A$. The only remaining estimate is now on $A_2$. However, we note that it is similar to $A_1^3$, and the same estimates apply. In conclusion, we may gather all these estimates and use \eqref{A norm} to find that
		\begin{align*}
			(1-	\frac{1}{4} )
			\|\Lambda^{s+1}\chi_2(\mathrm{D})\omega_{\mu}\nabla_{x,z}^{\mu} u_2 \|_{L^2(S^-)} \leq \ve \mu^{\frac{3}{4}} M(t_0+3)|\psi|_{\mathring{H}^{s+\frac{1}{2}}}.
		\end{align*}
		Returning to $I_2^1$, we obtain the bound
		\begin{align*}
			|I_2^1(\varphi)| \leq \ve \mu M(t_0+3)|\psi|_{\mathring{H}^{s+\frac{1}{2}}}|\varphi|_{L^2}.
		\end{align*}
		To conclude this step, it only remains to estimate $I_2^{2}(\varphi)$. However, this is just a simple version of $A_1^2$ and the following estimate is easy to deduce
		\begin{align*}
			|I_2(\varphi)| 
			& \leq |I_2^1(\varphi)| +|I_2^2(\varphi)| 
			\\ 
			& 
			\leq 
			\ve \mu M(t_0+3)|\psi|_{\mathring{H}^{s+\frac{1}{2}}}|\varphi|_{L^2}.
		\end{align*}
	
		\noindent
		\underline{Step 3.} \textit{Conclusion of proof}. We have from Step $1$. and Step $2$. that the result holds:
		\begin{align*}
			|\mathcal{G}^-_{\mu}[\ve \zeta]\psi + \sqrt{\mu} |\mathrm{D}|\psi|_{H^{s+\frac{1}{2}}} 
			& \leq  
			\sup\limits_{ |\varphi|_{L^2}=1} \Big{(} |I_1(\varphi)| + |I_2(\varphi)|\Big{)}
			\\ 
			 & \leq 	\ve \mu M(t_0 + 3)|\psi|_{\mathring{H}^{s+\frac{1}{2}}}. 
		\end{align*}

	\end{proof}

	\subsection{Symbolic analysis of $(\mathcal{G}_{\mu}^-)^{-1}\mathcal{G}_{\mu}^+$} The next step in studying the symbolic behaviour of $\mathcal{G}_{\mu}$ is to understand the composition of $(\mathcal{G}_{\mu}^-)^{-1}$ with $\mathcal{G}_{\mu}^+$.  

	\begin{cor}\label{Cor 3.4}
		Let $t_0 \geq 1$ and $\zeta \in H^{t_0 +3}(\R)$ be such that \eqref{non-cavitation} is satisfied. Then for all $0 \leq s \leq t_0$ and $f \in \mathscr{S}(\R)$ 
		 one can approximate the operator $(\mathcal{G}_{\mu}^-)^{-1}\mathcal{G}_{\mu}^+$ by
		\begin{equation*}
			\mathrm{Op}\Big( \frac{S^+}{S^-} \Big)f(x)
			= 
			-  \mathcal{F}^{-1} \big{(}\tanh(\sqrt{\mu} t(x, \xi)) \hat{f}(\xi) \big{)}(x),
		\end{equation*}
		where $t$ is defined by \eqref{tail operator}. Moreover, there holds,
		\begin{equation}\label{Symbol G-invG+}
			|(\mathcal{G}^-_{\mu}[\ve \zeta])^{-1}\mathcal{G}^+_{\mu}[\ve \zeta] f - \mathrm{Op}\Big( \frac{S^+}{S^-} \Big) f|_{\mathring{H}^{s+\frac{1}{2}}} \leq \ve \mu^{\frac{1}{4}} M(t_0 + 3)|f|_{H^{s-\frac{1}{2}}}. 
			\color{black}
		\end{equation}
			
	\end{cor}

	\begin{remark}\label{Rmrk negative Sob}
		From the symmetry of the Dirichlet-Neumann operators $\mathcal{G}^{\pm}$ they can be defined on Sobolev spaces of negative order by duality. See Remark $3.17$ in \cite{WWP}.
	\end{remark}
	
	\begin{proof}
		We define $\tilde{\psi}^- =  \mathrm{Op}\Big( \frac{S^+}{S^-} \Big) f$ and let $\tilde{\phi}^-$ be a solution of 
		\begin{equation*}
			\begin{cases}
				\nabla_{x,z}^{\mu} \cdot P(\Sigma^{-})\nabla_{x,z}^{\mu} \tilde{\phi}^- = 0 \hspace{0.7cm}\qquad \text{in} \quad \mathcal{S}^-
				\\
				\tilde{\phi}^- = \tilde{\psi}^- \hspace{4.05cm} \text{on} \quad  z = 0.
			\end{cases}
		\end{equation*}	
		Then we can make the definition
		\begin{equation*}
			\mathcal{G}^-_{\mu}[\ve \zeta]\tilde{\psi}^- = \partial_n \tilde{\phi}^-|_{z=0}.
		\end{equation*}
		Now, let $\phi^-$ be defined by the solution of
		\begin{equation}
			\begin{cases}
				\nabla_{x,z}^{\mu} \cdot P(\Sigma^{-})\nabla_{x,z}^{\mu} \phi^- = 0 \hspace{0.5cm}\qquad \text{in} \quad \mathcal{S}^-
				\\
				\partial_{n}^{P^-} \phi^- =  \mathcal{G}^+_{\mu}[\ve \zeta]f \hspace{2.3cm} \text{on} \quad  z = 0.
			\end{cases}
		\end{equation}
		Then the difference $u = \phi^- - \tilde{\phi}^-$ satisfies
		\begin{equation}
			\begin{cases}
				\nabla_{x,z}^{\mu} \cdot P(\Sigma^{-})\nabla_{x,z}^{\mu} u = 0 \hspace{1.7cm}\qquad \text{in} \quad \mathcal{S}^-
				\\
				\partial_{n}^{P^-} u =  \mathcal{G}^+_{\mu}[\ve \zeta]f - \mathcal{G}^-_{\mu}[\ve \zeta]\tilde{\psi}^- \hspace{1.4cm} \text{on} \quad  z = 0.
			\end{cases}
		\end{equation}
		It is straightforward to prove that these problems are well-defined where both $\phi^-$ and $\tilde{\phi}^-$ satisfy \eqref{Decay est}. Consequently, we can consider the variational equation:
		\begin{align*}
			\int_{\mathcal{S}^-} P(\Sigma^-)\nabla_{x,z}^{\mu} \Lambda^s u \cdot \nabla_{x,z}^{\mu}\Lambda^su \: \mathrm{d}x \mathrm{d}z 
			& =
			-\int_{\{z=0\}} \Lambda^s(\mathcal{G}^+_{\mu}[\ve \zeta]f - \mathcal{G}^-_{\mu}[\ve \zeta]\tilde{\psi}^-) \Lambda^s u_0 \: \mathrm{d}x
			\\ 
			& \hspace{1cm}
			+
			\int_{\mathcal{S}^-} [\Lambda^s,P(\Sigma^-)]\nabla_{x,z}^{\mu} u \cdot \nabla_{x,z}^{\mu}\Lambda^s u \: \mathrm{d}x \mathrm{d}z,
		\end{align*}
		where $u_0 = u|_{z=0}$ and trace inequality \eqref{trace 1} implies
		\begin{equation*}
			|u_0|_{\mathring{H}^{s+\frac{1}{2}}}\lesssim \mu^{-\frac{1}{4}} \|\Lambda^s \nabla_{x,z}^{\mu} u\|_{L^2(\mathcal{S}^-)}.
			\color{black}
		\end{equation*}
		To conclude the proof, we need the following inequality
		\begin{align}\label{claim ineq G+ - G-}
			\Big{|} \big{(} \Lambda^s(\mathcal{G}^+_{\mu}[\ve \zeta]f- \mathcal{G}^-_{\mu}[\ve \zeta]\tilde{\psi}^-) , \Lambda^s u_0  \big{)}_{L^2}  \Big{|} 
			\leq \ve \sqrt{\mu} M(t_0 +3) |f|_{H^{s-\frac{1}{2}}} |u_0|_{\mathring{H}^{s+\frac{1}{2}}}.
			\color{black}
		\end{align}
		Assuming the inequality holds, we can argue as in the proof of Proposition \ref{Inverse 2}, Step $2$. to obtain the result. 
		
		To prove \eqref{claim ineq G+ - G-}, we will decompose the left-hand side in several pieces. In particular, we first observe that
		\begin{align*}
			\big{(}\Lambda^s \mathcal{G}^+_{\mu}[\ve \zeta] f, \Lambda^s u_0 \big{)}_{L^2}
			& = 
			\big{(} \mathrm{Op}\Big(\frac{S^+}{|\mathrm{D}|}\Big)^{\ast}\Lambda^s f, \Lambda^s |\mathrm{D}|u_0 \big{)}_{L^2} + R_1 + R_2,
		\end{align*}
		where the rest is given by
		\begin{align*}
			R _1 & =\big{(}\big [\Lambda^s, \mathcal{G}^+_{\mu}[\ve \zeta]\big ] f, \Lambda^s u_0 \big{)}_{L^2},
			\\
			R_2 & = \big{(}\Lambda^{s-\frac{1}{2}} f, \Lambda^{\frac{1}{2}}\big{(} \mathcal{G}^+_{\mu}[\ve \zeta] - \mathrm{Op}(S^+)\big{)} \Lambda^s u_0 \big{)}_{L^2}.
		\end{align*}
		Here we estimate $R_1$ by Cauchy-Schwarz, \eqref{Commutator G+}, and \eqref{Basic est: B to D} to get that
		\begin{align*}
			|R_1| 
			& \leq \ve \mu M|f|_{\dot{H}^{s-\frac{1}{2}}_{\mu}} |u_0|_{\dot{H}^{s+\frac{1}{2}}_{\mu}} 
			\\ 
			& \leq \ve \sqrt{\mu} M|f|_{H^{s-\frac{1}{2}}}  |u_0|_{\mathring{H}^{s+\frac{1}{2}}}.
		\end{align*}
		While for $R_2$ we use \eqref{Symbolic G+} (with $k=1$) and \eqref{Basic est: B to D} to get that
		\begin{align*}
			|R_2| 
			& \leq \ve \sqrt{\mu} M(t_0+3)|f|_{H^{s-\frac{1}{2}}}| u_0|_{\mathring{H}^{s+\frac{1}{2}}}.
		\end{align*}
		Next, we decompose the remaining part where we get that
		\begin{align*}
			\big{(}\Lambda^s \mathcal{G}^-_{\mu}[\ve \zeta] \tilde\psi^-, \Lambda^s u_0 \big{)}_{L^2}
			& = 
			\big{(} \mathrm{Op}\Big(\frac{S^+}{|\mathrm{D}|}\Big)\Lambda^s f, \Lambda^s |\mathrm{D}|u_0 \big{)}_{L^2} + R_3 + R_4 + R_5,
		\end{align*}
		where the rest terms are defined by
		\begin{align*}
			R_3 & = \big{(}\big{[}\Lambda^s, \mathcal{G}^-_{\mu}[\ve \zeta]\big{]} \tilde\psi^-, \Lambda^s u_0 \big{)}_{L^2},
			\\
			R_4 & =  \big{(}\Lambda^{s-\frac{1}{2}} \tilde\psi^-, \Lambda^{\frac{1}{2}}\big{(} \mathcal{G}^-_{\mu}[\ve \zeta] - \mathrm{Op}(S^-)\big{)} \Lambda^s u_0 \big{)}_{L^2},
			\\ 
			R_5 & =  \big{(}\Lambda^{\frac{1}{2}}[\Lambda^{s},\mathrm{Op}\big{(} \frac{S^+}{S^-} \big{)} ]f, \Lambda^{-\frac{1}{2}} \mathrm{Op}(S^-)\Lambda^s u_0 \big{)}_{L^2}.
		\end{align*} 
		For $R_3$, the commutator estimate \eqref{Commutator G-}, and \eqref{PDO est tanh} yields,
		\begin{align*}
			|R_3| 
			& \leq \ve \sqrt{\mu} M|\tilde \psi^-|_{\mathring{H}^{s-\frac{1}{2}}} |u_0|_{\mathring{H}^{s+\frac{1}{2}}}
			\\ 
			& 
			\leq 
			\ve \sqrt{\mu} M |f|_{H^{s-\frac{1}{2}}} |u_0|_{\mathring{H}^{s+\frac{1}{2}}}.
		\end{align*}
		For $R_4$, we argue similarly but use Proposition \ref{Prop G- S-} with estimate \eqref{G- symbolic} to get that 
		\begin{align*}
			|R_4| \leq \ve \mu M(t_0+3)|f|_{H^{s-\frac{1}{2}}} |u_0|_{\mathring{H}^{s+\frac{1}{2}}}.
		\end{align*}
		Finally, for $R_5$, we apply Cauchy-Schwarz, the commutator estimate \eqref{Commutator Lambda s S+/S-} to get that
		\begin{align*}
			|R_5|
			&  \leq \sqrt{\mu}
			|[\Lambda^{s},\mathrm{Op}\big{(} \frac{S^+}{S^-} \big{)} ]f|_{H^{\frac{1}{2}}} |\Lambda^{s-\frac{1}{2}} |\mathrm{D} |u_0|_{L^2}
			\\
			&  
			\leq \ve \sqrt{\mu} M |f|_{H^{s-\frac{1}{2}}} |u_0|_{\mathring{H}^{s+\frac{1}{2}}}.
		\end{align*}
		Gathering all these observations, we can estimate the left-hand side of \eqref{claim ineq G+ - G-} by
		\begin{equation*}
			\mathrm{LHS}_{\eqref{claim ineq G+ - G-}} \leq \Big{|}\big{(} \Lambda^{\frac{1}{2}}(\mathrm{Op}\big( \frac{S^+}{|\mathrm{D}|}\big)^{\ast} - \mathrm{Op}\big( \frac{S^+}{|\mathrm{D}|}\big))\Lambda^s f, \Lambda^{s-\frac{1}{2}} |\mathrm{D}|u_0\big{)}_{L^2}\Big{|}+ \sum\limits_{i=1}^5 |R_i|.
		\end{equation*}
		Then to conclude the proof, we use Cauchy-Schwarz, the adjoint estimate \eqref{adjoint est on S+/D} to get that
		\begin{equation*}
			\mathrm{LHS}_{\eqref{claim ineq G+ - G-}} \leq  \ve \sqrt{\mu} M(t_0 + 3) |f|_{H^{s-\frac{1}{2}}} |u_0|_{\mathring{H}^{s+\frac{1}{2}}}.
		\end{equation*}
		
	\end{proof}

Before we proceed, we also note that under the provisions of Corollary \ref{Cor 3.4} one can make an approximation of $(\mathcal{G}_{\mu}^{-})^{-1} \partial_x$ in terms of the Hilbert transform:
\begin{cor}
	Let $t_0 \geq 1$ and $\zeta \in H^{t_0 +3}(\R)$ be such that \eqref{non-cavitation} is satisfied. Then for all $0 \leq s \leq t_0$ and $f \in \mathscr{S}(\R)$ there holds,
	\begin{equation}\label{est G- inv dx pdo}
		|(\mathcal{G}^-_{\mu}[\ve \zeta])^{-1} \partial_x  f  -\frac{1}{\sqrt{\mu}}\mathcal{H}  f|_{\mathring{H}^{s+\frac{1}{2}}} \leq \ve \mu^{-\frac{1}{4}} M(t_0 + 3)|f|_{H^{s-\frac{1}{2}}}. 
	\end{equation}
\end{cor}
\begin{proof}
	The proof is the same as Corollary \ref{Cor 3.4}, where we note that $S^+$ is traded for $\partial_x$ in the formula $\mathrm{Op}(\frac{\partial_x}{S^-}) = \frac{1}{\sqrt{\mu}}\mathcal{H}$. In particular, now $\tilde{\psi} = \mathrm{Op}(\frac{\partial_x}{S^-})f$ and  the proof follows from:
	\begin{align}\label{claim ineq dx - G-}
		\Big{|} \big{(} \Lambda^s(\partial_x f- \mathcal{G}^-_{\mu}[\ve \zeta]\tilde{\psi}^-) , \Lambda^s u_0  \big{)}_{L^2}  \Big{|} 
		\leq \ve  M(t_0 +3) |f|_{H^{s-\frac{1}{2}}} |u_0|_{\mathring{H}^{s+\frac{1}{2}}}.
		\color{black}
	\end{align}
	Then using the same notation as in the previous proof, we note that in this case  we have
	\begin{align*}
		\big{(} \Lambda^s(\partial_x f- \mathcal{G}^-_{\mu}[\ve \zeta]\tilde{\psi}^-) , \Lambda^s u_0  \big{)}_{L^2} 
		& = 
		\big{(}\big{[}\Lambda^s, \mathcal{G}^-_{\mu}[\ve \zeta]\big{]} \tilde\psi^-, \Lambda^s u_0 \big{)}_{L^2}
		\\
		&
	 	\hspace{0.5cm}
		+ 
		\big{(}\Lambda^{s-\frac{1}{2}} \tilde\psi^-, \Lambda^{\frac{1}{2}}\big{(} \mathcal{G}^-_{\mu}[\ve \zeta] - \mathrm{Op}(S^-)\big{)} \Lambda^s u_0 \big{)}_{L^2} 
		\\ 
		& = R_3 + R_4,
	\end{align*}
	where $R_1 = R_2 = R_5= 0$ and
	\begin{equation*}
		| R_3 | + |R_4| \leq \ve \sqrt{\mu} M(t_0+3)|\tilde \psi^-|_{\mathring{H}^{s-\frac{1}{2}}} |u_0|_{\mathring{H}^{s+\frac{1}{2}}} \leq \ve  M(t_0+3)|f|_{H^{s-\frac{1}{2}}} |u_0|_{\mathring{H}^{s+\frac{1}{2}}}.
	\end{equation*}
	Adding the estimates concludes the proof.
\end{proof}

	\subsection{Symbolic analysis of $(\mathcal{G}_{\mu}^-)^{-1}\mathcal{G}_{\mu}$}
	
		Next we will study the symbolic behaviour of $(\mathcal{G}_{\mu}^-)^{-1}\mathcal{G}_{\mu}$. To do so, we recall that we may write $\mathcal{G}_{\mu} = \mathcal{G}^+_{\mu} (\mathcal{J}_{\mu})^{-1}$ where the symbolic behaviour of $\mathcal{J}_{\mu}$ is captured by the symbol
		\begin{equation}\label{S_J}
			S_{J} = 1- \gamma \frac{S^+}{S^-}.
		\end{equation}
	
		\begin{cor}
		Let $t_0 \geq 1$ and $\zeta \in H^{t_0 +3}(\R)$ be such that \eqref{non-cavitation} is satisfied. Then for all $0 \leq s \leq t_0$ and $f \in \mathscr{S}(\R)$ one can approximate the operator $(\mathcal{G}^{-}_{\mu}[\ve \zeta])^{-1}\mathcal{G}_{\mu}[\ve \zeta]$ by
		\begin{equation*}
			\mathrm{Op}\Big( \frac{S^+}{S^-S_J} \Big)f(x)
			= 
			-
			\mathcal{F}^{-1}\Big{(} \frac{\tanh(\sqrt{\mu} t(x, \xi))}{ 1 +  \gamma \tanh(\sqrt{\mu} t(x, \xi))}\hat{f}(\xi)\Big{)}(x),
		\end{equation*}
		where $t$ is defined by \eqref{tail operator}. Moreover,  there holds,
		\begin{equation}\label{Est G-invG}
			|(\mathcal{G}^-_{\mu}[\ve \zeta])^{-1}\mathcal{G}_{\mu}[\ve \zeta] f - \mathrm{Op}\Big( \frac{S^+}{S^-S_J} \Big) f|_{\mathring{H}^{s+\frac{1}{2}}} \leq \ve  M(t_0 + 3)|f|_{H^{s-\frac{1}{2}}}. 
			\color{black}
		\end{equation}

		\end{cor}
	
	
		\begin{proof}
			We will give the proof in three steps. First, we will prove that $S_J$ provides a good symbolic description of $\mathcal{J}_{\mu}$. Then we will show that $1/S_J$ describes the inverse, and lastly, we combine each step with the previous results of this section to prove \eqref{Est G-invG}.\\

			\noindent
			\underline{Step 1.} We can approximate $\mathcal{J}_{\mu}$ with $\mathrm{Op}(S_J)$. Indeed, by definition and estimate \eqref{Symbol G-invG+} we get that
			\begin{align}\label{est J sJ}
				|\mathcal{J}_{\mu}[\ve \zeta]f - \mathrm{Op}(S_J)f|_{\mathring{H}^{s+\frac{1}{2}}} 
				& 
				=
				\gamma |(\mathcal{G}^-_{\mu}[\ve \zeta])^{-1}\mathcal{G}^+_{\mu}[\ve \zeta] f - \mathrm{Op}\big{(} \frac{S^+}{S^-}\big{)}f|_{\mathring{H}^{s+\frac{1}{2}}}
				\\
				& 
				\leq   \ve \mu^{\frac{1}{4}} M(t_0 +3)|f|_{H^{s-\frac{1}{2}}}.\notag
				\color{black}
			\end{align}
			
			\noindent
			\underline{Step 2.} We can approximate $(\mathcal{J}_{\mu})^{-1}$ with $\mathrm{Op}(1/S_J)$. To prove this fact, we simply employ Lemma \ref{Prop op J} with inequality \eqref{Est J} to deduce that
			\begin{align*}
				|(\mathcal{J}_{\mu}[\ve \zeta])^{-1}f - \mathrm{Op}(1/S_J)f|_{\dot{H}_{\mu}^{s+\frac{1}{2}}} 
				& = 
				|(\mathcal{J}_{\mu}[\ve \zeta])^{-1}\big(1 -\mathcal{J}_{\mu}[\ve \zeta] \mathrm{Op}(1/S_J)\big{)}f|_{\dot{H}_{\mu}^{s+\frac{1}{2}}} 
				\\ 
				& \leq M
				|\big(1 -\mathcal{J}_{\mu}[\ve \zeta] \mathrm{Op}(1/S_J)\big{)}f|_{\dot{H}_{\mu}^{s+\frac{1}{2}}} 
				\\ 
				& \leq M \big{(}
				|\big(1 -\mathrm{Op}(S_J) \mathrm{Op}(1/S_J)\big{)}f|_{\dot{H}_{\mu}^{s+\frac{1}{2}}} 
				\\
				& 
				\hspace{0.5cm}
				+
				|\big(\mathcal{J}_{\mu}[\ve \zeta]- \mathrm{Op}(S_J) \big{)}\mathrm{Op}(1/S_J)f|_{\dot{H}_{\mu}^{s+\frac{1}{2}}}  
				\\
				& = :
				R_1 + R_2.
			\end{align*}
			Here $R_1$ is estimated by \eqref{Op S_J 1/SJ}:
			\begin{align*}
				R_1 
				 \leq 
				\mu^{-\frac{k}{4}}|\big(1 -\mathrm{Op}(S_J) \mathrm{Op}(1/S_J)\big{)}f|_{H^{s+1-\frac{k}{2}}} 
				 \leq 
				\ve \mu^{-\frac{k}{4}} M |f|_{H^{s-\frac{k}{2}}}.
			\end{align*}
			While $R_2$ we use the estimate in Step 1. and \eqref{Op 1/SJ} to obtain
			\begin{align*}
				R_2 
				&
				\leq
				\mu^{-\frac{1}{4}}
				|\big(\mathcal{J}_{\mu}[\ve \zeta]- \mathrm{Op}(S_J) \big{)}\mathrm{Op}(1/S_J)f|_{\mathring{H}^{s+\frac{1}{2}}} 
				\\  
				& 
				\ \leq
				\ve
				M(t_0+3)
				|\mathrm{Op}(1/S_J)f|_{H^{s-\frac{1}{2}}}  
					\\  
				&
				\leq\ve
				M(t_0+3)|f|_{H^{s-\frac{1}{2}}}.
				\color{black}
			\end{align*}
			As a result, we have the following estimate
			\begin{align}\label{Est Jinv pdo}
				|(\mathcal{J}_{\mu}[\ve \zeta])^{-1}f - \mathrm{Op}(1/S_J)f|_{\dot{H}_{\mu}^{s+\frac{1}{2}}} 
				\leq 
				\ve\mu^{-\frac{k}{4}} 
				M(t_0+3)|f|_{H^{s-\frac{k}{2}}}.
				\color{black}
			\end{align}
			\underline{Step 3.} Estimate \eqref{Est G-invG} holds true. To prove this, we first use the definition $\mathcal{G}_{\mu} = \mathcal{G}_{\mu}^+ (\mathcal{J}_{\mu})^{-1}$ to make the decomposition:
			\begin{align*}
				\mathrm{LHS}_{\eqref{Est G-invG}} 
				& : = 
				|(\mathcal{G}^-_{\mu}[\ve \zeta])^{-1}\mathcal{G}^+_{\mu}[\ve \zeta] (\mathcal{J}_{\mu}[\ve \zeta])^{-1} f - \mathrm{Op}\Big( \frac{S^+}{S^-S_J} \Big) f|_{\mathring{H}^{s+\frac{1}{2}}}
				\\
				& 
				\leq 
				|(\mathcal{G}^-_{\mu}[\ve \zeta])^{-1}\mathcal{G}^+_{\mu}[\ve \zeta] \big{(}(\mathcal{J}_{\mu}[\ve \zeta])^{-1}  - \mathrm{Op}\Big( \frac{1}{S_J} \Big)\big{)} f|_{\mathring{H}^{s+\frac{1}{2}}}
				\\
				& 
				\hspace{0.5cm}
				+
				|(\mathcal{G}^-_{\mu}[\ve \zeta])^{-1}\mathcal{G}^+_{\mu}[\ve \zeta] \mathrm{Op}\Big{(} \frac{1}{S_J} \Big{)} f - \mathrm{Op}\Big( \frac{S^+}{S^-S_J} \Big) f|_{\mathring{H}^{s+\frac{1}{2}}}
				\\ 
				& = : R_3 + R_4. 
			\end{align*}
			For $R_3$ we use Proposition \ref{Inverse 2} with inequality \eqref{Inverse est 2} and then Step 2. to get that
			\begin{align*}
				R_3
				\leq \mu^{\frac{1}{4}}  M
				| \big{(}(\mathcal{J}_{\mu}[\ve \zeta])^{-1}  - \mathrm{Op}\Big( \frac{1}{S_J} \Big)\big{)} f|_{\dot{H}_{\mu}^{s+\frac{1}{2}}} 
				\leq \ve \mu^{\frac{1}{4}} M(t_0+3)
				| f |_{H^{s-\frac{1}{2}}}.
				\color{black}
			\end{align*}
			For $R_4$, we decompose it further:
			\begin{align*}
				R_4 
				& =
				|\Big{(}(\mathcal{G}^-_{\mu}[\ve \zeta])^{-1}\mathcal{G}^+_{\mu}[\ve \zeta] - \mathrm{Op}\Big( \frac{S^+}{S^-} \Big)\Big{)}\mathrm{Op}\Big{(} \frac{1}{S_J} \Big{)}  f|_{\mathring{H}^{s+\frac{1}{2}}}
				\\ 
				&
				\hspace{0.5cm}
				+
				|\Big{(} \mathrm{Op}\Big( \frac{S^+}{S^-} \Big) \mathrm{Op}\Big{(} \frac{1}{S_J} \Big{)}  - \mathrm{Op}\Big( \frac{S^+}{S^-S_J} \Big)\Big{)} f|_{\mathring{H}^{s+\frac{1}{2}}}
				\\
				& = :
				R_5 + R_6.
			\end{align*}
			For $R_5$ we use \eqref{Symbol G-invG+} and \eqref{Op 1/SJ} to get that
			\begin{align*}
				R_5 
				\leq 
				\ve \mu^{\frac{1}{4}}
				M(t_0+3)|\mathrm{Op}\Big{(} \frac{1}{S_J} \Big{)} f |_{H^{s-\frac{1}{2}}}
				\leq 
				\ve \mu^{\frac{1}{4}}
				M(t_0+3)|f|_{H^{s-\frac{1}{2}}}.
				\color{black}
			\end{align*}
			While for $R_6$ we simply employ estimate \eqref{Op S+/S- 1/SJ} which yields
			\begin{align*}
				R_6
				& = 
				\ve M|f|_{H^{s-\frac{1}{2}}}.
			\end{align*}
			Gathering all these estimates concludes the proof.

		\end{proof}

	\subsection{Symbolic description of $\mathcal{G}_{\mu}$} For the next result, we will give a simple approximation of $\mathcal{G}_{\mu}$. We will allow for a loss of derivatives in the estimate. However, the loss in regularity is translated into precision with respect to the small parameters.

	\begin{cor} Let $t_0 \geq 1$ and $\zeta \in H^{t_0 +2}(\R)$ be such that \eqref{non-cavitation} is satisfied. Then for all $0 \leq s \leq t_0$ and $\psi \in \dot{H}^{s+\frac{3}{2}}_{\mu}(\R)$ one can approximate the operator $\mathcal{G}_{\mu}[\ve \zeta]$ by  the Fourier multiplier $\mathcal{G}_{\mu}[0]$ defined by
	\begin{equation*}
		\mathcal{G}_{\mu}[0] = \sqrt{\mu}
		|\mathrm{D}| \frac{\tanh(\sqrt{\mu} | \mathrm{D}|)}{ 1 + \gamma \tanh(\sqrt{\mu} | \mathrm{D}|)},
	\end{equation*}
	and it satisfies the estimate
	\begin{equation}\label{Symbolic G}
		 |\mathcal{G}_{\mu}[\ve \zeta]\psi - \mathcal{G}_{\mu}[0]\psi|_{H^s} \leq \ve  \mu M |\psi|_{\dot{H}_{\mu}^{s+\frac{3}{2}}} .
	\end{equation}
	\end{cor}
	
	\begin{proof}

	For the proof of inequality \eqref{Symbolic G}, we let $\mathcal{G}_{\mu}[0] =  \mathcal{G}^+_{\mu}[0] (\mathcal{J}_{\mu}[0])^{-1}$ and make the following decomposition
	\begin{align*}
		|\mathcal{G}_{\mu}[\ve \zeta]\psi - \mathcal{G}_{\mu}^+[0] (\mathcal{J}_{\mu}[0])^{-1}\psi |_{H^s}
		& \leq 
		|\big{(}\mathcal{G}^+_{\mu}[\ve \zeta]-  \mathcal{G}_{\mu}^+[0]\big{)}(\mathcal{J}_{\mu}[\ve \zeta])^{-1}\psi |_{H^s}
		\\ 
		& 
		\hspace{0.5cm}
		+
		|  \mathcal{G}_{\mu}^+[0]\big{(}(\mathcal{J}_{\mu}[\ve \zeta])^{-1} -   (\mathcal{J}_{\mu}[0])^{-1}\big{)}\psi |_{H^s}
		\\ 
		& = 
		B_1 + B_2.
	\end{align*}
	For the estimate on $B_1$, we use the classical small amplitude expansion of $\mathcal{G}^+_{\mu}$ provided by Proposition $3.44$ (with $n=0$, $k=1$, $s'=s+\frac{1}{2}$) in \cite{WWP}, and then \eqref{Est J}:
	\begin{align*}
		B_1 
		\leq
		\ve  \mu  M|\psi|_{\dot{H}_{\mu}^{s+\frac{3}{2}}}.
	\end{align*}
	For an estimate on $B_2$, we first observe that
	\begin{equation*}
		\sqrt{\mu}|\xi|\tanh(\sqrt{\mu}|\xi|) \leq \mu \frac{|\xi|^2}{(1+\sqrt{\mu}|\xi|)^{\frac{1}{2}}}.
	\end{equation*}
	We may therefore use the product rule, combined with \eqref{Est J} and the shape derivative formulas \eqref{shape derivative J} and \eqref{shape derivative Jinv} to find that
	\begin{align*}
		B_2 
		& \leq
		\mu |\partial_x \big{(}(\mathcal{J}_{\mu}[\ve \zeta])^{-1}(1-\mathcal{J}_{\mu}[\ve \zeta] (\mathcal{J}_{\mu}[0])^{-1})\psi\big{)}|_{\dot{H}_{\mu}^{s+\frac{1}{2}}} 
		\\ 
		& \leq 
		\mu M \Big{(} |\mathrm{d}_{\zeta}(\mathcal{J}_{\mu}[\ve \zeta])^{-1}(\partial_x \zeta)(1-\mathcal{J}_{\mu}[\ve \zeta] (\mathcal{J}_{\mu}[0])^{-1})\psi|_{\dot{H}_{\mu}^{s+\frac{1}{2}}} 
		\\
		& 
		\hspace{0.5cm}
		+
		|\mathrm{d}_{\zeta}\mathcal{J}_{\mu}[\ve \zeta](\partial_x\zeta)  (\mathcal{J}_{\mu}[0])^{-1}\psi|_{\dot{H}_{\mu}^{s+\frac{1}{2}}} + |\partial_x \psi |_{\dot{H}_{\mu}^{s+\frac{1}{2}}} 	\Big{)}
		\\ 
		& 
		\leq \ve \mu M|\psi|_{\dot{H}_{\mu}^{s+\frac{3}{2}}} .
	\end{align*}

	\end{proof}

	\subsection{Symbolic analysis of $(\mathcal{J}_{\mu})^{-1}(\mathcal{G}_{\mu}^-)^{-1}\partial_x$}
	We end this section with the study of an operator that appears in the quasilinearisation of the equations. In particular, we will need the symbolic description for the energy estimates later (see inequality \eqref{What we want...}). The operator is well-defined by the first point in Corollary \ref{Cor inverse G-}, and can be approximated by  a first order operator with a skew-symmetric principal symbol.
	\begin{cor}\label{Cor last one}
		Let $t_0 \geq 1$ and $\zeta \in H^{t_0+3}(\R)$ be such that \eqref{non-cavitation} is satisfied. Then for all $0\leq s \leq t_0$ and $f \in H^{s+\frac{1}{2}}(\R)$, one can approximate the operator $\mathfrak{P}^2(\mathcal{J}_{\mu})^{-1}(\mathcal{G}_{\mu}^-)^{-1}\partial_x$ by 
		\begin{equation*}
				\mathrm{Op}\Big( \frac{\mathfrak{P}^2}{S_JS^-} \Big)\partial_x f(x)
				= -  \frac{1}{\sqrt{\mu}}\mathcal{F}^{-1}\Big{(} \frac{|\xi|(1+\sqrt{\mu}|\xi|)^{-1}}{ 1 +  \gamma \tanh(\sqrt{\mu} t(x, \xi))}i\xi\hat{f}(\xi)\Big{)}(x),
		\end{equation*}
		and for $k=0,1$, there holds,
		\begin{equation}
			|\big{[}\mathfrak{P}^2(\mathcal{J}_{\mu}[\ve \zeta])^{-1}(\mathcal{G}_{\mu}^-[\ve \zeta])^{-1}-\mathrm{Op}(\frac{\mathfrak{P}^2}{S_J S^-})\big{]}\partial_x f|_{H^{s+\frac{k}{2}}}\leq\ve \mu^{-1- \frac{k}{4}} M(t_0 +3)|(1+\sqrt{\mu}|\mathrm{D}|)^{\frac{1}{2}}f|_{H^s}.
		\end{equation}
		
	\end{cor}
	
	\begin{proof}
		We first decompose the main operator into three pieces:
		\begin{align}
			\mathfrak{P}^2(\mathcal{J}_{\mu}[\ve \zeta])^{-1}(\mathcal{G}_{\mu}^-[\ve \zeta])^{-1}\partial_xf 
			& =
			\mathrm{Op}(\frac{\mathfrak{P}^2}{S_J S^-})\partial_x f + R_1 + R_2,
		\end{align}
		where we define $f^{\sharp} = (\mathcal{J}_{\mu})^{-1}(\mathcal{G}_{\mu}^-)^{-1}\partial_x f$
		\begin{equation*}
			R_1 = \Big{(} 1- \mathrm{Op}(\frac{\mathfrak{P}^2}{S_J S^-})\mathrm{Op}(\frac{S_J S^-}{\mathfrak{P}^2}) \Big{)}\mathfrak{P}^2f^{\sharp},
		\end{equation*}
		and
		\begin{equation*}
				R_2 = \mathrm{Op}(\frac{\mathfrak{P}^2}{S_J S^-}) \big{(}	\mathrm{Op}(S^-S_J) -  \mathcal{G}^{-}_{\mu}[\ve \zeta]\mathcal{J}_{\mu}[\ve \zeta]\big{)}f^{\sharp}.
		\end{equation*}
		For the estimate on $R_1$ we use \eqref{1 - op b2...} 
		to get,
		\begin{align*}
			|R_1|_{H^{s+\frac{k}{2}}}
			\leq
			\ve \mu^{-\frac{k}{4}} M |\mathfrak{P}f^{\sharp}|_{H^s}.
		\end{align*}
		Then apply \eqref{Est J} and\eqref{Est derivative} to find that
		\begin{align} \label{B f sharp}
			 |\mathfrak{P}f^{\sharp}|_{L^2} 
			 & =
			 |(\mathcal{J}_{\mu}[\ve \zeta])^{-1}(\mathcal{G}_{\mu}^-[\ve \zeta])^{-1}\partial_x f|_{\dot{H}_{\mu}^{\frac{1}{2}}}
			 \\ 
			 &\notag
			 \leq 
			 \mu^{-\frac{1}{4}}M|(\mathcal{G}_{\mu}^-[\ve \zeta])^{-1}\partial_x f|_{\mathring{H}^{\frac{1}{2}}}
			  \\ 
			 & \notag
			 \leq 
			 \mu^{-\frac{3}{4}}M| f|_{\mathring{H}^{\frac{1}{2}}},
		\end{align}
		which yields the following estimate on $R_1$:
		\begin{align*}
			|R_1|_{H^{s+\frac{k}{2}}}
			\leq 
			 \ve \mu^{-1- \frac{k}{4}}M |(1+\sqrt{\mu}|\mathrm{D}|)^{\frac{1}{2}}f|_{H^s}.
		\end{align*}
		%
		%
		%
		%
		%
		%
		%
		%
		%
		For the estimate on $R_2$ we would like to employ the symbolic description of $\mathcal{G}_{\mu}^{\pm}$. However, since estimate \eqref{G- symbolic} arrives on $\mathring{H}^{\frac{1}{2}}(\R)$, we need to carefully decompose this term so that we can estimate  $f^{\sharp} \in \dot{H}^{s+\frac{1}{2}}_{\mu}(\R)$. First, we observe that
		%
		%
		%
		%
		%
		%
		%
		%
		%
		\begin{align*}
			\big{(}\mathrm{Op}(  S^-S_J ) -   \mathcal{G}^{-}_{\mu}[\ve \zeta]\mathcal{J}_{\mu}[\ve \zeta]\big{)}f^{\sharp} 
			&
			= 
			\Big{(}\big{(}\mathrm{Op}(S^-) -  \mathcal{G}^{-}_{\mu}[\ve \zeta]\big{)}\mathcal{J}_{\mu}[\ve \zeta]\Big{)}f^{\sharp}
			\\ 
			& \hspace{0.5cm}
			+
			\big{(}\mathrm{Op}(S^-S_J) -\mathrm{Op}(S^-)\mathrm{Op}(S_J)  \big{)}f^{\sharp}
			\\ 
			& 
			\hspace{0.5cm}
			+
			\Big{(}\mathrm{Op}(S^-)\big{(}\mathrm{Op}(S_J) -  \mathcal{J}_{\mu}[\ve \zeta]\big{)}\Big{)}f^{\sharp}
			\\  
			& = 
			r_1+r_2+r_3.
		\end{align*}
		Let $R_2^i =\Lambda^{\frac{k}{2}} \mathrm{Op}(\frac{\mathfrak{P}^2}{S_J S^-}) r_i$ with $i=1,2,3$. Then for the contribution of $R_2^1$, we first use  estimates \eqref{Op 1/SJ} and \eqref{B2/S-} to find that
		\begin{align*}
			| R_2^1|_{H^s}
			& \leq  \mu^{-\frac{3}{4}-\frac{k}{4}} M(t_0+3)|r_1|_{\mathring{H}^{s+\frac{1}{2}}}.
		\end{align*}
		Then apply Proposition \ref{Prop G- S-}, the definition of $f^{\sharp}$, and argue as in the proof of estimate \eqref{B f sharp} to deduce,
		\begin{align*}
			| R_2^1|_{H^s}
			& \leq \ve \mu^{\frac{1}{4}-\frac{k}{4}}   M(t_0+3)|(\mathcal{G}_{\mu}^-[\ve \zeta])^{-1} \partial_x f|_{\mathring{H}^{s+\frac{1}{2}}}
			\\ 
			& \leq \ve  \mu^{-1}M(t_0+3) |(1+\sqrt{\mu}|\mathrm{D}|)^{\frac{1}{2}}f|_{H^s}.
		\end{align*}
		For $R_2^2$, we apply  estimates \eqref{Op 1/SJ}, \eqref{B2/S-}, \eqref{S-SJ comp est}, and \eqref{B f sharp}:
		\begin{align*} 
			| R_2^2|_{H^s}
			\leq \ve \mu^{-1}M |(1+\sqrt{\mu}|\mathrm{D}|)^{\frac{1}{2}}f|_{H^s}.
		\end{align*} 
		Lastly, for $R_2^3$, we will decompose it further:
		\begin{align*}
			R_2^3
			& =
			-\gamma \Lambda^{\frac{k}{2}} \mathrm{Op}\big{(}\frac{\mathfrak{P}^2}{S_J S^-}\big{)} \Big{(}\mathrm{Op}(S^-) \mathrm{Op}(\frac{S^+}{S^-})-\mathrm{Op}(S^+)\Big{)}f^{\sharp}
			\\ 
			&
			\hspace{0.5cm}
			-
			\gamma \Lambda^{\frac{k}{2}} \mathrm{Op}\big{(}\frac{\mathfrak{P}^2}{S_J S^-}\big{)} \Big{(}\mathrm{Op}(S^+)-\mathcal{G}^+_{\mu}[\ve \zeta]\Big{)}f^{\sharp}
			\\ 
			&
			\hspace{0.5cm}
			-\gamma 
			\Lambda^{\frac{k}{2}}\mathrm{Op}\big{(}\frac{\mathfrak{P}^2}{S_J S^-}\big{)} \Big{(} \mathcal{G}_{\mu}^{-}[\ve \zeta] - \mathrm{Op}(S^-) \Big{)}
			(\mathcal{G}^-_{\mu}[\ve \zeta])^{-1}
			\mathcal{G}^+_{\mu}[\ve \zeta]f^{\sharp} 
			\\ 
			& = 
			R_2^{3,1} +	R_2^{3,2} + 	R_2^{3,3}.
		\end{align*}
		For $R_2^{3,1}$, we use \eqref{Op 1/SJ}, \eqref{B2/S-}, \eqref{S-SJ comp est}, and  \eqref{B f sharp} to find that
		\begin{equation*}
			|R_2^{3,1}|_{H^s} \leq \ve \mu^{-1} M |(1+\sqrt{\mu}|\mathrm{D}|)^{\frac{1}{2}}f|_{H^s}.
		\end{equation*}
		For $R_2^{3,2}$, we use \eqref{Op 1/SJ}, \eqref{B2/S-}, then estimate \eqref{Symbolic G+} to find that%
		\begin{equation*}
			|R_2^{3,2}|_{H^s}\leq \ve\mu^{1-\frac{k}{4}} M(t_0+3) |(1+\sqrt{\mu}|\mathrm{D}|)^{\frac{1}{2}}f|_{H^s}
		\end{equation*}
		For $R_2^{3,3}$, we use estimates \eqref{Op 1/SJ}, \eqref{B2/S-}, \eqref{G- symbolic}, \eqref{Inverse est 2}, \eqref{B f sharp} to find that
		%
		%
		%
		%
		\begin{align*}
			|R_2^{3,3}|_{H^s}
			& \leq
			\mu^{-1} M |	 \Big{(} \mathrm{Op}(S^-)- \mathcal{G}_{\mu}^{-}[\ve \zeta] \Big{)}
			(\mathcal{G}^-_{\mu}[\ve \zeta])^{-1}
			\mathcal{G}^+_{\mu}[\ve \zeta]f^{\sharp}|_{H^{s+\frac{1}{2}}}
			\\
			& \leq
			\ve  M(t_0+3) |(\mathcal{G}^-_{\mu}[\ve \zeta])^{-1}
			\mathcal{G}^+_{\mu}[\ve \zeta]f^{\sharp}|_{\mathring{H}^{s+\frac{1}{2}}}
			\\
			& \leq
			\ve \mu^{\frac{1}{4}} M(t_0+3) |f^{\sharp}|_{\dot{H}_{\mu}^{s+\frac{1}{2}}}
			\\
			& \leq  
			\ve \mu^{-1} M(t_0 +3)|(1+\sqrt{\mu}|\mathrm{D}|)^{\frac{1}{2}}f|_{H^s}.
		\end{align*}
		Gathering all these estimates implies
		\begin{align*}
			|R_2|_{H^s} 
			\leq  
			\ve \mu^{-1-\frac{k}{4}} M(t_0 +3)|(1+\sqrt{\mu}|\mathrm{D}|)^{\frac{1}{2}}f|_{H^s},
		\end{align*}
		and adding the estimates for $R_1$ and $R_2$ completes the proof.
		\color{black}
	
	\end{proof}

	\section{Quasilinearization of the  internal water wave system}\label{Quasilinearization}
	 In this section, we will put the internal water waves system \eqref{IWW} in a quasilinear form. This is done by applying time and space derivatives to the system and taking care of the principal terms. In particular, one needs shape derivative formulas for the Dirichlet-Neumann operator \eqref{Def G}, where we prove in Lemma \ref{Lemma shape derivative G} that
	\begin{equation*}
		\mathrm{d}_{\zeta} \mathcal{G}_{\mu}[\ve\zeta](h) \psi = - \ve \mathcal{G}_{\mu}[\ve\zeta]\big{(}h(\underline{w}^+ -\gamma  \underline{w}^-) \big{)} - \ve \mu \mathcal{I}[\mathbf{U}]h,
	\end{equation*}
	where $\underline{w}^{\pm}$ is defined below together with the operator $\mathcal{I}[\mathbf{U}]h$. In fact, the operator $\mathcal{I}[\mathbf{U}]h$ is one of the main quantities that we need to understand, and is where we will use the symbolic descriptions from the previous section.  Before stating the main result, we formally introduce the main operators involved, and that will be studied in detail later. 
	
	\begin{Def}\label{Def op} Let the functions $\psi^{\pm}$ serve as Dirichlet data for the elliptic problems \eqref{Phi+} and \eqref{Phi-}, then we define the horizontal components of the velocities at the surface  by
		\begin{equation*}
			\underline{w}^{\pm}
			= \frac{\mathcal{G}^{\pm}_{\mu}[\ve \zeta] \psi^+ + \ve \mu \partial_x \zeta \partial_x \psi^\pm}{1+ \ve^2 \mu (\partial_x \zeta)^2}.
		\end{equation*}
		We define the vertical component of the velocity at the surface by
		\begin{equation*}
			\underline{V}^\pm = \partial_x\psi^\pm - \ve \underline{w}^\pm \partial_x \zeta,
		\end{equation*}
		and
		\begin{equation*}
			[\![ \underline{V}^{\pm}]\!] =  \underline{V}^{+}-\underline{V}^{-}.
		\end{equation*}
		We refer to Corollary \ref{cor definitions} for the precise definition. 
		Moreover, we define the quantity related to the Rayleigh-Taylor criterion by
		\begin{equation}\label{a R-T}
			\mathfrak{a}
			 =  \Big{(}1
			+\ve \big{(} (\partial_t + \ve \underline{V}^{+}\partial_x)\underline{w}^{+} 
			-
			\gamma  (\partial_t + \ve \underline{V}^{-}\partial_x)\underline{w}^{-}
			\big{)}\Big{)},
		\end{equation}
		and a quantity related to the presence of surface tension:
		\begin{equation*}
			\mathcal{K}[\ve \sqrt{\mu} \partial_x \zeta] \bullet = (1+\ve^2\mu(\partial_x\zeta)^2)^{-\frac{3}{2}} \bullet . 
		\end{equation*}
		From these quantities, we let $\mathbf{U} = (\zeta, \psi)^T$ and define the linear operator, $\mathcal{I}[\mathbf{U}]$, of order one by,
	\begin{equation}\label{Def I}
		\mathcal{I}[\mathbf{U}] \bullet
		=
		\partial_x ( \bullet \underline{V}^+)+\gamma  \mathcal{G}_{\mu}[\ve \zeta]  (\mathcal{G}^-_{\mu}[\ve \zeta])^{-1}\partial_x	\Big{(} 
		\bullet [\![ \underline{V}^{\pm}]\!]\Big{)},
	\end{equation}
	and its adjoint reads,
	\begin{equation}\label{I star}
		\mathcal{I}[\mathbf{U}]^{\ast} \bullet
		=
		-\underline{V}^+\partial_x  \bullet
		-
		\gamma 
		[\![ \underline{V}^{\pm}]\!]  \partial_x\big{(}(\mathcal{G}^-_{\mu}[\ve \zeta])^{-1}\mathcal{G}_{\mu}[\ve \zeta] \bullet\big{)}.
	\end{equation}
	 Moreover, we define the instability operator:
		\begin{equation}\label{inst}
				\mathfrak{Ins}[\mathbf{U}] \bullet = \mathfrak{a}\bullet -	(1-\gamma)\gamma \ve^2 \mu	[\![ \underline{V}^{\pm}]\!]   \mathfrak{E}_{\mu}[\ve \zeta]	\big{(} 
			\bullet[\![ \underline{V}^{\pm}]\!] \big{)} - \mathrm{bo}^{-1} \partial_x \mathcal{K}[\ve \sqrt{\mu} \partial_x \zeta]\partial_x  \bullet,
		\end{equation}
		where 
		\begin{equation}\label{E inst}
			\mathfrak{E}_{\mu}[\ve \zeta] \bullet = \partial_x \circ (\mathcal{J}_{\mu}[\ve \zeta])^{-1}(\mathcal{G}^-_{\mu}[\ve \zeta])^{-1}\circ \partial_x \bullet.
		\end{equation}
		Then to put the internal water waves system in matrix form, it is convenient to introduce the notation
		\begin{equation*}
			\mathcal{A}[U] 
			=
			\begin{pmatrix}
				0 & - \frac{1}{\mu}\mathcal{G}_{\mu}[\ve \zeta]
				\\ 
				\mathfrak{Ins}[\mathbf{U}] & 0
			\end{pmatrix},
			\qquad 
			\mathcal{B}[\mathbf{U}] 
			=
			\begin{pmatrix}
				\ve \mathcal{I}[\mathbf{U}] & 0
				\\ 
				0 & - \ve \mathcal{I}[\mathbf{U}]^{\ast}
			\end{pmatrix},
		\end{equation*}
		where $\mathcal{A}$ and $\mathcal{B}$ corresponds to the principal part of the system. To account for surface tension, one also needs to track the dependence in the sub-principal part, which is defined for $\alpha = (\alpha_1,\alpha_2)\in \N^2$:
		\begin{equation*}
			\mathcal{C}_{\alpha}[U] 
			=
			\begin{pmatrix}
				0 & - \frac{1}{\mu}\mathcal{G}_{\mu,(\alpha)}[\ve \zeta]
				\\ 
				\mathrm{bo}^{-1} \mathcal{K}_{(\alpha)}[\ve \sqrt{\mu} \partial_x \zeta]   & 0
			\end{pmatrix},
		\end{equation*}
		where we let $(\partial_1, \partial_2) = (\partial_x, \partial_t)$, $F = (f_1,f_2)$ to define
		\begin{equation*}
			\mathcal{G}_{\mu,(\alpha)}[\ve \zeta] F= \sum \limits_{j=1}^2 \ \alpha_j \mathrm{d}_{\zeta}\mathcal{G}_{\mu}[\ve \zeta](\partial_j \zeta) f_j,
		\end{equation*}
		and we have:
		\begin{equation*}
			\mathcal{K}_{(\alpha)}[\partial_x \zeta] F
			=  
			- \partial_x \Big{(}\sum\limits_{j=1}^2 \mathrm{d}_{\zeta}\mathcal{K}[\partial_x \partial_j \zeta] \partial_x f_j + \mathcal{K}[\partial_x f_j] \partial_x \partial_j \zeta\Big{)}.
		\end{equation*}
		
	\end{Def}
	
	With these formulas, we can state the main result of the section.
	
	\begin{prop}\label{Prop Quasilinear system}
		Let $T>0$, $t_0 \geq 1$ and $N \in \N$ be such that $	N \geq 5$. Furthermore, let $\mathbf{U} = (\zeta, \psi)^T \in \mathscr{E}_{\mathrm{bo}, T}^N$ be such that \eqref{non-cavitation} is satisfied on $[0,T]$. Also, for any $\alpha = (\alpha^1,\alpha^2) \in \N^2$, $\check\alpha^j= \alpha- \mathbf{e}_j$, with $1 \leq |\alpha| \leq N$ we define $\partial_{x,t}^{\alpha} = \partial_x^{\alpha_1} \partial_t^{\alpha_2}$, and let $\underline{w}^{\pm}$ be defined as in \eqref{Def w pm}. Moreover, define $\underline{w}$ by
		\begin{align*}
			\underline{w} &  =  \underline{w}^+ - \gamma \underline{w}^-, 
		\end{align*}
		and
		\begin{equation*}
			\zeta_{(\alpha)} = \partial_{x,t}^{\alpha} \zeta,
			\quad
			\psi_{(\alpha)} = \partial_{x,t}^{\alpha}\psi - \ve \underline{w} \partial_{x,t}^{\alpha}\zeta,
			\quad \psi_{\langle\check{\alpha}\rangle} = (\psi_{(\check \alpha^1)},\psi_{(\check\alpha^2)})^T.
		\end{equation*}
		Then for $\mathbf{U}_{(\alpha)} = (\zeta_{(\alpha)}, \psi_{(\alpha)})^T$ and $\mathbf{U}_{\langle\check\alpha\rangle } = (\zeta_{(\alpha)}, \psi_{\langle \check \alpha\rangle})^T$, there holds
		\begin{align}\label{Quasi 1}
			\text{if} \: \: 1\leq|\alpha| <N :
			& \quad \hspace{4.1cm}
			\partial_t \mathbf{U}_{(\alpha)} + \mathcal{A}[\mathbf{U}]\mathbf{U}_{(\alpha)} = \ve (R_{\alpha},S_{\alpha})^T,
			\\ 
			\text{if} \qquad\: |\alpha| = N:
			& \quad \label{Quasi 2}
			\partial_t \mathbf{U}_{(\alpha)} + \mathcal{A}[\mathbf{U}]\mathbf{U}_{(\alpha)} + \mathcal{B}[\mathbf{U}]\mathbf{U}_{(\alpha)}  + \mathcal{C}[\mathbf{U}]\mathbf{U}_{\langle \check\alpha\rangle} = \ve (R_{\alpha},S_{\alpha})^T
		\end{align}
		The rest functions satisfy the estimates
		\begin{equation}\label{Est R and S}
			|R_{\alpha}|^2_{H^1_{\mathrm{bo}}} + |S_{\alpha}|^2_{\dot{H}^{\frac{1}{2}}_{\mu}} \leq C \mathcal{E}^N(\mathbf{U}) \big{(} 1 + \ve^2 \sqrt{\mu}|[\![ \underline{V}^{\pm}]\!]|_{L^{\infty}}^2|\zeta|^2_{<N+\frac{1}{2}>}\big{)},
		\end{equation}
		for some $C(M,\mathcal{E}^N(\mathbf{U}))>0$ nondecreasing function of its argument and where $|\zeta|_{<N+\frac{1}{2}>}$ is defined by
		\begin{equation*}
			|\zeta|_{<N+\frac{1}{2}>} = \sum \limits_{\alpha \in \N^2, |\alpha|=N} |\partial_{x,t}^{\alpha}\zeta|_{\mathring{H}^{\frac{1}{2}}}.
		\end{equation*}
	\end{prop}
	 For the proof, we need to carefully track the dependencies in the small parameters. However, this part is very similar to the one in \cite{LannesTwoFluid13} and can therefore be considered a technical point. We give the details and point out the differences in the Appendix $\mathrm{A}$, Section \ref{proof of quasi}. Now, before we proceed with the energy estimates, we need to give a rigorous meaning to the operators given in Definition \ref{Def op}. This will be done in separate subsections.

	\subsubsection{Properties of $\mathcal{I}[\mathbf{U}]\bullet$} First, we study the operator $\mathcal{I}[\mathbf{U}]\bullet$ which appears later in the shape derivative formulas for $\mathcal{G}_{\mu}$.

	\begin{prop}\label{prop I }
		Let $t_0 \geq 1$ and $\mathbf{U} = (\zeta, \psi)^T$, with $\zeta \in H^{t_0 +3}(\R)$ satisfying \eqref{non-cavitation} and $\psi \in \dot{H}^{t_0 + 2}(\R)$. Then we may define the operator 
		\begin{equation*}
			\mathcal{I}[\mathbf{U}] \bullet
			=
			\partial_x ( \bullet \underline{V}^+)+\gamma  \mathcal{G}_{\mu}[\ve \zeta]  (\mathcal{G}^-_{\mu}[\ve \zeta])^{-1}\partial_x	\Big{(} 
			\bullet[\![ \underline{V}^{\pm}]\!]\Big{)},
		\end{equation*}
		and it satisfies the following properties:
		\begin{itemize} 
			\item [1.] For all $0\leq s \leq t_0$ and $f \in H^{s+\frac{1}{2}}(\R)$ there holds,
			\begin{equation}\label{est I inst op}
				|\mathcal{I}[\mathbf{U}]f|_{H^{s-\frac{1}{2}}} \leq M |f|_{H^{s+\frac{1}{2}}} |\partial_x \psi|_{H^{t_0 + \frac{1}{2}}}, 
			\end{equation}
			and
			\begin{equation}\label{est I[U]h psi in Hs}
				|\mathcal{I}[\mathbf{U}]f|_{H^{s-\frac{1}{2}}} \leq M |f|_{H^{t_0+\frac{1}{2}}} |\partial_x \psi|_{H^{s + \frac{1}{2}}}.
			\end{equation}
		
			\item [2.] Let $\mathfrak{a} = 1 + \mathfrak{b}$, with $\mathfrak{b} \in H^{t_0 +1}(\R)$, and for all $f \in L^2(\R)$ there holds,
			\begin{equation}\label{Est I a}
				\big{(}\mathfrak{a} \mathcal{I}[\mathbf{U}]f,f\big{)}_{L^2} 
				\leq
				M(t_0+3)(1+|\mathfrak{b}|_{H^{t_0+1}})|\partial_x \psi|_{H^{t_0+1}}|f|^2_{L^2}. 
			\end{equation}

			\item [3.] Let $K \in H^{t_0+1}(\R)$, then for all $f \in H^1(\R)$ there holds,
			\begin{equation}
				\big{(} \partial_x (K\partial_x  \mathcal{I}[\mathbf{U}]f),f\big{)}_{L^2} 
				\leq
				M(t_0+3)|K|_{H^{t_0+1}}|\partial_x \psi|_{H^{t_0+1}}|f|^2_{H^1}. 
			\end{equation}

			\item [4.] For all $f \in \dot{H}_{\mu}^{\frac{1}{2}}(\R)$, $g \in H^{\frac{1}{2}}(\R)$, one has
			\begin{equation}
				\big{(} \mathcal{I}[\mathbf{U}]^{\ast}f,g\big{)}_{L^2} 
				\leq
				M|\partial_x \psi|_{H^{t_0+\frac{1}{2}}}|f|_{\dot{H}^{\frac{1}{2}}_{\mu}}(|g|_{L^2} + \mu^{\frac{1}{4}}|g|_{H^{\frac{1}{2}}}).
			\end{equation}
		
		\end{itemize}

		\begin{proof}
			We prove each point in separate steps. \\
			
			\noindent
			\underline{Step 1.} To prove the first point, we observe that the first part of $\mathcal{I}[\mathbf{U}]f$ is estimated by the product estimate \eqref{Classical prod est} and \eqref{V pm est}:
			\begin{align*}
				|\partial_x (f \underline{V}^+)|_{H^{s-\frac{1}{2}}}
				&  \leq
				|f|_{H^{s+\frac{1}{2}}} |\underline{V}^+|_{H^{\max\{s+\frac{1}{2},t_0\}}} 
				\\
				& \leq M |f|_{H^{s+\frac{1}{2}}} |\partial_x \psi|_{H^{t_0 + \frac{1}{2}}}.
			\end{align*}
			To estimate the remaining part we first estimate  $\mathcal{G}_{\mu}$ with \eqref{1 Est G} to find that
			\begin{align*}
				|\mathcal{G}_{\mu}[\ve \zeta](\mathcal{G}^-_{\mu}[\ve \zeta])^{-1}\big{(} \partial_x( f [\![ \underline{V}^{\pm}]\!])\big{)}|_{H^{s-\frac{1}{2}}} 
				\leq
				 \mu^{\frac{3}{4}}|(\mathcal{G}^-_{\mu}[\ve \zeta])^{-1}\big{(} \partial_x( f [\![ \underline{V}^{\pm}]\!])\big{)}|_{\dot{H}_{\mu}^{s+\frac{1}{2}}}.
			\end{align*}
			Then use \eqref{Est derivative}, the product estimate \eqref{Classical prod est}, and \eqref{V pm est} to see that
			\begin{align*}
				 \mu^{\frac{3}{4}}|(\mathcal{G}^-_{\mu}[\ve \zeta])^{-1}\big{(} \partial_x( f [\![ \underline{V}^{\pm}]\!])\big{)}|_{\dot{H}_{\mu}^{s+\frac{1}{2}}}  
				 & 
				 \leq  \mu^{\frac{1}{2}}|(\mathcal{G}^-_{\mu}[\ve \zeta])^{-1}\big{(} \partial_x( f [\![ \underline{V}^{\pm}]\!])\big{)}|_{\mathring{H}^{s+\frac{1}{2}}}  
				 \\ 
				 & \leq
				| f [\![ \underline{V}^{\pm}]\!]|_{H^{s+\frac{1}{2}}}  
				 \\
				 & 
				 \leq 
				 M |f|_{H^{s+\frac{1}{2}}} |\partial_x \psi|_{H^{t_0 + \frac{1}{2}}}.
			\end{align*}

			For the proof of \eqref{est I[U]h psi in Hs}, it is proved similarly where we only need to modify the part when we apply the product estimate. \\

			\noindent
			\underline{Step 2.} In this estimate, we note that the instability operator is a first-order differential operator acting on $f \mapsto \mathcal{I}[\mathbf{U}]f$. But the principal symbol is skew-symmetric. Indeed, for the first part, we use integration by parts
			\begin{align*}
				\big{(}\mathfrak{a} \partial_x(f\underline{V}^+),f\big{)}_{L^2} 
				& =
				\big{(}\mathfrak{a}(\partial_x \underline{V}^+) f,f\big{)}_{L^2} 
				- 
				\frac{1}{2}	\big{(}\big{(}\partial_x(\mathfrak{a}\underline{V}^+)\big{)} f,f\big{)}_{L^2}.
			\end{align*}
			Then Hölder's inequality, the Sobolev embedding $H^{t_0}(\R) \hookrightarrow L^{\infty}(\R)$, the product estimate \eqref{Classical prod est}, and \eqref{Def V pm}  gives
			\begin{align*}
				|\big{(}\mathfrak{a} \partial_x(f\underline{V}^+),f\big{)}_{L^2}|
				& =
				(1+ |\mathfrak{b}|_{H^{t_0+1}})|\underline{V}^+|_{H^{t_0+1}}|f|_{L^2}^2
				\\ 
				& 
				\leq 
				(1+ |\mathfrak{b}|_{H^{t_0+1}})|\partial_x \psi|_{H^{t_0+1}}|f|_{L^2}^2.
			\end{align*}
			For the second part, we use integration by parts together with the fact that $\mathcal{G}_{\mu}$ and $\mathcal{G}_{\mu}^-$ are symmetric to make the following decomposition
			\begin{align*}
				\big{(}
				\mathfrak{a}\mathcal{G}_{\mu}[\ve \zeta] (\mathcal{G}^-_{\mu}[\ve \zeta])^{-1}\big{(} \partial_x(f [\![ \underline{V}^{\pm}]\!])\big{)},f
				\big{)}_{L^2} 
				& =
				-
				\big{(}  f [\![ \underline{V}^{\pm}]\!] , \partial_x\Big( (\mathcal{G}^-_{\mu}[\ve \zeta])^{-1}\mathcal{G}_{\mu}[\ve \zeta](\mathfrak{a}f)\Big)
				\big{)}_{L^2} 
				\\
				& =
				-
				\big{(}  f [\![ \underline{V}^{\pm}]\!] , \partial_x\Big( (\mathcal{G}^-_{\mu}[\ve \zeta])^{-1}\mathcal{G}_{\mu}[\ve \zeta]-\mathrm{Op}\Big( \frac{S^+}{S^-S_J} \Big)\Big)(\mathfrak{a}f)
				\big{)}_{L^2} 
				\\
				&
				\hspace{0.5cm}
				+
				\big{(}  \partial_x(f [\![ \underline{V}^{\pm}]\!]) , \mathrm{Op}\Big( \frac{S^+}{S^-S_J} \Big)(\mathfrak{a}f)
				\big{)}_{L^2} 
				\\ 
				& = :
				R_1 + R_2.
			\end{align*}
			For $R_1$ we use Cauchy-Schwarz and estimate \eqref{Est G-invG} to deduce that
			\begin{align*}
				|R_1| 
				& 
				\leq 	
				|f [\![ \underline{V}^{\pm}]\!]|_{L^2} | ( (\mathcal{G}^-_{\mu}[\ve \zeta])^{-1}\mathcal{G}_{\mu}[\ve \zeta] -\mathrm{Op}\Big( \frac{S^+}{S^-S_J} \Big))(\mathfrak{a}f)|_{\mathring{H}^{1}}
				\\ 
				& 
				\leq 
				\ve M(t_0+3)
				|f [\![ \underline{V}^{\pm}]\!]|_{L^2} | \mathfrak{a}f|_{L^2}.
				\color{black}
			\end{align*}
			Then use the product estimate \eqref{Classical prod est} and \eqref{V pm est} to get that
			 \begin{equation*}
			 	|R_1| 
			 	 \leq 
			 	\ve  M(t_0+3)(1+ |\mathfrak{b}|_{H^{t_0}}) |\partial_x \psi|_{H^{t_0}}|f|^2_{L^2}.\color{black}
			 \end{equation*}
		 	For $R_2$, we will make two different decomposition where we first split it into several parts:
		 	\begin{align*}
		 		R_2 
		 		& =
		 		\big{(}  \Big{(}\mathrm{Op}\Big( \frac{S^+}{S^-S_J} \Big)^{\ast}-\mathrm{Op}\Big( \frac{S^+}{S^-S_J} \Big)\Big{)}\partial_x(f [\![ \underline{V}^{\pm}]\!]) , (\mathfrak{a}f)
		 		\big{)}_{L^2} 
		 		+
		 		\big{(} \mathrm{Op}\Big( \frac{S^+}{S^-S_J} \Big{)}(f \partial_x [\![ \underline{V}^{\pm}]\!]) , (\mathfrak{a}f)
		 		\big{)}_{L^2} 
		 		\\ 
		 		& 
		 		\hspace{0.5cm}
		 		+
		 		\big{(} [\mathrm{Op}\Big( \frac{S^+}{S^-S_J} \Big{)}, [\![ \underline{V}^{\pm}]\!]]\partial_xf , (\mathfrak{a}f)
		 		\big{)}_{L^2} 
		 		+
		 		\big{(}  [\![ \underline{V}^{\pm}]\!] \mathrm{Op}\Big( \frac{S^+}{S^-S_J} \Big{)} \partial_x f , (\mathfrak{a}f)
		 		\big{)}_{L^2} 
		 		\\
		 		& =
		 		R_2^1 + R_2^2 + R_2^3 +R_2^4. 
		 	\end{align*}
	 		For $R_2^1$ we use estimate \eqref{adjoint est S+/S-SJ}, for $R_2^2$ we use \eqref{est S+/S-SJ}, for $R_2^3$ we use \eqref{commutator est S+/S-SJ}, combined with \eqref{Classical prod est} and \eqref{V pm est} we get that
	 		\begin{equation*}
	 			|R_2^1| + |R_2^2| + |R_2^3| \leq M(1+ |\mathfrak{b}|_{H^{t_0}}) |\partial_x \psi|_{H^{t_0+1}}|f|^2_{L^2}.
	 		\end{equation*}
 			On the other hand,  to cancel $R_2^4$ we can use integration by parts on $R_2$. Indeed, we obtain that
 			\begin{align*}
 				R_2 
 				& =
 				-
 				\big{(}  (f [\![ \underline{V}^{\pm}]\!]) , \mathrm{Op}\Big( \partial_x\frac{S^+}{S^-S_J} \Big)(\mathfrak{a}f)
 				\big{)}_{L^2} 
 				-
 				\big{(}  (f [\![ \underline{V}^{\pm}]\!]) ,\mathrm{Op}\Big( \frac{S^+}{S^-S_J} \Big)(f\partial_x \mathfrak{b})
 				\big{)}_{L^2} 
 				\\ 
 				&
 				\hspace{0.5cm}
 				-
 				\big{(}  (f [\![ \underline{V}^{\pm}]\!]) ,[\mathrm{Op}\Big( \frac{S^+}{S^-S_J} \Big), \mathfrak{b}]\partial_xf 
 				\big{)}_{L^2}
 				-
 				\big{(}  (f [\![ \underline{V}^{\pm}]\!]) ,\mathfrak{a}\mathrm{Op}\Big( \frac{S^+}{S^-S_J} \Big) \partial_xf 
 				\big{)}_{L^2} 
 				\\ 
 				& = 
 				R_2^5 + R_2^6 + R_2^7 + R_2^8. 
 			\end{align*}
 			Here we estimate $R_2^5, R_2^6, R_2^7$ as above where we also use estimate \eqref{Op dx S+/S- 1/SJ} for $R_2^5$. Adding these two decompositions implies
 			\begin{equation*}
 				2|R_2| \leq \big{|}\sum \limits_{j=1}^8 R_2^j \big{|} \leq  M (1+ |\mathfrak{b}|_{H^{t_0+1}}) |\partial_x \psi|_{H^{t_0+1}}|f|^2_{L^2}.
 			\end{equation*}

		\noindent
		\underline{Step 3.} Since $K$ is symmetric and $f \mapsto \mathcal{I}[\mathbf{U}]f$ is skew-symmetric, we can absorb one derivative by integrating by parts as we did in the previous step.\\

		\noindent
		\underline{Step 4.} From the definition \eqref{I star} we have that
		\begin{align*}
				\big{(} \mathcal{I}[\mathbf{U}]^{\ast}f,g\big{)}_{L^2} 
				& = 
				-\big{(}\underline{V}^+\partial_x  f,g\big{)}_{L^2} 
				-
				\gamma \big{(}
				[\![ \underline{V}^{\pm}]\!]  \partial_x\big{(}(\mathcal{G}^-_{\mu}[\ve \zeta])^{-1}\mathcal{G}_{\mu}[\ve \zeta]f\big{)},g\big{)}_{L^2}
				\\ 
				& = 
				A_1 + \gamma A_2.
		\end{align*}
		For the first term, we introduce a commutator, then apply Hölder's inequality, Sobolev embedding, estimate \eqref{V pm est}, and \eqref{Commutator mu quart} to find that 
		\begin{align*}
			|A_1|
			& \leq 
			|[\underline{V}^+, (1+\sqrt{\mu}|\mathrm{D}|)^{\frac{1}{2}}]\frac{\partial_x }{(1+\sqrt{\mu}|\mathrm{D}|)^{\frac{1}{2}}} f|_{L^2}|g|_{L^2}
			\\ 
			& 
			\hspace{0.5cm}
			+
			|\underline{V}^+|_{L^{\infty}} |\frac{\partial_x }{(1+\sqrt{\mu}|\mathrm{D}|)^{\frac{1}{2}}}  f| |(1+\sqrt{\mu}|\mathrm{D}|)^{\frac{1}{2}}g|_{L^2}
			\\ 
			& 
			\leq
			M|\partial_x \psi|_{H^{t_0+\frac{1}{2}}}|f|_{\dot{H}^{\frac{1}{2}}_{\mu}}(|g|_{L^2} + \mu^{\frac{1}{4}}|g|_{H^{\frac{1}{2}}}).
		\end{align*}
		For the second term, we instead use commutator estimate \eqref{Commutator Dhalf} and \eqref{Inverse est 2} to obtain
		\begin{align*}
			|A_2|
			& \leq 
			|\big{[}[\![ \underline{V}^{\pm}]\!], |\mathrm{D}|^{\frac{1}{2}} \big{]} \mathcal{H}|\mathrm{D}|^{\frac{1}{2}}(\mathcal{G}^-_{\mu}[\ve \zeta])^{-1}\mathcal{G}_{\mu}[\ve \zeta]f|_{L^2} |g|_{L^2}
			+
			|[\![ \underline{V}^{\pm}]\!]|_{L^{\infty}}
			| (\mathcal{G}^-_{\mu}[\ve \zeta])^{-1}\mathcal{G}_{\mu}[\ve \zeta]f|_{\mathring{H}^{\frac{1}{2}}} |g|_{\mathring{H}^{\frac{1}{2}}}
			\\
			& \leq 
			M\mu^{\frac{1}{4}}|\partial_x \psi|_{H^{t_0+\frac{1}{2}}}|f|_{\dot{H}^{\frac{1}{2}}_{\mu}}(|g|_{L^2} + |g|_{H^{\frac{1}{2}}}).
		\end{align*}
		  
		\end{proof}
	 
	\end{prop}

	\subsection{Properties of $\mathfrak{Ins}[\mathbf{U}]\bullet$. } The instability operator is defined in terms of  $\mathfrak{E}_{\mu}$ which is given by
	\begin{equation*}
		\mathfrak{E}_{\mu}[\ve \zeta] \bullet = \partial_x \circ (\mathcal{J}_{\mu}[\ve \zeta])^{-1}(\mathcal{G}^-_{\mu}[\ve \zeta])^{-1}\circ \partial_x \bullet. 
	\end{equation*}
	Meaning we first need to study its properties, which is the topic of the next Proposition. 
	\begin{prop}\label{Prop e}
		Let $t_0\geq 1$ and $\zeta \in H^{t_0+2}(\R)$ be such that \eqref{non-cavitation} is satisfied. Then we have the following results:
		\begin{itemize}
			\item [1.] There exist a constant $c\leq M$ such that for all $f\in H^{\frac{1}{2}}(\R)$ there holds,
			\begin{equation}\label{Est E 1}
				0 \leq  \big{(}\mathfrak{E}_{\mu}[\ve \zeta] f,f \big{)}_{L^2} \leq \frac{c}{\mu}|(1+\sqrt{\mu}|\mathrm{D}|)^{\frac{1}{2}} f|_{L^2}^2.
			\end{equation}
			\item [2.] If we suppose further that $\zeta$ is time dependent and satisfies \eqref{non-cavitation} uniformly in time, then for all $f \in H^{\frac{1}{2}}(\R)$ there holds,
			\begin{equation}\label{Est E 2}
				|\big{(}\big{[}\partial_t, \mathfrak{E}_{\mu}[\ve \zeta]\big{]} f,f \big{)}_{L^2}| \leq \frac{\ve}{\mu} M |\partial_t \zeta|_{H^{t_0+1}}|(1+\sqrt{\mu}|\mathrm{D}|)^{\frac{1}{2}} f|_{L^2}^2.
			\end{equation}
			\item [3.] If we suppose further that $\zeta \in H^{t_0 +3}(\R)$. Then for all  $f \in H^{\frac{1}{2}}(\R)$ and $g \in H^{-\frac{1}{2}}(\R)$ there holds,
			\begin{equation}\label{est E East}
				|\big{(}f,\big{(}\mathfrak{E}_{\mu}[\ve \zeta]^{\ast}- \mathfrak{E}_{\mu}[\ve \zeta]\big{)} g\big{)}_{L^2}| \leq \ve \mu^{-\frac{3}{4}} M(t_0+3) |f|_{H^{\frac{1}{2}}} |g|_{H^{-\frac{1}{2}}}.
			\end{equation}
			
 		\end{itemize}
	\end{prop}
	\begin{proof} We give the proof in three separate points. \\ 
		
		\noindent
		\underline{Step 1.}	For the proof of the first point, we deduce the positivity by integrating parts:
		\begin{align*}
			\big{(}\mathfrak{E}_{\mu}[\ve \zeta] f,f \big{)}_{L^2} 
			& =
			- \big{(} (\mathcal{J}_{\mu}[\ve \zeta])^{-1}(\mathcal{G}^-_{\mu}[\ve \zeta])^{-1} \partial_xf,\partial_x f \big{)}_{L^2}
			\\ 
			& =
			- \big{(} g, (\mathcal{G}^-_{\mu}[\ve \zeta]) \mathcal{J}_{\mu}[\ve \zeta] g\big{)}_{L^2}
			\\ 
			& \geq 0
			,
		\end{align*}
		where we defined $g$ by
		\begin{equation*}
			g = (\mathcal{J}_{\mu}[\ve \zeta])^{-1}(\mathcal{G}^-_{\mu}[\ve \zeta])^{-1} \partial_xf,
		\end{equation*}
		and used that $\mathcal{G}^{-}_{\mu} \mathcal{J}_{\mu} = \mathcal{G}^-_{\mu} - \gamma \mathcal{G}^+_{\mu}$ is negative.
		
		For the upper bound, we use Plancherel's identity, and Cauchy-Schwarz inequality with the estimates \eqref{Est J} and \eqref{Est derivative} to find that
		\begin{align*}
			|\big{(}\mathfrak{E}_{\mu}[\ve \zeta] f,f \big{)}_{L^2}|
			& \leq 
			|(\mathcal{J}_{\mu}[\ve \zeta])^{-1}(\mathcal{G}^-_{\mu}[\ve \zeta])^{-1} \partial_xf|_{\dot{H}^{\frac{1}{2}}_{\mu}} |(1+\sqrt{\mu}|\mathrm{D}|)^{\frac{1}{2}} f|_{L^2}
			\\
			& \leq  
			\mu^{-\frac{1}{4}}
			M
			|(\mathcal{G}^-_{\mu}[\ve \zeta])^{-1} \partial_xf|_{\mathring{H}^{\frac{1}{2}}} |(1+\sqrt{\mu}|\mathrm{D}|)^{\frac{1}{2}} f|_{L^2}
			\\
			& \leq 
			\mu^{-\frac{3}{4}}
			M |f|_{\mathring{H}^{\frac{1}{2}}} |(1+\sqrt{\mu}|\mathrm{D}|)^{\frac{1}{2}} f|_{L^2}.
		\end{align*}

		\noindent
		\underline{Step 2.} By direct computations, we need to control the following two terms:
		\begin{align*}
			\big{[}\partial_t, \mathfrak{E}_{\mu}[\ve \zeta]\big{]} f 
			& = 
			\partial_x \circ (\mathrm{d}_{\zeta}(\mathcal{J}_{\mu}[\ve \zeta])^{-1}(\partial_t \zeta))\circ(\mathcal{G}^-_{\mu}[\ve \zeta])^{-1}\circ \partial_x f
			\\
			&
			\hspace{0.5cm}
			+
			\partial_x \circ (\mathcal{J}_{\mu}[\ve \zeta])^{-1}\circ\mathrm{d}_{\zeta}(\mathcal{G}^-_{\mu}[\ve \zeta])^{-1}(\partial_t \zeta)\circ \partial_x f
			\\ 
			& = f_1 + f_2.
		\end{align*}
		For the contribution of the first term, we use Plancherel's identity, Cauchy-Schwarz, and estimates \eqref{shape derivative Jinv} and \eqref{Est derivative}:
		\begin{align*}
			|\big{(} f_1, f\big{)}_{L^2}|
			& \leq  
			|\mathrm{d}_{\zeta}(\mathcal{J}_{\mu}[\ve \zeta])^{-1}(\partial_t \zeta)(\mathcal{G}^-_{\mu}[\ve \zeta])^{-1}\circ \partial_x f|_{\dot{H}^{\frac{1}{2}}_{\mu}} |(1+\sqrt{\mu}|\mathrm{D}|)^{\frac{1}{2}} f|_{L^2}
			\\ 
			& \leq \ve\mu^{-\frac{1}{4}} M |\partial_t \zeta|_{H^{t_0+1}}
			|(\mathcal{G}^-_{\mu}[\ve \zeta])^{-1} \partial_xf|_{\mathring{H}^{\frac{1}{2}}} |(1+\sqrt{\mu}|\mathrm{D}|)^{\frac{1}{2}} f|_{L^2}
			\\
			&
			\leq 
			\ve\mu^{-1} M |\partial_t \zeta|_{H^{t_0+1}}
			|(1+\sqrt{\mu}|\mathrm{D}|)^{\frac{1}{2}} f|_{L^2}^2.
		\end{align*}
		For the estimate on $f_2$ we use we use Plancherel's identity, Cauchy-Schwarz, \eqref{Est J}, \eqref{G- inv djG-}, and then \eqref{Est derivative} to obtain
		\begin{align*}
			|\big{(} f_2, f\big{)}_{L^2}|
			& \leq  
			\mu^{-\frac{1}{4}}|(\mathcal{G}^-_{\mu}[\ve \zeta])^{-1}\circ \mathrm{d}_{\zeta}\mathcal{G}^-_{\mu}[\ve \zeta](\partial_t \zeta)\circ (\mathcal{G}^-_{\mu}[\ve \zeta])^{-1}\circ \partial_x f|_{\mathring{H}^{\frac{1}{2}}} |(1+\sqrt{\mu}|\mathrm{D}|)^{\frac{1}{2}} f|_{L^2}
			\\
			&
			\leq 
			\ve\mu^{-1} M |\partial_t \zeta|_{H^{t_0+1}}
			|(1+\sqrt{\mu}|\mathrm{D}|)^{\frac{1}{2}} f|_{L^2}^2.\\
		\end{align*}

		\noindent
		\underline{Step 3.} We will split the left hand side of \eqref{est E East} into two parts and trade the differential operators with corresponding PDOs.  In particular, for the first term we observe that
		\begin{align*}
			\big{(}f, \mathfrak{E}_{\mu}[\ve \zeta]^{\ast} g\big{)}_{L^2}
			& = 
			\big{(} \Lambda^{\frac{1}{2}}\partial_x (\mathcal{J}_{\mu})^{-1}\Big{(}(\mathcal{G}_{\mu}^{-})^{-1}\partial_x - \mathrm{Op}(\frac{\partial_x}{S^-})\Big{)}f, \Lambda^{-\frac{1}{2}} g\big{)}_{L^2}
			\\ 
			& \hspace{0.5cm} 
			+
			\big{(}   \Lambda^{\frac{1}{2}}\partial_x  \Big{(}(\mathcal{J}_{\mu})^{-1} - \mathrm{Op}(\frac{1}{S_J})   \Big{)} \mathrm{Op}(\frac{\partial_x}{S^-})f, \Lambda^{-\frac{1}{2}}g\big{)}_{L^2} 
			+
			\big{(}\partial_x \mathrm{Op}(\frac{1}{S_J})  \mathrm{Op}(\frac{\partial_x}{S^-})f,g\big{)}_{L^2} 
			\\ 
			& = 
			A_1 + A_2 + A_3.
		\end{align*}
		Here $A_1$ is estimated by Cauchy-Schwarz inequality, \eqref{Est J},  \eqref{est G- inv dx pdo}, to find that
		\begin{align*}
			|A_1| 
			\leq |(\mathcal{J}_{\mu})^{-1}\Big{(}(\mathcal{G}_{\mu}^{-})^{-1}\partial_x - \mathrm{Op}(\frac{\partial_x}{S^-})\Big{)}f|_{\dot{H}^{\frac{3}{2}}_{\mu}} |g|_{H^{-\frac{1}{2}}}\leq 
			\ve \mu^{-\frac{1}{2}}M(t_0+3)|f|_{H^\frac{1}{2}}|g|_{H^{-\frac{1}{2}}},
		\end{align*}
		while $A_2$ is estimated similarly using  \eqref{Est Jinv pdo} with $s=k=1$ to find that
		\begin{equation*}
			|A_2| \leq \ve \mu^{-\frac{3}{4}} M(t_0 + 3)|f|_{H^{\frac{1}{2}}}|g|_{H^{-\frac{1}{2}}}.
		\end{equation*}
		Lastly, the contribution will cancel up to a smoothing estimate with the remaining terms
		\begin{align*}
			-
			\big{(}f, \mathfrak{E}_{\mu}[\ve \zeta] g\big{)}_{L^2}
			& = 
			\big{(} \mathfrak{P}^{-1}\partial_x f,  \mathfrak{P}(\mathcal{J}_{\mu})^{-1}\Big{(}(\mathcal{G}_{\mu}^{-})^{-1}\partial_x - \mathrm{Op}(\frac{\partial_x}{S^-})\Big{)}g\big{)}_{L^2}
			\\
			& 
			\hspace{0.5cm}
			+
			\big{(} \mathfrak{P}^{-1} \partial_x f, \mathfrak{P} \Big{(}(\mathcal{J}_{\mu})^{-1} - \mathrm{Op}(\frac{1}{S_J}) \Big{)}  \mathrm{Op}(\frac{\partial_x}{S^-})g\big{)}_{L^2}
			+
			\big{(} \partial_x f, \mathrm{Op}(\frac{1}{S_J}) \mathrm{Op}(\frac{\partial_x}{S^-})g\big{)}_{L^2}
			\\
			& = 
			B_1 + B_2 + B_3,
		\end{align*}
		where $|B_1| + |B_2|$ is estimated as $|A_1| + |A_2|$ to find that 
		\begin{equation*}
			|B_1| + |B_2| \leq \ve \mu^{-\frac{3}{4}} M(t_0 + 3)|f|_{H^{\frac{1}{2}}}|g|_{H^{-\frac{1}{2}}}.
		\end{equation*}
		On the other hand, $B_3$ can be compensated with $A_3$ by noting that
		\begin{align*}
			B_3 
			& = -
			\big{(}\Lambda^{\frac{1}{2}} \mathrm{Op}(\frac{\partial_x}{S^-})\mathrm{Op}(\frac{1}{S_J})^{\ast}\partial_x f,\Lambda^{-\frac{1}{2}} g
			\big{)}_{L^2}
			\\ 
			& = 
			-
			\big{(} \Lambda^{\frac{1}{2}}\mathrm{Op}(\frac{\partial_x}{S^-})\Big{(}\mathrm{Op}(\frac{1}{S_J})^{\ast} - \mathrm{Op}(\frac{1}{S_J})\Big{)}\partial_x f,\Lambda^{-\frac{1}{2}} g
			\big{)}_{L^2}
			\\
			& \hspace{0.5cm}
			-
			\big{(}\Lambda^{\frac{1}{2}} \big{[} \mathrm{Op}(\frac{\partial_x}{S^-}), \mathrm{Op}(\frac{1}{S_J})\big{]}\partial_x f, \Lambda^{-\frac{1}{2}}g
			\big{)}_{L^2}
			-
			\big{( \Lambda^{\frac{1}{2}} \big{[}\mathrm{Op}(\frac{1}{S_J}), \partial_x\big{]} }\mathrm{Op}(\frac{\partial_x}{S^-}) f, \Lambda^{-\frac{1}{2}}g
			\big{)}_{L^2} - A_3
			\\ 
			& = 
			B_3^1 + B_3^2 + B_3^3 - A_3,
		\end{align*}
		where $|	B_3^1| + |B_3^2| + |B_3^3 |$ is estimated using PDO estimates \eqref{Op SJ 1},  \eqref{Op SJ 3}, and \eqref{Op SJ 2} respectively, which yields the following bound
		\begin{equation*}
			 |	B_3^1| + |B_3^2| + |B_3^3 | \leq \ve \mu^{-\frac{1}{2}}M|f|_{H^{\frac{1}{2}}} |g|_{H^{-\frac{1}{2}}}.
		\end{equation*} 
		Adding all these estimates implies the final result.
		\begin{align*}
			|\big{(}f,\big{(}\mathfrak{E}_{\mu}[\ve \zeta]^{\ast}- \mathfrak{E}_{\mu}[\ve \zeta]\big{)} g\big{)}_{L^2}| \leq \ve \mu^{-\frac{3}{4}} M(t_0+3) |f|_{H^{\frac{1}{2}}} |g|_{H^{-\frac{1}{2}}}.
		\end{align*}
		
	\end{proof}

	\begin{remark}
		In the first point, we can define the smallest constant $c\leq M$ such that \eqref{Est E 1} holds by 
		\begin{equation}\label{def e}
			\mathfrak{e} (\zeta)
			= \sup\limits_{f \in H^{\frac{1}{2}}(\R), f\neq 0} \mu  \frac{\big{(}(\mathcal{J}_{\mu}[\ve \zeta])^{-1}(\mathcal{G}^-_{\mu}[\ve \zeta])^{-1} \partial_x f, \partial_x f \big{)}_{L^2}}{|1+\sqrt{\mu}|\mathrm{D}|^{\frac{1}{2}}f|_{L^2}^2}.
		\end{equation}
		Also, it is evident from the proof that for $f,g \in H^{\frac{1}{2}}(\R)$ that there holds
		\begin{equation}\label{est on E fg}
			 \big{(}\mathfrak{E}_{\mu}[\ve \zeta] f,g \big{)}_{L^2} \leq \frac{c}{\mu} |(1+\sqrt{\mu}|\mathrm{D}|)^{\frac{1}{2}} f|_{L^2}|(1+\sqrt{\mu}|\mathrm{D}|)^{\frac{1}{2}} g|_{L^2}.
		\end{equation}
	\end{remark}

	For the next result, we will treat the properties of the instability operator. In particular, we will show that under the stability criterion, we can have a coercivity-type estimate. This is essential for the well-posedness theory to work and relies on the surface tension parameter through the Bond number $\mathrm{bo}^{-1}>0$. To be clear, we restate the stability criterion in the introduction:
		\begin{equation*}
			0 <  \mathfrak{d}(\mathbf{U}) : = \inf \limits_{\R} \mathfrak{a}- \Upsilon \mathfrak{c} (\zeta)  |[\![ \underline{V}^{\pm}]\!]|_{L^{\infty}}^4,
		\end{equation*}
		where  
		\begin{equation*}
			\Upsilon = \frac{\mathrm{bo}}{2} (1-\gamma)^2\gamma^2 \ve^4\mu 
		\end{equation*}
		and $\mathfrak{a}$ is given by \eqref{def a}, $\mathfrak{e}$ is given in \eqref{def e}, and we define
		\begin{align*}
			\mathfrak{c}(\zeta)  = 
			\mathfrak{e}(\zeta)^2(1+\ve^2 \mu |\partial_x \zeta|_{L^{\infty}}^2)^{\frac{3}{2}}.
		\end{align*}
		%
		%
		%

		\begin{prop}\label{Prop stability} Let $\ve,\mu, \mathrm{bo}^{-1} \in(0,1)$ such that $\varepsilon^2 \mu \leq \mathrm{bo}^{-1}$, $T>0$, $t_0 \geq 1$, $\mathcal{E}^{\lceil t_0+2\rceil}(\mathbf{U})$ be defined by \eqref{Energy functional N}, and $\mathbf{U} = (\zeta, \psi) \in \mathscr{E}^{\lceil t_0+2\rceil}_{\mathrm{bo},T}$ be such that  the non-cavitation condition \eqref{non-cavitation} holds and satisfies the stability criterion \eqref{Stability criterion} on $[0,T]$.  Then we have the following set of inequalities on the same time interval:
		\begin{itemize}
			\item [1.] For all $u \in H^1_{\gamma, \mathrm{bo}}(\R)$ and $\mathfrak{Ins}(\mathbf{U})$ defined by \eqref{inst} one has
			\begin{equation}\label{upper bound ins}
				\big{(}u, \mathfrak{Ins}(\mathbf{U})u\big{)}_{L^2} \leq C\big{(} \mathcal{E}^{\lceil t_0+2\rceil }(\mathbf{U})\big{)} |u|_{H^1_{\mathrm{bo}}}^2,
			\end{equation}
			where $C>0$ is a nondecreasing function of its argument.\\ 
			\item [2.] For $a(\mathbf{U})$ defined by
			\begin{equation}\label{def a}
				a(\mathbf{U}) = (1-\gamma)\gamma \ve^2  \mathfrak{e} (\zeta) | [\![ \underline{V}^{\pm}]\!] |_{L^{\infty}}^2,
			\end{equation}
			and there is some $C_1>0$ and $b(\mathbf{U})$ defined by
			\begin{equation}\label{def b}
				b(\mathbf{U}) = (1-\gamma)\gamma \ve^2\sqrt{\mu} C_1 \mathfrak{e} (\zeta) | [\![ \underline{V}^{\pm}]\!] |_{H^{t_0+1}}^2,
			\end{equation}
			such that for $\mathfrak{d}$ defined by \eqref{Stability criterion} there holds,
			\begin{equation}\label{lower bound ins}
				\min\big{\{}\mathfrak{d}(\mathbf{U}), \frac{1}{M}\big{\}} |u|_{H^1_{\mathrm{bo}}}^2
				\leq 
				\big{(}u, \mathfrak{Ins}(\mathbf{U})u\big{)}_{L^2} + a(\mathbf{U})|u|_{L^2}^2 
				+
				b(\mathbf{U})|u|_{H^{-\frac{1}{2}}}^2.
				\color{black}
			\end{equation}
			\item [3.]  Lastly, there is a control on $u \in \mathring{H}^{\frac{1}{2}}(\R)$ through the inequality
			\begin{equation}\label{zeta half norm}
				(1-\gamma)\gamma \ve^2
				\sqrt{\mu}\mathfrak{e}(\zeta) \max\limits_{|\alpha|\leq 1} |\partial^{\alpha}_{x,t}[\![ \underline{V}^{\pm}]\!]|_{L^{\infty}}^2  | u|^2 \leq \varepsilon C\big{(} \mathcal{E}^{\lceil t_0+2\rceil }(\mathbf{U})\big{)} \Big{(}|u|_{L^2}^2 + \mathrm{bo}^{-1} |\partial_x u|_{L^2}^2\Big{)}.
			\end{equation}
			where $C>0$ is a nondecreasing function of its argument.\\ 
			%
			%
			%
		\end{itemize}
	\end{prop}

	\begin{proof}  The proof is similar to the one of Lemma $11$ in \cite{LannesTwoFluid13}. However, we give a short proof for the convenience of the reader to track the constants that are responsible for the definitions above. We divide the proof into three steps \\

	\noindent
	\underline{Step 1.} For the proof of \eqref{upper bound ins}, we have to deal with the terms:
	\begin{align*}
		\big{(}u, \mathfrak{Ins}(\mathbf{U})u\big{)}_{L^2}  
		& =
		\big{(}u, \mathfrak{a} u\big{)}_{L^2}  
		-	
		(1-\gamma)\gamma \ve^2 \mu	\big{(}u,  [\![ \underline{V}^{\pm}]\!]   \mathfrak{E}_{\mu}[\ve \zeta]	\big{(} 
		u[\![ \underline{V}^{\pm}]\!] \big{)} \big{)}_{L^2} 
		\\ 
		& 
		\hspace{0.5cm}
		-
		\mathrm{bo}^{-1} \big{(}u, \partial_x \mathcal{K}[\ve \sqrt{\mu} \partial_x \zeta]\partial_x  u\big{)}_{L^2} 
		\\ 
		& = 
		I_1 + I_2 + I_3.
	\end{align*}
	For $I_1$, we have that
	\begin{equation}\label{est I1}
		\big{(} \inf \limits_{x\in \R} \mathfrak{a} \big{)} |u|_{L^2}^2 \leq I_1 \leq C\big{(} \mathcal{E}^{\lceil t_0+2\rceil }(\mathbf{U})\big{)} |u|_{L^2}^2,
	\end{equation}
	where the upper bound is a consequence of the Hölder inequality and estimate \eqref{est on a}. For the estimate on $I_2$, we employ \eqref{Est E 1} to find that
	\begin{align*}
		-(1-\gamma)\gamma \ve^2 \mathfrak{e}(\zeta)
		|(1+\sqrt{\mu} |\mathrm{D}|)^{\frac{1}{2}}(u [\![ \underline{V}^{\pm}]\!])|_{L^2}^2 \leq I_2 \leq 0.
	\end{align*}
	For the estimate on $I_3$, we have by integration by parts that
	\begin{equation}\label{est I3}
		\mathrm{bo}^{-1} (1+ \ve^2 \mu |\partial_x \zeta|^2_{L^{\infty}})^{-\frac{3}{2}} |\partial_x u|_{L^2}^2 \leq I_3 \leq \mathrm{bo}^{-1}|\partial_x u|_{L^2}^2.
	\end{equation}
	So the upper bound in \eqref{upper bound ins} is proved.\\

	\noindent
	\underline{Step 2.} For the proof of \eqref{lower bound ins}, we need to work on the lower bound in $I_2$. First, we can replace the lower bound using the commutator estimate \eqref{Commutator est} and Hölder's inequality to find that
	\begin{align}\label{eq 5.7 in lannes}
		|\langle \sqrt{\mu}\mathrm{D}\rangle^{\frac{1}{2}}(u  [\![ \underline{V}^{\pm}]\!])|_{L^2}^2\notag
		& \leq 
		C_1 \sqrt{\mu}| [\![ \underline{V}^{\pm}]\!]|_{H^{t_0+1}}^2|u|_{H^{-\frac{1}{2}}}^2\notag 
		+
		\sqrt{\mu}
		| [\![ \underline{V}^{\pm}]\!]|_{L^{\infty}}^2||\mathrm{D}|^{\frac{1}{2}}u|_{L^2}^2
		\\ 
		& 
		\hspace{0.5cm}
		+
		| [\![ \underline{V}^{\pm}]\!]|^2_{L^{\infty}}|u|_{L^2}^2,
	\end{align}
	for some constant $C_1>0$. Then we may define $b(\mathbf{U})$ by \eqref{def b} and use the expression for $a(\mathbf{U})$ to find that
	\begin{align*}
		I_2 \geq -b(\mathbf{U})|u|_{H^{-\frac{1}{2}}}^2 - a(\mathbf{U})\big{(} |u|_{L^2}^2 
		+  \sqrt{\mu} |u|^2_{\mathring{H}^{\frac{1}{2}}}\big{)}.
	\end{align*}
	Adding all these estimates, one finds that 
	\begin{align*}
		\mathrm{RHS}_{\eqref{lower bound ins}} :& = 
		a(\mathbf{U})|u|_{L^2}^2 
		+
		b(\mathbf{U})|u|_{H^{-\frac{1}{2}}}
		+
		\big{(}u, \mathfrak{Ins}(\mathbf{U})u\big{)}_{L^2}  
		\\ 
		& = 
		 a(\mathbf{U})|u|_{L^2}^2 
		+
		b(\mathbf{U})|u|_{H^{-\frac{1}{2}}}^2
		+
		I_1 + I_2 + I_3
		\\ 
		& \geq 
		\big{(} \inf \limits_{x\in \R} \mathfrak{a} \big{)} |u|_{L^2}^2 
		+
		\mathrm{bo}^{-1} (1+ \ve^2 \mu |\partial_x \zeta|^2_{L^{\infty}})^{-\frac{3}{2}} |\partial_x u|_{L^2}^2
		- 
		a(\mathbf{U})\sqrt{\mu}  |u|^2_{\mathring{H}^{\frac{1}{2}}},
	\end{align*}
	where the last term is controlled using interpolation and Young's inequality:
	\begin{align*}
		\sqrt{\mu} a(\mathbf{U}) |u|^2_{\mathring{H}^{\frac{1}{2}}} 
		& \leq
		\sqrt{\mu} a(\mathbf{U})|u|_{L^2}  |\partial_x u|_{L^2}
		\\ 
		&  \leq 
		 \frac{\mathrm{bo} \mu}{2}(1+ \ve^2 \mu |\partial_x \zeta|^2_{L^{\infty}})^{\frac{3}{2}} a(\mathbf{U})^2|u|_{L^2}^2 + \frac{1}{2\mathrm{bo}}(1+ \ve^2 \mu |\partial_x \zeta|^2_{L^{\infty}})^{-\frac{3}{2}} |\partial_x u|_{L^2}^2.
	\end{align*}
	Consequently, we find that
	\begin{align*}
		\mathrm{RHS}_{\eqref{lower bound ins}} 
		& \geq 
		\Big{(} \inf \limits_{x\in \R} \mathfrak{a} -\frac{\mathrm{bo} \mu}{2}(1+ \ve^2 \mu |\partial_x \zeta|^2_{L^{\infty}})^{\frac{3}{2}} a(\mathbf{U})^2\Big{)} |u|_{L^2}^2 
		+
		\frac{1}{2\mathrm{bo}}(1+ \ve^2 \mu |\partial_x \zeta|^2_{L^{\infty}})^{-\frac{3}{2}} |\partial_x u|_{L^2}^2,
		\\
		& 
		\geq 
		\Big{(} \inf \limits_{x\in \R} \mathfrak{a} - \Upsilon   \mathfrak{c} (\zeta) | [\![ \underline{V}^{\pm}]\!] |_{L^{\infty}}^4 \Big{)} |u|_{L^2}^2 
		+
		(M \mathrm{bo})^{-1} |\partial_x u|_{L^2}^2,
	\end{align*}
	where we used Sobolev embedding and defined the quantities 
	\begin{equation*}
		   \Upsilon \cdot  \mathfrak{c} (\zeta) : =   \frac{\mathrm{bo} }{2} (1-\gamma)^2\gamma^2\ve^4  \mu   \cdot \mathfrak{e}^2 (\zeta)  (1+ \ve^2 \mu |\partial_x \zeta|^2_{L^{\infty}})^{\frac{3}{2}},
	\end{equation*}
	in accordance with Definition \ref{Def stability crit}.
	\color{black}
	\\ 
	
	\noindent
	\underline{Step 3.} To prove the third point, we  use interpolation, Sobolev embedding, \eqref{V pm est}, \eqref{est on dtVpm}, and that $\varepsilon^2 \mu \leq \mathrm{bo}^{-1}$ to obtain 
	\begin{align*}
		\mathrm{LHS}_{\eqref{zeta half norm}} & =
		(1-\gamma)\gamma \ve^2
		\sqrt{\mu}\mathfrak{e}(\zeta) \max\limits_{|\alpha|\leq 1} |\partial^{\alpha}_{x,t}[\![ \underline{V}^{\pm}]\!]|_{L^{\infty}}^2  | u|^2_{\mathring{H}^{\frac{1}{2}}} 
		\\
		& \leq \varepsilon\Big{(} \varepsilon^{-2}
		\Upsilon \mathfrak{c} (\zeta) \max\limits_{|\alpha|\leq 1} |\partial^{\alpha}_{x,t}[\![ \underline{V}^{\pm}]\!]|_{L^{\infty}}^4 |u|_{L^2}^2
		+
		\mathrm{bo}^{-1} (1+  \varepsilon^2\mu |\partial_x \zeta|_{L^{\infty}})^{-\frac{3}{2}} |\partial_x u|_{L^2}^2\Big{)}
		\\
		& \leq\varepsilon  
		\Big{(} C\big{(} \mathcal{E}^{\lceil t_0+2\rceil}(\mathbf{U})\big{)}|u|_{L^2}^2 + \mathrm{bo}^{-1} |\partial_x u|_{L^2}^2\Big{)}.
	\end{align*}
	%
	%
	%
	\end{proof}

	\begin{remark}\label{Rmrk 1 bo numb}\ 
		Here we made the restriction $\varepsilon^2 \mu \leq \mathrm{bo}^{-1}$ to prove the third point. Alternatively, we  could have made stronger version of the stability criterion \eqref{Stability criterion}: 
		\begin{equation}\label{Crit ve mu}
			 \inf \limits_{x\in \R} \mathfrak{a} - \varepsilon^{-2}
			 \Upsilon \mathfrak{c} (\zeta) \max\limits_{|\alpha|\leq 1} |\partial^{\alpha}_{x,t}[\![ \underline{V}^{\pm}]\!]|_{L^{\infty}}^4 >0.
		\end{equation}
		In this case, we could have used it to prove \eqref{zeta half norm} without any restriction on the small parameters. However, in order to use this criterion later in the proof of Theorem \ref{Thm 1} we would need to make the restriction $\varepsilon\mu \leq \mathrm{bo}^{-1}$ (see Remark \ref{Rmrk 3 bo numb}). Also, due to a technical point in the energy estimate we need the restrction $\ve^2 \leq \mathrm{bo}^{-1}$ (see Remark \ref{Rmrk 2 bo numb}).\color{black}
		
	\end{remark}


	\section{A priori estimates}\label{Energy est IWW}
	We are now in the position to derive energy estimates for the internal water waves system \eqref{IWW}. To define a natural energy to the system, we distinguish between the cases $\alpha=0$ and $1\leq |\alpha|\leq N.$ For $\alpha = 0$, we have enough regularity on the data to control the solutions with the energy:
	\begin{equation}\label{E0}
		E^0(\mathbf{U}) = |\Lambda^{t_0+\frac{5}{2}}\zeta|_{H^1_{\mathrm{bo}}}^2 + \frac{1}{\mu} \big{(}\Lambda^{t_0+\frac{5}{2}}\psi, \mathcal{G}_{\mu}[0]\Lambda^{t_0+\frac{5}{2}}\psi\big{)}_{L^2},
	\end{equation}
	where $\mathcal{G}[0]\psi$ is defined by formula \eqref{G[0]}:
	\begin{equation*}
		\mathcal{G}[0]  \psi = \sqrt{\mu} |\mathrm{D}| \frac{\mathrm{tanh}(\sqrt{\mu}|\mathrm{D}|)}{1+\gamma \mathrm{tanh}(\sqrt{\mu}|\mathrm{D}|)} \psi.
	\end{equation*}
	For $1\leq|\alpha|\leq N$, we use Proposition \ref{Prop Quasilinear system} and exploit its quaslinear structure to define a suitable symmetrizer. In particular, we will need to cancel specific terms in the energy estimates, and it will be done by introducing the symmetrizer:
	\begin{equation}\label{Q: symmetrizer}
		Q(\mathbf{U}) =  Q^{(1)}(\mathbf{U})+ Q^{(2)}(\mathbf{U})
		=
		\begin{pmatrix}
			\mathfrak{Ins}[\mathbf{U}] & 0 \\
			0 &  \frac{1}{\mu}\mathcal{G}_{\mu}[\ve \zeta]
		\end{pmatrix}
		+ 
		\begin{pmatrix}
			a(\mathbf{U})+b(\mathbf{U})\Lambda^{-1} & 0 \\
			0 & 0
		\end{pmatrix}.
	\end{equation}
	where $a(U)$ and $b(U)$ are defined by \eqref{def a} and \eqref{def b}. With this symmetrizer at hand, we define energies by
	\begin{equation}\label{Energy alpha}
		E_{\alpha}( \mathbf{U}) = \big{(} \mathbf{U}_{(\alpha)}, Q(\mathbf{U}) \mathbf{U}_{(\alpha)} \big{)}_{L^2}.
	\end{equation}
	Finally, taking the sum over all $\alpha$ will be the main quantity used to control the principal part of \eqref{IWW}:
	\begin{equation}\label{Def Energy}
		E^{k}( \mathbf{U}) =  	E^{0}( \mathbf{U}) +  \sum\limits_{1\leq|\alpha|\leq k} E_{\alpha}( \mathbf{U}). 
	\end{equation}
	
	Before proceeding, we make three comments on the choice of the energy.
	
	\begin{remark}\label{Remark on the energy} \color{white} space \color{black}
		\begin{itemize}
			\item[1.] The term $Q^{(1)}(\mathbf{U})$ in \eqref{Q: symmetrizer}, is chosen specifically to cancel the principal part:
			\begin{equation*}
					\big{(} \mathcal{A}[\mathbf{U}]\mathbf{U}_{(\alpha)}, Q^{(1)}(\mathbf{U}) \mathbf{U}_{(\alpha)}\big{)}_{L^2} =0, 
			\end{equation*}
			which appears naturally in the energy estimates. See estimate of $B_2^1$ in the proof of Proposition \ref{Prop Energy est} below. 
			
			\item[2.] The role of the term $Q^{(2)}(\mathbf{U})$  is to make the energy equivalent to the energy norm. In particular, for \eqref{Def Energy} to be coercive, we need $Q^{(2)}(\mathbf{U})$ and the stability criterion \eqref{Stability criterion} to have a lower bound on $\mathfrak{Ins}[\mathbf{U}]$ in the energy space.
			
			\item[3.] To close the energy estimates, we actually need to modify \eqref{Def Energy}. This is because $\mathfrak{Ins}[\mathbf{U}]$ is a second-order operator and therefore makes a contribution to the sub-principal part of the equation (when $|\alpha| = N$). In particular, we correct the energy to cancel the terms
			\begin{equation*}
				\big{(} \mathcal{C}[\mathbf{U}]\mathbf{U}_{\langle \check\alpha\rangle}, Q^{(1)}(\mathbf{U}) \mathbf{U}_{(\alpha)}\big{)}_{L^2},
			\end{equation*}
			which will appear in the a priori estimates. To do so, we define the quantity 
			\begin{equation*}
				\tilde Q(\mathbf{U})
				=
				\begin{pmatrix}
					\mathrm{bo}^{-1} \mathcal{K}_{(\alpha)}[\ve \sqrt{\mu} \partial_x \zeta] & 0 \\
					0 &  \frac{1}{\mu}\mathcal{G}_{\mu, (\alpha)}
				\end{pmatrix},
			\end{equation*}
			where we define the modified energy for the internal water waves equation by 
			\begin{equation}\label{modified energy}
				E_{{\text{\tiny IWW}}}^N(\mathbf{U}) = E^N(\mathbf{U}) + C_2C E^{N-1}(\mathbf{U}) + \sum\limits_{|\alpha| = N}\big{(}\mathbf{U}_{(\alpha)}, \tilde Q(\mathbf{U})\mathbf{U}_{\langle\check \alpha\rangle}\big{)}_{L^2},
			\end{equation}
			for some constants $C,C_2>0$ to be fixed in the proof. 
		
			
		\end{itemize}

	\end{remark}

	\begin{prop}\label{Prop Energy est} Let $\ve, \mu, \mathrm{bo}^{-1} \in (0,1)$, such that $\varepsilon^2\leq \mathrm{bo}^{-1}$, $t_0 = 1$, $N\geq 5$, and let $\mathbf{U} =  (\zeta, \psi)^T  \in \mathscr{E}_{\mathrm{bo},\gamma}^{N,t_0}$ be a solution to \eqref{IWW} on a time interval $[0,T]$ for some $T >0 $.  Assume that $\mathbf{U}$ satisfies the non-cavitation condition \eqref{non-cavitation} and the stability criterion \eqref{Stability criterion} on $[0,T]$. Then, for the modified energy defined by  \eqref{modified energy}  there is a nondecreasing function of its argument \ $C = C(\mathcal{E}^N(\mathbf{U}), h_{\min}^{-1}, (\mathfrak{d}(\mathbf{U}))^{-1})>0$ \color{black} such that,\color{black}
		\begin{equation}\label{Energy estimate}
			\frac{d}{dt}  E_{{\text{\tiny IWW}}}^N(\mathbf{U}) \leq \ve C E_{{\text{\tiny IWW}}}^N(\mathbf{U}),
		\end{equation}
		for all $0<t<T$. Furthermore, for the energy \eqref{Def Energy} is equivalent to 
		\begin{equation*}
			\mathcal{E}^{N}(\mathbf{U}) = |\partial_x\psi|_{H^{t_0+2}}^2 + \sum \limits_{\alpha\in\N^2, |\alpha|\leq N} |\zeta_{(\alpha)}|_{H^1_{\mathrm{bo}}}^2 + |\psi_{(\alpha)}|^2_{\dot{H}^{\frac{1}{2}}_{\mu}}.
		\end{equation*}
		In particular, there holds,
		\begin{equation}\label{equiv of norms k}
			\frac{1}{C}\mathcal{E}^N (\mathbf{U}) \leq E^N(\mathbf{U}) \leq C   \mathcal{E}^N (\mathbf{U}),
		\end{equation}
		and
		\begin{equation}\label{equiv of modified norms}
			\frac{1}{C}\mathcal{E}^N (\mathbf{U}) \leq  E_{{\text{\tiny IWW}}}^N(\mathbf{U}) \leq C   \mathcal{E}^N (\mathbf{U}),
		\end{equation}
		for all $0<t<T$.
	\end{prop}
 
	\begin{proof}
		We first prove \eqref{equiv of norms k} and \eqref{equiv of modified norms} in two separate steps before turning to the proof of \eqref{Energy estimate}. \\

		\noindent
		\underline{Proof of \eqref{equiv of norms k}.} First consider the case $k=0$, \  where we only deal with space derivatives\color{black}. Let  $( \mathcal{G}_{\mu}[0])^{\frac{1}{2}}$ be the square root of the symbol associated to $\mathcal{G}_{\mu}[0]$ and then use Plancherel's identity to see that
		\begin{equation*}
			 \frac{1}{\mu} \big{(}\Lambda^{\frac{7}{2}}\psi, \mathcal{G}_{\mu}[0]\Lambda^{\frac{7}{2}}\psi\big{)}_{L^2} = \frac{1}{\mu} |( \mathcal{G}_{\mu}[0])^{\frac{1}{2}}\psi|_{H^{\frac{7}{2}}}^2.
		\end{equation*}
		Then use \eqref{equivalence alpha = 0} and the definition \eqref{E0} to obtain that 
		\begin{equation*}
		 |\Lambda^{\frac{7}{2}}\zeta|_{H^1_{\mathrm{bo}}}^2+	\frac{1}{C}|\partial_x \psi |_{H^{3}}^2  \leq E^0(\mathbf{U})\leq C  | \psi |_{\dot{H}_{\mu}^{4}}^2 +  \mathcal{E}^N(\mathbf{U}). 
		\end{equation*}
		The lower bound will be absorbed in $\mathcal{E}^N(\mathbf{U})$ when summing over all $k$. While for the upper bound, we relate $\psi$ with the definition of $\psi_{(\alpha)} = \partial_{x,t}^{\alpha} \psi - \varepsilon \underline{w}\zeta_{(\alpha)}$ through estimate \eqref{est psi}:
		\begin{equation*}
			\ 
			| \psi |_{\dot{H}_{\mu}^{(N-\frac{1}{2}) + \frac{1}{2}}}^2 \leq  M \mathcal{E}^{N}(\mathbf{U}).
		\end{equation*}

		For  the case $k = N$, we use the definition of the individual energies  \eqref{Energy alpha} to find that
		\begin{align*}
			E_{\alpha}(\mathbf{U}) 
			& =
			\big{(}\zeta_{(\alpha)}, \mathfrak{Ins}[U] \zeta_{(\alpha)}\big{)}_{L^2}
			+
			\big{(}  \zeta_{(\alpha)}, (a(\mathbf{U}) + b(\mathbf{U})\Lambda^{-1})\zeta_{(\alpha)}\big{)}_{L^2}
			+
			\frac{1}{\mu} \big{(}\psi_{(\alpha)}, \mathcal{G}_{\mu}[\ve \zeta]\psi_{(\alpha)}\big{)}_{L^2}.
		\end{align*}
		The first two terms are estimated by  \eqref{lower bound ins} and \eqref{upper bound ins}. While the last term is controlled by \eqref{G Coercive} and \eqref{G continuous form}, which implies the result
		\begin{equation*}
			\frac{1}{C}\mathcal{E}^N(\mathbf{U}) \leq E^N(\mathbf{U}) \leq C   \mathcal{E}^N(\mathbf{U}).
		\end{equation*}\\
		
		\noindent
		\underline{Proof of \eqref{equiv of modified norms}.} Next, we prove \eqref{equiv of modified norms}. To estimate the additional term appearing in \eqref{modified energy} when $|\alpha|=N$, we first observe that (suppressing the argument in $\zeta$):
		\begin{align*}
			\frac{1}{\mu} |\big{(}\psi_{(\alpha)}, \mathcal{G}_{{\mu},(\alpha)}\psi_{\langle\check{\alpha}\rangle}\big{)}_{L^2}|
			& \leq 
			\frac{1}{\mu}
			\sum\limits_{j=1}^2 \alpha_j \Big{(}
			|\big{(} \psi_{(\alpha)},\mathrm{d}\mathcal{G}_{\mu}^+(\partial_j\zeta) (\mathcal{J}_{\mu})^{-1} \psi_{(\check \alpha^j)}\big{)}_{L^2}|
			\\ 
			&
			\hspace{0.5cm} 
			+
			|\big{(}  \psi_{(\alpha)},  \mathcal{G}_{\mu}^+(\mathcal{J}_{\mu})^{-1}  (\mathrm{d} \mathcal{J}_{\mu}(\partial_j \zeta)) (\mathcal{J}_{\mu})^{-1}\psi_{(\check \alpha^j)}\big{)}_{L^2}|\Big{)},
		\end{align*}
		where $\check{\alpha}^j=  \alpha- \mathbf{e}_j$ and $\mathcal{G}_{{\mu},(\alpha)}$ is given in Definition \ref{Def op}. For the first term we use  \eqref{Est djG mu ve} with \eqref{Est J}, while the for the second term we also use \eqref{dual est G+ s}, and \eqref{shape derivative Jinv} to find that
		\begin{equation}\label{est G alpha}
			\frac{1}{\mu} |\big{(}\psi_{(\alpha)}, \mathcal{G}_{{\mu},(\alpha)}\psi_{\langle\check{\alpha}\rangle}\big{)}_{L^2}| 
			\leq \ve M |\psi_{(\alpha)}|_{\dot{H}^{\frac{1}{2}}_{\mu}}
			\sum\limits_{j=1}^2 |\psi_{(\check{\alpha}^j)}|_{\dot{H}^{\frac{1}{2}}_{\mu}}.
		\end{equation}
		Then since the surface tension term can be treated directly by integration by parts, we use Young's inequality to find that
		\begin{align*}
			\sum \limits_{|\alpha| = N}|\big{(}\mathbf{U}_{(\alpha)}, \tilde Q(\mathbf{U})\mathbf{U}_{\langle\check \alpha\rangle}\big{)}_{L^2}| 
			&\leq \sum \limits_{|\alpha| = N}
			\mathrm{bo}^{-1} |\big{(}\zeta_{(\alpha)}, \mathcal{K}_{(\alpha)} \zeta_{\langle \check{\alpha}\rangle }\big{)}_{L^2}|
			+
			\frac{1}{\mu} |\big{(}\psi_{(\alpha)}, \mathcal{G}_{{\mu},(\alpha)}\psi_{\langle \check{\alpha}\rangle }\big{)}_{L^2}| 
			\\ 
			&\leq 
			\frac{1}{2C_2} \mathcal{E}^N(\mathbf{U}) + \frac{C_2}{2}\mathcal{E}^{N-1}(\mathbf{U}),
		\end{align*}
		for any $C_2>0$. Then let $C_2 \geq C$ and use \eqref{equiv of norms k} with the definition of the modified energy \eqref{modified energy} to see that
		\begin{align*}
			\frac{1}{2}E^N(\mathbf{U})  
			\leq E_{IWW}^N(\mathbf{U}) \leq C_2CE^N(\mathbf{U}).
		\end{align*}
		%
		%
		%

		Lastly, we will prove energy estimate \eqref{Energy estimate} by considering two cases.\\

		\noindent 
		\underline{An energy estimate in the case $k=0$.} We first use \eqref{IWW} to make the decomposition
		\begin{align*}
			\partial_t \zeta & = \frac{1}{\mu} \mathcal{G}_{\mu}[0]\psi +  \ve \mathcal{N}_1(\mathbf{U})\\
			\partial_t \psi   & = -\big{(}1 - \mathrm{bo}^{-1} \partial_x^2\big{)}\zeta + \ve \mathcal{N}_2(\mathbf{U}),
		\end{align*}
		where
		\begin{equation*}
			\mathcal{N}_1(\mathbf{U})  = \frac{1}{\mu \ve}(\mathcal{G}_{\mu}[\ve \zeta] - \mathcal{G}_{\mu}[0])\psi,
		\end{equation*}
		and 
		\begin{align*}
			\mathcal{N}_2(\mathbf{U}) 
			&  =  \frac{1}{\ve\mathrm{bo}}
			\partial_x \Big{(} \big{(} \frac{1}{\sqrt{1+\ve^2 \mu (\partial_x\zeta)^2}} - 1\big{)}\partial_x \zeta \Big{)}
			-
			\frac{1}{2}  
			\big{(}
			 (\partial_x \psi^{+})^2 
			-
			\gamma    (\partial_x \psi^{-})^2 
			\big{)}
			-
			\mathcal{N}(\mathbf{U}).
		\end{align*}
		The decomposition emphasizes the terms of order $\ve$, while the linear terms are canceled by our choice of the energy:
		\begin{align*}
			\frac{1}{2}\frac{\mathrm{d}}{\mathrm{d}t}E^0(U)
			&  = 
			\big{(}\Lambda^{\frac{7}{2}}\partial_t\zeta, \big{(}1- \mathrm{bo}^{-1} \partial_x^2\big{)}\Lambda^{\frac{7}{2}}\zeta\big{)}_{L^2}+ \frac{1}{\mu} \big{(}\Lambda^{\frac{7}{2}}\partial_t\psi, \mathcal{G}_{\mu}[0]\Lambda^{\frac{7}{2}}\psi\big{)}_{L^2}
			\\ 
			& = 
			\ve \big{(} \Lambda^{\frac{7}{2}}\mathcal{N}_1(\mathbf{U}), \big{(}1 - \mathrm{bo}^{-1} \partial_x^2\big{)}\Lambda^{\frac{7}{2}}\zeta\big{)}_{L^2}
			+
			\frac{\ve}{\mu} \big{(}\Lambda^{\frac{7}{2}} \mathcal{N}_2(\mathbf{U}), \mathcal{G}_{\mu}[0]\Lambda^{\frac{7}{2}}\psi\big{)}_{L^2}.
			\\
			& =
			A_1 + A_2.
		\end{align*}
		For the estimate on $A_1$, we use Plancherel's identity,  Cauchy-Schwarz inequality, and \  \eqref{Symbolic G} with $s=t_0 = 3$  to find that
		\begin{align*}
			|A_1| 
			& \leq 
			\ve 
			|\mathcal{N}_1(\mathbf{U})|_{H^{3}}|\Lambda^{5}\zeta|_{H_{\gamma, \mathrm{bo}}^{1}}
			\\
			& \leq 
			\ve M(5) |\psi|_{\dot{H}_{\mu}^{\frac{9}{2}}}|\Lambda^{5}\zeta|_{H^{1}_{\gamma, \mathrm{bo}}}.
		\end{align*}
		Then to estimate $\psi$, we use that $ |\psi|_{\dot{H}_{\mu}^{\frac{9}{2}}}\leq  |\psi|_{\dot{H}_{\mu}^{(N-1)+\frac{1}{2}}}$ together with \eqref{est psi}. Combining these estimates implies
		\begin{align*}
			|A_1| 
			 \leq 
			\ve C( \mathcal{E}^N(\mathbf{U})) \mathcal{E}^N(\mathbf{U}).
		\end{align*}
		\color{black}
		For the estimate on $A_2$, we integrate by parts and use the algebra property of $H^{\frac{7}{2}}(\R)$ to find that
		\begin{align*}
			|A_2| & \lesssim
			\frac{\ve}{\mu} \mathrm{bo}^{-1}  |\Lambda^5 \zeta|_{H^1} |\mathcal{G}_{\mu}[0]\psi|_{H^{3}} 
			+ 
			\frac{\ve}{\mu}( |\partial_x \psi^+|^2_{{H^{4}}}  +  |\partial_x\psi^-|^2_{{H^{4}}} 
			+ |\mathcal{N}(\mathbf{U})|_{H^4})|\mathcal{G}_{\mu}[0]\psi|_{H^{3}},
		\end{align*}
		where we recall the definition
		\begin{equation*}
			\mathcal{N}[\ve \zeta,\psi] 
			=  
			\frac{1}{2\mu} \frac{\gamma (\mathcal{G}^{-}_{\mu}[\ve \zeta] \psi^{-} + \ve \mu \partial_x \zeta  \partial_x \psi^{-})^2 - (\mathcal{G}^{+}_{\mu}[\ve \zeta] \psi^{+} + \ve \mu \partial_x \zeta  \partial_x \psi^{+})^2 }{(1+   \varepsilon^2 \mu (\partial_x \zeta)^2)}.
		\end{equation*}
		To conclude, we use \eqref{Est 2 G0} to control  $\frac{1}{\mu}\mathcal{G}_{\mu}[0]\psi$ together with \eqref{est psi} in the first term:
		\begin{equation*}
			\frac{\mathrm{bo}^{-1}}{\mu}  |\Lambda^5 \zeta|_{H^1} |\mathcal{G}_{\mu}[0]\psi|_{H^{3}} \lesssim |\Lambda^5 \zeta|_{H^1_{ \mathrm{bo}}} |\psi|_{\dot{H}^{\frac{9}{2}}_{\mu}} \leq \mathcal{E}^{N}(\mathbf{U}).
		\end{equation*}
		For the second term,  we use Corollary \ref{cor definitions} and estimates \eqref{Est V pm} with \eqref{est psi} to get that
		\begin{equation*}
			|\partial_x \psi^+|^2_{{H^{4}}} + |\partial_x \psi^-|_{{H^{4}}}^2 \leq M |\partial_x\psi|_{H^{4}}^2 \leq    M |\psi|_{\dot{H}^{5}_{\mu}}^2\leq  M \mathcal{E}^{N}(\mathbf{U}).
		\end{equation*}		
		Lastly, to deal with the terms in $\mathcal{N}(\mathbf{U})$ we use the algebra property of ${H^{4}}(\R)$, estimate \eqref{W+ est low reg} to obtain
		\begin{align*}
			|\mathcal{N}(\mathbf{U})|_{{H^{4}}}
			 \leq    
			M (\frac{\gamma}{\mu}|\underline{w}^-|^2_{{H^{4}}} + \frac{1}{\mu}|\underline{w}^+|_{{H^{4}}}^2)
			 \leq 
			M |\psi|_{\dot{H}_{\mu}^{5}}^2 \leq
			M \mathcal{E}^N(\mathbf{U}),
		\end{align*}
		Gathering all these estimates, we find that
		\begin{equation}\label{est 0}
			\frac{\mathrm{d}}{\mathrm{d}t}E^0(U) \leq \ve C( \mathcal{E}^N(\mathbf{U}))  \mathcal{E}^N(\mathbf{U}).
		\end{equation}\\

		\noindent 
		\underline{An energy estimate in the case $1\leq k \leq N$.} We first derive an estimate on each $E_{\alpha}(\mathbf{U})$ defined by \eqref{Energy alpha} for $1\leq |\alpha|\leq N$. To do so, we use that $Q(\mathbf{U})$ is symmetric up to a lower order term to find that
		\begin{align*}
			\frac{1}{2}\frac{\mathrm{d}}{\mathrm{d} t} E_{\alpha}(\mathbf{U}) 
			& =
			\frac{1}{2}
			\big{(} \mathbf{U}_{(\alpha)}, \big{(}\partial_t Q(\mathbf{U})\big{)} \mathbf{U}_{(\alpha)}\big{)}_{L^2} 
			+
			\big{(} \partial_t \mathbf{U}_{(\alpha)}, Q(\mathbf{U}) \mathbf{U}_{(\alpha)}\big{)}_{L^2} 
			\\ 
			& \hspace{0.5cm}
			+
			\frac{1}{2}
			\big{(}\mathbf{U}_{(\alpha)}, \big{(}Q(\mathbf{U}) - Q(\mathbf{U})^{\ast}\big{)}  \partial_t \mathbf{U}_{(\alpha)}\big{)}_{L^2} 
			\\ 
			& =
			B_1 + B_2 + B_3,
		\end{align*}
		where $B_3$ is the contribution of the lower order term and reads
		\begin{align*}
			B_3 = -\frac{1}{2}(1-\gamma)\gamma \ve^2 \mu \Big{((}\zeta_{(\alpha)}[\![ \underline{V}^{\pm}]\!], \big{(}\mathfrak{E}_{\mu}[\ve \zeta] - \mathfrak{E}_{\mu}[\ve \zeta]^{\ast}\big{)} \big{(} \partial_t\zeta_{(\alpha)} [\![ \underline{V}^{\pm}]\!]\big{)}\Big{)}_{L^2}.
		\end{align*} 

		\underline{Control of $B_1$.}  By definition of $Q(\mathbf{U})$ given by \eqref{Q: symmetrizer}, we find that
		\begin{align*}
			B_1
			& =
			\frac{1}{2}
			\big{(} [\partial_t, \mathfrak{Ins}(\mathbf{U})]\zeta_{(\alpha)}, \zeta_{(\alpha)}\big{)}_{L^2} 
			+
			\frac{1}{2\mu }
			\big{(} [\partial_t, \mathcal{G}]\psi_{(\alpha)}, \psi_{(\alpha)}\big{)}_{L^2}
			\\ 
			& 
			\hspace{0.5cm}
			+
			\frac{1}{2}
			\big{(} (\partial_t a)\zeta_{(\alpha)}, \zeta_{(\alpha)}\big{)}_{L^2}
			+
			\frac{1}{2}
			\big{(} (\partial_t b)\zeta_{(\alpha)}, \Lambda^{-1}\zeta_{(\alpha)}\big{)}_{L^2}
			& 
			\\ 
			& = 
			B_1^1 + B_1^2 + B_1^3 + B_1^4.
		\end{align*}
		For the estimate on $B_1^1$, we have by definition and integration by parts that
		\begin{align*}
			B_1^1 
			& =
			\frac{1}{2}
			\big{(}\zeta_{(\alpha)},\big{(} \partial_t \mathfrak{a}\big{)} \zeta_{(\alpha)}\big{)}_{L^2}  
			+
			\frac{1}{2\mathrm{bo}}\big{(} \partial_x \zeta_{(\alpha)}, \big{[}\partial_t, \mathcal{K}[\ve \sqrt{\mu} \partial_x \zeta]\big{]}\partial_x  \zeta_{(\alpha)}\big{)}_{L^2} 
			\\ 
			& 
			\hspace{0.5cm}
			-	
			(1-\gamma)\gamma \ve^2 \mu	\big{(} \big{(} \partial_t [\![ \underline{V}^{\pm}]\!] \big{)}\zeta_{(\alpha)},     \mathfrak{E}_{\mu}[\ve \zeta]	\big{(} 
			\zeta_{(\alpha)}[\![ \underline{V}^{\pm}]\!] \big{)} \big{)}_{L^2} 
			\\ 
			& 
			\hspace{0.5cm}
			-
			\frac{1}{2}
			(1-\gamma)\gamma \ve^2 \mu	\big{(}  [\![ \underline{V}^{\pm}]\!] \zeta_{(\alpha)}, \big{[}\partial_t, \mathfrak{E}_{\mu}[\ve \zeta] \big{]}	\big{(} 
			\zeta_{(\alpha)}[\![ \underline{V}^{\pm}]\!] \big{)} \big{)}_{L^2}
			\\
			& = 
			B_1^{1,1} + B_1^{1,2} + B_1^{1,3} + B_1^{1,4}.
		\end{align*}
		\ For the estimate on $B_1^{1,1}$, we have an estimate on $\partial_t \mathfrak{a}$ in $L^{\infty}(\R)$ given by \eqref{est on dta}. In particular,  the Hölder  inequality implies that
		\begin{align*}
			|B_1^{1,1}|
			& \leq \frac{1}{2}
			|\partial_t \mathfrak{a}|_{H^{s}} |\zeta_{(\alpha)}|_{L^2}^2
			\\ 
			& \leq  \varepsilon C\big{(}\mathcal{E}^{N}(\mathbf{U})\big{)} \mathcal{E}^{N}(\mathbf{U}),
		\end{align*}
		where we used that $t_0+3 \leq N$.
		\color{black}

		For the estimate on $B_1^{1,2}$, it only depends on $\partial_t\mathcal{K}[\ve \sqrt{\mu} \partial_x \zeta]$ which is in $L^{\infty}(\R)$. Indeed, it follows from an estimate on $\partial_x \partial_t \zeta$, where we use the Sobolev embedding:
		\begin{align*}
			|\partial_x \partial_t \zeta|_{L^{\infty}}^2 \leq |\partial_t \zeta|_{H^2}^2 \leq \sum \limits_{\alpha\in \N^2, |\alpha|\leq3} |\zeta_{(\alpha)}|_{L^2}^2 \leq \mathcal{E}^N(\mathbf{U}).
		\end{align*} 
		As a result, we obtain by Hölder's inequality and the definition of the energy that
		\begin{align*}
			|B_1^{1,2}| \leq  \ve C\big{(}\mathcal{E}^{N}(\mathbf{U})\big{)} \mathcal{E}^N(\mathbf{U}).
		\end{align*}
		For the estimate on $B_1^{1,3}$, we use estimates \eqref{est on E fg}, \eqref{eq 5.7 in lannes}, \eqref{V pm est}, and \eqref{est on dtVpm} to find that
		\begin{align*}
			B_1^{1,3} 
			& \lesssim  (1-\gamma)\gamma \mathfrak{e}(\zeta)\ve^2 
			|(1+\sqrt{\mu}|\mathrm{D}|)^{\frac{1}{2}} (\zeta_{(\alpha)}[\![ \underline{V}^{\pm}]\!])|_{L^2}	|(1+\sqrt{\mu}|\mathrm{D}|)^{\frac{1}{2}} (\zeta_{(\alpha)}\partial_t[\![ \underline{V}^{\pm}]\!])|_{L^2}
			\\ 
			& \leq  \ve   C\big{(} \mathcal{E}^N(\mathbf{U}) \big{)} \mathcal{E}^N(\mathbf{U})
			+
			(1-\gamma)\gamma  \mathfrak{e}(\zeta) \ve^2
			\sqrt{\mu} \max\limits_{|\alpha|\leq 1} |\partial_{x,t}^{\alpha}[\![ \underline{V}^{\pm}]\!]|_{L^{\infty}}^2  | \zeta_{(\alpha)}|^2 
			\\ 
			& 
			\leq \ve C \mathcal{E}^N(\mathbf{U}).
		\end{align*}
		Likewise, the estimate on $B_1^{1,4}$ we simply use \eqref{Est E 2} instead of \eqref{Est E 1} to obtain the same bound and we deduce that
		\begin{align*}
			|	B_1^1 |\leq \ve C \mathcal{E}^N(\mathbf{U}).
		\end{align*}
		Clearly, the estimates on $B_1^3$ and $B_1^{4}$ is estimated similarly to $B_1^{1,1}$, while we estimate $B_1^2$ we have by direct computations that
		\begin{align*}
			|B_1^{2}| \leq  \frac{1}{2\mu }
			\big{(} [\partial_t, \mathcal{G}_{\mu}^+] (\mathcal{J}_{\mu})^{-1}\psi_{(\alpha)}, \psi_{(\alpha)}\big{)}_{L^2}
			+
			\frac{1}{2\mu }
			\big{(}  \mathcal{G}_{\mu}^+(\mathcal{J}_{\mu})^{-1}  (\mathrm{d} \mathcal{J}_{\mu}(\partial_t \zeta)) (\mathcal{J}_{\mu})^{-1} \psi_{(\alpha)}, \psi_{(\alpha)}\big{)}_{L^2}.
		\end{align*}
		For the first term we use Proposition $3.6$ of \cite{Alvarez-SamaniegoLannes08a} with \eqref{Est J}, while the for the second term we also use \eqref{dual est G+ s},  \eqref{Est J}, and then \eqref{shape derivative J} to find that
		\begin{align*}
			|	B_1^{2} |\leq \ve C \mathcal{E}^N(\mathbf{U}).
		\end{align*}
		Collecting each estimate so far gives
		\begin{align*}
			|B_1| \leq \ve C \mathcal{E}^N(\mathbf{U}).
		\end{align*}

		\underline{Control of $B_2$.} We use system \eqref{Quasi 2} to find that
		\begin{align*}
			B_2 
			& = 
			-\big{(} \mathcal{A}[\mathbf{U}]\mathbf{U}_{(\alpha)}, Q(\mathbf{U}) \mathbf{U}_{(\alpha)}\big{)}_{L^2} 
			-
			\big{(} \mathcal{B}[\mathbf{U}]\mathbf{U}_{(\alpha)}, Q(\mathbf{U}) \mathbf{U}_{(\alpha)}\big{)}_{L^2} 
			\\ 
			& \hspace{0.5cm}
			-\big{(} \mathcal{C}[\mathbf{U}]\mathbf{U}_{\langle \check{\alpha}\rangle}, Q(\mathbf{U}) \mathbf{U}_{(\alpha)}\big{)}_{L^2} 
			+
			\ve 
			\big{(} (R_{\alpha},S_{\alpha})^T, Q(\mathbf{U}) \mathbf{U}_{(\alpha)}\big{)}_{L^2}  
			\\ 
			& =
			B_2^1 + B_2^2 + B_2^3 + B_2^4.
		\end{align*}
		
		\noindent 
		\textit{Control of $B_2^1$.} From  Definition \ref{Def op} and \eqref{Q: symmetrizer} we have that 
		\begin{equation*}
			\big{(} \mathcal{A}[\mathbf{U}]\mathbf{U}_{(\alpha)}, Q^{(1)}(\mathbf{U}) \mathbf{U}_{(\alpha)}\big{)}_{L^2}  =0.
		\end{equation*}
		We may therefore decompose $B_2^1$ into two pieces:
		\begin{align*}
			B_2^1 
			& = 
			-\big{(} \mathcal{A}[\mathbf{U}]\mathbf{U}_{(\alpha)}, Q_2(\mathbf{U}) \mathbf{U}_{(\alpha)}\big{)}_{L^2} 
			\\ 
			& = 
			a(\mathbf{U})\big{(} \frac{1}{\mu}\mathcal{G}_{\mu}\psi_{(\alpha)}, \zeta_{(\alpha)}\big{)}_{L^2}  
			+
			b(\mathbf{U})
			\big{(} \frac{1}{\mu} \mathcal{G}_{\mu}\psi_{(\alpha)}, \Lambda^{-1}\zeta_{(\alpha)}\big{)}_{L^2}.
		\end{align*}
		Both terms are estimated directly by \eqref{G continuous form}, \eqref{V pm est} combined with Sobolev embedding to estimate $a(\mathbf{U})$, to find that
		\begin{align*}
			|B_2^1|
			& \leq 
			M  |\psi_{(\alpha)}|_{\dot{H}^{\frac{1}{2}}_{\mu}}\big{(}a(\mathbf{U}) |\zeta_{(\alpha)}|_{\dot{H}^{\frac{1}{2}}_{\mu}} + b(\mathbf{U}) |\Lambda^{-1}\zeta_{(\alpha)}|_{\dot{H}^{\frac{1}{2}}_{\mu}} \big{)}
			\\
			& \leq 
			C  |\psi_{(\alpha)}|_{\dot{H}^{\frac{1}{2}}_{\mu}}\big{(}  \ve^2 |\partial_x \zeta_{(\alpha)}|_{L^2} + b(\mathbf{U}) |\zeta_{(\alpha)}|_{L^2} \big{)}
			\\ 
			& 
			\leq  \ve C \mathcal{E}^N(\mathbf{U}),
		\end{align*}
		where we used that $\ve \leq \mathrm{bo}^{-\frac{1}{2}}$. 
		\\

	\noindent   
	\textit{Control of $B_2^2$.} We first decompose $B_2^2$ into four parts
	\begin{align*}
		B_2^2 
		& =
		\ve \big{(}\mathcal{I}[\mathbf{U}]\zeta_{(\alpha)}, \mathfrak{Ins}[\mathbf{U}]\zeta_{(\alpha)} \big{)}_{L^2}
		-
		\frac{\ve}{\mu}\big{(}\mathcal{I}^{\ast}[\mathbf{U}]\psi_{(\alpha)}, \mathcal{G}_{\mu}\psi_{(\alpha)} \big{)}_{L^2}
		\\ 
		& 
		\hspace{0.5cm}
		+
		\ve a(\mathbf{U}) \big{(}\mathcal{I}[\mathbf{U}]\zeta_{(\alpha)}, \zeta_{(\alpha)} \big{)}_{L^2}
		+
		\ve b(\mathbf{U})\big{(}\mathcal{I}[\mathbf{U}]\zeta_{(\alpha)}, \Lambda^{-1}\zeta_{(\alpha)} \big{)}_{L^2}
		\\ 
		& = 
		B_2^{2,1} + B_2^{2,2} + B_2^{2,3} + B_2^{2,4}.
	\end{align*}
	We first observe that the two last terms are easily treated with an estimate of the type \eqref{Est I a} to get that
	\begin{align*}
		|B_2^{2,3}| + |B_2^{2,4}|
		& \leq \ve C \mathcal{E}^N(\mathbf{U}).
	\end{align*}
	For the remaining two terms, we will need to work some more.  \\ 

	\noindent 
	\textit{Control of $B_2^{2,1}$.} Clearly, by \eqref{zeta half norm} it is enough to prove that
	\begin{align}\label{What we want...}
		|B_2^{2,1}| \leq \mathcal{E}^N(\mathbf{U})\big{(} 1+  \ve^2 \sqrt{\mu}| [\![ \underline{V}^{\pm}]\!]|_{W^{1,\infty}}^2|\zeta|^2_{<N+\frac{1}{2}>}\big{)}.
	\end{align}
	Since the quantities involved in this depend on $(\mathcal{G}^{-}_{\mu})^{-1}$, we need to adapt the proof of  \cite{LannesTwoFluid13}.  However, the proof relies on the symbolic expression of $(\mathcal{J}_{\mu})^{-1}(\mathcal{G}^-_{\mu})^{-1}\partial_x$, and this estimate is provided in Corollary \ref{Cor last one} which has the same outcome as the one in \cite{LannesTwoFluid13}. 
	 In particular, we refer the reader to the proof of estimate $(5.15)$ in this paper. \\

	\noindent 
	\textit{Control of $B_2^{2,2}$.} This estimate is similar to estimate $(5.16)$ in \cite{LannesTwoFluid13}, but we give the details here to account for the difference in $\mathcal{G}^-_{\mu}$. In particular, we have by definition \eqref{I star}, that we must provide an estimate on the terms
	\begin{align*}
		B_2^{2,2}
		& = 
		\frac{\ve}{\mu}\big{(}	\underline{V}^+\partial_x \psi_{(\alpha)}, \mathcal{G}_{\mu}\psi_{(\alpha)} \big{)}_{L^2}
		+
		\gamma \frac{\ve}{\mu}\big{(}[\![ \underline{V}^{\pm}]\!]  \partial_x\big{(}(\mathcal{G}^-_{\mu} )^{-1}\mathcal{G}_{\mu}  \psi_{(\alpha)}\big{)}, \mathcal{G}_{\mu}\psi_{(\alpha)} \big{)}_{L^2}
		\\ 
		& = B_{2}^{2,2,1} +  B_{2}^{2,2,2}.
	\end{align*} 
	We let $g = (\mathcal{J}_{\mu})^{-1} \psi_{(\alpha)}$ and use the definition $\mathcal{G}_{\mu} = \mathcal{G}_{\mu}^+ (\mathcal{J}_{\mu})^{-1}$ to find that
	\begin{align*}
		B_{2}^{2,2,1} 
		& =
		\frac{\ve}{\mu}\big{(}	\underline{V}^+\partial_x \mathcal{J}_{\mu} g, \mathcal{G}^+_{\mu}g \big{)}_{L^2}
		\\ 
		& = 
		\frac{\ve}{\mu}\big{(}	\underline{V}^+ \partial_x g, \mathcal{G}^+_{\mu}g \big{)}_{L^2}
		-
		\gamma \frac{\ve}{\mu}\big{(}	\underline{V}^+ \partial_x (\mathcal{G}^-_{\mu})^{-1} \mathcal{G}^+_{\mu} g, \mathcal{G}^+_{\mu}g \big{)}_{L^2}
		\\ 
		& =
		B_{2}^{2,2,1,1} + 	B_{2}^{2,2,1,2}.
	\end{align*}
	For the estimate on $B_{2}^{2,2,1,1}$, we use \eqref{Prop 3.30}, combined with Sobolev embedding, \eqref{Est J}, and \eqref{V pm est} to obtain,
	\begin{align*}
		|B_{2}^{2,2,1,1}| 
		& \leq
		\ve M |\underline{V}^+ |_{W^{1,\infty}}|g|_{\dot{H}^{\frac{1}{2}}_{\mu}}^2
		\\ 
		& \leq
		\ve M |\underline{V}^+ |_{H^{2}}|\psi_{(\alpha)}|_{\dot{H}^{\frac{1}{2}}_{\mu}}^2
		\\
		& 
		\leq 
		\ve C \mathcal{E}^N(\mathbf{U}).
	\end{align*}
	For the estimate on $B_{2}^{2,2,1,2}$, we let $\tilde{g} = (\mathcal{G}^-_{\mu})^{-1} \mathcal{G}^+_{\mu}\psi_{(\alpha)} $ and use \eqref{ext prop 3.30}. It is then straightforward to obtain the desired bound using also \eqref{Inverse est 2}: 
	\begin{align*}
		|B_{2}^{2,2,1,1}|
		& \leq 
		\gamma \frac{\ve}{\mu}|\big{(}	\underline{V}^+ \partial_x  \tilde g, \mathcal{G}^-_{\mu}\tilde g \big{)}_{L^2}|
		\\ 
		& \leq
		\ve \mu^{-\frac{1}{2}} M |\underline{V}^+ |_{W^{1,\infty}}|\tilde g|_{\mathring{H}^{\frac{1}{2}}}^2
		\\
		& 
		\leq 
		\ve C \mathcal{E}^N(\mathbf{U}).
	\end{align*}
	Lastly, for the estimate on $B_2^{2,2,2}$ we let $g^{\sharp} = (\mathcal{G}^-_{\mu} )^{-1}\mathcal{G}_{\mu}  \psi_{(\alpha)}$, and again use \eqref{ext prop 3.30} with \eqref{Inverse 2} and \eqref{Est J}:
	\begin{align*}
		|B_2^{2,2,2}|
		& = 
		\gamma \frac{\ve}{\mu}|\big{(}[\![ \underline{V}^{\pm}]\!]  \partial_x g^{\sharp}, \mathcal{G}^-_{\mu} g^{\sharp}\big{)}_{L^2}|
		\\ 
		&
		\leq
		\ve \mu^{-\frac{1}{2}} M |(\mathcal{G}^-_{\mu} )^{-1}\mathcal{G}^+_{\mu}\mathcal{J}_{\mu}^{-1}\psi_{(\alpha)}|_{\mathring{H}^{\frac{1}{2}}}^2
		\\
		& 
		\leq 
		\ve C \mathcal{E}^N(\mathbf{U}).
	\end{align*}
	From these estimates, we conclude the bound on $B_2^2$ where we find that
	\begin{align*}
		|B_2^2| 	
		\leq 
		\ve C \mathcal{E}^N(\mathbf{U}).
	\end{align*}

	\noindent 
	\textit{Control of $B_2^3$.} To get an estimate of order $\ve$ we need to cancel the following term (when $k=N$):
	\begin{align*}
		B_2^{3,1} 
		: =
		\big{(} \mathcal{C}[\mathbf{U}]\mathbf{U}_{\langle \check\alpha \rangle }, Q^{(1)}(\mathbf{U}) \mathbf{U}_{(\alpha)}\big{)}_{L^2}  
		= - 	\big{(} \mathcal{A}[\mathbf{U}]\mathbf{U}_{(\alpha)}, \tilde Q(\mathbf{U}) \mathbf{U}_{\langle\check\alpha\rangle}\big{)}_{L^2} .
	\end{align*}
	We will cancel $B_2^{3,1}$ later by introducing $\tilde Q$ into the energy (see Remark \ref{Remark on the energy}). Therefore, we may decompose $(B_2^3-B_2^{3,1})$ into two parts:
	\begin{align*}
		(B_2^3-B_2^{3,1})
		& =
		\frac{1}{\mu}a(\mathbf{U}) \big{(}\mathcal{G}_{\mu,(\alpha)}\psi_{\langle \check \alpha \rangle}, \zeta_{(\alpha)} \big{)}_{L^2}
		+
		\frac{1}{\mu}b(\mathbf{U}) \big{(}\mathcal{G}_{\mu,(\alpha)}\psi_{\langle\check \alpha\rangle}, \Lambda^{-1}\zeta_{(\alpha)} \big{)}_{L^2}
		\\ 
		& = 
		B_2^{3,2} + B_2^{3,3}.
	\end{align*}
	However, these terms are easily dealt with using estimates on $\mathcal{G}^+_{\mu}$ and $(\mathcal{J}_{\mu})^{-1}$  (similar to \eqref{est G alpha}) and that $\ve \leq \mathrm{bo}^{-\frac{1}{2}}$ to obtain
	\begin{align*}
		|B_2^{3,2}| + |B_2^{3,3}| \leq   \ve C \mathcal{E}^N(\mathbf{U}).
	\end{align*}

	\noindent 
	\textit{Control of $B_2^4$.} For $B_2^4$ we have to estimate the following terms,
	\begin{align*}
		B_2^4
		& =
		\ve \big{(}R_{\alpha}, \mathfrak{Ins}[\mathbf{U}]\zeta_{(\alpha)} \big{)}_{L^2}
		+
		\frac{\ve }{\mu}\big{(}S_{\alpha}, \mathcal{G}_{\mu}\psi_{(\alpha)} \big{)}_{L^2}
		\\
		& 
		\hspace{0.5cm}
		+
		\ve a(\mathbf{U}) \big{(}R_{\alpha}, \zeta_{(\alpha)} \big{)}_{L^2}
		+
		\ve b(\mathbf{U}) \big{(}R_{\alpha}, \Lambda^{-1}\zeta_{(\alpha)} \big{)}_{L^2}.
	\end{align*}
	For the first term, we use \eqref{upper bound ins} and \eqref{Est R and S} to obtain,
	\begin{align*}
		\ve |\big{(}R_{\alpha}, \mathfrak{Ins}[\mathbf{U}]\zeta_{(\alpha)} \big{)}_{L^2}| 
		\leq  
		\ve C \mathcal{E}^1(\mathbf{U})|R_{\alpha}|_{H^1_{\mathrm{bo}}}
		\leq 
		\ve C \mathcal{E}^N(\mathbf{U}).
	\end{align*}
	For the second term is treated by \eqref{G continuous form}, \eqref{Est R and S}, and \eqref{zeta half norm} to get
	\begin{align*}
		\frac{\ve }{\mu}|\big{(}S_{\alpha}, \mathcal{G}_{\mu}\psi_{(\alpha)} \big{)}_{L^2}| 
		\leq \ve M  |S_{\alpha}|_{\dot{H}^{\frac{1}{2}}} |\psi_{(\alpha)}|_{\dot{H}^{\frac{1}{2}}} 
		\leq \ve C \mathcal{E}^N(\mathbf{U}).
	\end{align*}
	The remaining two estimates satisfy the same upper bound, and so gathering all these estimates, we find that
	\begin{align}\label{Energy step 1}
		\frac{1}{2}\frac{\mathrm{d}}{\mathrm{d} t} E_{\alpha}(\mathbf{U}) - B_1^{3,1} \leq \ve  C \mathcal{E}^N(\mathbf{U}).
	\end{align}
	As pointed out above, we need to cancel  $B_2^{3,1}$ and is done by modifying the energy by
	\begin{align*}
		\frac{1}{2}\frac{\mathrm{d}}{\mathrm{d} t}  \big{(}\mathbf{U}_{(\alpha)}, \tilde Q(\mathbf{U})\mathbf{U}_{\langle \check{\alpha}\rangle }\big{)}_{L^2} 
		& =
		\frac{1}{2}
		\big{(} \mathbf{U}_{(\alpha)}, \big{(}\partial_t \tilde Q(\mathbf{U})\big{)} \mathbf{U}_{\langle \check{\alpha}\rangle }\big{)}_{L^2} 
		-
		\big{(} \mathcal{A}[\mathbf{U}]\mathbf{U}_{(\alpha)}, \tilde Q(\mathbf{U}) \mathbf{U}_{\langle \check{\alpha}\rangle }\big{)}_{L^2} 
		\\ 
		& 
		\hspace{0.5cm}
		-
		\big{(} \mathcal{B}[\mathbf{U}]\mathbf{U}_{(\alpha)}, \tilde Q(\mathbf{U}) \mathbf{U}_{\langle \check{\alpha}\rangle }\big{)}_{L^2} 
		-
		\big{(} \mathcal{C}[\mathbf{U}]\mathbf{U}_{\langle\check{\alpha}\rangle}, \tilde Q(\mathbf{U}) \mathbf{U}_{\langle \check{\alpha}\rangle }\big{)}_{L^2} 
		\\ 
		& \hspace{0.5cm}
		+
		\ve 
		\big{(} (R_{\alpha},S_{\alpha})^T, \tilde Q(\mathbf{U}) \mathbf{U}_{\langle \check{\alpha}\rangle }\big{)}_{L^2}  
		\\ 
		& =
		\tilde B_1 +\tilde  B_2 + \tilde B_3 + \tilde B_4 + \tilde B_5. 
	\end{align*}
	Where we observe that $\tilde B_2 = - B_2^{3,1}$ and gives the desired cancellation (when $k=N$), while the remaining terms are treated as above and satisfy
	\begin{equation*}
		|\tilde B_1| + |\tilde B_3| + |\tilde B_4| + |\tilde B_5| \leq \ve C \mathcal{E}^N(\mathbf{U}).
	\end{equation*}
	We therefore find that 
	\begin{align}\label{Working progress}
		\frac{1}{2}\frac{\mathrm{d}}{\mathrm{d} t} \Big{(} E_{\alpha}(\mathbf{U}) + \big{(}\mathbf{U}_{(\alpha)}, \tilde Q(\mathbf{U})\mathbf{U}_{\langle \check{\alpha}\rangle }\big{)}_{L^2} \Big{)} \leq  \ve C \mathcal{E}^N(\mathbf{U}) + |B_3|,
	\end{align} 
	where we recall that $\tilde Q= 0$ when $|\alpha| < N$. \\

	\underline{Control of $B_3$.} For the estimate on $B_3$, we use \eqref{est E East}, \eqref{Classical prod est}, \eqref{V pm est}, $\ve^2 \leq \mathrm{bo}^{-1}$, to find that
	\begin{align*}
		|B_3| 
		& \leq 
		\ve^3 \mu^{\frac{1}{4}}M(t_0+3)|\zeta_{(\alpha)}[\![ \underline{V}^{\pm}]\!]|_{H^{\frac{1}{2}}} |\partial_t\zeta_{(\alpha)} [\![ \underline{V}^{\pm}]\!]|_{H^{-\frac{1}{2}}}
		\\ 
		& 
		\leq 
		\ve^2 \mu^{\frac{1}{4}} C(\mathcal{E}^N(\mathbf{U})) |\zeta_{(\alpha)}|_{H^{1}_{\mathrm{bo}}} |\partial_t \zeta_{(\alpha)}|_{H^{-\frac{1}{2}}}.
	\end{align*}
	Now for the control of $\partial_t \zeta_{(\alpha)}$, we use equation \eqref{Quasi 2} to see that
	\begin{align*}
		|\partial_t \zeta_{(\alpha)}|_{H^{-\frac{1}{2}}} 
		& \leq 
		\frac{1}{\mu} |\mathcal{G}_{\mu}\psi_{(\alpha)}|_{H^{-\frac{1}{2}}} +\frac{1}{\mu} |\mathcal{G}_{\mu, (\alpha)}\psi_{(\alpha)}|_{H^{-\frac{1}{2}}}+ \ve | \mathcal{I}[\mathbf{U}] \zeta_{(\alpha)}|_{H^{-\frac{1}{2}}}  + \ve |R_{\alpha}|_{H^{-\frac{1}{2}}},
	\end{align*}
	where $R_{\alpha}$ satisfies \eqref{Est R and S}. Then to conclude this estimate, we simply employ \eqref{1 Est G}, \eqref{est I inst op}, and  \eqref{Est R and S} to deduce that
	\begin{equation*}
		|B_3|  \leq \ve C \mathcal{E}^N(\mathbf{U}). 
	\end{equation*}

	\noindent
	\underline{Proof of estimate \eqref{Energy estimate}.} To conclude, we use \eqref{est 0} combined with \eqref{Working progress} where we sum over $0\leq |\alpha|\leq N$ and then apply estimate \eqref{equiv of modified norms} to find that
	\begin{align*} 
			\frac{1}{2}\frac{\mathrm{d}}{\mathrm{d} t}   E_{{\text{\tiny IWW}}}^N(\mathbf{U}) 
			\leq
			\ve C E_{{\text{\tiny IWW}}}^N(\mathbf{U}). 
	\end{align*}
	
	\color{black}

	\end{proof}

	\begin{remark}\label{Rmrk 2 bo numb}
		The  restiction  $\varepsilon^2 \leq \mathrm{bo}^{-1}$ stems from the estimate on $B_2^1$ and $B_3$.
	\end{remark}

	\section{Proof of Theorem \ref{Thm 1}}\label{Pf main thm}
	
	The strategy for the proof of Theorem \ref{Thm 1} is classical, and we refer the reader to \cite{WWP} Chapter $4$ for a detailed proof in the case of the water waves equations. See also Chapter $9$ in the same book in the case of the water waves equations with surface tension. We will now give the main steps involved.
	
	\begin{proof}
		The first step is to regularize the internal water waves equations \eqref{IWW}, and use the Fixed Point Theorem to deduce the existence of a solution. In fact, the solution is smooth. However, the existence time depends on the regularization parameter and shrinks to zero when taking the limit. To extend the existence time,  one applies Proposition \ref{Prop Energy est} and then a compactness argument to deduce a limit. Note however, this requires us to verify the stability criteria and the non-cavitatation condition  on a long-time scale where we merely assume the non-cavitation condition \eqref{non-cavitation}, and the stability criterion \eqref{Stability criterion} for the initial data\footnote{This is an abuse of notation, note that terms involving time derivatives are not initially prescribed. But they can be relplaced with spatial derivatives using shape derivative formulas and the equation for $\partial_t \zeta$ and $\partial_t \psi$. }. For clarity, we recall \eqref{Stability criterion} (in the case of the non-cavitation condition see for instance Lemma 5.3 in \cite{Paulsen22}):

		\begin{equation*}
			0<\mathfrak{d}(\mathbf{U}_0)  = \inf \mathfrak{a}|_{t=0} - \Upsilon \mathfrak{c}(\zeta|_{t=0}) |[\![ \underline{V}^{\pm}|_{t=0}]\!]|_{L^{\infty}}^4.
		\end{equation*}

			To verify the stability criterion on a long time scale follows classically from a boot-strap argument and
			\begin{equation*}
				\sup\limits_{t \in [0,\frac{T}{\varepsilon}]}\mathcal{E}^N(\mathbf{U}) \leq C \mathcal{E}^N(\mathbf{U}_0), 
			\end{equation*}
			if we can verify the inequality
			\begin{align}\label{FTC}
				\mathfrak{d}(\mathbf{U})(t) \geq  \frac{1}{2}\mathfrak{d}(\mathbf{U}_0),
			\end{align}
			for any $t \in [0,\frac{T}{\varepsilon}]$ where $T = C^{-1}>0$ with $C = C(\mathcal{E}^N(\mathbf{U}_0), {h^{-1}_{\min}}, \mathfrak{d}(\mathbf{U}_0)^{-1})$. To prove \eqref{FTC}, we use \eqref{est on dta} to estimate $\partial_t \mathfrak{a}$, \eqref{Est E 2} to estimate $\mathfrak{c}(\partial_t \zeta)$, and for the estimates on $\underline{V}^{\pm}$ we use the Sobolev embedding and \eqref{est on dtVpm} to find that
			\begin{equation*}
			|\partial_t	\mathfrak{d}(\mathbf{U})(t) |_{L^{\infty}} \leq (\varepsilon+\Upsilon) C\big{(} \mathcal{E}^N(\mathbf{U})(t)\big{)},
			\end{equation*}
			for all $t \in [0,\frac{T}{\varepsilon}]$. From this estimate we may employ the Fundamental Theorem of Calculus to deduce that
			\begin{align*}
				\mathfrak{d}(\mathbf{U})(t) 
				& =
				\mathfrak{d}(\mathbf{U}_0) - \int_{0}^{t} (\partial_t\mathfrak{d}(\mathbf{U}))(s)\: \mathrm{d}s\\
				& \geq 
				\mathfrak{d}(\mathbf{U}_0)
				- \frac{T}{\varepsilon}  (\varepsilon+\Upsilon) C,
			\end{align*}
			which implies \eqref{FTC} by taking $T$ small enough, depending only on $C$, and $\Upsilon \sim \varepsilon^4 \mu \mathrm{bo} \lesssim \varepsilon$ which is ensured by the assumption that $\varepsilon^2   \leq \mathrm{bo}^{-1}$.\color{black} 
		
		 	For uniqueness, one needs to have an estimate of the difference between two solutions, which can be proved with estimates similar to the ones used for the proof of Proposition \ref{Prop Energy est}.
	\end{proof}

	\begin{remark}\label{Rmrk 3 bo numb}
	Here we needed to make the restriction $\Upsilon \lesssim \varepsilon$  Also, note if we have imposed the stronger criterion \eqref{Crit ve mu} we would have $\varepsilon^{-2}\Upsilon \lesssim \varepsilon$ and implies $\varepsilon \mu \leq \mathrm{bo}^{-1}$. 

\end{remark}

	 	\section{Proof of Theorem \ref{W-P System BO}}\label{Sec Wp system BO}
	 	
	 	Since the long-time well-posedness of the unidirectional model \eqref{Full disp Benjamin} is classical (see Remark \ref{Remark wp}) we only give the proof for \eqref{System Full disp B}. The strategy of the proof is the same as for \eqref{IWW}. It relies on the energy method, where we need to find a suitable symmetrizer for the system. For simplicity, we define
	 	$$u = \mathrm{t}(\mathrm{D})v \quad \text{and} \quad \mathbf{U} = (\zeta, u)^T,$$ 
	 	using it to write  \eqref{System Full disp B} on the compact form:\color{black} 
	 \begin{equation}\label{Eq 1 M}
	 	\partial_t \mathbf{U} + \mathcal{M}(\mathbf{U}) \mathbf{U} = \mathbf{0},
	 \end{equation}
	 with
	 \begin{equation}\label{M 1}
	 	\mathcal{M}(\mathbf{U}) = 
	 	\begin{pmatrix}
	 		\ve \mathrm{t}(\mathrm{D})\big{(} u \partial_x \bullet \big{)}  & \partial_x + \ve \mathrm{t}(\mathrm{D})\big{(}\zeta \partial_x\bullet \big{)} \\
	 		\mathrm{k}(\mathrm{D})\partial_x & 	\ve \mathrm{t}(\mathrm{D})\big{(} u \partial_x \bullet \big{)}
	 	\end{pmatrix},
	 \end{equation}
	 To define an energy, we introduce the symmetrizer
	 \begin{equation}\label{S: symmetrizer}
	 	\mathcal{Q}(\mathbf{U}) 
	 	=
	 	\begin{pmatrix}
	 		(1-\mathrm{bo}^{-1} \partial_x^2)& 0 \\
	 		0 & 	\mathrm{t}^{-1}(\mathrm{D}) + \varepsilon \zeta 
	 	\end{pmatrix}.
	 \end{equation}
	 Then the energy associated with $X^s_{\mu,\mathrm{bo}}(\R)$, given in Definition \ref{Energy space X}, for the weakly dispersive Benjamin system \eqref{System Full disp B} can be written as
	 \begin{equation*}
	 	E_{\text{\tiny wBs}}^s( \mathbf{U}) = \big{(} \Lambda^s \mathbf{U}, \mathcal{Q}(\mathbf{U})\Lambda^s \mathbf{U} \big{)}_{L^2}.
	 \end{equation*}
	 We will use this energy to deduce a bound on the solutions of \eqref{Eq 1 M}. As noted in the proof Theorem \ref{Thm 1}, this is not enough. One also needs a uniqueness type estimate to establish the well-posedness with the energy method. This estimate is derived similarly,  but since we will need it for the proof of the full justification, we also state it in the next result. For convenience,  we let two solutions of \eqref{Eq 1 M} be given by $\mathbf{U}_1 = (\zeta_1, u_1 )^T$ and $\mathbf{U}_2 = (\zeta_2, u_2)^T$ where we define  $(\eta,w) = (\zeta_1 - \zeta_2, u_1-u_2)$. Then $\mathbf{W} = (\eta,w)^T$ solves 
	 \begin{equation}\label{lin W}
	 	\partial_t \mathbf{W} +  \mathcal{M}(\mathbf{U}_1)\mathbf{W} =  \mathbf{F}, 
	 \end{equation}
	 with $\mathcal{M}$ defined as  in \eqref{M 1} and  the source term is given by
	 \begin{equation} \label{F: source term}
	 	\mathbf{F} 
	 	=- \big(\mathcal{M}(\mathbf{U}_1) - \mathcal{M}(\mathbf{U}_2)  \big) \mathbf{U}_2.
	 \end{equation}
	 The energy associated to \eqref{lin W} is given in terms of the symmetrizer $\mathcal{Q}(\mathbf{U}_1)$ defined in \eqref{S: symmetrizer} and reads
	 \begin{equation}\label{tilde Energy s}
	 	\tilde{E}^s_{\text{\tiny wBs}}(\mathbf{W}) 
	 	: =
	 	\big( \Lambda^s  \mathbf{W}, \mathcal{Q}(\mathbf{U}_1) \Lambda^s \mathbf{W}\big)_{L^2}. 
	 \end{equation}
	 
	 From these energies, we have the following \textit{a priori} estimate and an estimate of the difference between two solutions.

	 \begin{prop}\label{Energy fully disp BO} Let $\ve, \mu, \mathrm{bo}^{-1} \in (0,1) $, $s> \frac{3}{2}$,  and $ (\zeta,u) \in C([0,T]; X^s_{\mu,\mathrm{bo}}(\R) )$ be a solution to \eqref{Eq 1 M} on a time interval $[0,T]$ for some $T >0 $.  Moreover, assume that \eqref{non-cavitation} holds for $\zeta$ uniformly in time and let $K(s) = C_1(  h_{\min}^{-1},|(\zeta, u)|_{X^s_{\mu,\mathrm{bo}}})>0$ be a nondecreasing function of its argument. Then, for the energy given in Definition  \ref{Def Energy}, there holds,

	 	\begin{itemize}
	 		\item [1.] For all $0<t<T$, we have that
	 		\begin{equation}\label{Energy full dispersion}
	 			\frac{\mathrm{d}}{\mathrm{d}t} E^s_{\text{\tiny wBs}}(\mathbf{U}) \leq \ve  K(s)E_s(\mathbf{U}).
	 		\end{equation}
	 		\item [2.] For all $0<t<T$, there exist $C_2>0$ such that
	 		\begin{equation}\label{equiv. full dispersion}
	 			(K(s))^{-1}|(\zeta, u)|^2_{X^s_{\mu,\mathrm{bo}}} \leq E^s_{\text{\tiny wBs}}(\mathbf{U}) \leq K(s)  |(\zeta, u)|^2_{X^s_{\mu,\mathrm{bo}}}.
	 		\end{equation}
	 	\end{itemize}
	 	Moreover, let $ (\zeta_1,u_1),(\zeta_2,u_2) \in C([0,T]; X^s_{\mu,\mathrm{bo}}(\R) )$ be solutions to \eqref{Eq 1 M} on a time interval $[0,T]$ and satisfying the condition \eqref{non-cavitation}  uniformly in time. Define the difference to be $\mathbf{W}=(\eta,w) = (\zeta_1 - \zeta_2, u_1 - u_2)$ and for  $i=1,2$ let $\tilde{K}(s) = C_2(  h_{\min}^{-1},|(\zeta_i, v_i)|_{X^s_{\mu,\mathrm{bo}}})>0$ be a nondecreasing function of its argument. Then, for the energy defined by \eqref{tilde Energy s}, there holds
	 	\begin{itemize}
	 		\item [3.] For all $0<t<T$, we have that
	 		\begin{equation}\label{Energy 1}
	 			\frac{\mathrm{d}}{\mathrm{d}t} \tilde{E}^0_{\text{\tiny wBs}}(\mathbf{W}) \leq \ve  \tilde{K}(s) \tilde{E}^0_{\text{\tiny wBs}}(\mathbf{W}).
	 		\end{equation}
	 		\item [4.] For all $0<t<T$, we have that
	 		\begin{equation}\label{Energy 2}
	 			\frac{\mathrm{d}}{\mathrm{d}t} \tilde{E}^s_{\text{\tiny wBs}}(\mathbf{W})  \leq \ve \tilde{K}(s+1)\tilde{E}^s_{\text{\tiny wBs}}(\mathbf{W}).
	 		\end{equation}
	 		
	 		\item [5.] For all $0<t<T$ and $r\geq 0$, there exist $C_3(h_{\min}^{-1},|\zeta|_{L^{\infty}})>0$ nondecreasing function of its argument such that
	 		\begin{equation}\label{equiv 2}
	 			C_3^{-1}|(\eta,w)|^2_{X^r_{\mu,\mathrm{bo}}} \leq \tilde{E}^r_{\text{\tiny wBs}}(\mathbf{W}) \leq C_3 |(\eta,w)|^2_{X^r_{\mu,\mathrm{bo}}}.
	 		\end{equation}
	 	\end{itemize}

	 \end{prop}

	 \begin{remark}\label{On loss of der}
	 	In estimate \eqref{Energy 2}, we observe that there is a loss of derivative. This is not sufficient for the proof of the continuous dependence with respect to initial data. In fact, one would have to refine this estimate using a Bona-Smith argument \cite{BonaSmith75}. Since the system is almost identical to one in \cite{Paulsen22} we only detail the main differences. With this in mind, we simply allow for this rough estimate and use it to deduce the convergence estimate when comparing its solution with the ones of the internal water waves system.
	 \end{remark}

	 \begin{proof}[Proof of Proposition \ref{Energy fully disp BO}] The proof of point $3-5$ is similar to the proof of the first two points. We will therefore only prove \eqref{Energy full dispersion}  and \eqref{equiv. full dispersion}.

	 	We first prove estimate \eqref{equiv. full dispersion}.  By definition, we have  that
	 	\begin{align*}
	 		E^s_{\text{\tiny wBs}}(\mathbf{U}) =|\Lambda^s \zeta |_{L^2}^2 + \mathrm{bo}^{-1}|\Lambda^s \partial_x \zeta|^2_{L^2} 
	 		+
	 		 \big( \Lambda^s  u,(\mathrm{t}^{-1}(\mathrm{D}) +\ve \zeta)  \Lambda^s u \big)_{L^2}.
	 	\end{align*}
	 	Then, as a result of the non-cavitation condition \eqref{non-cavitation} and \eqref{coercivity} we find that
	 	\begin{align*}
			\big( \Lambda^s  u,(\mathrm{t}^{-1}(\mathrm{D}) +\ve \zeta) \Lambda^s  u \big)_{L^2} \geq \frac{h_0}{2}|u|_{H^s}^2 + C\sqrt{\mu} |u|_{\mathring{H}^{s+\frac{1}{2}}}^2,
	 	\end{align*}
	 	for any $h_0\in(0,1-h_{\min})$ and some $C>0$. The reverse inequality is a consequence of \eqref{coercivity},  Hölder's inequality, and the Sobolev embedding with $s>\frac{1}{2}$:
	 	\begin{align*}
	 		E^s_{\text{\tiny wBs}}(\mathbf{U}) 
	 		& \leq
	 		|\Lambda^s \zeta |_{L^2}^2 
	 		+
	 		\mathrm{bo}^{-1}|\Lambda^s \partial_x \zeta|^2_{L^2}  
	 		+
	 		|\Lambda^s \mathrm{t}^{-\frac{1}{2}}(\mathrm{D}) u|_{H^s}^2 
	 		+
	 		\ve |\zeta|_{L^{\infty}}|\Lambda^su|_{H^s}^2
	 		\\  
	 		& \leq C (|\zeta|_{H^{s+1}_{\mathrm{bo}}}^2 + |\Lambda^su|_{H^s}^2 +\sqrt{\mu} |u|_{\mathring{H}^{s+\frac{1}{2}}}^2 ).
	 	\end{align*}

	 	Next, we prove \eqref{Energy full dispersion}. By using \eqref{Eq 1 M} and the fact that $\mathcal{Q}(\mathbf{U})$ is self-adjoint, we compute
	 	\begin{align*}
	 		\frac{1}{2}\frac{\mathrm{d}}{\mathrm{d}t} E^s_{\text{\tiny wBs}}(\mathbf{U})
	 		& =
	 		-
	 		\big(\Lambda^s \mathcal{M}(\mathbf{U}) \mathbf{U} , \mathcal{Q}(\mathbf{U}) \Lambda^s\mathbf{U}\big)_{L^2}
	 		+
	 		\frac{1}{2}\big( \Lambda^s \mathbf{U}, ( \partial_t \mathcal{Q}(\mathbf{U})) \Lambda^s \mathbf{U}\big)_{L^2}
	 		\\
	 		& = 
	 		-
	 		\ve \big(\Lambda^s (u \partial_x \zeta) , \mathrm{k}(\mathrm{D}) \Lambda^s\zeta \big)_{L^2}
	 		-
	 		\big(\Lambda^s(\partial_x u + \ve \mathrm{t}(\mathrm{D})(\zeta \partial_x u))  , (1-\mathrm{bo}^{-1} \partial_x^2) \Lambda^s\zeta\big)_{L^2}
	 		\\
	 		&
	 		\hspace{0.5cm}
	 		-
	 		\big(\Lambda^s \mathrm{k}(\mathrm{D}) \partial_x \zeta , (\mathrm{t}^{-1}(\mathrm{D}) + \varepsilon \zeta ) \Lambda^s u\big)_{L^2}
	 		 \\
	 		 & \hspace{0.5cm}
	 		 -
	 		 \ve \big(\Lambda^s \mathrm{t}(\mathrm{D})(u\partial_x u) , (\mathrm{t}^{-1}(\mathrm{D}) + \varepsilon \zeta ) \Lambda^s u\big)_{L^2}
	 		  +
	 		  \frac{1}{2}	\big( \Lambda^s u,  (\partial_t \zeta ) \Lambda^s  u \big)_{L^2}. 
	 	\end{align*} 
	 	We now note that the terms are similar to the ones in the proof of Proposition 3.4. in \cite{Paulsen22}.  The only difference is the presence of the zero order differential operator $\sqrt{\mathrm{I}}(\mathrm{D})$ in the definition of $\sqrt{\mathrm{k}}(\mathrm{D})$ and $\sqrt{\mathrm{t}}(\mathrm{D})$. However, the final estimates still remain the same. In particular, the proof of  Proposition 3.4 in \cite{Paulsen22} relies on
	 	\begin{align*}
	 		|\sqrt{\mathrm{k}}(\mathrm{D})f|_{H^s}   & \leq C|f|_{H^{s+1}_{\mathrm{bo}}},
	 		\\ 
	 		|\sqrt{\mathrm{t}}(\mathrm{D})f|_{H^s}   & \leq C |f|_{H^s},
	 		\\
	 		|[\Lambda^s \sqrt{\mathrm{k}}(\mathrm{D}),  f]\partial_x g|_{L^2} & \leq C |f|_{H^{s+1}_{\mathrm{bo}}} |g|_{H^{s+1}_{\mathrm{bo}}},
	 		\\
	 		|[ \Lambda^s \sqrt{\mathrm{t}}(\mathrm{D}),  f]\partial_x g|_{L^2} & \leq C |f|_{H^s} |g|_{H^s},
	 		\\
	 		|\partial_x [\sqrt{\mathrm{t}}(\mathrm{D}),  f]g|_{L^2}& \leq C |f|_{H^{t_0+1}} |g|_{L^2},
	 	\end{align*}
 		which are provided in Propositions \ref{prop B 2} and \ref{prop B 5} in the Appendix. From these estimates one can follow \cite{Paulsen22} line by line to deduce the final result.

	 \end{proof}

	 \begin{proof}[Proof of Theorem \ref{W-P System BO}]
	 	The proof is an application of the energy method where we use the estimates in Proposition \ref{Energy fully disp BO} and Remark \ref{On loss of der}. We refer the reader to \cite{Paulsen22}, Theorem 1.11. for a similar proof. 
	 \end{proof}

	\section{Proof of Theorem \ref{Thm: Consistency}}\label{Sec Consistency}
	For the derivation of \eqref{Bo equation}, we follow the strategy given in \cite{Bona_Saut_Lannes_08}, where we formulated \eqref{IWW} in terms of the $(\zeta, v)$ solving system \eqref{IWW - v}.	The convenience of this formulation is apparent since we can use the shallow water expansion of $\mathcal{G}_{\mu}^+$ derived by Emerald in \cite{Emerald21}:
		%
		%
		%
	%
	%
	%
	\begin{prop}Let $t_0\geq 1$, $s \geq 0$, and $\zeta \in H^{s + 3}(\mathbb{R})$ such that it satisfies \eqref{non-cavitation}. Then for all  $ \psi^+ \in  \dot{H}^{s + 4}(\mathbb{R})$ there holds,
		\begin{equation}\label{exp G+}
			|\mathcal{G}_{\mu}^+[\ve \zeta]\psi^+ - \partial_x\big{(} -\mu(1+\ve \zeta)\mathrm{T}(\mathrm{D})\partial_x \psi^+\big{)}|_{H^s}\leq \mu^2\ve  M(s+3)|\partial_x \psi^+|_{H^{s+3}},
		\end{equation}
		where $\mathrm{T}(\mathrm{D}) = \frac{\tanh(\sqrt{\mu} |\mathrm{D}|)}{\sqrt{\mu}|\mathrm{D}|}$.
	\end{prop}
	Finally, the last ingredient before we turn to the proof of Theorem \ref{Thm: Consistency} is an expansion of the interface operator defined by \eqref{interface op}.
	\begin{prop}\label{H: Infinite depth/BO regime} Let $t_0\geq 1$, $s \geq 0$, and $\zeta \in H^{s + 3}(\mathbb{R})$ such that it satisfies \eqref{non-cavitation}. Then for all  $ \psi^+ \in  \dot{H}^{s + 4}(\mathbb{R})$ there holds,
		\begin{align}\label{exp H}
			\Big| \mathbf{H}_{\mu}[\ve \zeta]\psi^+- \big{(}-  \tanh(\sqrt{\mu} |\mathrm{D}|)\partial_x\psi^+\big{)}  \Big|_{H^{s}} 
			\leq
			\ve 
			\sqrt{\mu} M(s+3)|\partial_x \psi^+|_{H^{s+3}}.
		\end{align}

	\end{prop}	
	\begin{remark}\label{Rmrk higher d} \color{white} space \color{black}
		\begin{itemize}
			\item The expansion of $\mathbf{H}_{\mu}[\ve \zeta]$ is a fully dispersive version of the infinite depth expansion mentioned in \cite{Bona_Saut_Lannes_08} (see Remark $20$ of this paper). \\
			
			\item This expansion can be extended to higher dimensions and may be a useful tool to derive other systems and equations from the internal water waves equations in the case one layer is of infinite depth. 
		\end{itemize}
	\end{remark}
	\begin{proof}[Proof of Proposition \ref{H: Infinite depth/BO regime}] The proof is divided into three steps and relies on making an approximate solution of $\Phi^-$ solving \eqref{For exp phi-}.\\ 
		
		\noindent 
		\underline{Step 1.} \textit{Approximate solution for \eqref{For exp phi-}}. First we make the transformation $\phi^- = \Phi^-\circ \Sigma^-$ defined on $\mathcal{S}^-$:
		\begin{equation*}
			\begin{cases}
				\nabla_{x,z}^{\mu} \cdot P(\Sigma^-)\nabla_{x,z}^{\mu} \phi^- = 0 \quad \text{in} \quad  \mathcal{S}^-
				\\
				\partial_n^{P^-} \phi^{-}|_{z=0} = \mathcal{G}^+_{\mu}[\ve \zeta]\psi^+ , \quad \lim \limits_{z\rightarrow \infty} |\nabla^{\mu}_{x,z} \phi^{-} | = 0,
			\end{cases}
		\end{equation*}
		where we recall from estimate \eqref{est on phi- in prop 2.4} in the proof of Proposition \ref{Inverse 2} that
		\begin{equation*}
			\| \Lambda^s \nabla_{x,z}^{\mu} \phi^- \|_{L^2(\mathcal{S}^-)} \leq \sqrt{\mu}M(s+2)|\psi^+|_{\dot{H}^{s+\frac{1}{2}}_{\mu}}.
		\end{equation*}
		To construct an approximation, we can verify that
		\begin{align*}
			\phi_{\mathrm{app}}^-  =- \tanh(\sqrt{\mu} |\mathrm{D}|)\mathrm{e}^{-z\sqrt{\mu}|\mathrm{D}|}\psi^+,
		\end{align*}	
		solves
		\begin{equation}\label{BO: phi_0}
			\begin{cases}
				(\mu\partial_x^2+\partial_z^2)\phi_{\mathrm{app}}^- = 0,\\
				\partial_z\phi_{\mathrm{app}}^-\:_{|_{ z = 0}} = -\mu\partial_x^2 \mathrm{T}(\mathrm{D})\psi^+,
			\end{cases}
		\end{equation}
		and satisfies the decay estimates \eqref{Decay est} by using Plancherel's identity. Defining the difference $u = \phi^- - \phi_{\mathrm{app}}^-$, then we have by construction that
		\begin{equation}\label{sol of u}
			\begin{cases}
				\nabla_{x,z}^{\mu} \cdot P(\Sigma^-)\nabla_{x,z}^{\mu} u = \ve \mu r_1 \quad \text{in} \quad  \mathcal{S}^-
				\\
				\partial_n^{P^-}u |_{z=0} = \mathcal{G}^+_{\mu}[\ve \zeta]\psi^+ + \mu \partial_x^2 \mathrm{T}(\mathrm{D}) \psi^+ + \ve \mu r_2, \quad \lim \limits_{z\rightarrow \infty} |\nabla^{\mu}_{x,z} u | = 0,
			\end{cases}
		\end{equation}
		where $r_1$ reads
		\begin{align*}
			 r_1 
			& = 
			-\frac{1}{\ve \mu}\nabla_{x,z}^{\mu} \cdot P(\Sigma^-)\nabla_{x,z}^{\mu} \phi_{\mathrm{app}}^-
			\\
			& =
			\partial_x\big{(}(\partial_x \zeta)  \partial_z \phi_{\mathrm{app}}^-\big{)} 
			+
			\partial_z\Big{(}(\partial_x \zeta)\partial_x \phi_{\mathrm{app}}^- 
			-
			\ve (\partial_x \zeta)^2
			\partial_z \phi_{\mathrm{app}}^-\Big{)},
		\end{align*}
		and the second rest is given by
		\begin{align*}
			r_2  
			& =
			\frac{1}{\ve \mu}
			\Big{(}\mathbf{e}_z \cdot P(\Sigma^-)\nabla_{x,z}^{\mu} \phi_{\mathrm{app}}^-  - \partial_z \phi_{\mathrm{app}}^-\Big{)}|_{z=0}
			\\
			& =
			 (\partial_x\zeta) \partial_x \tanh(\sqrt{\mu}|\mathrm{D}|)\psi^+ + \ve \sqrt{\mu} (\partial_x \zeta)^2 |\mathrm{D}| \tanh(\sqrt{\mu}|\mathrm{D}|) \psi^+.\\
		\end{align*}

		\noindent
		\underline{Step 2.} \textit{Estimate on $\mathrm{LHS}_{\eqref{exp H}}$}. By construction we have that 
		\begin{equation*}
			 \mathbf{H}_{\mu}[\ve \zeta]\psi^+- \big{(}-  \tanh(\sqrt{\mu} |\mathrm{D}|)\partial_x\psi^+\big{)}  = \partial_x u_0,
		\end{equation*}
		where we let $u_0 : = u|_{z=0}$. Then by trace estimate \eqref{trace 1} we obtain the bound
		\begin{align}\label{est on dx u0}
			|\partial_x u_0|_{H^{s}} 
			\leq
			||\mathrm{D}|^{\frac{1}{2}}\Lambda^{s} u_0|_{\mathring{H}^{\frac{1}{2}}}
			\leq  \mu^{-\frac{1}{4}}
			\| |\mathrm{D}|^{\frac{1}{2}}\Lambda^{s} \nabla_{x,z}^{\mu} u\|_{L^2(\mathcal{S}^-)}.
		\end{align}
		To conclude, we need an estimate on the right hand side of this inequality. \\

		\noindent
		\underline{Step 3.} \textit{Estimate on $|\mathrm{D}|^{\frac{1}{2}}\Lambda^{s} \nabla_{x,z}^{\mu} u$ in $L^2(\mathcal{S}^-)$.} We let $\mathrm{m}(\mathrm{D}) = |\mathrm{D}|^{\frac{1}{2}}\Lambda^{s}$ and observe by the coercivity of $P(\Sigma^-)$, given by \eqref{Coercivity-},  that
		\begin{align*}
			\frac{1}{1+M}\|\mathrm{m}(\mathrm{D}) \nabla_{x,z}^{\mu}u\|_{L^2(\mathcal{S}^-)}^2 
			& \leq 
			\int_{\mathcal{S}^-} P(\Sigma^-)\mathrm{m}(\mathrm{D})\nabla_{x,z}^{\mu}u \cdot \mathrm{m}(\mathrm{D}) \nabla_{x,z}^{\mu}u \: \mathrm{d}x \mathrm{dz}
			\\ 
			& =
			\int_{\mathcal{S}^-} \mathrm{m}(\mathrm{D}) P(\Sigma^-)\nabla_{x,z}^{\mu}u \cdot \mathrm{m}(\mathrm{D}) \nabla_{x,z}^{\mu}u \: \mathrm{d}x \mathrm{dz} 
			\\ 
			& 
			\hspace{0.5cm}
			+
			\int_{\mathcal{S}^-} [\mathrm{m}(\mathrm{D}), P(\Sigma^-)]\nabla_{x,z}^{\mu}u \cdot \mathrm{m}(\mathrm{D}) \nabla_{x,z}^{\mu}u \: \mathrm{d}x \mathrm{dz} 
			\\ 
			& = 
			I_1 + I_2.
		\end{align*}
		For the estimate of $I_1$, we use the Divergence Theorem together with \eqref{sol of u} to obtain the relation
		\begin{align*}
			I_1 
			& =
			\int_{\R} (\mathrm{m}(\mathrm{D})\partial_n^{P^-} u|_{z=0}) \mathrm{m}(\mathrm{D}) u_0 \: \mathrm{d}x
			- 
			\varepsilon \mu \int_{\mathcal{S}^-} \mathrm{m}(\mathrm{D}) r_1 \:  \mathrm{m}(\mathrm{D}) u \: \mathrm{d}x \mathrm{d}z
			\\ 
			& = 
			I_1^1 + I_1^2.
		\end{align*}
		To estimate $I_1^1$, we use Plancherel's identity, Cauchy-Schwarz inequality, trace estimate \eqref{trace 1},  estimate \eqref{exp G+} to approximate $\mathcal{G}^+_{\mu}$, and basic product estimates to treat $r_2$:\footnote{The term  $I_1^1$ gives the main restriction on the precision of the parameters and the regularity of $(\zeta,\psi^+)$.} 
		\begin{align*}
			|I_1^1| 
			&  =
			\Big{|} 
			\int (\Lambda^{s}\partial_n^{P^-} u|_{z=0})
			|\mathrm{D}|^{\frac{1}{2}}   
			\mathrm{m}(\mathrm{D}) u_0
			\: \mathrm{d}x
			\Big{|}
			\\ 
			& \leq 
			\big{(}| \mathcal{G}^+_{\mu}[\ve \zeta]\psi^+ + \mu \partial_x^2 \mathrm{T}(\mathrm{D})\psi^+|_{H^{s}} +  \ve\mu |r_2|_{H^{s}}\big{)}|\mathrm{m}(\mathrm{D}) u_0|_{\mathring{H}^{\frac{1}{2}}}
			\\ 
			& \leq 
			\ve \mu^{\frac{3}{4}} M(s+3)|\partial_x \psi^+|_{H^{s+3}} \|\mathrm{m}(\mathrm{D}) \nabla_{x,z}^{\mu}u\|_{L^2(\mathcal{S}^-)}.
		\end{align*}
		To estimate $I_1^2$, we use the definition of $r_1$, Plancherel's identity and the Hilbert transform identity $\partial_x = - \mathcal{H}|\mathrm{D}|$ to make the decomposition:
		\begin{align*}
			I_1^{2} 
			& =
			\varepsilon \mu
			\int_{\mathcal{S}^-} \Lambda^s\mathcal{H}|\mathrm{D}|^{\frac{1}{2}}\big{(}(\partial_x \zeta)  \partial_z \phi_{\mathrm{app}}^-\big{)} \: |\mathrm{D}|^1 \mathrm{m}(\mathrm{D}) u\: \mathrm{d}x \mathrm{d}z 
			\\ 
			& \hspace{0.5cm}
			-
			\varepsilon \mu
			\int_{\mathcal{S}^-}  \mathrm{m}(\mathrm{D}) 	\partial_z\Big{(}(\partial_x \zeta)\partial_x \phi_{\mathrm{app}}^- 
			-
			\ve (\partial_x \zeta)^2
			\partial_z \phi_{\mathrm{app}}^-\Big{)}\: \mathrm{m}(\mathrm{D}) u\: \mathrm{d}x \mathrm{d}z	
			\\	
			& = 
			I_1^{2,1} + I_1^{2,2}.	
		\end{align*}
		To estimate $I_1^{2,1}$, we use Cauchy-Schwarz inequality, the definition $\partial_z\phi_{\mathrm{app}}^-$ to gain $\sqrt{\mu}$ and the smoothing estimate:  
		\begin{align}\notag 
			\| \tanh(\sqrt{\mu}|\mathrm{D}|)|\mathrm{D}|e^{-z\sqrt{\mu}|\mathrm{D}|}\psi^+\|_{L^2(\mathcal{S}^-)}
			& \leq \notag 
			 \mu^{-\frac{1}{4}} ||\mathrm{D}|^{\frac{1}{2}}\tanh(\sqrt{\mu}|\mathrm{D}|)\psi^{+}|_{L^2}
			\\ 
			& \leq \mu^{\frac{1}{4}}\notag
			||\mathrm{D}|\frac{\tanh(\sqrt{\mu}|\mathrm{D}|)}{\sqrt{\mu}|\mathrm{D}|}\psi^{+}|_{H^{\frac{1}{2}}}
			\\
			& \leq \mu^{\frac{1}{4}} |\partial_x \psi^+|_{H^{\frac{1}{2}}}, \label{smoothing 2}
		\end{align}
		to deduce the bound
		\begin{align*}
			|I_2^1|
			\leq \varepsilon \mu  M(s+2) |\partial_x\psi^+ |_{H^{s+1}} \|\mathrm{m}(\mathrm{D})\nabla_{x,z}^{\mu}u\|_{L^2(\mathcal{S}^-)}.
		\end{align*}
		On the other hand, to estimate $I_1^{2,2}$ we also integrate by parts and the trace estimate \eqref{trace 1} to find that
		\begin{align*}
			|I_2^2| 
			& \leq 
			\varepsilon \mu 
			\Big{|}
			\int_{\R}
			\Lambda^s\Big{(}(\partial_x \zeta)\partial_x \phi_{\mathrm{app}}^- 
			-
			\ve (\partial_x \zeta)^2
			\partial_z \phi_{\mathrm{app}}^-\Big{)}|_{z=0} |\mathrm{D}|^{\frac{1}{2}} \mathrm{m}(\mathrm{D}) u_0\: \mathrm{d}x
			\Big{|}
			\\
			& 
			\hspace{0.5cm}
			+
			\varepsilon \mu 
			\Big{|}
			\int_{\mathcal{S}^-} 
			|\mathrm{D}|^{\frac{1}{2}}
			\Lambda^s\Big{(}(\partial_x \zeta)\partial_x \phi_{\mathrm{app}}^- 
			-
			\ve (\partial_x \zeta)^2
			\partial_z \phi_{\mathrm{app}}^-\Big{)} \mathrm{m}(\mathrm{D}) \partial_z u\: \mathrm{d}x
			\Big{|}
			\\ 
			& 
			\leq
			\varepsilon \mu^{\frac{3}{4}}
			M(s+2) |\partial_x\psi^+ |_{H^{s+1}} \|\mathrm{m}(\mathrm{D})\nabla_{x,z}^{\mu}u\|_{L^2(\mathcal{S}^-)}.
		\end{align*}
		Gathering these estimates yields the following bound on $I_1$:
		\begin{equation*}
			|I_1| \leq \varepsilon \mu^{\frac{3}{4}}
			M(s+3) |\partial_x\psi^+ |_{H^{s+1}} \|\mathrm{m}(\mathrm{D})\nabla_{x,z}^{\mu}u\|_{L^2(\mathcal{S}^-)}.
		\end{equation*}
		Lastly, to estimate $I_2$ we use Cauchy-Schwarz inequality,  \eqref{est on phi- in prop 2.4}, \eqref{smoothing 2},  and that the nonzero terms in the commutator $[\mathrm{m}(\mathrm{D}),P(\Sigma^-)]$ is of order $\varepsilon \sqrt{ \mu}$:
		\begin{align*}
			|I_2| 
			&  \leq
			\|[\mathrm{m}(\mathrm{D}), P(\Sigma^-)]\nabla_{x,z}^{\mu}u \|_{L^2(\mathcal{S}^-)} \| \mathrm{m}(\mathrm{D}) \nabla_{x,z}^{\mu}u \|_{L^2(\mathcal{S}^-)}
			\\
			& \leq 
			\varepsilon\sqrt{\mu} M(s+2) \big{(}
			\|\Lambda^{s+\frac{1}{2}} \nabla_{x,z}^{\mu}\phi^-\|_{L^2(\mathcal{S}^-)}
			+
			\|\Lambda^{s+\frac{1}{2}} \nabla_{x,z}^{\mu}\phi^-_{\mathrm{app}}\|_{L^2(\mathcal{S}^-)}
			\big{)}
			\| \mathrm{m}(\mathrm{D}) \nabla_{x,z}^{\mu}u \|_{L^2(\mathcal{S}^-)}
			\\
			& \leq 
			\varepsilon\mu 
			M(s+3)|\partial_x \psi^+|_{H^{s+2}}\| \mathrm{m}(\mathrm{D}) \nabla_{x,z}^{\mu}u \|_{L^2(\mathcal{S}^-)},
		\end{align*}
		which implies that
		\begin{align*}
			\||\mathrm{D}|^{\frac{1}{2}}\Lambda^s \nabla_{x,z}^{\mu}u\|_{L^2(\mathcal{S}^-)} \leq \varepsilon\mu^{\frac{3}{4}} M(s+3)|\partial_x \psi^+|_{H^{s+3}},
		\end{align*}
		and from estimate \eqref{est on dx u0} we make the conclusion:
		\begin{equation*}
			|\partial_x u_0|_{H^{s}}  \leq \varepsilon \sqrt{\mu} M(s+3)|\partial_x \psi^+|_{H^{s+3}}.
		\end{equation*}

	\end{proof}

	Having these expansions at hand, we can turn to the main proofs of the section.

	\begin{proof}[Proof of Theorem \ref{Thm: Consistency}] We first observe that we can have bound on $v = \partial_x \psi$ in terms of the intial data. Indeed, using the definition of $\psi_{(\alpha)}$ with \eqref{Classical prod est} and \eqref{W+ est low reg}  yields,
		\begin{align*}
			|\partial_x \psi |_{H^{N-\frac{1}{2}}} 
			 \leq 
			|\psi|_{\dot{H}_{\mu}^{N+\frac{1}{2}}}
			& \leq
			\sum \limits_{\alpha \in \N^2 \: : \: |\alpha|\leq N} |\psi_{(\alpha)}|_{\dot{H}_{\mu}^{\frac{1}{2}}} + \ve |\underline{w}\zeta_{(\alpha)}|_{\dot{H}_{\mu}^{\frac{1}{2}}} 
			\\ 
			& \leq 
			\sum \limits_{\alpha \in \N^2 \: : \: |\alpha|\leq N} |\psi_{(\alpha)}|_{\dot{H}_{\mu}^{\frac{1}{2}}} +  M |\psi|_{\dot{H}^{2}} |\zeta_{(\alpha)}|_{H^1_{ \mathrm{bo}}}
			\\ 
			& 
			\leq 
			 C( \mathcal{E}^N(\mathbf{U})).
		\end{align*}
		since $\varepsilon\leq\mathrm{bo}^{-\frac{1}{2}}$. Then Theorem \ref{Thm 1} with estimate \eqref{bound on sol E}  implies
		\begin{equation*}
			\sup_{t\in [0,\ve^{-1}T] }\big{(}|\zeta|_{H^{N+1}_{\mathrm{bo}}}^2 + |v|_{H^{N-\frac{1}{2}}}^2\big{)} \leq  C(\mathcal{E}^N(\mathbf{U}_0), {h^{-1}_{\min}}, \mathfrak{d}(\mathbf{U}_0)^{-1}).
		\end{equation*}	
		With this estimate in mind, we now give the proof in two steps where we first derive \eqref{System Full disp BO} in the sense of consistency.  \\ 
		
		\noindent
		\underline{Step 1.} To derive \eqref{System Full disp B} we let $C : = C(\mathcal{E}^N(\mathbf{U}_0), {h^{-1}_{\min}}, \mathfrak{d}(\mathbf{U}_0)^{-1})$ and $R$ be some generic function satisfying
		\begin{equation}\label{est R}
			|R|_{H^{N-5}}\leq C.
		\end{equation}
		Then we can simplify \eqref{IWW - v}. In particular, for the second equation in \eqref{IWW}, we first estimate the nonlinear terms. For  $\mathcal{N}[\ve\zeta,\psi^{\pm}]$ we apply estimates \eqref{Classical prod est}, \eqref{exp G+}, and \eqref{bound on sol E} to obtain,
		\begin{equation*}
			|\mathcal{N}[\ve\zeta,\psi^{\pm}]|_{H^{N-5}}\leq \mu  C (\mathcal{E}^N(\mathbf{U})) \leq\mu C,
		\end{equation*}
		and for the term involving surface tension, we observe that
		\begin{align*}
			\Big|\partial_x 
			\bigg{(}\partial_x \zeta 
			\Big{(} 
			\frac{1}{\sqrt{1+ \ve^2 \mu (\partial_x \zeta )^2}} - 1 
			\Big{)}
			\bigg{)}
			\Big|_{H^{N-5}}\leq \varepsilon^2\mu C (\mathcal{E}^N(\mathbf{U})) \leq \varepsilon^2\mu C.
		\end{align*}
		While the first equation in \eqref{IWW - v}, we use \eqref{exp G+} to deduce that
		\begin{equation}\label{IWW - pf}
			\begin{cases}
				\partial_t \zeta + \partial_x^2\mathrm{T}(\mathrm{D})\psi^+ + \ve \partial_x(\zeta \partial_x \mathrm{T}(\mathrm{D})\psi^+) = \ve \mu R
				\\
				\partial_t v + \big{(}1 - \mathrm{bo}^{-1} \partial_x^2\big{)}\partial_x \zeta 
				+
				\frac{\ve}{2}  \partial_x
				\big{(}
				(\partial_x \psi^{+})^2 
				-
				\gamma (\mathbf{H}_{\mu}[\ve\zeta]\psi^+)^2 
				\big{)}
				=    \mu \ve R.
			\end{cases}
		\end{equation}		
		Then to relate the equation to $v$, we observe by \eqref{exp H}:
		%
		%
		%
		%
		\begin{align*}
			\partial_x \psi^+
			& =
			\gamma \mathbf{H}_{\mu} [\ve \zeta]\psi^+ + v \\
			& = 
			-\gamma \tanh(\sqrt{\mu} |\mathrm{D}|) \partial_x \psi^+  + v +  \ve \sqrt{\mu} R,
		\end{align*}
		so that
		\begin{equation*}
			\partial_x \psi^+
			 =
			(1+\gamma \tanh(\sqrt{\mu} |\mathrm{D}|))^{-1}v + \ve \sqrt{\mu} R.
		\end{equation*}
		%
		%
		%
		%
		%
		%
		From this relation, we can express the first equation in terms of $\zeta$ and $v$ by using  the expansions of $\mathcal{G}^+_{\mu}$ given by \eqref{exp G+},  $\mathrm{T}(\mathrm{D})$ given by \eqref{est T},  $\mathrm{I}(\mathrm{D}) = (1+\gamma \tanh(\sqrt{\mu} |\mathrm{D}|))^{-1}$ given by \eqref{inverse tanh}, and recalled here:
		\begin{equation}\label{approxs}
			|(\mathrm{T}(\mathrm{D}) -1)f|_{H^s}  \leq \mu C |\partial_x^2f|_{H^{s}}, \quad  |(\mathrm{I}(\mathrm{D})-1)f|_{H^s}  \leq \sqrt{\mu} C |\partial_xf|_{H^{s}}.
		\end{equation}
		Indeed, from these observations there holds,
		\begin{align*}
			\partial_x^2\mathrm{T}(\mathrm{D})\psi^+ + \ve \partial_x(\zeta \partial_x \mathrm{T}(\mathrm{D})\psi^+) 
			& = 
			\mathrm{I}(\mathrm{D})\mathrm{T}(\mathrm{D})\partial_x v  
			+
			\ve \partial_x(\zeta v) + \ve\sqrt{\mu} R. 
		\end{align*}
		For the second equation, we have that
		\begin{equation*}
			(\partial_x \psi^+)^2 =v^2
			+  \sqrt{\mu} R,
		\end{equation*}
		and
		\begin{align*}
			(\mathbf{H}_{\mu}[\ve \zeta]\psi^+)^2 
			& =  \mu R.
		\end{align*}
		To summarize the approximations so far, we have the following system:
		\begin{equation*}
			\begin{cases}
				\partial_t \zeta
				+
				\mathrm{I}(\mathrm{D})\mathrm{T}(\mathrm{D})\partial_x v  +\ve \partial_x(\zeta v) =  \ve \sqrt{\mu}R
				\\
				\partial_t v 
				+
				\big{(}1- \mathrm{bo}^{-1} \partial_x^2\big{)}\partial_x \zeta 
				+
				\ve v\partial_xv =   \ve\sqrt{\mu}R.
			\end{cases}
		\end{equation*}
		Next, we ease the notation by defining 
		$$\mathrm{t}(\mathrm{D}) = \mathrm{I}(\mathrm{D})\mathrm{T}(\mathrm{D}), \quad \mathrm{k}(\mathrm{D}) = \mathrm{t}(\mathrm{D})\big{(}1 - \mathrm{bo}^{-1} \partial_x^2\big{)},$$ 
		and 
		$$u = \mathrm{t}(\mathrm{D})v.$$
		Then  use  \eqref{approxs}  to  find the desired system 
		\begin{equation*}
			\begin{cases}
				\partial_t \zeta
				+
				\partial_x u +\ve \mathrm{t}(\mathrm{D})\partial_x(\zeta u) =  \ve \sqrt{\mu}R
				\\
				\partial_t u
				+
				\mathrm{k}(\mathrm{D})\partial_x \zeta 
				+
				\frac{\ve}{2} \mathrm{t}(\mathrm{D})\partial_x (u)^2 =   \ve\sqrt{\mu}R.
			\end{cases}
		\end{equation*}\\

			\noindent
			\underline{Step 2.} To derive \eqref{Full disp Benjamin} we can make an approximation $u^{\text{\tiny wB}} $ in terms of $\zeta^{\text{\tiny wB}}$ at order $\mathcal{O}(\varepsilon( \ve +\sqrt{\mu} +  \mathrm{bo}^{-1}))$ solving \eqref{Full disp Benjamin}. Then we will show that this solution is consistent with the weakly dispersive Benjamin system \eqref{System Full disp B}.  
			%
			%
			%
			%
			%
			%
			The proof is given in two steps, where we first make formal computations and then prove the consistency once we have constructed $u^{\text{\tiny wB}}$.	\\

			\noindent
			\underline{Step 2.1.} \textit{Formal derivation:}
			To construct the lowest order approximation of $u^{\text{\tiny wB}}$ in terms of the solution $\zeta^{\text{\tiny wB}}$, we  first consider system \eqref{System Full disp B} at order $\ve$: 
			\begin{equation*}
				\begin{cases}
					\partial_t \zeta^{\text{\tiny wB}}
					+
					 \partial_xu^{\text{\tiny wB}} = \mathcal{O}(\ve)
					\\
					\partial_t u^{\text{\tiny wB}}
					+
					\mathrm{k}(\mathrm{D}) \partial_x \zeta^{\text{\tiny wB}}
					=\mathcal{O}(\ve).
				\end{cases}
			\end{equation*}\\ 
			We observe $u^{\text{\tiny wB}}$ is equal to  $\sqrt{	\mathrm{k}}(\mathrm{D})\zeta^{\text{\tiny wB}}$ at first order. Then having an approximation at first order we can now make a correction at higher order:\color{black}
			\begin{equation*}
				u^{\text{\tiny wB}}= \sqrt{	\mathrm{k}}(\mathrm{D})\zeta^{\text{\tiny wB}}+ \ve A,
			\end{equation*}
			for some function $A$ depending on the solution $\zeta^{\text{\tiny wB}}$.  In particular, we construct $A$ by plugging the ansatz into \eqref{System Full disp B} and using \eqref{est sqrt K}:
			\begin{equation*}
				\begin{cases}
					\partial_t \zeta^{\text{\tiny wB}}
					+
					\sqrt{	\mathrm{k}}(\mathrm{D})\partial_x \zeta^{\text{\tiny wB}}  
					+
					\ve\big{(}   \partial_xA + \partial_x((\zeta^{\text{\tiny wB}})^2)\big{)} = \mathcal{O}\big{(}\varepsilon( \varepsilon + \sqrt{\mu} + \mathrm{bo}^{-1} ) \big{)}
					\\
					\partial_t \zeta^{\text{\tiny wB}}
					+
					\sqrt{	\mathrm{k}}(\mathrm{D})\partial_x \zeta^{\text{\tiny wB}}
					+
					\ve\big{(}\partial_tA + \zeta^{\text{\tiny wB}} \partial_x\zeta^{\text{\tiny wB}}\big{)} = \mathcal{O}\big{(}\varepsilon( \varepsilon + \sqrt{\mu} + \mathrm{bo}^{-1} ) \big{)}.
				\end{cases}
			\end{equation*}
			Now using the nonlocal transport equations to see that $\partial_x A = -\partial_t A$ (up to an order $\mathcal{O}(\ve + \sqrt{\mu})$), we find that
			\begin{equation*}
					2 \partial_x A = -\zeta^{\text{\tiny wB}} \partial_x \zeta^{\text{\tiny wB}}.
			\end{equation*}
			We therefore let $v^{\text{\tiny wB}}$ be given by
			\begin{equation}\label{v app}
				v^{\text{\tiny wB}} 
				=
				\mathrm{t}^{-1}(\mathrm{D})u^{\text{\tiny wB}} 
				=
				\mathrm{t}^{-1}(\mathrm{D})
				\big{(}\sqrt{	\mathrm{k}}(\mathrm{D})\zeta^{\text{\tiny wB}}  - \frac{\ve}{4} (\zeta^{\text{\tiny wB}})^2\big{)}.
			\end{equation}

			\noindent
			\underline{Step 2.2.} \textit{Rigorous derivation:} For any $ \zeta_0 \in H^{N+1}_{\mathrm{bo}}(\R)$ we have that  
			\begin{equation*}
				v_0 = \mathrm{t}^{-1}(\mathrm{D})\big{(}\sqrt{	\mathrm{k}}(\mathrm{D}) \zeta_0   - \frac{\ve}{4} (\zeta_0)^2\big{)} \in H^{N-\frac{1}{2}}(\R).
			\end{equation*}
			Indeed, using Plancherel's identity, the algebra property of $H^{N-\frac{1}{2}}(\R)$, we observe that
			\begin{align}\notag
				|v_0|_{H^{N-\frac{1}{2}}} 
				& \leq
				|\zeta_0|_{H^{N+\frac{1}{2}}_{\mathrm{bo}}} + \ve \Big{|} \frac{\sqrt{\mu} |\mathrm{D}|}{\tanh(\sqrt{\mu}|\mathrm{D}|)} \zeta_0^2  \Big{|}_{H^{N-\frac{1}{2}}}
				\\\notag
				& 
				\lesssim 	|\zeta_0|_{H^{N+\frac{1}{2}}_{\mathrm{bo}}} +
				\ve |\zeta_0|_{H^{N-\frac{1}{2}}}^2 + \ve |\partial_x (\zeta_0^2)|_{H^{N-\frac{1}{2}}}
				\\ \label{est v0 wb}
				& \leq |\zeta_0|_{H^{N+1}_{\mathrm{bo}}}^2,
			\end{align}
			where we used that $\ve \leq \mathrm{bo}^{-\frac{1}{2}}$. Moreover, there is a time $T_1>0$ and a unique solution\footnote{See Remark \ref{Remark wp}.} 	$\zeta^{\text{\tiny wB}} \in C\big{(}[0,\ve^{-1}T_1] \: : \: H^{N+1}_{\mathrm{bo}}(\R)  \big{)}$ associated to $\zeta_0$ solving the equation
			\begin{equation*}
				\partial_t \zeta^{\text{\tiny wB}}
				+
				\sqrt{	\mathrm{k}}(\mathrm{D})\partial_x \zeta^{\text{\tiny wB}}
				+
				\frac{3\ve }{2}\zeta^{\text{\tiny wB}} \partial_x\zeta^{\text{\tiny wB}} =0
			\end{equation*}
			and satisfying
			\begin{equation}\label{Bound sol wBo}
				\sup\limits_{t\in[0,\ve^{-1}T]}|\zeta^{\text{\tiny{wB}}}|_{H^{N+1}_{\mathrm{bo}}} \leq |\zeta_0|_{H^{N+1}_{\mathrm{bo}}}.
			\end{equation}
			Moreover, we can define $v^{\text{\tiny wB}} \in C\big{(}[0,\ve^{-1}T_1] \: : \: H^{N-\frac{1}{2}}(\R)  \big{)}$  by \eqref{v app} and $u^{\text{\tiny wB}} = \mathrm{t}(\mathrm{D})v^{\text{\tiny wB}}$ in the same space. Then  by using estimate \eqref{est sqrt K}  we deduce from \eqref{Bound sol wBo} that
			\begin{equation*}
				\begin{cases}
					\partial_t \zeta^{\text{\tiny wB}}
					+
					 \partial_{ x} u^{\text{\tiny wB}}  +\ve\mathrm{t}(\mathrm{D}) \partial_{x}(\zeta^{\text{\tiny wB}}u^{\text{\tiny wB}}) = \varepsilon( \varepsilon + \sqrt{\mu} + \mathrm{bo}^{-1} )R
					\\
					\partial_t u^{\text{\tiny wB}}
					+
					\sqrt{	\mathrm{k}}(\mathrm{D})\partial_x  \zeta^{\text{\tiny wB}}
					+
					\frac{\ve}{2} \mathrm{t}(\mathrm{D}) \partial_{{x}}\big{(}u^{\text{\tiny wB}}\big{)}^2 = \varepsilon( \varepsilon + \sqrt{\mu} + \mathrm{bo}^{-1} )R,
				\end{cases}
			\end{equation*}
			for $R$ satisfying \eqref{est R} and concludes the proof of this step.  \\

			\noindent
			\underline{Step 3.} For the consistency result in the case of the Benjamin equation we  use \eqref{precise est sqrt K}:
			\begin{equation*}
				\Big{|}	\big{(}\sqrt{k}(\mathrm{D}) -  (1- \frac{\gamma}{2} \sqrt{\mu}|\mathrm{D}| - \frac{\partial_x^2}{2\mathrm{bo}})\big{)}f\Big{|}_{H^s} \lesssim (\mu + \frac{\sqrt{\mu}}{\mathrm{bo}}) |\partial_x^2 f|_{H^{s+1}}.
			\end{equation*}\\

				\noindent
			\underline{Step 4.} For the consistency result in the case of the BO equation we also  neglect the surface tension effects.
			%
			%
			%
			%
			%
			%
			%
			%
			%

	\end{proof}

	\section{Proof of Theorem \ref{Justification} and Corollary \ref{Cor ilw}}\label{Sec just}

	\begin{proof}[Proof of Theorem \ref{Justification}]

	First, we let $N\geq 8$ and take initial data $(\zeta_0 ,  \psi_0)$ satisfying the assumptions of Theorem \ref{Thm 1}. Then using \eqref{Est on Sol}, we can define the solutions of the internal  water waves equations \eqref{IWW - v} with variables 
	$$(\zeta, v) \in C([0,\ve^{-1}T_1] ; H^{N+1}_{\mathrm{bo}}(\R)\times H^{N-\frac{1}{2}}(\R)),$$
	from the data $(\zeta_0,\partial_x\psi_0) \in H^{N+1}_{\mathrm{bo}}(\R)\times H^{N-\frac{1}{2}}(\R)$ for some $T_1>0$. Next, we use Theorem \ref{Thm: Consistency} and the matrix notation \eqref{M 1} to say that on the same time interval the functions 
	$$u = \mathrm{t}(\mathrm{D})v \quad \text{and} \quad \mathbf{U} = (\zeta, u)^T,$$ 
	solves
	\begin{equation*}
		\partial_t \mathbf{U} +  \mathcal{M}(\mathbf{U}) \mathbf{U} 
		=
		\varepsilon \sqrt{\mu} (R,R)^T,
	\end{equation*}
	for any $t \in [0, \ve^{-1}T_1]$ and  where the rest satisfies
	\begin{equation}\label{est on rest}
		|R|_{H^{N-5}}\leq C,
	\end{equation}
	for  $C = C(\mathcal{E}^N(\mathbf{U}_0), {h^{-1}_{\min}}, \mathfrak{d}(\mathbf{U}_0)^{-1})>0$ nondecreasing function of its argument. We will now use this to prove estimates \eqref{CV BOs}, \eqref{CV wB},  \eqref{CV B}, and \eqref{CV BO} in four separate steps. \\

	\noindent
	\underline{Step 1.} \textit{Proof of estimate \eqref{CV BOs}}. We take the same data $(\zeta_0,v_0)$ with $v_0 = \partial_x \psi_0$, we by Plancherel's identity we observe that
	\begin{equation*}
		|u_0|_{H^{N-\frac{1}{2}}}^2+\sqrt{\mu}|u_0|_{\mathring{H}^{N}}^2 = |\mathrm{t}(\mathrm{D})v_0|_{H^{N-\frac{1}{2}}}^2+\sqrt{\mu}||\mathrm{D}|^{\frac{1}{2}}\mathrm{t}(\mathrm{D})v_0|_{H^{N-\frac{1}{2}}}^2 \leq |v_0|_{H^{N-\frac{1}{2}}}^2,
	\end{equation*}
	and \eqref{Est on Sol} implies $(\zeta_0, u_0) \in X_{\mu,\mathrm{bo}}^{N-\frac{1}{2}}(\R)$. We may therefore apply Theorem \ref{W-P System BO} to deduce the existence of $T_2>0$ such that 
	$$\mathbf{U}^{\text{\tiny   wBs} } = (\zeta^{\text{\tiny   wBs} }, u^{\text{\tiny   wBs} } ) \in C([0,\ve^{-1}T_2] ; X_{\mu,\mathrm{bo}}^{N-\frac{1}{2}}(\R)),$$
	solves system \eqref{Eq 1 M}:
	\begin{equation*} 
		\partial_t \mathbf{U}^{\text{\tiny   wBs} } +  \mathcal{M}( \mathbf{U}^{\text{\tiny  wBs} })\mathbf{U}^{\text{\tiny   wBs} }=\mathbf{0},
	\end{equation*}
	for any $t \in [0,\ve^{-1}T_2]$. We may now take the difference between the two solutions
	$$\mathbf{W} = (\eta, w)^T = \mathbf{U} - \mathbf{U}^{\text{\tiny   wBs} },$$
	to find that
	\begin{equation*} 
		\partial_t \mathbf{W} +   \mathcal{M}(\mathbf{U}) \mathbf{W}=  \mathbf{F} + \varepsilon \sqrt{\mu} (R,R)^T, 
	\end{equation*}	
	for any $t \in [0,\ve^{-1}\min\{T_1,T_2\}]$ where $\mathbf{F}$ is defined by  \eqref{F: source term}. Then using the estimate \eqref{Energy 2} with  $N-6>\frac{3}{2}$ and adding the rest term we find that
	\begin{align*}
		\frac{\mathrm{d}}{\mathrm{d}t} \tilde{E}_{\text{\tiny{wBs}}}^{N-6}(\mathbf{W}) 
		& \leq \varepsilon\sqrt{\mu} 
		|\big(\mathcal{Q}( \mathbf{U}^{\text{\tiny   wBs}} ) \Lambda^{N-6}(R,R)^T,  \Lambda^{N-6} \mathbf{W}\big)_{L^2}| +  \tilde{K}(N-5)  \tilde{E}_{\text{\tiny{wBs}}}^{N-6}(\mathbf{W}).
	\end{align*}
	where the energy is defined by \eqref{tilde Energy s}.	Furthermore, by definition of $\mathcal{Q}(\mathbf{U}_1 )$, the Hölder inequality, the Sobolev embedding,  \eqref{equiv 2}, and \eqref{est on rest} we obtain  
	\begin{align*}
		\frac{\mathrm{d}}{\mathrm{d}t} \tilde{E}_{\text{\tiny{wBs}}}^{N-6}(\mathbf{W}) 
		& \leq 
		\varepsilon \sqrt{\mu} |R|_{H^{N-5}}(\tilde{E}_{\text{\tiny{wBs}}}^{N-6}(\mathbf{W}))^{\frac{1}{2}} +
		\ve \tilde{K}(N-5)  \tilde{E}_{\text{\tiny{wBs}}}^{N-6}(\mathbf{W})
		\\ 
		& \leq 
		\varepsilon \sqrt{\mu} 
		C (\tilde{E}_{\text{\tiny{wBs}}}^{N-6}(\mathbf{W}))^{\frac{1}{2}} +
		\ve \tilde{K}(N-5)  \tilde{E}_{\text{\tiny{wBs}}}^{N-6}(\mathbf{W})
		.
	\end{align*}
	Then Grönwall's inequality,  \eqref{equiv 2}, \eqref{Bound sol Full disp B system}, and \eqref{Est on Sol}  yields
	\begin{equation*} \label{convergence est in proof}
		|(\eta, w)|_{X_{\mu,\mathrm{bo}}^{N-6}} \leq 
		 \varepsilon \sqrt{\mu} t
		  C e^{\ve \tilde{K}(N-5)t}.
	\end{equation*}
	Finally, we observe that
	$$\ve \tilde{K}(N-5)t\leq C \min\{T_1,T_2\},$$
	and from the rough estimate
	$$| v-v^{\text{\tiny wBs}}|_{H^{N-7}} = |\mathrm{t}(\mathrm{D})^{-1}w|_{H^{N-7}} \leq  |w|_{H^{N-6}},$$ 
	that there holds,
	\begin{align*}
		| (\zeta-\zeta^{\text{\tiny wBs}}, v-v^{\text{\tiny wBs}}) |_{L^{\infty}([0,t]; H^{N-7}\times H^{N-7})} 
		& \leq  |(\eta, w)|_{L^{\infty}([0,t]: X^{N-6}_{\mu, \mathrm{bo}})} 
		\\
		& \leq   \varepsilon  \sqrt{\mu} t C,
	\end{align*}
	for all $t \in [0,\ve^{-1} \min\{T_1,T_2\}]$. This completes the proof of estimate \eqref{CV BOs}.\\

	\noindent
	\underline{Step 2.} \textit{Proof of estimate \eqref{CV wB}}.
	For $\zeta_0\in H^{N+1}_{\mathrm{bo}}(\R))$ there  exist a time  $T_3>0$ such that 
	$$\zeta^{\text{\tiny   wB} }  \in C([0,\ve^{-1}T_3] ;H^{N+1}_{\mathrm{bo}}(\R)),$$
	solves \eqref{Full disp Benjamin} and from it we define 
	\begin{equation*}
		v^{\text{\tiny wB}} 
		=
		\mathrm{t}^{-1}(\mathrm{D})u^{\text{\tiny wB}} 
		=
		\mathrm{t}^{-1}(\mathrm{D})
		\big{(}\sqrt{	\mathrm{k}}(\mathrm{D})\zeta^{\text{\tiny wB}}  - \frac{\ve}{4} (\zeta^{\text{\tiny wB}})^2\big{)}.
	\end{equation*}
	Moreover, by estimate  \eqref{est v0 wb} and \eqref{Bound sol wBo}  we have that
	\begin{equation*}
		\sup\limits_{t\in[0,\ve^{-1}T_3]}|(\zeta^{\text{\tiny   wB} }, v^{\text{\tiny wB}})|_{H^{N+1}_{\mathrm{bo}}\times H^{N-\frac{1}{2}}} \leq C(|\zeta_0|_{H^{N+1}_{\mathrm{bo}}}).
	\end{equation*}
	Then if we define $\mathbf{U}^{\text{\tiny   wB} } = (\zeta^{\text{\tiny   wB} }, u^{\text{\tiny wB}})^T$, where $u^{\text{\tiny wB}} = \mathrm{t}(\mathrm{D})v^{\text{\tiny wB}}$, and then use Theorem \ref{Thm: Consistency} and argue as above to find that
	\begin{align*}
		|\mathbf{U}  - \mathbf{U}^{\text{\tiny wB}} |_{L^{\infty}([0,t]; H^{N-6}\times H^{N-6})}
		& \leq 
		|\mathbf{U}  - \mathbf{U}^{\text{\tiny wBs}} |_{L^{\infty}([0,t]; H^{N-6}\times H^{N-6})}
		\\ 
		& \hspace{0.5cm}
		+
		|\mathbf{U}^{\text{\tiny wBs}} - \mathbf{U}^{\text{\tiny wB}} |_{L^{\infty}([0,t]; H^{N-6}\times H^{N-6})}
		\\
		& \leq   \varepsilon(\ve +  \sqrt{\mu} +  \mathrm{bo}^{-1})  t C,
	\end{align*}
	for all $t \in [0,\ve^{-1}\min\{T_1,T_2,T_3\}]$.\\

	\noindent 
	\underline{Step 3.} \textit{Proof of estimate \eqref{CV B}}. There exist a time $T_4 > 0$ such that $\zeta^{\text{\tiny   B} }  \in C([0,\ve^{-1}T_4] ;H^{N+1}_{\mathrm{bo}}(\R))$ solves \eqref{Benjamin} that is bounded by its intial data. Moreover, by Theorem  \ref{Thm: Consistency} the solution satisfies
	\begin{equation*} 
		\partial_t \zeta^{\text{\tiny B}}
		+
		\sqrt{k}(\mathrm{D})\partial_x \zeta^{\text{\tiny B}}
		+
		\frac{3\ve }{2}\zeta^{\text{\tiny B}}  \partial_x\zeta^{\text{\tiny B}} =  (\mu +  \frac{\sqrt{\mu}}{\mathrm{bo}}) R.
	\end{equation*}
	Then if we define the difference $\tilde \eta =(\zeta^{\text{\tiny   wB} }- \zeta^{\text{\tiny   B} })$ it is straightforward to deduce that
	\begin{equation*}
		\frac{1}{2} \frac{\mathrm{d}}{\mathrm{d} t} |\tilde \eta|_{H^{N-5}}^2 \leq \ve (|\zeta^{\text{\tiny wB}}|_{H^{N-4}} + |\zeta^{\text{\tiny B}}|_{H^{N-4}})|\tilde\eta|_{H^{N-5}}^2 + (\mu  +  \frac{\sqrt{\mu}}{\mathrm{bo}})|R|_{H^{N-5}} |\tilde \eta|_{H^{N-5}}.
	\end{equation*}
	Then by Grönwall's inequality, the bound on the solution, and \eqref{est on rest} we find that
	\begin{equation*}
		 |\zeta^{\text{\tiny   wB} }- \zeta^{\text{\tiny   B} }|_{H^{N-5}} \leq (\mu + \frac{\sqrt{\mu}}{\mathrm{bo}}) tC,
	\end{equation*}
	and the result is  a direct consequence of the previous steps. \\ 

	\noindent
	\underline{Step 4.} \textit{Proof of estimate \eqref{CV BO}}.  The proof is the same as in the previous step.

\end{proof}

\begin{proof}[Proof of Corollary \ref{Cor ilw}]
	Since the well-posedness of the ILW equation is well-known for regular initial data we are simply left to prove a convergence estimate between its solutions and the ones of BO. However, comparing the linear terms we have that 
	\begin{equation*}
		\Big{|}|\mathrm{D}|(1- \coth(\sqrt{\mu^-}|\mathrm{D}|))f\Big{|}_{L^2} \leq \frac{1}{\sqrt{\mu^-}} |f|_{L^2},
	\end{equation*}
	using Lemma 2.3. given in \cite{chapouto2024deepwaterlimitintermediatelong}. Then we can deduce an estimate between the BO equation and the ILW  equation and apply Theorem \ref{Justification} to conclude.
\end{proof}

		\section*{Acknowledgements}
	
	This research was supported by a Trond Mohn Foundation grant. The material is also based on discussions with David Lannes at Institut Mittag-Leffler in Djursholm, Sweden during the program \lq\lq Order and Randomness in Partial Differential Equations\rq\rq \: in Fall, 2023 that was supported by the Swedish Research Council under grant no.~2016-06596. 
	
	I want to thank Vincent Duchêne for several important comments and for pointing out the full dispersion expansion of the interface operator. 	I also would like to thank Jean-Claude Saut for suggesting the problem, Louis Emerald for helpful remarks, and my advisor, Didier Pilod, for his comments on the introduction, support, and friendship.  Lastly, I would like to thank the anonymous referees for their careful reading and helpful comments.

	\appendix
	 \section{Properties of the Dirichlet-Neumann operators} 
	 In this section we will give several results on $\mathcal{G}^{\pm}_{\mu}[\ve\zeta]$ and $\mathcal{G}_{\mu}[\ve\zeta]$. Then we will use these results to share some details on the proof Proposition \ref{Prop Quasilinear system}. But as we will shortly see, the main quantities in \eqref{IWW} can be estimated in terms of the principal unknown $(\zeta, \psi)$ where the estimates are very similar to the ones in \cite{LannesTwoFluid13}. We will therefore give more details when there is a difference, and when the estimates are the same, we will simply refer to the results in \cite{LannesTwoFluid13}.

\subsection{Properties of $\mathcal{G}_{\mu}^+$} We start this section by giving a precise definition of the positive Dirichlet-Neumann operator $\mathcal{G}^+_{\mu}$.
\begin{Def}\label{Def G+}
	Let $t_0>\frac{1}{2}$, $\psi^+ \in \dot{H}_{\mu}^{\frac{3}{2}}(\R)$, and $\zeta\in H^{t_0 +2}(\R)$ be such that \eqref{non-cavitation} is satisfied. Let $\Phi^+$ be the unique solution in $\dot{H}^2(\Omega_t^+)$ of the boundary value problem 
	\begin{equation*} 
		\begin{cases}
			\Delta^{\mu}_{x,z} \Phi^+ = 0 \hspace{1.3cm}\qquad \text{in} \quad \Omega_t^+
			\\
			\Phi^+ = \psi^+ \hspace{2.54cm} \text{on} \quad  z = \varepsilon \zeta
			\\
			\partial_{z} \Phi^+  = 0 \hspace{2.5cm} \text{on} \quad  z = -1,
		\end{cases}
	\end{equation*}	
	then $\mathcal{G}^+_{\mu}[\ve \zeta]\psi^+ \in H^{\frac{1}{2}}(\R)$ is defined by
	\begin{equation*}
		\mathcal{G}^+_{\mu}[\varepsilon \zeta]\psi^+ = (\partial_z \Phi^+ 
		- 
		\mu \varepsilon 
		\partial_x \zeta  \partial_x \Phi^+){|_{z = \varepsilon \zeta}}.
	\end{equation*}
\end{Def}

\begin{remark}
	Under the provisions of Definition \ref{Def G+} and let $\phi^+ = \Phi\circ\Sigma^+$ where $\Sigma^+$ is the diffeomorphism from the fixed domain $\mathcal{S}^+$ onto $\Omega^+_t$ given in Definition \ref{diffeomorphism +-},  then we have that
	\begin{equation}\label{Pb phi+}
		\begin{cases}
			\nabla^{\mu}_{x,z} \cdot P(\Sigma^{+})\nabla_{x,z}^{\mu} \phi^+ = 0 \hspace{0.7cm}\qquad \text{in} \quad \mathcal{S}^+
			\\
			\phi^+ = \psi^+ \hspace{4.05cm} \text{on} \quad  z = 0,
			\\
			\partial_{n}^{P^+}\phi^+  = 0 \hspace{3.7cm} \text{on} \quad  z = -1,
		\end{cases}
	\end{equation}	
	where
	\begin{align*}
		\partial_{n}^{P^+} =  \mathbf{e}_{z} \cdot  P(\Sigma^+) \nabla^{\mu}_{x,z}.
	\end{align*}
	Moreover, we have an equivalent expression of $\mathcal{G}^{+}[\zeta]\psi^+$ given by
	\begin{equation}\label{G+}
		\mathcal{G}^+_{\mu}[\varepsilon \zeta]\psi^+  = \partial_n^{P^+}\phi^+|_{z=0}.
	\end{equation}
\end{remark}

\begin{remark}
	For $\ve = 0$ we have that $\mathcal{G}^+$ becomes
	\begin{equation}\label{G+[0]}
		\mathcal{G}^+_{\mu}[0]\psi^+ = \sqrt{\mu}|\mathrm{D}| \tanh(\sqrt{\mu}|\mathrm{D}|)\psi^+.
	\end{equation}
\end{remark}

Next, we will state several results on the Dirichlet-Neumann operator that are taken from \cite{WWP,LannesTwoFluid13}.
\begin{prop}\label{Prop G+ dual est}
	Let $t_0> \frac{1}{2}$, $s \in [0,t_0 + 1]$ and $\zeta \in H^{t_0+2}(\R)$ be such that \eqref{non-cavitation} is satisfied. Then we have the following properties:
	
	\begin{itemize}
		\item [1.]  For $\psi^+ \in \dot{H}^{s+{\frac{1}{2}}}_{\mu}(\R)$ there is a (variational) solution of \eqref{Pb phi+} satisfying the estimates
		\begin{equation}\label{est psi +}
			\sqrt{\mu}|\psi^+|_{\dot{H}_{\mu}^{s+\frac{1}{2}}} \leq M \| \Lambda^s\nabla_{x,z}^{\mu} \phi^+\|_{L^2(\mathcal{S}^+)},
		\end{equation}
		and
		\begin{equation}\label{est phi +}
			\| \Lambda^s\nabla_{x,z}^{\mu} \phi^+\|_{L^2(\mathcal{S}^+)}
			\leq \sqrt{\mu} M |\psi^+|_{\dot{H}_{\mu}^{s+\frac{1}{2}}}.
		\end{equation}
		\item [2.] We may extend definition \ref{Def G+} for
		\begin{align*}
			\mathcal{G}^+_{\mu}[\ve \zeta ] : \dot{H}_{\mu}^{s+\frac{1}{2}}(\R) \rightarrow H^{s-\frac{1}{2}}(\R),
		\end{align*}
		where for all $\psi_1,\psi_2 \in \dot{H}_{\mu}^{s+\frac{1}{2}}(\R)$, there holds,
		\begin{equation*}
			\int_{\R} \psi_1 \mathcal{G}^+_{\mu}[\varepsilon \zeta] \psi_2 \: \mathrm{d}x = \int_{\mathcal{S}^+} \nabla_{x,z}^{\mu} \phi_1^+ \cdot P^+(\Sigma^+)\nabla_{x,z}^{\mu} \phi_2^+ \: \mathrm{d}x \mathrm{d}z.
		\end{equation*}
		\item [3.] For all $\psi_1,\psi_2 \in \dot{H}_{\mu}^{s+\frac{1}{2}}(\R)$, there holds,
		\begin{align}\label{dual est G+ s}
			\big{(} \Lambda^s \psi_1,  \Lambda^s \mathcal{G}^+_{\mu}[\varepsilon \zeta] \psi_2\big{)}_{L^2} 
			\leq
			\mu M|\psi_1|_{\dot{H}_{\mu}^{s+\frac{1}{2}}}|\psi_2|_{\dot{H}_{\mu}^{s+\frac{1}{2}}}.
		\end{align}
		\item [4.] For $\psi^+ \in H^{s+\frac{1}{2}}_{\mu}(\R)$ the following estimates hold
		\begin{align}\label{Est G+}
			& \forall s \in [0,t_0+\frac{3}{2}],
			\quad
			|\mathcal{G}^+_{\mu}[\varepsilon \zeta]\psi^+|_{H^{s-\frac{1}{2}}} \leq \mu^{\frac{3}{4}} M |\psi|_{\dot{H}_{\mu}^{s+\frac{1}{2}}},
			\\ 
			& \forall s \in [0,t_0+1], \:
			\quad
			|\mathcal{G}^+_{\mu}[\varepsilon \zeta]\psi^+|_{H^{s-\frac{1}{2}}} \leq \mu M |\psi|_{\dot{H}_{\mu}^{s+1}}. \label{Est G+ mu}
		\end{align}
		\item [5.] For all $\psi_1 \in \dot{H}_{\mu}^{s-\frac{1}{2}}(\R)$ and $\psi_2 \in \dot{H}_{\mu}^{s+\frac{1}{2}}(\R)$, there holds
		\begin{equation}\label{Commutator G+}
			\big{(}\big [\Lambda^s,\mathcal{G}^+_{\mu}[\varepsilon \zeta]\big ]\psi_1, \Lambda^s\psi_2\big{)}_{L^2}
			\leq 
			\mu \ve 
			M|\psi_1|_{\dot{H}_{\mu}^{s-\frac{1}{2}}}|\psi_2|_{\dot{H}_{\mu}^{s+\frac{1}{2}}}.
		\end{equation}
		%
		%
		\item [6.] For all $V \in H^{t_0+1}(\R)$ and $\psi^+\in H_{\mu}^{\frac{1}{2}}(\R)$ there holds,
		\begin{equation}\label{Prop 3.30}
			\big{(}(V \partial_x \psi^+ ), \frac{1}{\mu}\mathcal{G}_{\mu}^+[\ve \zeta]\psi^+\big{)}_{L^2} \leq M |V|_{W^{1,\infty}}|\psi^+|_{\dot{H}^{\frac{1}{2}}_{\mu}}^2.
		\end{equation}
	\end{itemize}
\end{prop}
	
	\begin{remark}
		The regularity on $\zeta$ here is not optimal. Specifically, the dependence on $|\zeta|_{H^{t_0+2}}$ in the definition of $M$ can be lowered in the estimates above. However, we do not give these estimates since we will in other instances require more regularity on the free surface (as an example, see estimate \eqref{Symbolic G+}).
	\end{remark}
	\begin{remark}
		The last estimate \eqref{Commutator G+} is taken from \cite{LannesTwoFluid13} (see equation number $(2.24)$). However, $(2.24)$ has an $\ve$ missing due to a small typo and has been added here \cite{Personal_comm}. 
	\end{remark}

 The next result concerns the shape derivative of $\mathcal{G}_{\mu}^+$. The result is found in \cite{WWP,LannesTwoFluid13}.
\begin{prop}
	Let $t_0>\frac{1}{2}$, $s \in [0,t_0+1]$, and for any $\zeta \in H^{t_0+2}(\R)$ satisfying \eqref{non-cavitation} we have the following properties:
	%
	%
	%

	\begin{itemize}
	\item [1.] For $\psi^+ \in \dot{H}_{\mu}^{s+\frac{1}{2}}(\R)$  the shape derivative of $\mathcal{G}^+_{\mu}[\ve \zeta]$ at $\zeta\in H^{t_0+2}(\R)$ in the direction of $h\in  H^{t_0+2}(\R)$ is given by the formula
	\begin{equation}\label{shape derivative G+}
		\mathrm{d}_{\zeta}\mathcal{G}^+_{\mu}[\ve \zeta](h) \psi^+ = -\ve \mathcal{G}^+_{\mu}[\ve \zeta](h \underline{w}^+) - \ve \mu \partial_x(h\underline{V}^+). 
	\end{equation}
	\item [2.] For all $0\leq s \leq t_0 + 1$, $j \geq 1$ there holds
	\begin{equation}\label{directional derivative G+ mu ve}
		|\mathrm{d}^j_{\zeta} \mathcal{G}^+_{\mu}[\ve \zeta](h) \psi^+|_{H^{s-\frac{1}{2}}} \leq \ve^j\mu^{\frac{3}{4}}M \prod\limits_{m=1}^j |h_m|_{H^{\max\{s,t_0\}+1}} |\psi^+|_{\dot{H}_{\mu}^{s+\frac{1}{2}}}.
	\end{equation}

	\item [3.] For all $0\leq s \leq t_0 + \frac{1}{2}$, $j \geq 1$, and $\psi^+ \in \dot{H}^{s + \frac{1}{2}}_{\mu}(\R)$, there holds,
	\begin{equation}\label{directional derivative G mu1 ve}
		|\mathrm{d}^j_{\zeta} \mathcal{G}^+_{\mu}[\ve \zeta](h) \psi^+|_{H^{s-\frac{1}{2}}} \leq \ve^j\mu M \prod\limits_{m=1}^j |h_m|_{H^{\max\{s+ \frac{1}{2},t_0\}+1}} |\psi^+|_{\dot{H}_{\mu}^{s+1}}.
	\end{equation}
	\item [4.]  For all $0\leq s \leq t_0 + \frac{1}{2}$, $j \geq 1$, $\psi^+ \in \dot{H}^{\max\{s,t_0\}+1}_{\mu}(\R)$, there holds,
	\begin{equation}\label{6. Est dj G+ ve mu}
		|\mathrm{d}_{\zeta}^j \mathcal{G}^+_{\mu}[\ve \zeta](h) \psi^+|_{H^{s-\frac{1}{2}}} \leq \ve^j\mu M|h_k|_{H^{s+\frac{1}{2}}} \prod\limits_{m \neq k}^j |h_m|_{H^{\max\{s, t_0\} + \frac{3}{2}}} |\psi|_{\dot{H}_{\mu}^{\max\{s, t_0\}+1}}.
	\end{equation}
	\item [5.]  For all $0\leq s \leq t_0$, $j \geq 1$, $\psi^+ \in \dot{H}^{\max\{s+\frac{1}{2},t_0\}+1}_{\mu}(\R)$, there holds,
	\begin{equation}\label{7. Est dj G+ ve mu}
		|\mathrm{d}_{\zeta}^j \mathcal{G}^+_{\mu}[\ve \zeta](h) \psi^+|_{H^{s-\frac{1}{2}}} \leq \ve^j\mu M|h_k|_{H^{s+1}} \prod\limits_{m \neq k}^j |h_m|_{H^{\max\{s+ \frac{1}{2},t_0\} + \frac{3}{2}}} |\psi^+|_{\dot{H}_{\mu}^{\max\{s+\frac{1}{2},t_0\}+1}}.
	\end{equation}

	\item [6.] For all $0\leq s \leq t_0 + 1$, $j \geq 1$, $\psi_1, \psi_2 \in \dot{H}^{s + \frac{1}{2}}_{\mu}(\R)$, there holds,
	\begin{equation}\label{Est djG mu ve}
		\big{|}\big{(}\Lambda^s\mathrm{d}^j \mathcal{G}^+_{\mu}[\ve \zeta](\mathbf{h}) \psi_1,  \Lambda^s\psi_2\big{)}_{L^2}\big{|} \leq \ve^j \mu M \prod\limits_{m=1}^j |h_m|_{H^{\max\{s,t_0\}+1}} |\psi_1|_{\dot{H}^{s+\frac{1}{2}}} |\psi_2|_{\dot{H}^{s+\frac{1}{2}}_{\mu}}.
	\end{equation}
	\item [7.]For for all $0\leq s \leq t_0 + \frac{1}{2}$, $j \geq 1$, $\psi_1 \in  \dot{H}^{\max\{s,t_0\}+1}_{\mu}(\R) ,\psi_2 \in \dot{H}^{s + \frac{1}{2}}_{\mu}(\R)$, there holds,
	\begin{equation}\label{Est djG+ neq}
		\big{|}\big{(}\Lambda^s\mathrm{d}^j \mathcal{G}^+_{\mu}[\ve \zeta](\mathbf{h}) \psi_1,  \Lambda^s\psi_2\big{)}_{L^2}\big{|} \leq \ve^j \mu M |h_l|_{H^{s+\frac{1}{2}}}\prod\limits_{m\neq l}^j |h_m|_{H^{\max\{s,t_0\}+\frac{3}{2}}} |\psi_1|_{\dot{H}^{\max\{s,t_0\}+1}} |\psi_2|_{\dot{H}^{s+\frac{1}{2}}_{\mu}}.
	\end{equation}
	\end{itemize}
\end{prop}

\subsection{Properties of $\mathcal{G}_{\mu}^-$}  In this section will define and give the main properties of the negative Dirichlet-Neumann operator.

\begin{Def}\label{Def G-} Let $t_0\geq 1$, $\psi\in \mathring{H}^{\frac{3}{2}}(\R)$, and $\zeta\in H^{t_0 +2}(\R)$ be such that \eqref{non-cavitation} is satisfied. Let $\Phi^-$ be the unique solution in $\dot{H}^2(\Omega_t^-)$ of the boundary value problem 
	\begin{equation*} 
		\begin{cases}
			\Delta_{x,z}^{\mu} \Phi^- = 0 \hspace{1.3cm}\qquad \text{in} \quad \Omega_t^-
			\\
			\Phi^- = \psi \hspace{2.8cm} \text{on} \quad  z = \ve \zeta,
		\end{cases}
	\end{equation*}	
	then  $\mathcal{G}^-_{\mu}[\ve\zeta]\psi^- \in H^{\frac{1}{2}}(\R)$ is defined by 
	\begin{equation}\label{form D-N operator}
		\mathcal{G}^-_{\mu}[\ve\zeta]\psi^- = (\partial_z \Phi^- 
		- 
		\ve \mu
		\partial_x\zeta \partial_x \Phi^-){|_{z = \ve \zeta}}.
	\end{equation}
\end{Def}

\begin{remark}\label{Remark extension of G-}
	As noted in Remark $3.50$ $(2)$ in \cite{WWP}, we can define the negative Dirichlet-Neumann operator by formula \eqref{form D-N operator} (or formula \eqref{G-} below) for  $\psi^- \in  \mathring{H}^{s+\frac{1}{2}}(\R)$ if $t_0>\frac{1}{2}$ and $s \geq \max\{0,1-t_0\}$, where the restriction is a consequence of \eqref{Classical prod est}. We therefore put $t_0\geq 1$ for simplicity.
\end{remark}

\begin{remark}
	The scaling for $\mathcal{G}_{\mu}^-[\ve \zeta]$ is different from the one used in \cite{WWP}, where the author used the scaling natural for infinite depth. In that case one would change $\mu = 1$ and  $\ve$ by $\epsilon = \frac{a}{\lambda}$. In our case, the internal water wave model depends on both scales and is explained in detail in Subsection \ref{SubSec nondim}.
\end{remark}

\begin{remark}
	For $\ve \zeta  = 0$ we have that $\mathcal{G}^-_{\mu}$ becomes
	\begin{equation}\label{G-[0]}
		\mathcal{G}_{\mu}^-[0]\psi^- = -\sqrt{\mu}|\mathrm{D}| \psi^-.
	\end{equation}
\end{remark}

\begin{remark}
	Under the provisions of Definition \ref{Def G-} and let $\phi^- = \Phi\circ\Sigma^-$ where $\Sigma^+$ is the diffeomorphism from the fixed domain $\mathcal{S}^-$ onto $\Omega^-_t$ given in Definition \ref{diffeomorphism +-},  then we have that
	\begin{equation}\label{Elliptic pb G-}
		\begin{cases}
			\nabla_{x,z}^{\mu} \cdot P(\Sigma^{-})\nabla_{x,z}^{\mu} \phi^- = 0 \hspace{0.7cm}\qquad \text{in} \quad \mathcal{S}^-
			\\
			\phi^- = \psi \hspace{4.3cm} \text{on} \quad  z = 0,
		\end{cases}
	\end{equation}	
	and we have an equivalent expression of $\mathcal{G}^{-}_{\mu}[\ve\zeta]\psi^-$ given by
	\begin{equation}\label{G-}
		\mathcal{G}^-_{\mu}[\ve\zeta]\psi^-  =  \partial_n^{P^-}\phi^-|_{z=0},
	\end{equation}
	where $\partial_{n}^{P^-} =  \mathbf{e}_{z} \cdot  P(\Sigma^-) \nabla^{\mu}_{x,z}$.
	\end{remark}
 Here we use the results in \cite{WWP} adapted to the current scaling.

\begin{prop}\label{Prop G- dual est}
	Let $t_0\geq 1$, $s \in [0,t_0 + 1]$ and $\zeta \in H^{t_0+2}(\R)$ be such that \eqref{non-cavitation} is satisfied. We have the following properties:
	\begin{itemize}
		\item [1.]  For $\psi \in \mathring{H}^{s+{\frac{1}{2}}}(\R)$ there is a (variational) solution of \eqref{Pb phi+} satisfying the estimate 
		\begin{equation}\label{est psi -}
		\ 	|\psi^-|_{\mathring{H}^{s+\frac{1}{2}}} \leq \mu^{-\frac{1}{4}} M \| \Lambda^s\nabla_{x,z}^{\mu} \phi^-\|_{L^2(\mathcal{S}^-)},\color{black}
		\end{equation}
		and
		\begin{equation}\label{est phi -}
			\| \Lambda^s\nabla_{x,z}^{\mu}\phi^-\|_{L^2(\mathcal{S}^-)} 
			\leq
			\mu^{\frac{1}{4}} M |\psi^-|_{\mathring{H}^{s+\frac{1}{2}}}.
		\end{equation}
		\item [2.] By remark \ref{Remark extension of G-} we may extend Definition \ref{Def G-} for
		\begin{align*}
			\mathcal{G}^-_{\mu}[\ve\zeta]: \mathring{H}^{s+\frac{1}{2}}(\R) \rightarrow H^{s-\frac{1}{2}}(\R),
		\end{align*}
		where for all $\psi_1,\psi_2 \in \mathring{H}^{s+\frac{1}{2}}(\R)$, there holds,
		\begin{equation*}
			\int_{\R} \psi_1 \mathcal{G}^-_{\mu}[\ve\zeta] \psi_2 \: \mathrm{d}x =- \int_{\mathcal{S}^-} \nabla_{x,z}^{\mu}  \phi_1 \cdot P^-(\Sigma^-)\nabla_{x,z}^{\mu}\phi_2 \: \mathrm{d}x \mathrm{d}z.
		\end{equation*}
		\item [3.] For all $\psi_1,\psi_2 \in \mathring{H}^{s+\frac{1}{2}}(\R)$, there holds,
		\begin{align}\label{dual est G- s}
			\big{(} \Lambda^s \psi_1,  \Lambda^s \mathcal{G}^-_{\mu}[\ve\zeta] \psi_2\big{)}_{L^2} 
			\leq
			\sqrt{\mu}M|\psi_1|_{\mathring{H}^{s+\frac{1}{2}}}|\psi_2|_{\mathring{H}^{s+\frac{1}{2}}}.
		\end{align}
		\item [4.] For $\psi^- \in \mathring{H}^{s+\frac{1}{2}}(\R)$ the following estimates hold
		\begin{equation}\label{Est G-}
			|\mathcal{G}^-_{\mu}[\ve\zeta]\psi^-|_{H^{s-\frac{1}{2}}} \leq \sqrt{\mu} M |\psi^-|_{\mathring{H}^{s+\frac{1}{2}}}.
		\end{equation}

		\item [5.] For all $\psi_1\in \mathring{H}^{s-\frac{1}{2}}(\R)$ and $\psi_2 \in \mathring{H}^{s+\frac{1}{2}}(\R)$, there holds,

		\begin{equation}\label{Commutator G-}
			\big{(}\big{[}\Lambda^s,\mathcal{G}^-_{\mu}[\ve\zeta]\big{]}\psi_1, \Lambda^s\psi_2\big{)}_{L^2}
			\leq \ve  \sqrt{\mu}
			M|\psi_1|_{\mathring{H}^{s-\frac{1}{2}}}|\psi_2|_{\mathring{H}^{s+\frac{1}{2}}}.
		\end{equation}

		\item [6.] For all $V \in H^{t_0+1}(\R)$ and $\psi^-\in \mathring{H}^{\frac{1}{2}}(\R)$ there holds,	
		
		\begin{equation}\label{ext prop 3.30}
			\big{(}(V \partial_x \psi^- ), \frac{1}{\mu}\mathcal{G}_{\mu}^-[\ve \zeta]\psi^-\big{)}_{L^2} \leq \mu^{-\frac{1}{2}}M |V|_{W^{1,\infty}}|\psi^-|^2_{\mathring{H}^{\frac{1}{2}}}.
		\end{equation}
	\end{itemize}
\end{prop}

\begin{remark}
	If we compare the estimates for $\mathcal{G}_{\mu}^+$ with the ones above, we note that there is a $\sqrt{\mu}$ missing. However, this is a simple consequence of the functional setting where there is an additional gain in $\mu$ from the observation that
	\begin{equation*}
		\sqrt{\mu} |\xi| \mathrm{tanh}(\sqrt{\mu}|\xi|) \sim \mu \frac{|\xi|^2}{(1+ \sqrt{\mu}|\xi|)}.
	\end{equation*}
\end{remark}

\begin{remark}
	The inequality \eqref{Commutator G-} is a direct extension of inequality $(2.24)$ in \cite{LannesTwoFluid13}.	While the last inequality \eqref{ext prop 3.30} is the extension of Proposition $3.30$ in \cite{WWP} to infinite depth. The extension is straightforward and is explained on page $88$ of this book. 
\end{remark}

Lastly, we will need a shape derivative formula for $\mathcal{G}_{\mu}^-$ and estimates on higher order shape derivatives \cite{WWP}: 

\begin{prop}	Let $t_0\geq 1$, $s \in [0,t_0+1]$, and take $\zeta \in H^{t_0+2}(\R)$. Then we have the following properties: \\ 
	
	\begin{itemize}
	\item [1.]  For $\psi^- \in \mathring{H}^{s+\frac{1}{2}}(\R)$ the shape derivative of $\mathcal{G}^-_{\mu}[\ve \zeta]$ at $\zeta\in H^{t_0+2}(\R)$ in the direction of $h\in  H^{t_0+2}(\R)$ is given by the formula
	\begin{equation}\label{shape derivative G-}
		\mathrm{d}_{\zeta}\mathcal{G}^-_{\mu}[\ve \zeta](h) \psi^- = -\ve \mathcal{G}^-_{\mu}[\ve \zeta](h \underline{w}^-) - \ve \mu \partial_x(h\underline{V}^-). 
	\end{equation}

	\item [2.] For all $0\leq s \leq t_0 + 1$, $j \geq 1$ and $\psi_1, \psi_2 \in \mathring{H}^{s+\frac{1}{2}}(\R)$, there holds, 
	\begin{equation}\label{Est djG- mu ve}
		\big{|}\big{(}\Lambda^s\mathrm{d}^j \mathcal{G}^-_{\mu}[\ve \zeta](\mathbf{h}) \psi_1,  \Lambda^s\psi_2\big{)}_{L^2}\big{|} \leq \ve^j\sqrt{\mu}  M \prod\limits_{m=1}^j |h_m|_{H^{\max\{s,t_0\}+1}} |\psi_1|_{\mathring{H}^{s+\frac{1}{2}}} |\psi_2|_{\mathring{H}^{s+\frac{1}{2}}}.
	\end{equation}
	\item [3.]For all $0\leq s \leq t_0 + \frac{1}{2}$, $j \geq 1$, $\psi_1 \in\mathring{H}^{\max\{s,t_0\}+\frac{3}{2}}(\R), \psi_2 \in \mathring{H}^{s+\frac{1}{2}}(\R)$, there holds,
	\begin{equation}\label{Est djG- neq}
		\big{|}\big{(}\Lambda^s\mathrm{d}^j \mathcal{G}^-_{\mu}[\ve \zeta](\mathbf{h}) \psi_1,  \Lambda^s\psi_2\big{)}_{L^2}\big{|} \leq \ve^j\sqrt{\mu}  M  |h_l|_{H^{s+\frac{1}{2}}}\prod\limits_{m\neq l}^j |h_m|_{H^{\max\{s,t_0\}+\frac{3}{2}}} |\psi_1|_{\mathring{H}^{\max\{s,t_0\}+1}} |\psi_2|_{\mathring{H}^{s+\frac{1}{2}}}.
	\end{equation}
	\end{itemize}
\end{prop}

\subsection{Corollaries from the results in Section \ref{Properties of G}} \label{Appendix Corollaries} We will give an important generalization of Proposition \ref{Inverse 2}  where we prove that we can trade $\mathcal{G}_{\mu}^+$ with any operator $\mathrm{Op}(A): X \rightarrow H^{s-\frac{1}{2}}(\R)$ satisfying
\begin{equation}\label{Condition on A}
	\big{|} \big{(} \Lambda^s \mathrm{Op}(A) f_1, \Lambda^s f_2 \big{)}_{L^2}\big{|} \leq M M_A(f_1) \sqrt{\mu}|f_2|_{\mathring{H}^{s+\frac{1}{2}}},
\end{equation}
where $M_{A}(\psi)$ is some positive number defined by the norm on $X$. This can be seen from Step $1.$ in the proof, which is the only place where we use that $(\mathcal{G}_{\mu}^-)^{-1}$ is composed with $\mathcal{G}_{\mu}^+$. We have the following result. 

\begin{cor}\label{Cor inverse G-}
	 Let $t_0 \geq 1$, $s\in [0, t_0+1]$, and $\zeta \in H^{t_0+2}(\R)$ be such that \eqref{non-cavitation} is satisfied.  Then for $f \in \mathscr{S}(\R)$ and  $\mathrm{Op}(A)$ satisfying condition \eqref{Condition on A},  the mapping
	\begin{equation}\label{inverse G A}
		(\mathcal{G}^-_{\mu}[\ve \zeta])^{-1}  \mathrm{Op}(A)
		\:
		:
		\:
		\begin{cases}
			X & \rightarrow   \mathring{H}^{s + \frac{1}{2}}(\R)\\
			f & \mapsto  (\mathcal{G}^-_{\mu}[\ve \zeta])^{-1}  \mathrm{Op}(A)f
		\end{cases} 
	\end{equation}
	is well-defined and satisfies
	\begin{equation}\label{Inverse est A}
		| (\mathcal{G}^-_{\mu}[\ve \zeta])^{-1} \mathrm{Op}(A) f |_{\mathring{H}^{s+ \frac{1}{2}}}\leq M M_{A}(f).
	\end{equation}
	Moreover, we have the particular cases of $\mathrm{Op}(A)$:\\ 
	
	\begin{itemize}

	\item [1.] If $\mathrm{Op}(A) = \partial_x$, then $X = \mathring{H}^{s+\frac{1}{2}}(\R)$ and there holds,
	\begin{equation}\label{Est derivative}
		|(\mathcal{G}^{-}_{\mu}[\ve \zeta])^{-1} \partial_x f|_{\mathring{H}^{s+\frac{1}{2}}} \leq \mu^{-\frac{1}{2}}M |f|_{\mathring{H}^{s+\frac{1}{2}}}.
	\end{equation}

	\item [2.] If $\mathrm{Op}(A) = \mathrm{d}_{\zeta}^j \mathcal{G}^+_{\mu}[\ve \zeta]$ for $j\geq 1$, then $X = \dot{H}_{\mu}^{s+\frac{1}{2}}(\R)$ and for $\mathbf{h} = (h_1,...,h_j) \in H^{t_0+2}(\R)^j$ there holds,
	\begin{equation}\label{G- inv djG+}
		|(\mathcal{G}^{-}_{\mu}[\ve \zeta])^{-1} \mathrm{d}_{\zeta}^j \mathcal{G}^+_{\mu}[\ve \zeta]f|_{\mathring{H}^{s+\frac{1}{2}}} \leq  \ve^j \mu^{\frac{1}{4}} M \prod\limits_{m=1}^j |h_m|_{H^{\max\{s,t_0\}+1}}  |f|_{\dot{H}_{\mu}^{s+\frac{1}{2}}},
	\end{equation}
	and for $X=  \dot{H}_{\mu}^{\max\{s, t_0\}+1}(\R)$ there holds,
	\begin{equation}\label{2: G- inv djG+}
		|(\mathcal{G}^{-}_{\mu}[\ve \zeta])^{-1} \mathrm{d}_{\zeta}^j \mathcal{G}^+_{\mu}[\ve \zeta]f|_{\mathring{H}^{s+\frac{1}{2}}} \leq \ve^j\mu^{\frac{1}{4}} M|h_k|_{H^{s+\frac{1}{2}}} \prod\limits_{m \neq k}^j |h_m|_{H^{\max\{s, t_0\} + \frac{3}{2}}} |f|_{\dot{H}_{\mu}^{\max\{s, t_0\}+1}}.
	\end{equation}

	\item [3.] If $\mathrm{Op}(A) = \mathrm{d}^j \mathcal{G}^-_{\mu}[\ve \zeta]$ for $j\geq 1$, then $X = \mathring{H}^{s+\frac{1}{2}}(\R)$ and for $\mathbf{h} = (h_1,...,h_j) \in H^{t_0+2}(\R)^j$ there holds,
	\begin{equation}\label{G- inv djG-}
		|(\mathcal{G}^{-}_{\mu}[\ve \zeta])^{-1} \mathrm{d}^j_{\zeta} \mathcal{G}^-_{\mu}[\ve \zeta]f|_{\mathring{H}^{s+\frac{1}{2}}} 
		\leq
		\ve^j M \prod\limits_{m=1}^j |h_m|_{H^{\max\{s,t_0\}+1}}   |f|_{\mathring{H}_{\mu}^{s+\frac{1}{2}}},
	\end{equation}
	and for $X =  \mathring{H}^{\max\{s, t_0\}+1}(\R)$ there holds,
	\begin{equation}\label{2: G- inv djG-}
		|(\mathcal{G}^{-}_{\mu}[\ve \zeta])^{-1} \mathrm{d}^j_{\zeta} \mathcal{G}^+_{\mu}[\ve \zeta]f|_{\mathring{H}^{s+\frac{1}{2}}} 
		\leq 
		\ve^j M|h_k|_{H^{s+\frac{1}{2}}} \prod\limits_{m \neq k}^j |h_m|_{H^{\max\{s, t_0\} + \frac{3}{2}}} |f|_{\mathring{H}^{\max\{s, t_0\}+1}}.
	\end{equation}
	\end{itemize}
	
\end{cor}

\begin{proof}
	As explained above, we only need to consider the specific cases of $\mathrm{Op}(A)$, where \eqref{Inverse est A} is deduced as in the proof of Proposition \ref{Inverse 2}, Step 1., where we use  \eqref{trace 1} (instead of \eqref{trace 2}) in combination with \eqref{Condition on A}. In other words,  since we have \eqref{Inverse est A} at hand, we simply need to verify inequality \eqref{Condition on A} and identify the constant $M_{A}(f_1)$.
	\\
	
	\noindent
	\underline{Step 1.} For the proof of \eqref{Est derivative} we use that $\partial_x = - \mathcal{H}|\mathrm{D}|,$ where $\mathcal{H}$ is the Hilbert transform and then use Plancherel's identity together with Cauchy-Schwarz inequality to deduce the bound
	\begin{align*}
		\big{|} \big{(} \Lambda^s \partial_x f_1, \Lambda^s f_2 \big{)}_{L^2}\big{|}
		\leq M |f_1|_{\mathring{H}^{s+\frac{1}{2}}}|f_2|_{\mathring{H}^{s+\frac{1}{2}}}.
	\end{align*}
	%
	%
	%

	\noindent
	\underline{Step 2.} For the proof of \eqref{G- inv djG+}, we use estimate \eqref{Est djG mu ve} and\eqref{Basic est: B to D} to get the estimate
	\begin{align*}
		\big{|}\big{(}\Lambda^s\mathrm{d}^j \mathcal{G}^+_{\mu}[\ve \zeta](\mathbf{h}) f_1,  \Lambda^sf_2\big{)}_{L^2}\big{|} 
		& \leq
		\ve^j \mu M  |f_1|_{\dot{H}_{\mu}^{s+\frac{1}{2}}} |f_2|_{\dot{H}^{s+\frac{1}{2}}_{\mu}} \prod\limits_{m=1}^j |h_m|_{H^{\max\{s,t_0\}+1}}
		\\ 
		& 
		\leq \ve^j \mu^{\frac{3}{4}} M |f_1|_{\dot{H}_{\mu}^{s+\frac{1}{2}}} |f_2|_{\mathring{H}^{s+\frac{1}{2}}} \prod\limits_{m=1}^j |h_m|_{H^{\max\{s,t_0\}+1}}.
	\end{align*}
	While the proof of \eqref{2: G- inv djG+} is obtained by \eqref{Est djG+ neq}. \\

	\noindent
	\underline{Step 3.} This step is a direct consequence of estimates \eqref{Est djG- mu ve}
	and \eqref{Est djG- neq}, where estimate \eqref{Est djG- mu ve} implies
	\begin{align*}
		\big{|}\big{(}\Lambda^s\mathrm{d}^j \mathcal{G}^-_{\mu}[\ve \zeta](\mathbf{h}) f_1,  \Lambda^sf_2\big{)}_{L^2}\big{|} 
		& \leq
		\ve^j \sqrt{\mu} M  |f_1|_{\mathring{H}^{s+\frac{1}{2}}} |f_2|_{\mathring{H}^{s+\frac{1}{2}}} \prod\limits_{m=1}^j |h_m|_{H^{\max\{s,t_0\}+1}}.
	\end{align*}
\end{proof}

	\subsection{Basic definitions} We start this section by defining the main quantities involved in  \eqref{IWW} and relating them to the primary variables $(\zeta, \psi)$.  

\begin{cor}\label{cor definitions}
	Let $t_0 \geq 1$, $s \in [0, t_0 +1]$ , and $\zeta \in H^{t_0 + 2}(\R)$ satisfying \eqref{non-cavitation}. Moreover define $\psi^{\pm}$ as in Remark \ref{rmk id psi- psi+} and as the trace of $\phi^{\pm}$ satisfying \eqref{Transmission pb}. Then there holds:

	\begin{itemize}
		\item [1.]
		The tangential velocity is given by the mapping
		\begin{equation}\label{Def grad psi pm}
			\mathcal{V}^{\pm}_{\parallel}
			\:
			:
			\:
			\begin{cases}
				\dot{H}_{\mu}^{s + \frac{1}{2}}(\R) & \rightarrow  H^{s - \frac{1}{2}}(\R)
				\\
				\psi & \mapsto  \partial_x \psi^{\pm}
			\end{cases} 
		\end{equation}
		is well-defined and  satisfies
		\begin{equation}\label{Est V pm}
			s \in [0, t_0 + 1], \quad 	|\mathcal{V}^{\pm}_{\parallel}|_{H^{s}} \leq M |\partial_x\psi|_{H^s}.
		\end{equation}
		\item [2.]
		The horizontal component of the velocity is given by the mappings
		\begin{equation}\label{Def w pm}
			\underline{w}^{\pm}
			\:
			:
			\:
			\begin{cases}
				\dot{H}_{\mu}^{s + \frac{1}{2}}(\R) & \rightarrow  H^{s - \frac{1}{2}}(\R)
				\\
				\psi & \mapsto  \frac{\mathcal{G}^{\pm}_{\mu}[\ve \zeta] \psi^{\pm} + \ve \mu \partial_x \zeta \partial_x \psi^\pm}{1+ \ve^2 \mu (\partial_x \zeta)^2},
			\end{cases} 
		\end{equation}
		is well-defined and satisfies
		\begin{align}\label{W+ est low reg}
			& s \in [0, t_0 +1],
			\quad 
			|\underline{w}^\pm|_{H^{s-\frac{1}{2}}} \leq \mu^{\frac{3}{4}}M |\psi|_{\dot{H}^{s+\frac{1}{2}}_{\mu}},  
		\end{align}
		and
		\begin{align}\label{w pm mu}
			& s \in [0, t_0 + \frac{1}{2}], \quad 
			|\underline{w}^{\pm} |_{H^{s-\frac{1}{2}}} \leq  \mu M |\psi|_{\dot{H}^{s+1}_{\mu}}.
		\end{align}
		\item [3.]
		The vertical component of the velocity is given by the mappings
		\begin{equation}\label{Def V pm}
			{\underline{V}}^{\pm}
			\:
			:
			\:
			\begin{cases}
				\dot{H}_{\mu}^{s + \frac{1}{2}}(\R) & \rightarrow  H^{s - \frac{1}{2}}(\R)
				\\
				\psi & \mapsto \mathcal{V}^{\pm}_{\parallel} - \ve \underline{w}^\pm  \partial_x \zeta
			\end{cases} 
		\end{equation}
		is well-defined and satisfies
		\begin{equation}\label{V pm est}
			s \in [0, t_0 + \frac{1}{2}], \quad |\underline{V}^{\pm} |_{H^s} \leq M |\partial_x \psi|_{H^s}.
		\end{equation}
		%
		%
		%
		%
		%
		%
		
		\item [4.]  Let $r\geq 0$, and $\delta, N \in \N$ such that $r+\delta \leq N-1$. Then we have the estimate
		\begin{equation}\label{est psi}
			| \Lambda^{r+\frac{1}{2}} \partial_t^{\delta}\psi |_{\dot{H}_{\mu}^{\frac{1}{2}}}^2 \leq  M \mathcal{E}^{N}(\mathbf{U}).
		\end{equation}
		\item [5.] Recall the definition of  $	\mathfrak{a}$:
		\begin{equation*}
			\mathfrak{a}
			=  \Big{(}1
			+\ve \big{(} (\partial_t + \ve \underline{V}^{+}\partial_x)\underline{w}^{+} 
			-
			\gamma  (\partial_t + \ve \underline{V}^{-}\partial_x)\underline{w}^{-}
			\big{)}\Big{)}.
		\end{equation*}
		Then for $C>0$ being a nondecreasing function of its argument, we have that
		 \begin{align}\label{est on a}
		 	|\mathfrak{a}|_{L^{\infty}} \leq C\big{(}\mathcal{E}^{\lceil t_0+2\rceil}(\mathbf{U})\big{)},
		 \end{align}
	 	and
	 	\begin{align}\label{est on dta}
	 		|\partial_t \mathfrak{a}|_{L^{\infty}} \leq  \varepsilon C\big{(}\mathcal{E}^{\lceil t_0+3\rceil}(\mathbf{U})\big{)}.
	 	\end{align}
	\end{itemize}
	\color{black}
	
\end{cor}

\begin{proof}
	We prove each point in separate steps. \\

	\noindent
	\underline{Step 1.} For $\mathcal{V}^{+}_{\parallel}$ we will use formula $\psi^+ = (\mathcal{J}_{\mu})^{-1} \psi$ and Lemma  \ref{Prop op J} with estimate \eqref{Est J} to get
	\begin{align*}
		|\mathcal{V}^{+}_{\parallel}|_{H^{s}} 
		& \leq 
		M| (\mathcal{J}_{\mu}[\ve \zeta])^{-1} \psi |_{\dot{H}^{s+\frac{1}{2}}_{\mu}}
		+
		\mu^{\frac{1}{4}}
		M| (\mathcal{J}_{\mu}[\ve \zeta])^{-1} \psi |_{\dot{H}^{(s+\frac{1}{2})+\frac{1}{2}}_{\mu}}
		\\ 
		& \leq 
		M |\partial_x \psi |_{H^s}.
	\end{align*}
	For $\mathcal{V}^{-}_{\parallel}$ we use formula $\psi^{-}  = (\mathcal{G}_{\mu}^-)^{-1}\mathcal{G}_{\mu}^+ \psi^+$, then Proposition \ref{Inverse 2} with estimate \eqref{Inverse est 2.1}, and the above estimates to deduce:
	\begin{align*}
		|\mathcal{V}^{-}_{\parallel}|_{H^{s}} 
		= 
		|\partial_x\big{(}(\mathcal{G}_{\mu}[\ve \zeta]^-)^{-1}\mathcal{G}^+_{\mu}[\ve \zeta] \psi^+\big{)}|_{H^s}
		\leq M
		|\partial_x \psi^+|_{H^s}
		\leq M
		|  \partial_x \psi |_{H^s}.
	\end{align*}\\
	
	\noindent
	\underline{Step 2.} We consider the first estimate of $\underline{w}^+$, where we use estimate \eqref{Est G+}, the product estimates \eqref{Classical prod est}   and definition of $\psi^+$ with \eqref{Est J} to get:
	\begin{align*}
		|\underline{w}^{+}|_{H^{s-\frac{1}{2}}} 
		& \leq 
		|\mathcal{G}^+_{\mu}[\ve \zeta] \psi^+|_{H^{s-\frac{1}{2}}} + \mu \ve |\partial_x\zeta|_{H^{\max\{t_0,s-\frac{1}{2}\}}} |\partial_x \psi^+|_{H^{s-\frac{1}{2}}}
		\\ 
		& 
		\leq \mu^{\frac{3}{4}}M |\psi^+|_{\dot{H}^{s+\frac{1}{2}}_{\mu}}
		\\ 
		& 
		\leq \mu^{\frac{3}{4}}M |\psi|_{\dot{H}^{s+\frac{1}{2}}_{\mu}}.
	\end{align*}

	For the second estimate on $\underline{w}^+$, we argue similarly using \eqref{Est G+ mu} instead of \eqref{Est G+}.\\

	For $\underline{w}^-$ we first use the definition of $\psi^- =  (\mathcal{G}_{\mu}^-)^{-1}\mathcal{G}_{\mu}^+ \psi^+$ and apply the same estimates, combined with \eqref{Inverse est 2}:
	\begin{align*}
		|\underline{w}^-|_{H^{s-\frac{1}{2}}} 
		& \leq 
		|\mathcal{G}^+_{\mu}[\ve \zeta] \psi^+|_{H^{s-\frac{1}{2}}} 
		+
		\ve \mu |\partial_x\zeta|_{H^{\max\{t_0,s-\frac{1}{2}\}}} |\partial_x \big{(}(\mathcal{G}^-_{\mu}[\ve \zeta])^{-1}\mathcal{G}^+_{\mu}[\ve \zeta] \psi^+\big{)}|_{H^{s-\frac{1}{2}}}
		\\ 
		& 
		\leq \mu^{\frac{3}{4}}M |\psi|_{\dot{H}^{s+\frac{1}{2}}_{\mu}} 
		+
		\ve \mu  M |(\mathcal{G}^-_{\mu}[\ve \zeta])^{-1}\mathcal{G}^+_{\mu}[\ve \zeta] \psi^+|_{\mathring{H}^{s+\frac{1}{2}}}
		\\
		& 
		\leq \mu^{\frac{3}{4}}M |\psi|_{\dot{H}^{s+\frac{1}{2}}_{\mu}}.
	\end{align*}

	For the second estimate on $\underline{w}^-$, we argue as above, where we use \eqref{Est G+ mu} instead of \eqref{Est G+}. \\

	\noindent
	\underline{Step 3.} We now prove  \eqref{V pm est}. For the first part of $V^{\pm}$ we use estimates \eqref{Est V pm}, while the second part is estimated by the product estimate \eqref{Classical prod est} and then \eqref{W+ est low reg} to see that
	\begin{align*}
		|V^{\pm} |_{H^s} 
		& \leq 
		|\mathcal{V}_{\parallel}^{\pm} |_{H^s} + \ve |\underline{w}^{\pm} \partial_x \zeta|_{H^s}
		\\ 
		& \leq 
		M |\partial_x \psi|_{H^s} 
		+
		\ve |\partial_x \zeta |_{H^{\max\{t_0,s\}}} |\underline{w}^{\pm} |_{H^{(s+\frac{1}{2})-\frac{1}{2}}}
		\\
		& \leq 
		M( |\partial_x \psi|_{H^s} 
		+
		\mu^{\frac{3}{4}} |\psi|_{\dot{H}_{\mu}^{s+1}}),
		\\
		& \leq 
		M |\partial_x \psi|_{H^s}.
	\end{align*}
	\noindent
	\underline{Step 4.} By definition we relate $\psi$ with $\psi_{(\alpha)} = \partial_{x,t}^{\alpha} \psi - \varepsilon \underline{w}\zeta_{(\alpha)}$,  where $\underline{w} =  \underline{w}^+ - \gamma \underline{w}^-$, and use it to obtain the estimate
	\begin{align*}
		|\Lambda^{r+\frac{1}{2}} \partial_t^{\delta} \psi |_{\dot{H}_{\mu}^{\frac{1}{2}}} 
            & \leq
		\sum \limits_{\alpha \in \N^2 \: : \: |\alpha|\leq N-1}
		|\Lambda^{\frac{1}{2}}\partial_{x,t}^{\alpha} \psi|_{\dot{H}_{\mu}^{\frac{1}{2}}}
		\leq
		\sum \limits_{\alpha \in \N^2 \: : \: |\alpha|\leq N-1} |\Lambda^{\frac{1}{2}}\psi_{(\alpha)}|_{\dot{H}_{\mu}^{\frac{1}{2}}} + \ve |\Lambda^{\frac{1}{2}}(\underline{w}\zeta_{(\alpha)})|_{\dot{H}_{\mu}^{\frac{1}{2}}}.
		\color{black}
	\end{align*}
	Then  \eqref{Classical prod est} and \eqref{W+ est low reg} yields,
	\begin{align*}
		\ve |\underline{w}\zeta_{(\alpha)}|_{\dot{H}_{\mu}^{1}} 
		\leq 
		\ve\mu^{-\frac{1}{4}} |\underline{w}\zeta_{(\alpha)}|_{H^{1}} 
		\leq 
		\ve \sqrt{\mu} M |\partial_x \psi|_{H^2} |\zeta_{(\alpha)}|_{H^1}.
	\end{align*} 
	Combining these estimates implies
	\begin{equation*}
		| \partial_t^{\delta}\psi |_{\dot{H}_{\mu}^{(r+\frac{1}{2}) + \frac{1}{2}}}^2 \leq  M \mathcal{E}^{N}(\mathbf{U}).
	\end{equation*}\\
	\underline{Step 5.}
	We prove first \eqref{est on a}.  By Sobolev embedding with $s>t_0$, we have that
	\begin{align*}
		|\mathfrak{a}|_{L^{\infty}} \lesssim 1 
		+ 
		\varepsilon(|\partial_t \underline{w}^+|_{H^{s-\frac{1}{2}}} +  |\partial_t \underline{w}^-|_{H^{s-\frac{1}{2}}}   )
		+
		\varepsilon^2 (|\underline{V}^+|_{H^{s-\frac{1}{2}}} |\partial_x\underline{w}^+|_{H^{s-\frac{1}{2}}}  
		+
		|\underline{V}^-|_{H^{s-\frac{1}{2}}} |\partial_x\underline{w}^-|_{H^{s-\frac{1}{2}}}).
	\end{align*}
	For the estimate on $\underline{w}^+$ we use that it satisfies \eqref{directional derivative G+ mu ve} (with $t_0<s\leq t_0+1$), which in combination with the estimates in Step 2. yields
	\begin{align*}
		|\partial_t \underline{w}^+|_{H^{s-\frac{1}{2}}} \leq \varepsilon \mu^{\frac{3}{4}} M |\partial_t \zeta|_{H^{s+1}}|\psi^+|_{H^{s+\frac{1}{2}}_{\mu}} + \mu^{\frac{3}{4}} M |\partial_t\psi^+|_{H^{s+\frac{1}{2}}_{\mu}}.
	\end{align*}
	Moreover, for $\partial_t\psi^+ = \partial_t \big{(}(\mathcal{J}_{\mu})^{-1}\psi\big{)}$ we use estimate \eqref{shape derivative J} for the shape derivative of the inverse of $\mathcal{J}_{\mu}$ to deduce that
	\begin{equation*}
		|\partial_t\psi^+|_{H^{s+\frac{1}{2}}_{\mu}} \leq \ve M |\partial_t \zeta|_{H^{s+1}}|\psi|_{H^{s+\frac{1}{2}}_{\mu}} + |\partial_t\psi|_{H^{s+\frac{1}{2}}_{\mu}},
	\end{equation*}
	which in turn implies
	\begin{align}\label{est on w in pf}
		|\partial_t \underline{w}^+|_{H^{s-\frac{1}{2}}} \leq \varepsilon \mu^{\frac{3}{4}} M |\partial_t \zeta|_{H^{s+1}}|\psi|_{H^{s+\frac{1}{2}}_{\mu}} + \mu^{\frac{3}{4}} M |\partial_t\psi|_{H^{s+\frac{1}{2}}_{\mu}}.
	\end{align}
	The estimate on $\partial_t\underline{w}^-$, $\partial_x \underline{w}^{\pm}$ is deduced similarly, and using \eqref{V pm est} we find that
	\begin{align*}
		|\mathfrak{a}|_{L^{\infty}} 
		&  \lesssim 1 + \varepsilon \mu^{\frac{3}{4}} M \Big{(}|\partial_t \zeta|_{H^{s+1}}|\psi|_{H^{s+\frac{1}{2}}_{\mu}} +  |\partial_t\psi|_{H^{s+\frac{1}{2}}_{\mu}} \Big{)}
		\\
		& 
		\hspace{0.5cm} \notag
		+
		\varepsilon^2 \mu^{\frac{3}{4}} M \Big{(}| \partial_x\zeta|_{H^{s+1}}|\psi|_{H^{s+\frac{1}{2}}_{\mu}} +  |\psi|_{H^{s+\frac{3}{2}}_{\mu}} \Big{)}|\psi|_{H^{s+\frac{1}{2}}_{\mu}}.
	\end{align*}
	Then since the energy includes time derivatives we may employ \eqref{est psi} to find the bound:
	\begin{align*}
		|\mathfrak{a}|_{L^{\infty}} \lesssim 1 + \ve C\big{(}\mathcal{E}^{\lceil t_0+2\rceil}(\mathbf{U})\big{)}.
	\end{align*}
	On the other hand, the final estimate is proved similarly by accounting for one extra time derivative and that
	\begin{equation*}
		|\partial_t \underline{V}^{+}|_{H^{s-\frac{1}{2}}} \leq |\partial_t \big{(}(\mathcal{J}_{\mu})^{-1}\psi\big{)}|_{\dot{H}_{\mu}^{s+\frac{1}{2}}} + \varepsilon |\partial_t (\underline{w}^+ \partial_x \zeta)|_{H^{s-\frac{1}{2}}}   \leq C\big{(}\mathcal{E}^{\lceil t_0+2\rceil}(\mathbf{U})\big{)}.
	\end{equation*}
	using \eqref{est on w in pf} and \eqref{shape derivative J}. The same estimate holds for $\partial_t \underline{V}^{+}$ using similar arguments and we find that
	\begin{equation}\label{est on dtVpm}
		|\partial_t \underline{V}^{\pm}|_{H^{s-\frac{1}{2}}} \leq C\big{(}\mathcal{E}^{\lceil t_0+2\rceil}(\mathbf{U})\big{)},
	\end{equation}
	which  allows us to conclude the proof of the Corollary.
	\color{black}
\end{proof}

	\subsection{Proof of  Proposition \ref{Prop Quasilinear system}}\label{proof of quasi} We will in this section give the details on the proof of Proposition \ref{Prop Quasilinear system}. Since the strategy in the proof is exactly the same as in \cite{LannesTwoFluid13} and the main quantities involved satisfy the same estimates (see Corollary \ref{cor definitions}), we only prove the steps that are unique to the current regime. The first step is to derive linearization formulas. for the Dirichlet-Neumann operator $\mathcal{G}_{\mu}$.

		\subsubsection{Linearization formulas} The main step in the quasilinearisation of the internal water waves system \eqref{IWW} is to get linearization formulas for
		\begin{equation*}
			\mathcal{G}_{\mu}[\ve \zeta]= \mathcal{G}^+_{\mu}[\ve \zeta] \Big{(} 1 -   \gamma (\mathcal{G}^-_{\mu}[\ve \zeta])^{-1} \mathcal{G}^+_{\mu}[\ve \zeta] \Big{)}^{-1}\psi.
		\end{equation*}
		Here it will be important to get an explicit formula for the shape derivative of $\mathcal{G}_{\mu}$ with respect to $\zeta$ and track the dependencies in the small parameters. 
	\begin{prop}\label{linearization formulas}
		Let $T>0$, $t_0 \geq 1$ and $N \in \N$ be such that $	N \geq 5$. Furthermore, let $U = (\zeta, \psi)^T \in \mathscr{E}_{\mathrm{bo}, T}^{N, t_0}$ be such that \eqref{non-cavitation} is satisfied on $[0,T]$. For all $\alpha = (\alpha^1,\alpha^2) \in \N^2$, $\check{\alpha}^j = \alpha- \mathbf{e}_j$, with $1 \leq |\alpha| \leq N$ and define $\partial_{x,t}^{\alpha} = \partial_x^{\alpha_1} \partial_t^{\alpha_2}$, and let $\underline{w}^{\pm}$ be as defined in \eqref{Def w pm}. Moreover, define $\underline{w}$ by
		\begin{align*}
			\underline{w} &  =  \underline{w}^+ - \gamma \underline{w}^-, 
		\end{align*}
		and
		\begin{equation*}
			\zeta_{(\alpha)} = \partial_{x,t}^{\alpha} \zeta, \quad
			\psi_{(\alpha)} = \partial_{x,t}^{\alpha}\psi - \ve \underline{w} \partial_{x,t}^{\alpha}\zeta,\quad \psi_{\langle\check{\alpha}\rangle} = (\psi_{\check \alpha^1)},\psi_{(\check \alpha^2)}).
		\end{equation*}
		Then for all $1\leq |\alpha| \leq N$, one has
		\begin{align*}
			\text{if} \: \: \alpha <N :
			& \quad
			\frac{1}{\mu}\partial_{x,t}^{\alpha}(\mathcal{G}_{\mu}[\ve \zeta]\psi) = \frac{1}{\mu}\mathcal{G}_{\mu}[\ve \zeta]\psi_{(\alpha)} + \ve R_{(\alpha)},
			\\ 
			\text{if} \: \: \alpha \leq N:
			& \quad
			\frac{1}{\mu}\partial_{x,t}^{\alpha}(\mathcal{G}_{\mu}[\ve \zeta]\psi) = \frac{1}{\mu}\mathcal{G}_{\mu}[\ve \zeta]\psi_{(\alpha)} - \ve \mathcal{I}[U]\zeta_{(\alpha)} 
			+
			\frac{1}{\mu}\mathcal{G}_{\mu,(\alpha)}[\ve \zeta]\psi_{\langle\check{\alpha}\rangle} 
			+
			\ve R_{\alpha},
		\end{align*}
		where the linear operators $\mathcal{I}[U]$ and $\mathcal{G}_{\mu, (\alpha)}$ are given in Definition \ref{Def op}, while $R_{\alpha}$ is a function that satisfies the estimate
		\begin{equation}\label{Est on R}
			|R_{\alpha}(t)|^2_{H^1_{\mathrm{bo}}} \leq C \mathcal{E}^N(\mathbf{U}(t)),
		\end{equation}
		for some $C(M,\mathcal{E}^N(\mathbf{U}))>0$ nondecreasing function of its argument and for all $t \in [0,T]$.
		
	\end{prop}

	\begin{remark}
		The principal part of the linearization formula for $\mathcal{G}_{\mu}$ is given in terms of $\mathcal{G}_{\mu}\psi_{(\alpha)}$ and $\mathcal{I}[U]\zeta_{\alpha}$. While the additional term $\mathcal{G}_{\mu, (\alpha)}\psi_{\langle\check{\alpha}\rangle}$ is sub-principal that offers no difficulty in the proof, but is needed to deal with surface tension in the energy estimates (see the third point in Remark \ref{Remark on the energy}). 
	\end{remark}
	
	We will now give the main ingredients in proving the linearization formulas in Proposition \ref{linearization formulas}. These formulas involve the precise formulation of the directional derivative of $\mathcal{G}_{\mu}$, and we therefore give the definition.
	
	\begin{remark} Let $\psi \in H^{s+\frac{1}{2}}(\R)$ and $\zeta \in H^{t_0+2}(\R)$, then  $\zeta \mapsto \mathcal{G}_{\mu}[\ve \zeta] \psi$ is smooth and the directional derivative, in the direction of $h\in H^{t_0+2}(\R)$, is given by
		\begin{equation*}
			\mathrm{d}_{\zeta} \mathcal{G}_{\mu}[\ve \zeta](h)\psi 
			=
			\lim \limits_{\nu \rightarrow 0} \frac{\mathcal{G}_{\mu}[\ve\zeta+\nu h]\psi- \mathcal{G}_{\mu}[\ve \zeta]\psi  }{\nu}.
		\end{equation*}
	The smoothness follows from the shape analyticity of  $\mathcal{G}^{\pm}_{\mu}$ \cite{WWP}. 
		
	\end{remark}
	
	\begin{lemma}\label{Lemma shape derivative G}
		Let $t_0 \geq 1$ and $\zeta \in H^{t_0+2}(\R)$ be such that \eqref{non-cavitation} is satisfied.
		\begin{itemize}
			\item [1.] For all $\psi \in \dot{H}_{\mu}^{\frac{1}{2}}(\R)$, $h \in H^{t_0+2}(\R)$, and $U = (\zeta, \psi)^T$, one has
			\begin{equation}\label{d G}
				\mathrm{d}_{\zeta} \mathcal{G}_{\mu}[\ve \zeta](h) \psi = - \ve \mathcal{G}_{\mu}[\ve \zeta]\big{(}h(\underline{w}^+ - \gamma \underline{w}^-) \big{)} - \ve \mu \mathcal{I}[U]h.
			\end{equation}
			where $\mathcal{I}[U]h$ is defined by
			\begin{equation*}
				\mathcal{I}[U] h 
				=
				\partial_x ( h \underline{V}^+)+\gamma  \mathcal{G}_{\mu}[\ve \zeta]  (\mathcal{G}^-_{\mu}[\ve \zeta])^{-1}\partial_x	\Big{(} 
				h(\underline{V}^+-\underline{V}^-)\Big{)}.
			\end{equation*}
			\item [2.] For all $0\leq s \leq t_0 + 1$, $j \geq 1$, there holds,
			\begin{equation}\label{2. Est dj G ve mu3/4}
				|\mathrm{d}_{\zeta}^j \mathcal{G}_{\mu}[\ve \zeta](h) \psi|_{H^{s-\frac{1}{2}}} \leq \ve^j\mu^{\frac{3}{4}}M \prod\limits_{m=1}^j |h_m|_{H^{\max\{s,t_0\}+1}} |\psi|_{\dot{H}_{\mu}^{s+\frac{1}{2}}}.
			\end{equation}
			\item [3.] For all $0\leq s \leq t_0 + \frac{1}{2}$, $j \geq 1$, there holds,
			\begin{equation}\label{3. Est dj G ve mu}
				|\mathrm{d}_{\zeta}^j \mathcal{G}_{\mu}[\ve \zeta](h) \psi|_{H^{s-\frac{1}{2}}} \leq \ve^j\mu M \prod\limits_{m=1}^j |h_m|_{H^{\max\{s+ \frac{1}{2},t_0\}+1}} |\psi|_{\dot{H}_{\mu}^{s+1}}.
			\end{equation}
			\item [4.]  For all $0\leq s \leq t_0 + \frac{1}{2}$, $j \geq 1$, there holds,
			\begin{equation}\label{4. Est dj G ve mu}
				|\mathrm{d}_{\zeta}^j \mathcal{G}_{\mu}[\ve \zeta](h) \psi|_{H^{s-\frac{1}{2}}} \leq \ve^j\mu^{\frac{3}{4}} M|h_k|_{H^{s+\frac{1}{2}}} \prod\limits_{m \neq k}^j |h_m|_{H^{\max\{s, t_0\} + \frac{3}{2}}} |\psi|_{\dot{H}_{\mu}^{\max\{s, t_0\}+1}}.
			\end{equation}
			\item [5.]  For all $0\leq s \leq t_0$, $j \geq 1$, there holds,
			\begin{equation}\label{5. Est dj G ve mu}
				|\mathrm{d}_{\zeta}^j \mathcal{G}_{\mu}[\ve \zeta](h) \psi|_{H^{s-\frac{1}{2}}} \leq \ve^j\mu M|h_k|_{H^{s+1}} \prod\limits_{m \neq k}^j |h_m|_{H^{\max\{s+ \frac{1}{2},t_0\} + \frac{3}{2}}} |\psi|_{\dot{H}_{\mu}^{\max\{s+\frac{1}{2},t_0\}+1}}.
			\end{equation}
			%
			%

		\end{itemize}
	\end{lemma}
	
	\begin{proof}
		The proof is similar to the one in \cite{LannesTwoFluid13}, but we have to track the dependence in the small parameters for the current scaling and adapt them to a different functional setting. We prove each point separately. \\

		\noindent  
		\underline{Step 1.} We do the details here, where we note that it is important not to compose the inverse of $\mathcal{G}_{\mu}^+$ with $\mathcal{G}_{\mu}^-$ (see Remark \ref{order of comp}). With this in mind, we observe by direct computations (suppressing the argument in $\zeta$) that
		\begin{align*}
			\mathrm{d} (\mathcal{G}_{\mu}\circ \mathcal{J}_{\mu})(h) 
			& =
			\mathrm{d} \mathcal{G}_{\mu}(h )\circ \mathcal{J}_{\mu} + \mathcal{G}_{\mu}\circ \mathrm{d}\mathcal{J}_{\mu}(h). 
		\end{align*}
		Then use definition \eqref{Def G}  and compose the above identity  with $\mathcal{J}_{\mu}^{-1}$ from the right to obtain that
		\begin{equation}\label{dG pf 1}
			\mathrm{d} \mathcal{G}_{\mu}(h) 
			=
			\mathrm{d}\mathcal{G}_{\mu}^+(h) \circ (\mathcal{J}_{\mu})^{-1} 
			-
			\mathcal{G}_{\mu}\circ \mathrm{d}\mathcal{J}_{\mu}(h) \circ (\mathcal{J}_{\mu})^{-1}.
		\end{equation}
		Then we note by Proposition \ref{rmk id psi- psi+} that $(\mathcal{J}_{\mu})^{-1}\psi = \psi^+$
		allowing us to use formula \eqref{shape derivative G+} and express the first term by
		\begin{align*}
			\mathrm{d}\mathcal{G}_{\mu}^{+}(h)(\mathcal{J}_{\mu})^{-1}\psi 
			& =
			\mathrm{d}\mathcal{G}_{\mu}^{+}(h)\psi^+
			\\
			& =
			-\ve \mathcal{G}_{\mu}^+(h \underline{w}^+) - \ve \mu \partial_x(h\underline{V}^+).
		\end{align*}
		For the second term of \eqref{dG pf 1}, we first make the observation
		\begin{align*}
			(\mathcal{G}^{-}_{\mu})^{-1}\circ\mathrm{d}\mathcal{G}^{+}_{\mu}(h) 
			& =
			(\mathcal{G}^{-}_{\mu})^{-1}\circ\mathrm{d}(\mathcal{G}^{-}_{\mu} \circ\big{(} (\mathcal{G}^-_{\mu})^{-1}\circ\mathcal{G^+_{\mu}}\big{)})(h) 
			\\
			& = 
			(\mathcal{G}^{-}_{\mu})^{-1}\circ \Big{(}\mathrm{d} \mathcal{G}^{-}_{\mu}(h )\circ \big{(} (\mathcal{G}^-_{\mu})^{-1}\circ\mathcal{G}^+_{\mu}\big{)} - \gamma^{-1} \mathcal{G}^{-}_{\mu}\circ \mathrm{d}\mathcal{J}_{\mu}(h) \Big{)}
			\\ 
			& =
			(\mathcal{G}^{-}_{\mu})^{-1} \circ \mathrm{d} \mathcal{G}^{-}_{\mu}(h )\circ (\mathcal{G}^-_{\mu})^{-1}\circ\mathcal{G}^+_{\mu} 
			+
			\gamma^{-1}\mathrm{d}\mathcal{J}_{\mu}(h),
		\end{align*}
		so that
		\begin{equation*}
			\mathrm{d}\mathcal{J}_{\mu}(h)
			=
			\gamma 
			(\mathcal{G}^{-}_{\mu})^{-1}\circ\Big{(}\mathrm{d}\mathcal{G}^{+}_{\mu}(h) 
			-\mathrm{d} \mathcal{G}^{-}_{\mu}(h )\circ (\mathcal{G}^-_{\mu})^{-1}\circ\mathcal{G}^+_{\mu}\Big{)}.
		\end{equation*}
		Then by Remark \ref{rmk id psi- psi+} we have that $\psi^- =  (\mathcal{G}_{\mu}^-)^{-1} \circ \mathcal{G}_{\mu}^+\psi^+$, together with the shape derivative formulas \eqref{shape derivative G+} and \eqref{shape derivative G-}, we deduce that
		\begin{align*}
			\mathrm{d}\mathcal{J}_{\mu}(h) \circ (\mathcal{J}_{\mu})^{-1}\psi
			& =
			\mathrm{d}\mathcal{J}_{\mu}(h) \psi^+
			\\ 
			& =
			\gamma 
			(\mathcal{G}^{-}_{\mu})^{-1}\Big{(}\mathrm{d}\mathcal{G}^{+}_{\mu}(h) \psi^+
			-
			\mathrm{d} \mathcal{G}^{-}_{\mu}(h )\psi^-\Big{)}
			\\
			& =-
			\ve \gamma  (\mathcal{G}^-_{\mu})^{-1}	\Big{(} \mathcal{G}^+_{\mu}(h \underline{w}^+) - \mathcal{G}^-_{\mu}(h \underline{w}^-) \Big{)}
			-
			\gamma \ve \mu  (\mathcal{G}^-_{\mu})^{-1}	\Big{(} 
			\partial_x(hV^+)-  \partial_x(hV^-)\Big{)}.
		\end{align*}
		From \eqref{dG pf 1} and definition of $\mathcal{I}[U]h$ given by \eqref{Def I}, we may collect the above identities and factorize the leading terms together to find that
		\begin{align*}
			\mathrm{d} \mathcal{G}_{\mu}(h) \psi 
			& = 
			-\ve \mathcal{G}^+_{\mu}(h \underline{w}^+)
			-
			\ve \gamma  \mathcal{G}_{\mu}\circ(\mathcal{G}^-_{\mu})^{-1}	\Big{(} \mathcal{G}^+_{\mu}(h \underline{w}^+) - \mathcal{G}^-_{\mu}(h \underline{w}^-) \Big{)}
			\\ 
			& 
			\hspace{0.5cm}
			-
			\ve \mu  \Big{(} \partial_x(hV^+)  
			+ 
			\gamma  \mathcal{G}_{\mu} \circ (\mathcal{G}^-_{\mu})^{-1}	\Big{(} 
			\partial_x(hV^+)- \partial_x(hV^-)\Big{)}\Big{)}
			\\ 
			& =
			-\ve \Big{(}1
			+
			\gamma  \mathcal{G}_{\mu}\circ(\mathcal{G}^-_{\mu})^{-1}	\Big{)}\circ  \mathcal{G}^+_{\mu}(h \underline{w}^+) 
			+
			\ve  \gamma \mathcal{G}_{\mu}(h \underline{w}^-) 
			- \mu \ve \mathcal{I}(U)h.
		\end{align*}
		Clearly, the proof is a consequence of the following relation
		\begin{align}\label{Observation G}
			\Big{(}
			1
			+
			\gamma \mathcal{G}_{\mu}\circ(\mathcal{G}_{\mu}^-)^{-1}	\Big{)}\circ  \mathcal{G}^+_{\mu}
			= 
			\mathcal{G}_{\mu}.
		\end{align}
		and can be seen in the identity 
		\begin{align*}
			\mathcal{G}^+_{\mu}
			& =
			\mathcal{G}_{\mu} \circ \mathcal{J}_{\mu}
			\\ 
			& = 
			\mathcal{G}_{\mu}
			-
			\gamma \mathcal{G}_{\mu} \circ  (\mathcal{G}^-_{\mu})^{-1}\circ \mathcal{G}^+_{\mu}.
		\end{align*}

		\noindent
		\underline{Step 2.} We first use the previous step and \eqref{dG pf 1} to write
		\begin{align}\label{to be est}
			\mathrm{d} \mathcal{G}_{\mu}(h) \psi
			& =
			\mathrm{d}\mathcal{G}^+_{\mu}(h)\psi^+
			-
			\mathcal{G}_{\mu}\circ \mathrm{d}\mathcal{J}_{\mu}(h)\psi^+.
		\end{align}
		Then for the first term, we use \eqref{directional derivative G+ mu ve} and Proposition \eqref{Prop op J} with estimate \eqref{Est J} to get that
		\begin{align*}
			|\mathrm{d}^j\mathcal{G}^+_{\mu}(h)\psi^+|_{H^{s-\frac{1}{2}}}
			& \leq 
			\ve^j\mu^{\frac{3}{4}}M \prod\limits_{m=1}^j |h_m|_{H^{\max\{s,t_0\}+1}} |(\mathcal{J}_{\mu})^{-1} \psi|_{\dot{H}_{\mu}^{s+\frac{1}{2}}},
			\\ 
			& \leq 
			\ve^j\mu^{\frac{3}{4}}M \prod\limits_{m=1}^j |h_m|_{H^{\max\{s,t_0\}+1}} | \psi|_{\dot{H}_{\mu}^{s+\frac{1}{2}}}.
		\end{align*}
		For the second term, we first let $j=1$  and observe by direct calculations (or see previous step)  that
		\begin{align*}
			|\mathrm{d}\mathcal{J}_{\mu} \psi^+|_{\mathring{H}^{s+\frac{1}{2}}} 
			&  \leq
			|(\mathcal{G}_{\mu}^{-})^{-1}\mathrm{d}\mathcal{G}_{\mu}^{+}(h_1) \psi^+|_{\mathring{H}^{s+\frac{1}{2}}} 
			+
			|(\mathcal{G}_{\mu}^{-})^{-1}\mathrm{d} \mathcal{G}_{\mu}^{-}(h_1 )\psi^-|_{\mathring{H}^{s+\frac{1}{2}}} 
			\\ 
			& =: R_1 + R_2.
		\end{align*}
		Here we treat $R_1$ by \eqref{G- inv djG+} and then \eqref{Est J} to get that
		\begin{equation*}
			R_1 \leq \ve \mu^{\frac{1}{4}}M |h_1|_{H^{\max\{s,t_0\}+1}} | \psi|_{\dot{H}_{\mu}^{s+\frac{1}{2}}}.
		\end{equation*}
		For $R_2$ we use \eqref{G- inv djG-}, then the relation $\psi^- =  (\mathcal{G}^-_{\mu})^{-1}\mathcal{G}_{\mu}^+\psi^+$ together with estimates \eqref{Inverse est 2} and \eqref{Est J} to obtain, 
		\begin{align*}
			R_2 
			&  \leq
			\ve M |h_1|_{H^{\max\{s,t_0\}+1}} | \psi^-|_{\mathring{H}^{s+\frac{1}{2}}}
			\\ 
			& \leq \ve \mu^{\frac{1}{4}}M |h_1|_{H^{\max\{s,t_0\}+1}} | \psi|_{\dot{H}_{\mu}^{s+\frac{1}{2}}}.
		\end{align*}
		As a result,  for $j=1$, we have by \eqref{1 Est G} (for $0\leq s \leq t_0+1$) and the above estimates that
		\begin{align*}
			|\mathcal{G}_{\mu}\circ \mathrm{d}\mathcal{J}_{\mu}(h)\psi^+|_{H^{s-\frac{1}{2}}} 
			& \leq
			\mu^{\frac{3}{4}} M |\mathrm{d}\mathcal{J}_{\mu}(h)\psi^+|_{\dot{H}_{\mu}^{s+\frac{1}{2}}} 
			\\ 
			& \leq 
			\sqrt{\mu}M  |\mathrm{d}\mathcal{J}_{\mu}(h)\psi^+|_{\mathring{H}^{s+\frac{1}{2}}}
			\\ 
			& 
			\leq 
			\ve \mu^{\frac{3}{4}}M |h_1|_{H^{\max\{s,t_0\}+1}}  |\psi|_{\dot{H}_{\mu}^{s+\frac{1}{2}}}.
		\end{align*}
		The remaining cases follow recursively and can be proved by induction.\\

		\noindent
		\underline{Step 3$-$5.} The proof follows similarly, where for the first term in \eqref{to be est} we instead use \eqref{directional derivative G mu1 ve} or \eqref{6. Est dj G+ ve mu}, or \eqref{7. Est dj G+ ve mu}. While for the second term, we can gain precision in $\mu$ using \eqref{2 Est G}.
		\\ 

	\end{proof}

		\begin{remark}
		In Step $1.$ of the proof, we found a formula for the shape derivative of $\mathcal{J}_{\mu}$:
		\begin{equation*}
			\mathrm{d}_{\zeta}\mathcal{J}_{\mu}[\ve \zeta](h)
			=
			\gamma 
			(\mathcal{G}^{-}_{\mu}[\ve \zeta])^{-1}\circ\Big{(}\mathrm{d}_{\zeta}\mathcal{G}^{+}_{\mu}[\ve \zeta](h) 
			-\mathrm{d}_{\zeta} \mathcal{G}^{-}_{\mu}[\ve \zeta](h )\circ (\mathcal{G}^-_{\mu}[\ve \zeta])^{-1}\circ\mathcal{G^+_{\mu}[\ve \zeta]}\Big{)}.
		\end{equation*}
		Furthermore, in Step $2.$, we proved  the  estimate
		\begin{align}\label{shape derivative J}
			|\mathrm{d}_{\zeta}\mathcal{J}_{\mu}[\ve \zeta](\partial_x\zeta) \psi^+|_{\mathring{H}^{s+\frac{1}{2}}} 
			\leq \ve \mu^{\frac{1}{4}}M |\partial_x \zeta|_{H^{\max\{s,t_0\}+1}} | \psi|_{\dot{H}_{\mu}^{s+\frac{1}{2}}}.
		\end{align}
		From these expressions, we can deduce an estimate on the inverse:
		\begin{align}\label{shape derivative Jinv} 
			|\mathrm{d}_{\zeta}(\mathcal{J}_{\mu}[\ve \zeta])^{-1}(\partial_x\zeta) \psi^+ |_{\dot{H}_{\mu}^{s+\frac{1}{2}}} 
			& =
			|(\mathcal{J}_{\mu}[\ve \zeta])^{-1}\circ \mathrm{d}_{\zeta}\mathcal{J}_{\mu}[\ve \zeta](\partial_x\zeta) \circ (\mathcal{J}_{\mu}[\ve \zeta])^{-1} \psi|_{\dot{H}_{\mu}^{s+\frac{1}{2}}} 
			\\ 
			& \leq \ve M |\partial_x \zeta|_{H^{\max\{s,t_0\}+1}} | \psi|_{\dot{H}_{\mu}^{s+\frac{1}{2}}}.\notag
		\end{align}
		For an application of these estimates, see proof of \eqref{Symbolic G}.
	\end{remark}

	\begin{remark} In step $1$. in the proof we also have some convenient identities for $\mathcal{G}_{\mu}$. We saw that
		\begin{equation*}
			\Big{(}
			1
			+
			\gamma \mathcal{G}_{\mu}[\ve \zeta]\circ(\mathcal{G}^-_{\mu}[\ve \zeta])^{-1}	\Big{)}\circ  \mathcal{G}^+_{\mu}[\ve \zeta]
			= 
			\mathcal{G}_{\mu}[\ve \zeta].
		\end{equation*}
		Composing this identity with $(\mathcal{G}^{+}_{\mu}[\ve \zeta])^{-1}$ on the right, we get a quantity we will use later:
		\begin{equation}\label{id 1}
			\big(1+ \gamma \mathcal{G}_{\mu}[\ve \zeta]  (\mathcal{G}^-_{\mu}[\ve \zeta])^{-1}\big{)} = \mathcal{G}_{\mu}[\ve \zeta]  (\mathcal{G}^{+}_{\mu}[\ve \zeta])^{-1}
		\end{equation}
		Furthermore, we have that
		\begin{align*}
			-\gamma (\mathcal{G}^-_{\mu}[\ve \zeta])^{-1}\mathcal{G}_{\mu}[\ve \zeta]  
			& =
			-\gamma (\mathcal{G}^-_{\mu}[\ve \zeta])^{-1}\mathcal{G}^+_{\mu}[\ve \zeta]  (\mathcal{J}_{\mu}[\ve \zeta])^{-1}
			\\
			& = 
			 1 - (\mathcal{J}_{\mu}[\ve \zeta])^{-1},
		\end{align*}
		which implies
		\begin{equation}\label{id 2}
			(\mathcal{J}_{\mu}[\ve \zeta])^{-1} + \gamma  (\mathcal{G}^-_{\mu}[\ve \zeta])^{-1}\mathcal{G}_{\mu}[\ve \zeta] = 1 + 2 \gamma (\mathcal{G}^-_{\mu}[\ve \zeta])^{-1}\mathcal{G}_{\mu}[\ve \zeta].
		\end{equation}
		and 
		%
		%
		%
		\begin{equation}\label{id 3}
				\gamma (\mathcal{G}^-_{\mu}[\ve \zeta])^{-1}\mathcal{G}^+_{\mu}[\ve \zeta]\mathcal{J}_{\mu}[\ve \zeta]^{-1}  (\mathcal{G}^{+}_{\mu}[\ve \zeta])^{-1} = -\gamma\mathcal{J}_{\mu}[\ve \zeta]^{-1} (\mathcal{G}^-_{\mu}[\ve \zeta])^{-1}.
		\end{equation}
	\end{remark}

	\begin{proof}[Proof of Proposition  \ref{linearization formulas}] 
		The proof relies on the estimates provided by Lemma \ref{Lemma shape derivative G}. However, these estimates are the same as in \cite{LannesTwoFluid13}, and so we refer the reader to the proof of Proposition $6$ in this paper.
	\end{proof}

	To prove the main result of this section, we also need the following Lemma:
	
	\begin{lemma}\label{Lemma Gpm inv}
		Under the assumptions of Proposition \ref{linearization formulas}, one has for $|\alpha|\leq N-1$ that
		\begin{equation*}
			\mathcal{G}^{\pm}_{\mu}[\ve \zeta]\psi_{(\alpha)} = r_{\alpha}^{\pm},
		\end{equation*}
		In the case when $|\alpha| = N$, then
		\begin{equation*}
			\mathcal{G}^{\pm}_{\mu}[\ve \zeta] \psi^{\pm}_{(\alpha)} 
			=
			\mathcal{G}_{\mu}[\ve \zeta]\psi_{(\alpha)}  
			-
			\gamma \ve \mu  \mathcal{G}_{\mu}[\ve \zeta]  (\mathcal{G}^{\mp}_{\mu}[\ve \zeta])^{-1}\partial_x	\Big{(} 
			\zeta_{(\alpha)}(\underline{V}^+-\underline{V}^-)\Big{)} + r_{\alpha}^{\pm},
		\end{equation*}
		where the residual terms $r^{\pm}_{\alpha}$ satisfies
		\begin{equation*}
			|(\mathcal{G}^{\pm}_{\mu}[\ve \zeta])^{-1}r_{\alpha}^{\pm}|_{\dot{H}^{\frac{3}{2}}_{\mu}}^2 \leq C \mathcal{E}^N(\mathbf{U})(1+\gamma \ve^2 \sqrt{\mu} |\zeta|_{<N+\frac{1}{2}>}^2),
		\end{equation*} 
		for some $C(M,\mathcal{E}^N(\mathbf{U}))>0$ nondecreasing function of its argument  and 
		$$|\zeta|_{<N+\frac{1}{2}>} = \sum \limits_{\alpha \in \N^2, |\alpha|=N} |\partial_{x,t}^{\alpha}\zeta|_{\mathring{H}^{\frac{1}{2}}}.$$
	\end{lemma}
	\begin{remark}\label{Remark est G inv}
		Here it is crucial to use that $\mathring{H}^{s+\frac{1}{2}}(\R) \subset \dot{H}_{\mu}^{s+\frac{1}{2}}(\R)$ where we have the estimate
		\begin{equation*}
			|f|_{\dot{H}_{\mu}^{s+\frac{1}{2}}} \leq \mu^{-\frac{1}{4}}|f|_{\mathring{H}^{s+\frac{1}{2}}}.
		\end{equation*}
		To regain the precision in $\mu$, we note that is provided in Proposition \ref{Inverse 2} and Corollary \ref{Cor inverse G-}.
	\end{remark}
	\begin{proof}
		The proof is similar to the proof of Proposition $7$ in \cite{LannesTwoFluid13}, so we will just point out the differences. To that end, we focus on the case $|\alpha| = N$ and apply $\partial^{\alpha}_{x,t}$ to the relation $\mathcal{G}^{\pm} _{\mu}[\ve \zeta]= \mathcal{G}_{\mu}[\ve \zeta]$. Then using identities \eqref{shape derivative G+}, \eqref{shape derivative G-}, and \eqref{d G} (or Proposition \ref{linearization formulas} without sub-principal terms) to find that
		\begin{equation*}
			\mathcal{G}^{\pm}_{\mu}[\ve \zeta] \psi_{(\alpha)}^{\pm} - \ve \mu \partial_x(\zeta_{(\alpha)}\underline{V}^{\pm}) 
			=
			 \mathcal{G}_{\mu}[\ve \zeta]\psi_{(\alpha)} - \ve \mu \mathcal{I}[U]\zeta_{(\alpha)} + r_{\alpha}^{\pm},
		\end{equation*}
		where $r_{\alpha}^{\pm}$ is on the form
		\begin{equation*}
			\mathrm{d}_{\zeta}^j\mathcal{G}_{\mu}[\ve \zeta] (\partial^{l^1}_{x,t}\zeta, ..., \partial^{l^j}_{x,t}\zeta)\partial^{\delta}_{x,t}\psi -
			\mathrm{d}_{\zeta}^j\mathcal{G}_{\mu}[\ve \zeta] ^{\pm}(\partial^{l^1}_{x,t}\zeta, ..., \partial^{l^j}_{x,t}\zeta)\partial^{\delta}_{x,t}\psi^{\pm}, 
		\end{equation*}
		with $\sum \limits_{i=1}^j |l^j| + |\delta| = N$, $0\leq |\delta|\leq N-1$ and $|l^i|\leq N$. The proof of estimate is the same as in the Proposition $7$ in \cite{LannesTwoFluid13} after using Remark \ref{Remark est G inv}. Moreover, applying the definition of $\mathcal{I}[U]\zeta_{(\alpha)}$ we find that
		\begin{align*}
			\mathcal{G}^{+}_{\mu}[\ve \zeta] \psi_{(\alpha)}^+ & =  \mathcal{G}_{\mu}[\ve \zeta]\psi_{(\alpha)}  
			-
			\gamma \ve \mu  \mathcal{G}_{\mu}[\ve \zeta]  (\mathcal{G}^-_{\mu}[\ve \zeta])^{-1}\partial_x	\Big{(} 
			\zeta_{(\alpha)}(\underline{V}^+-\underline{V}^-)\Big{)}.
		\end{align*}
		Similarly, for $\mathcal{G}^{-}_{\mu}$ we observe that
		\begin{align*}
			\mathcal{G}^{-}_{\mu}[\ve \zeta] \psi_{(\alpha)}^- & =  \mathcal{G}_{\mu}[\ve \zeta]\psi_{(\alpha)}  
			-
			 \ve \mu\big(1+ \gamma \mathcal{G}_{\mu}[\ve \zeta]  (\mathcal{G}^-_{\mu}[\ve \zeta])^{-1}\big{)}\partial_x	\Big{(} 
			\zeta_{(\alpha)}(\underline{V}^+-\underline{V}^-)\Big{)}.
		\end{align*}
		Then we conclude by relation \eqref{id 1}.
		 
	\end{proof}

	We will now turn to the proof of Proposition \ref{Prop Quasilinear system}. However, since the quantities involved satisfy the same estimates as in \cite{LannesTwoFluid13}, we only point out the main differences (see also for a similar approach \cite{Iguchi_09}).

	\begin{proof}[Proof of Proposition \ref{Prop Quasilinear system}] Let $\psi^{\pm}$ and $\psi$ be defined as in the transmission problem \eqref{Transmission pb} throughout the proof. Then for the first equation, we simply apply $\partial_{x,t}^{\alpha}$ and conclude by Proposition \ref{linearization formulas}. 
		
	For the second equation, we need to prove that
	\begin{align}\notag
		\text{if} \: \: 1\leq|\alpha| <N :
		& \quad \hspace{2.3cm}
		\partial_t \psi_{(\alpha)} + \mathfrak{Ins}[\mathbf{U}]\zeta_{(\alpha)}= \ve S_{\alpha},
		\\ 
		\text{if} \qquad\: |\alpha| = N:
		& \quad \label{Quasi 2 2}
		\partial_t \psi_{(\alpha)} + \mathfrak{Ins}[\mathbf{U}]\zeta_{(\alpha)}-\ve \mathcal{I}[\mathbf{U}]^{\ast} \psi_{(\alpha)}
		= \ve S_{\alpha}.
	\end{align}
	We will only prove the most difficult case $|\alpha| = N$, but first need an observation. Let $\partial$ denote either the derivative with respect to $x$ or $t$. Then the following identity holds,
	\begin{align}\label{Claim in pf of quasi}
		\partial_t& \partial \psi +  \partial \zeta 
		-
		\frac{\ve}{\mu}(\underline{w}^+- \gamma \underline{w}^-)\partial_t \partial\zeta 
		& +\ve \underline{V}^+(\partial_x \partial \psi^+ - \ve \underline{w}^+ \partial_x \partial \zeta) 
		-
		\gamma \ve \underline{V}^-(\partial_x \partial \psi^- - \ve \underline{w}^- \partial_x \partial \zeta)\notag
		\\
		& =-\frac{1}{\mathrm{bo}}  \frac{1}{\ve \sqrt{\mu }} \partial \kappa(\ve \sqrt{\mu}\zeta).
 	\end{align}
	To prove the claim \eqref{Claim in pf of quasi}, we apply $\partial$ to the second equation \eqref{IWW} to find that
	\begin{equation}
		\partial_t\partial  \psi + \partial \zeta 
		+
		\ve\big{(}
		 (\partial_x   \psi^{+}) (\partial_x  \partial   \psi^{+}) 
		-
		\gamma (\partial_x \psi^{-})(\partial_x \partial    \psi^{-})
		\big{)}
		+
		\ve \partial \mathcal{N}[\ve \zeta,\psi^{\pm}]
		=- \frac{1}{\mathrm{bo}}\frac{1}{\ve \sqrt{\mu}} \partial \kappa(\ve \sqrt{\mu}\zeta),
	\end{equation}
	where
	\begin{align*}
	\ve \partial \mathcal{N}[\ve \zeta,\psi^{\pm}] 
	& =
	- \partial \Bigg{(}
		\frac{\ve}{2 \mu}(1+\ve^2 \mu (\partial_x\zeta)^2)
	\Big{(}
		(\underline{w}^+)^2 - \gamma (\underline{w}^-)^2
	\Big{)}
	\Bigg{)}
	\\ 
	& =
	-
	\ve^3  (\partial_x\zeta)  (\partial_x \partial \zeta)
	\Big{(}
	(\underline{w}^+)^2 - \gamma (\underline{w}^-)^2
	\Big{)}
	\\ 
	& 
	\hspace{0.5cm}
	-
	\frac{1}{ \mu}(1+\ve^2 \mu (\partial_x\zeta)^2)
	\Big{(}
	(\underline{w}^+)(\partial \underline{w}^+)  - \gamma (\underline{w}^-)(\partial \underline{w}^-)
	\Big{)}.
	\end{align*}	
	To conclude, we observe by definition that (or use Lemma $4.13$ \cite{WWP}):
	\begin{align*}
		\ve (\partial_x   \psi^{\pm}) & (\partial_x  \partial   \psi^{\pm}) 
		-
		\ve^3  (\partial_x\zeta)  (\partial_x \partial \zeta)
		(\underline{w}^{\pm})^2
		-
		\frac{\ve}{ \mu}(1+\ve^2 \mu (\partial_x\zeta)^2)
		(\underline{w}^{\pm})(\partial \underline{w}^{\pm})
		\\ 
		& = 
		\ve \underline{V}^{\pm}(\partial_x \partial \psi^{\pm} - \ve \underline{w}^{\pm} \partial_x \partial \zeta)- \frac{\ve}{\mu}\underline{w}^{\pm}\partial\big(\mathcal{G}_{\mu}^{\pm}[\ve \zeta]\psi^{\pm}\big).
	\end{align*}
	By adding these observations and trading $\frac{1}{\mu}\mathcal{G}^+_{\mu}[\ve \zeta]\psi^+$ and $\frac{1}{\mu}\mathcal{G}^-_{\mu}[\ve \zeta]\psi^-$ with $\partial_t \zeta$, we deduce the identity \eqref{Claim in pf of quasi}. 
	
	The next step is to let $\alpha = \beta + \delta$ with $|\delta|=1$, where we trade $\partial$ with $\partial^{\delta}_{x,t}$ in \eqref{Claim in pf of quasi}. Then we apply $\partial^{\beta}_{x,t}$ to this equation, where we claim that it will result in the following equation:
	\begin{align}\label{Claim 2 quasi}
		\partial_t& \partial^{\alpha}_{x,t}\psi +  \partial^{\alpha}_{x,t} \zeta 
		-
		\frac{\ve}{\mu}(\underline{w}^+
		- 
		\gamma \underline{w}^-)\partial_t \partial^{\alpha}_{x,t}\zeta
		\\ 
		&
		+
		\ve \underline{V}^+(\partial_x  \psi^+_{(\alpha)} 
		+
		\ve   (\partial_x \underline{w}^+)\partial_x \partial^{\alpha}_{x,t} \zeta) 
		-
		\gamma \ve \underline{V}^-(\partial_x \psi^-_{(\alpha)} 
		+
		\ve  (\partial_x \underline{w}^-)\partial_x \partial^{\alpha}_{x,t} \zeta)\notag
		\\
		& =- \frac{1}{\mathrm{bo}}  \frac{1}{\ve \sqrt{\mu}} \partial^{\alpha}_{x,t} \kappa(\ve \sqrt{\mu}\zeta) + \ve S_{\alpha},\notag 
	\end{align}
	where $S_{\alpha}$ is some generic function, satisfying 
	\begin{equation}\label{Claim 2.1 quasi}
		|S_{\alpha}|^2_{\dot{H}^{\frac{1}{2}}_{\mu}} \leq C \mathcal{E}^N(U) \big{(} 1 +   \ve^2 \sqrt{\mu}|\underline{V}^+-   \underline{V}^-|^2_{L^{\infty}}|\zeta|^2_{<N+\frac{1}{2}>}\big{)}.
	\end{equation}
	The proof of this fact is the same as in \cite{LannesTwoFluid13}, where the estimate relies on the following inequalities
	\begin{equation}\label{claim - est psi}
		|\psi_{(\alpha)}^{\pm}|_{\dot{H}^{\frac{1}{2}}_{\mu}} \leq 1 + \gamma \ve \mu^{\frac{1}{4}}|\underline{V}^+-  \underline{V}^-|_{L^{\infty}}|\zeta|_{<N+\frac{1}{2}>},
	\end{equation}
	and
	\begin{equation}\label{claim - est w}
		\frac{1}{\sqrt{\mu}} |\partial^{\beta}\underline{w}^{\pm}|_{\dot{H}^{\frac{1}{2}}_{\mu}} \leq 1 + \gamma \ve \mu^{\frac{1}{4}}|\underline{V}^+- \underline{V}^-|_{L^{\infty}}|\zeta|_{<N+\frac{1}{2}>}.
	\end{equation}
	This can be seen in the proof of Lemma $9$ in  \cite{LannesTwoFluid13}. In our case, the proof is a consequence of Lemma \ref{Lemma Gpm inv} and Corollary \ref{cor definitions}. However, the quantities involved satisfy the same estimates and therefore complete the estimate on $S_{\alpha}$.

	 We may now further decompose \eqref{Claim 2 quasi}, where we may now use the definition of $ \psi_{(\alpha)}$ to find that
	 \begin{align*}
	 	\partial_t \psi_{(\alpha)} 
	 	+
	 	\mathfrak{a}\zeta_{(\alpha)}
	 	+
	 	\ve \underline{V}^+\partial_x \psi^+_{(\alpha)}
	 	-
	 	\gamma \ve \underline{V}^-\partial_x \psi^-_{(\alpha)} 
	 	= -\frac{1}{\mathrm{bo}}
	 	\frac{1}{\ve \sqrt{\mu}} \partial^{\alpha} \kappa(\ve \sqrt{\mu}\zeta) + \ve S_{\alpha},\notag 
	 \end{align*}
	where we identify $\mathfrak{a}$ by
	\begin{equation*}
		\mathfrak{a} = 
		 \Big{(}1 
		 +
		 \ve \big{(} (\partial_t + \ve \underline{V}^{+}\partial_x)\underline{w}^{+} 
		 -
		 \gamma (\partial_t + \ve \underline{V}^{-}\partial_x)\underline{w}^{-}
	     \big{)}\Big{)}.
	\end{equation*}
	To conclude, we simply need to work on the terms:
	\begin{align}\label{so that}
		\ve \underline{V}^+\partial_x \psi^+_{(\alpha)}
		-
		\gamma \ve \underline{V}^-\partial_x \psi^-_{(\alpha)} 
		& = 
	 	\frac{\ve}{2}
		\big{(}\underline{V}^+ +  \underline{V}^- \big{)}\partial_x \psi_{(\alpha)}
		+
		\frac{\ve}{2}\big{(}\underline{V}^+ -  \underline{V}^- \big{)}\partial_x\big{(} \psi_{(\alpha)}^+ + \gamma \psi^-_{(\alpha)}\big{)}.
	\end{align}
	Then by identities \eqref{id 2}, \eqref{id 3}, and Lemma \ref{Lemma Gpm inv} implies
	\begin{align*}
		& \psi_{(\alpha)}^+ + \gamma \psi^-_{(\alpha)}
		\\ 
		& 
		=
		\big{(} ( \mathcal{J}_{\mu}[\ve \zeta]^{-1} + \gamma (\mathcal{G}^-_{\mu}[\ve \zeta])^{-1}\mathcal{G}_{\mu}[\ve \zeta]\big{)}\psi_{(\alpha)}	
		\\
		&
		\hspace{0.5cm}
		-
		\gamma \ve \mu   \Big{(}
		 \mathcal{J}_{\mu}[\ve \zeta]^{-1} (\mathcal{G}^{-}_{\mu}[\ve \zeta])^{-1}
		+
		\gamma(\mathcal{G}^-_{\mu}[\ve \zeta])^{-1}\mathcal{G}_{\mu}[\ve \zeta]  (\mathcal{G}^{+}_{\mu}[\ve \zeta])^{-1}
		\Big{)}
		\partial_x	\big{(} 
		\zeta_{(\alpha)}(\underline{V}^+-\underline{V}^-)\big{)}	+\tilde{r}_{\alpha}
		\\
		& = 
		\big{(}1 + 2\gamma (\mathcal{G}^-_{\mu}[\ve \zeta])^{-1}\mathcal{G}_{\mu}[\ve \zeta]\big{)}\psi_{(\alpha)}	
		-
		(1-\gamma)\gamma \ve \mu \mathcal{J}_{\mu}[\ve \zeta]^{-1}(\mathcal{G}^-_{\mu}[\ve \zeta])^{-1}\circ \partial_x	\big{(} 
		\zeta_{(\alpha)}(\underline{V}^+-\underline{V}^-)\big{)} 
		+ 
		\tilde{r}_{\alpha},
	\end{align*}
	with
	\begin{equation*}
		\partial_x \tilde{r}_{\alpha} = \partial_x\big{(} (\mathcal{G}^{+}_{\mu}[\ve \zeta])^{-1}r_{\alpha}^+ + (\mathcal{G}^{-}_{\mu}[\ve \zeta])^{-1}r_{\alpha}^-\big{)}.
	\end{equation*}
	Here $r_{\alpha}^{\pm}$ is the term in Lemma \ref{Lemma Gpm inv}, which is of lower order. So that \eqref{so that} becomes
	\begin{equation*}
		\ve \underline{V}^+\partial_x \psi^+_{(\alpha)}
		-
		\gamma \ve \underline{V}^-\partial_x \psi^-_{(\alpha)} 
		=
		- 
		\ve \mathcal{I}[\mathbf{U}]^{\ast}\psi_{(\alpha)}
		-
		(1-\gamma)\gamma \ve^2 \mu	(\underline{V}^+-\underline{V}^-)  \mathfrak{E}_{\mu}[\ve \zeta]	\big{(} 
		\zeta_{(\alpha)}(\underline{V}^+-\underline{V}^-)\big{)},
	\end{equation*}
	where $\mathfrak{E}_{\mu}[\ve \zeta]$ reads
	\begin{equation*}
		\mathfrak{E}_{\mu}[\ve \zeta] \bullet = \partial_x \circ \mathcal{J}_{\mu}[\ve \zeta]^{-1}(\mathcal{G}^-_{\mu}[\ve \zeta])^{-1}\circ \partial_x \bullet,
	\end{equation*}
	and we identify $\mathcal{I}[\mathbf{U}]^{\ast}$ by
	\begin{equation*}
		\mathcal{I}[\mathbf{U}]^{\ast} \bullet
		=
		-\underline{V}^+\partial_x  \bullet
		-
		\gamma 
		(\underline{V}^+-\underline{V}^-) \partial_x\big{(}(\mathcal{G}^-_{\mu}[\ve \zeta])^{-1}\mathcal{G}_{\mu}[\ve \zeta] \bullet\big{)}.
	\end{equation*}
	The surface tension term is linearized with the formula
	\begin{equation*}
		\frac{1}{\mathrm{bo}}\frac{1}{\ve \sqrt{\mu}} \partial^{\alpha} \kappa(\ve \sqrt{\mu}\zeta) 
		=
		\mathrm{bo}^{-1}\partial_{x} \mathcal{K}(\ve \sqrt{\mu}\partial_x\zeta)\partial_x\zeta_{(\alpha)}
		+
		\mathcal{K}_{(\alpha)}[\sqrt{\mu}\ve \partial_x\zeta]\zeta_{\langle\check{\alpha}\rangle} + \ve S_{\alpha},
	\end{equation*}
	where $\mathcal{K}$ is given in Definition \ref{Def op}. We also identify $\mathfrak{Ins}[\mathbf{U}]$ by
	\begin{equation*}
		\mathfrak{Ins}[\mathbf{U}] \bullet = \mathfrak{a}\bullet -	(1-\gamma)\gamma \ve^2 \mu	[\![ \underline{V}^{\pm}]\!]   \mathfrak{E}_{\mu}[\ve \zeta]	\big{(} 
		\bullet[\![ \underline{V}^{\pm}]\!] \big{)} - \mathrm{bo}^{-1} \partial_x \mathcal{K}[\ve \sqrt{\mu} \partial_x \zeta]\partial_x  \bullet .
	\end{equation*}

	\end{proof}

	\section{Tools}

	\subsection{Estimates on Fourier multipliers and classical estimates}
	In this section, we will give basic multiplier estimates. To be precise, we will give a definition of the Fourier multipliers.
	\begin{Def}\label{definition order of a Fourier multiplier}
		We say that a Fourier multiplier $\mathrm{F}$ is of order $s$ $(s\in\mathbb{R})$ and write $\mathrm{F} \in \mathcal{S}^s$ if $\xi \in \mathbb{R} \mapsto F(\xi) \in \mathbb{C}$ is smooth and satisfies
		\begin{align*}
			\forall \xi \in \mathbb{R}, \forall\beta\in\mathbb{N}, \qquad \sup_{\xi\in\mathbb{R}}\: \langle\xi\rangle^{\beta-s}|\partial^{\beta}F(\xi)| < \infty.
		\end{align*}
		We also introduce the seminorm
		\begin{align*}
			\mathcal{N}^s(\mathrm{F}) = \sup_{\beta\in\mathbb{N}, \beta\leq 4} \:\sup_{\xi\in\mathbb{R}}\: \langle\xi\rangle^{\beta-s}|\partial^{\beta}F(\xi)|.
		\end{align*}
	\end{Def}
	
	The first result is several basic multiplier estimates that are used throughout the paper.

	\begin{prop}\label{prop B 2} Let $\mu, \gamma, \mathrm{bo}^{-1} \in (0,1)$ and $f \in \mathscr{S}(\R)$. Then there exist a universal constant $C>0$ such that:
		
		\begin{itemize}
			\item [1.] For the Fourier multiplier
			\begin{equation*}
				\mathcal{G}_{\mu}[0] f (x)= \sqrt{\mu} \mathcal{F}^{-1}\Big{(} |\xi| \frac{\mathrm{tanh}(\sqrt{\mu}|\xi|)}{1+ \gamma\mathrm{tanh}(\sqrt{\mu}|\xi|)} \hat{f}(\xi) \Big{)}(x),
			\end{equation*}
			there holds,
			\begin{equation}\label{Est G0}
				|\mathcal{G}_{\mu}[0] f|_{L^2} \leq \mu^{\frac{3}{4}}C |f|_{\dot{H}^{1}_{\mu}},
			\end{equation}
			and
			\begin{equation}\label{Est 2 G0}
				|\mathcal{G}_{\mu}[0] f|_{L^2} \leq \mu C |f|_{\dot{H}^{\frac{3}{2}}_{\mu}}.
			\end{equation} 
			\item [2.] For the Fourier multiplier
			\begin{equation*}
				(\mathcal{G}_{\mu}[0])^{\frac{1}{2}}f(x)= \mathcal{F}^{-1}\bigg{(}\Big{(}\sqrt{\mu} |\xi| \frac{\mathrm{tanh}(\sqrt{\mu}|\xi|)}{1+\gamma\mathrm{tanh}(\sqrt{\mu}|\xi|)}\Big{)}^{\frac{1}{2}} \hat{f}(\xi) \bigg{)}(x),
			\end{equation*}
			there holds,
			\begin{equation}\label{equivalence alpha = 0}
				\frac{1}{C} |f|_{\dot{H}_{\mu}^{\frac{1}{2}}} \leq \frac{1}{\sqrt{\mu}}|(\mathcal{G}_{\mu}[0])^{\frac{1}{2}} f|_{L^2} \leq C |f|_{\dot{H}_{\mu}^{\frac{1}{2}}}.
			\end{equation}

			\item [3. ] There holds,
			\begin{equation}\label{Basic est: B to D}
				|f|_{\dot{H}_{\mu}^{s+\frac{1}{2}}} 
				\leq
				\mu^{-\frac{1}{4}} |\partial_x f|_{H^{s-\frac{1}{2}}} 
				\leq 
				\mu^{-\frac{1}{4}} | f|_{\mathring{H}^{s+\frac{1}{2}}},
			\end{equation}
			and
			\begin{equation}\label{Basic..}
				|f|_{\dot{H}_{\mu}^{s+\frac{1}{2}}} 
				\leq
				|\partial_x f|_{H^{s}}.
			\end{equation}
			Moreover, for $\mathfrak{P} = |\mathrm{D}|(1+\sqrt{\mu}|\mathrm{D}|)^{-\frac{1}{2}}$ and $S^- = -\sqrt{\mu} |\mathrm{D}|$ there holds,
			\begin{equation}\label{B2/S-}
				|\frac{\mathfrak{P}^2}{S^-} f|_{H^{s+\frac{1}{2}}} \leq \mu^{-\frac{3}{4}} |f|_{\dot{H}_{\mu}^{s+\frac{1}{2}}}.
			\end{equation}
		
			\item [4.] Define Fourier multipliers
			\begin{align*}
				\mathrm{T}(\mathrm{D})f(x) & = \mathcal{F}^{-1} \Big{(}\frac{\tanh(\sqrt{\mu}|\xi|)}{\sqrt{\mu}|\xi|} \hat{f}(\xi) \Big{)}(x), 
				\\
				\mathrm{I}(\mathrm{D})f(x) & =  \mathcal{F}^{-1} \Big{(}(1+\gamma \tanh(\sqrt{\mu} |\xi|))^{-1}\hat{f}(\xi) \Big{)}(x)
			\end{align*}
			Then there holds,
			\begin{align}
					\label{est T}
					|(\mathrm{T}(\mathrm{D}) -1)f|_{H^s}  & \leq \mu C |\partial_x^2f|_{H^{s}},
					\\
					\label{inverse tanh}
					|(\mathrm{I}(\mathrm{D}) -1)f|_{H^s} & \leq \sqrt{\mu} C |\partial_xf|_{H^{s}}.
			\end{align}
			Moreover, let $\sqrt{\mathrm{k}}(\mathrm{D})$ be the equare root of the symbol
			\begin{equation*}
				\mathrm{k}(\mathrm{D}) = \mathrm{I}(\mathrm{D})\mathrm{T}(\mathrm{D}) \big{(}(1 - \mathrm{bo}^{-1} \partial_x^2\big{)}.
			\end{equation*}
			Then there holds,
			\begin{align}\label{est sqrt k p}
				|\sqrt{\mathrm{k}}(\mathrm{D})f|_{H^s}  & \leq C( |f|_{H^s} +  \mathrm{bo}^{-\frac{1}{2}}|\partial_xf|_{H^{s}}),
				\\ 
				\label{est K}
				|(\mathrm{k}(\mathrm{D}) -1)f|_{H^s}  & \leq (\sqrt{\mu}+\mathrm{bo}^{-1}) C |\partial_xf|_{H^{s+1}},
				\\ 
				\label{est sqrt K}
				|(\sqrt{\mathrm{k}}(\mathrm{D}) -1)f|_{H^s}  & \leq (\sqrt{\mu}+\mathrm{bo}^{-1}) C |\partial_xf|_{H^{s+1}},
				\\
				\label{precise est sqrt K}
				\Big{|}	\big{(}\sqrt{k}(\mathrm{D}) -  (1- \frac{\gamma}{2} \sqrt{\mu}|\mathrm{D}| - \frac{\partial_x^2}{2\mathrm{bo}})\big{)}f\Big{|}_{H^s}& \leq
				(\mu +  \frac{\sqrt{\mu}}{\mathrm{bo}})C |\partial_x^2 f|_{H^{s+1}}.
			\end{align}
			\item 5. Define the Fourier multiplier
			\begin{equation*}
				\mathrm{t}^{-\frac{1}{2}}(\mathrm{D})f (x)
				=
				\mathcal{F}^{-1} \Big{(}\Big{(}(1+\gamma \tanh(\sqrt{\mu}|\xi|))\frac{\sqrt{\mu}|\xi|}{\tanh(\sqrt{\mu}|\xi|)} \Big{)}^{\frac{1}{2}}\hat{f}(\xi) \Big{)}(x).
			\end{equation*}
			Then for any $h_0 \in (0,1)$ there is a $C>0$ such that,
			\begin{equation}\label{coercivity}
				(1-\frac{h_0}{2})|f|_{H^s}^2 + \sqrt{\mu}C|f|^2_{\mathring{H}^{s+\frac{1}{2}}} \leq |\mathrm{t}^{-\frac{1}{2}}(\mathrm{D})f|_{H^s}^2 \leq C(|f|_{H^s}^2 + \sqrt{\mu}|f|^2_{\mathring{H}^{s+\frac{1}{2}}}).
			\end{equation}
		\end{itemize}
		
	\end{prop}
	
	\begin{proof}
		We first consider the  symbol in frequency, where the elementary inequality holds
		\begin{equation*}
			\frac{1}{2}\sqrt{\mu} |\xi| \mathrm{tanh}(\sqrt{\mu}|\xi|)\leq \sqrt{\mu} |\xi| \frac{\mathrm{tanh}(\sqrt{\mu}|\xi|)}{1+\gamma\mathrm{tanh}(\sqrt{\mu}|\xi|)} \leq 	\sqrt{\mu} |\xi| \mathrm{tanh}(\sqrt{\mu}|\xi|).
		\end{equation*}
		Then by splitting in high and low frequency, we can prove that there is a number $C>0$ such that
		\begin{equation*}
			\frac{\mu}{C}\frac{|\xi|^2}{(1+ \sqrt{\mu}|\xi|)} \leq \sqrt{\mu} |\xi| \mathrm{tanh}(\sqrt{\mu}|\xi|) \leq \mu C\frac{|\xi|^2}{(1+ \sqrt{\mu}|\xi|)}.
		\end{equation*}
		To conclude, use Plancherel's identity for \eqref{Est G0} to find that
		\begin{align*}
			|\mathcal{G}_{\mu}[0] f|_{L^2}^2 
			 \lesssim
			\mu^2 \int_{\R} \frac{|\xi|^4}{(1+\sqrt{\mu}|\xi|)^2} |\hat{f}(\xi)|^2 \: \mathrm{d}\xi \lesssim
			\mu^{\frac{3}{2}} \int_{\R} \langle \xi \rangle^{1} \frac{|\xi|^2}{(1+\sqrt{\mu}|\xi|)} |\hat{f}(\xi)|^2 \: \mathrm{d}\xi,
		\end{align*}
		while \eqref{Est 2 G0} and \eqref{equivalence alpha = 0} follows directly.

		The proof of the estimates in point three follows directly by Plancherel's identity and elementary inequalities.

		For the proof of the fourth point, we observe for $r\geq 0$ that
		\begin{equation*}
			\tanh(r) = \int_0^r (1- \tanh^2(s)) \: \mathrm{d}s.
		\end{equation*}
		Then since $\tanh(r)\leq r$, we have that
		\begin{equation*}
			|(T(\xi) -1)|  \lesssim \mu  \xi^2,
		\end{equation*}
		and may conclude by Plancherel's identity that \eqref{est T} holds true, while the remaining points follows similarly. For the sake of completeness we give the proof of \eqref{precise est sqrt K}. To do so, we let $r = \sqrt{\mu}|\xi|$ and for $0<r<1$ we observe:
		\begin{align*}
			k(r) 
			& = 
			\frac{\tanh(r)}{r} \frac{1}{1+\gamma \tanh(r)} \big{(}1+ \frac{r^2}{\mu \mathrm{bo}}\big{)}
			\\
			& = 1 -\gamma r +\frac{r^2}{\mu \mathrm{bo}} +
			A_1 + A_2 + A_3,
		\end{align*}
		where use  Taylor expansions to see that
		\begin{align*}
			A_1 
			& := \Big{(}  \frac{\tanh(r)}{r} -1 \Big{)}\frac{1}{1+\gamma \tanh(r)} \big{(}1+ \frac{r^2}{\mu \mathrm{bo}}\big{)} = cr^2\big{(}1+ \frac{r^2}{\mu \mathrm{bo}}\big{)} 
			\\
			A_2 
			&: =
			\Big{(}\frac{1}{1+\gamma \tanh(r)}-1\Big{)}\frac{r^2}{\mu \mathrm{bo}}= c \frac{r^3}{\mu \mathrm{bo}} 
			\\
			A_3 
			& := 
			\frac{1}{1+\gamma \tanh(r)} - (1 -\gamma r)= cr^2, 
		\end{align*}
		for some universal constant $c$. Now using a Taylor expansion of $r \mapsto \sqrt{r}$ implies that
		\begin{equation*}
			\sqrt{k}(r)  
			=
			1- \frac{\gamma}{2} r + \frac{r^2}{2\mu \mathrm{bo}}  + c\big{(}r^2+ \frac{r^3}{\mu \mathrm{bo}} \big{)},
		\end{equation*}
		for $0<r<1$. While for $r \geq 1$, we trivially have that
		\begin{equation*}
			\Big{|}\sqrt{k}(r)  - \Big{(}1- \frac{\gamma}{2} r + \frac{r^2}{2\mu \mathrm{bo}}\Big{)}\Big{|}
			\lesssim
			1+ \frac{r^2}{\mu \mathrm{bo}} + r
			\leq 
			r^2+ \frac{r^3}{\mu \mathrm{bo}}.
		\end{equation*}
		In particular, by Plancherel's identity we find that
		\begin{equation*}
			\Big{|}	\big{(}\sqrt{\mathrm{k}}(\mathrm{D}) -  (1- \frac{\gamma}{2} \sqrt{\mu}|\mathrm{D}| - \frac{\partial_x^2}{2\mathrm{bo}})\big{)}f\Big{|}_{H^s} \lesssim (\mu +\frac{\sqrt{\mu}}{\mathrm{bo}})|\partial_x^2 f|_{H^{s+1}}.
		\end{equation*}
	
		For the proof of \eqref{coercivity}, we note that the upper bound is trivial while the lower bound follows from $\tanh(r) \leq r$, $0\leq \tanh(r)\leq 1$, and the inequality
		\begin{equation*}
			(1+\gamma \tanh(r))\frac{r}{\tanh(r)} \geq  (1-\frac{h_0}{2})\frac{r}{\tanh(r)} +  \frac{h_0}{2}\frac{r}{\tanh(r)} \geq (1-\frac{h_0}{2})+ \frac{h_0}{2}r.
		\end{equation*}
	\end{proof}

	We will also use the following product estimates (see Proposition B. 2. and Proposition B. 4. in \cite{WWP}).
	\begin{lemma} Let $t_0 > \frac{d}{2}$, $s\geq - t_0$, $f \in H^{\max\{t_0,s\}}(\R^d)$, and take $g\in H^{s}(\R^d)$ then
		\begin{equation}\label{Classical prod est}
			|fg|_{H^s} \lesssim |f|_{H^{\max\{t_0,s\}}}|g|_{H^s}.
		\end{equation}
		Moreover, if there exist $c_0>0$ and $1+g \geq c_0$ then
		\begin{equation}\label{prod est division}
			\Big{|}\frac{f}{1+g}\Big{|}_{H^s} \lesssim C(c_0,|g|_{L^{\infty}})(1+|f|_{H^s})|g|_{H^s}.
		\end{equation}
	\end{lemma}

	Lastly, we will use several commutator estimates for Fourier multipliers. The first result is a generalization of the classical Kato-Ponce estimate\footnote{See \cite{KatoPonce1988} for the case $\mathrm{F} =\Lambda^s$.} and is given by Proposition B. 8. in \cite{WWP}:
	\begin{prop}\label{Commutator estimates}
		Let $t_0 > \frac{1}{2}$, $s \geq 0$. In the case $\mathrm{F} \in \mathcal{S}^s$, $f\in  H^{\max\{t_0+1,s\}}(\mathbb{R})$ then, for all $g \in H^{s-1}(\mathbb{R})$,
		\begin{align}\label{Commutator est}
			|[\mathrm{F},f]g|_{L^2} \leq \mathcal{N}^s(\mathrm{F})|f|_{H^{\max\{t_0+1,s\}}}|g|_{H^{s-1}}. 
		\end{align}
		In the case $\mathrm{F} \in \mathcal{S}^0$, $f\in  H^{t_0+1}(\mathbb{R})$ then, for all $g \in L^{2}(\mathbb{R})$,
		\begin{align}\label{Commutator est 2}
			|[\mathrm{F},f]g|_{H^1} \leq \mathcal{N}^0(\mathrm{F})|f|_{H^{t_0+1}}|g|_{L^2}. 
		\end{align}
	\end{prop}

	Moreover, we will also need some commutator estimates on multipliers with non-smooth symbol.

	\begin{prop}\label{prop B 5}
		Let  $\mu, \mathrm{bo}^{-1}\in (0,1)$, $f,g \in \mathscr{S}(\R)$,  $t_0 > \frac{1}{2}$ then  there is a universal constant $C>0$ such that
		\begin{equation}\label{Commutator Dhalf}
			|[f,|\mathrm{D}|^{\frac{1}{2}}] g| \leq C |f|_{H^{t_0+ \frac{1}{2}}} |g|_{L^2}
		\end{equation}
		\begin{equation}\label{Commutator mu quart}
			|[f, (1+\sqrt{\mu}|\mathrm{D}|)^{\frac{1}{2}}]g|_{L^2} \leq C|f|_{H^{t_0+\frac{1}{2}}} |g|_{L^2}.
		\end{equation}
		Moreover, let
		\begin{align*}
			\mathrm{\sqrt{T}}(\mathrm{D})f(x) & = \mathcal{F}^{-1} \Big{(}
			\Big{(}
			\frac{\tanh(\sqrt{\mu}|\xi|)}{\sqrt{\mu}|\xi|} \Big{)}^{\frac{1}{2}}\hat{f}(\xi) \Big{)}(x), 
			\\
			\sqrt{\mathrm{I}}(\mathrm{D})f(x) & =  \mathcal{F}^{-1} \Big{(}
			(1+\gamma \tanh(\sqrt{\mu} |\xi|))^{-\frac{1}{2}}\hat{f}(\xi) \Big{)}(x),
		\end{align*}
		and let $\sqrt{\mathrm{k}}(\mathrm{D})$ and $\sqrt{\mathrm{t}}(\mathrm{D})$ be defined by
		\begin{align*}
			\sqrt{\mathrm{k}}(\mathrm{D}) 
			& = 
			\sqrt{I}(\mathrm{D}) \sqrt{T}(\mathrm{D}) \sqrt{1+\mathrm{bo}^{-1}|\mathrm{D}|^2}
			\\ 
			\sqrt{\mathrm{t}}(\mathrm{D})f & = \sqrt{I}(\mathrm{D}) \sqrt{T}(\mathrm{D}) 
		\end{align*}
		then for $s\geq t_0+1$ there holds
		\begin{equation}\label{est com 1}
			|[\Lambda^s \sqrt{\mathrm{k}}(\mathrm{D}),  f]\partial_x g|_{L^2} \leq C |f|_{H^{s+1}_{\mathrm{bo}}} |g|_{H^{s+1}_{\mathrm{bo}}},
		\end{equation}
		\begin{equation}\label{est com 2}
			|[ \Lambda^s \sqrt{\mathrm{t}}(\mathrm{D}),  f]\partial_x g|_{L^2} \leq C |f|_{H^s} |g|_{H^s},
		\end{equation}
		\begin{equation}\label{est com 3}
			|\partial_x [\sqrt{\mathrm{t}}(\mathrm{D}),  f]g|_{L^2} \leq C |f|_{H^{t_0+1}} |g|_{L^2}.
		\end{equation}

	\end{prop}

	\begin{proof}
		For the proof of \eqref{Commutator Dhalf}, we write the commutator as a convolution product in frequency:
		\begin{align*}
			\big{|}  [f, |\mathrm{D}|^{\frac{1}{2}}] g\big{|}_{L^2}
			=
			\bigg{|}
			\int_{\R} \Big( 
			|\xi|^{\frac{1}{2}} - |\rho|^{\frac{1}{2}} 
			\Big) \hat{f}(\xi-\rho)\hat{g}(\rho)\: d\rho \bigg{|}_{L^2_{\xi}}.
		\end{align*}
		Then the proof is a direct consequence of the estimate $|\xi|^{\frac{1}{2}}-|\rho|^{\frac{1}{2}} \leq 1 + |\xi - \rho|^{\frac{1}{2}}$ when combined with Minkowski integral inequality and Cauchy-Schwarz inequality.
		
		Inequality \eqref{Commutator mu quart} is proved similarly.

		For the remaining estimates, they are proved in \cite{Paulsen22} in the case $	\sqrt{I}(\mathrm{D})$ is identity with the same result. However, noting that $\sqrt{I}(\mathrm{D})$ is bounded on $L^2(\R)$ we observe that
		\begin{align*}
			|[\mathrm{a}(\mathrm{D}) \mathrm{b}(\mathrm{D}),  f] \partial_xg |_{L^2}
			& \leq 
			|\mathrm{a}(\mathrm{D}) [ \mathrm{b}(\mathrm{D}),f]\partial_xg|_{L^2}
			+
			|[\mathrm{a}(\mathrm{D}) ,  f] \mathrm{b}(\mathrm{D})\partial_xg|_{L^2}
			\\ 
			& =
			A_1 + A_2.
		\end{align*}
		Then to prove \eqref{est com 1} we take $\mathrm{a}(\mathrm{D}) = \sqrt{I}(\mathrm{D})$ and $\mathrm{b}(\mathrm{D}) =  \sqrt{T}(\mathrm{D}) \sqrt{1+\mathrm{bo}^{-1}|\mathrm{D}|^2}$ to deduce that
		\begin{align*}
			A_1 
			=
			|\mathrm{a}(\mathrm{D})[ \mathrm{b}(\mathrm{D}),  f]\partial_x g|_{L^2} 
			\leq
			|[ \mathrm{b}(\mathrm{D}),  f] \partial_xg|_{L^2} \leq C |f|_{H^{s+1}_{\mathrm{bo}}} |g|_{H^{s+1}_{\mathrm{bo}}},
		\end{align*}
		where the last estimate is proved in \cite{Paulsen22}, Lemma 2.7. For the second term it is clearly enough to prove that
		\begin{equation*}
			A_2 \leq |[\mathrm{a}(\mathrm{D}) ,  f] \mathrm{b}(\mathrm{D})\partial_xg|_{L^2} \lesssim |f|_{H^s}|\mathrm{b}(\mathrm{D})g|_{L^2}.
		\end{equation*}
		In other words, we must prove that the commutator $[\mathrm{a}(\mathrm{D}) , f]$ absorbs one derivative uniformly in $\mu$. In particular, writing it as a convolution product in frequency, we find  that:
		\begin{equation*}
			[\mathrm{a}(\mathrm{D}) , f] \partial_x g = \int_{\R} \Big( 
			a(\xi)-a(\rho) 
			\Big) \hat{f}(\xi-\rho) \:\widehat{\partial_x g}(\rho)\: d\rho,
		\end{equation*}
		and we can conclude if 
		\begin{equation}\label{est claim}
			k(\xi, \rho)| = |a(\xi)-a(\rho)| \langle \rho \rangle \leq \langle \xi- \rho\rangle.
		\end{equation}
		To prove \eqref{est claim}, we consider three cases. First, if $|\rho|\leq 1$ we have trivially that
		\begin{equation*}
			|a(\xi)-a(\rho)| \langle \rho \rangle \lesssim 1,
		\end{equation*}
		since $\xi \mapsto a(\xi)$ is bounded by one. Next, we consider the region  $|\xi|\geq |\rho|>1$. Then we observe that since $\xi \mapsto a(\xi)$ is decreasing then
		\begin{equation*}
			\Big{(}\frac{a(\xi)}{a(\rho)}\Big{)}^{2} \leq \frac{a(\xi)}{a(\rho)} \leq 1,
		\end{equation*}
		from which we deduce that
		\begin{align*}
			k(\xi,\rho) 
			& =
			\Big{(}1-\frac{a(\xi)}{a(\rho)}\Big{)}a(\rho) \langle \rho \rangle
			\\
			& \leq 
			\bigg{(} 1 - \Big{(}\frac{a(\xi)}{a(\rho)}\Big{)}^{2} \bigg{)}a(\rho)\langle \rho \rangle.
		\end{align*}
		Now, use the definition of $a$ to simplify the expression further: 
		\begin{align*}
				k(\xi,\rho)
				& \leq  \big{(} \tanh(\sqrt{\mu}|\xi|) - \tanh(\sqrt{\mu}|\rho|)\big{)} \frac{\gamma a(\rho)}{1+\gamma \tanh(\sqrt{\mu}|\xi|)}\langle \rho \rangle.
		\end{align*}
		and then use the Mean value Theorem and that $|\rho|<|\xi|$ to deduce the estimate
		\begin{align*}
			k(\xi,\rho)
			& \lesssim  \sqrt{\mu}e^{-2\sqrt{\mu}|\rho|}\langle \rho \rangle \lesssim 1.
		\end{align*}
		The remaining case  $1<|\rho|$ and $|\xi|<|\rho|$ follows by reversing the role of $\xi$ and $\rho$ and use that $\langle \rho \rangle \lesssim \langle \rho - \xi \rangle + \langle \xi \rangle$ to conclude.  
		
		For the proof of \eqref{est com 2} and \eqref{est com 3}, we use Lemma 2.12 in \cite{Paulsen22} in the case $\sqrt{\mathrm{I}}(\mathrm{D})$ is identity, then argue as above where  $\sqrt{\mathrm{I}}(\mathrm{D})$  can be treated as a zero order differential operator.
	\end{proof}

	\subsection{Estimates on pseudo-differential operators} In this section, we will give estimates of pseudo-differential operators whose symbol depends on the free surface. In particular, the framework needs to handle symbols of limited smoothness which is developed in \cite{Lannes_06}. To be precise, we give the definition of the objects we will study. 
	
	\begin{Def}
		Let $m \in \R$ and $t_0> \frac{1}{2}$. A symbol $\sigma(x,\xi)$ belongs to the class $\Gamma_{t_0}^m$ if and only if 
		\begin{equation*}
			\sigma|_{\R \times \{|\xi|\leq 1\}} \in L^{\infty}(\{|\xi|\leq 1\} \: : \: H^{t_0}(\R))
		\end{equation*}
		and for all $\beta \in\N$ one has
		\begin{equation*}
			\sup\limits_{|\xi|\geq \frac{1}{4}} \langle \xi \rangle^{\beta-m} |\partial^{\beta}_{\xi}\sigma(\cdot, \xi)|_{H^{t_0}}<\infty.
		\end{equation*}
		We also introduce the seminorm
		\begin{equation*}
			\mathcal{L}_{k,s}^m(\sigma) 
			: = 
			\sup_{\beta\leq k} \:\sup_{|\xi|\geq \frac{1}{4}}\: \langle\xi\rangle^{\beta-m}|\partial_{\xi}^{\beta}\sigma(\cdot, \xi)|_{H^s},
		\end{equation*}
		Moreover, let $l_{k,s}$ be a measure of the information in low frequencies given by
		\begin{equation*}
			l_{s}(\sigma) : = \sup_{\beta \leq k, |\xi|\leq 1} |\sigma(\cdot,\xi)|_{H^s},
		\end{equation*}
		and define
		\begin{equation*}
			|\sigma|_{H^s_{(m)}}: = l_{s}(\sigma)+ \mathcal{L}^m_{4,s}(\sigma).
		\end{equation*}

	\end{Def}

	The main tool we will use to derive our estimate is provided  in \cite{Lannes_06}.
	
	\begin{prop}\label{PDO Est Lannes}
		Let $f \in \mathscr{S}(\R)$,  $t_0\in(\frac{1}{2},s_0]$. 	Then we have the following estimates:
		\begin{itemize}
			\item [1.] For $\sigma \in \Gamma_{s_0}^m$,  $s \in (-t_0,t_0)$, and $m \in \R$, there holds,
			\begin{equation}\label{product est Op}
				|\mathrm{Op}(\sigma)f|_{H^s} \leq	|\sigma|_{H^{t_0}_{(m)}}  |f|_{H^{s+m}},
			\end{equation}
			\item [2.] Let $m_1,m_2 \in \R$, $\sigma_1 \in \Gamma_{s_0+\max\{m_1,0\}+1}^{m_1}$ and $\sigma_2 \in \Gamma_{s_0+\max\{m_2,0\}+1}^{m_2}$ such that $-t_0<s+m_j\leq t_0+1$ and $-t_0<s\leq t_0+1$.
			%
			%
			%
			%
			%
			%
			%
			%
			%
			%
			%
			%
			Then one has,
			\begin{align}
				\label{Commutator Op general}
				|[\mathrm{Op}(\sigma_1),\mathrm{Op}(\sigma_2)]f|_{H^s} 
				& \lesssim	
				|\sigma^1|_{H^{t_0+1}_{(m_1)}} 	|\sigma^2|_{H^{t_0+1}_{(m_2)}} 
				|f|_{H^{s+m_1+m_2-1}},
				\\
				\label{adjoint est}
				|\big{(}\mathrm{Op}(\sigma_1)^{\ast}-\mathrm{Op}(\overline{\sigma}_1)\big{)}f|_{H^s}
				& \lesssim
				|\sigma^1|_{H^{t_0+1}_{(m_1)}} 
				|f|_{H^{s+m_1-1}},
				\\ 
				\label{Comp est}
				|\mathrm{Op}(\sigma_1\sigma_2)-\mathrm{Op}(\sigma_1) \circ \mathrm{Op}(\sigma_2)f|_{H^s} 
				& \lesssim
				 |\sigma^1|_{H^{t_0+1}_{(m_1)}} 	|\sigma^2|_{H^{t_0+1}_{(m_2)}}  |f|_{H^{s+m_1+m_2-1}}.
			\end{align}

		\end{itemize}
		
	\end{prop}

\begin{remark}
	The estimates provided in \cite{Lannes_06} hold for a larger class of symbols, but here, we only state the result for a symbol that is bounded in the origin and smooth for large frequencies. To be precise, 	estimate \eqref{product est Op} follows from Theorem 1 in this paper and  \eqref{Commutator Op general} is an application of Theorem 8. While estimate \eqref{adjoint est} and  \eqref{Comp est} are not given in \cite{Lannes_06}. However,  as stated in footnote 8 and 9 of \cite{LannesTwoFluid13}, the composition estimate is used to prove Theorem 8, while the adjoint estimate can be deduced by similar methods.
\end{remark}

Another important result is on Fourier multipliers that are merely bounded in low frequency. 

\begin{Def}	Let $m \in \R$ and $t_0> \frac{1}{2}$.	We say $\sigma \in \mathcal{M}^{m}$  if and only if 
	$$ \sigma^1 \in L^{\infty}(\R)$$
	and for all $\beta \in\N$ one has 
	\begin{equation*}
		\sup\limits_{|\xi| \geq \frac{1}{4}} \langle \xi \rangle^{\beta-m} |\partial^{\beta}_{\xi}\sigma^1(\xi)|<\infty.
	\end{equation*}
\end{Def}

In particular, we have from Theorem $6$ (i) \cite{Lannes_06}:

\begin{prop} 
	Let  $m_1,m_2 \in \R$,  and $\frac{1}{2}<t_0 \leq s_0$. Let $\sigma^1 \in \mathcal{M}^{m_1}$ and $\sigma^{2} \in  \Gamma^{m_2}_{s_0 + \max\{m_1,0\} +1}$. Then for all $s\in \R$ such that $-t_0< s \leq t_0+1$ and $-t_0<s+m_1 \leq t_0+1$, there holds,
	\begin{equation}\label{Thm 6}
		|[\sigma^1(\mathrm{D}),\mathrm{Op}(\sigma^2)]f|_{H^s}
		\lesssim 
		\big{(} 
		\sup\limits_{|\xi|\leq 1}|\sigma^1(\xi)|
		+
		\sup\limits_{\beta \leq 4|\xi| \geq \frac{1}{4}} \langle \xi \rangle^{\beta-m_1} |\partial^{\beta}_{\xi}\sigma^1(\xi)| 
		\big{)}
		|\sigma^2|_{H^{t_0+1}_{(m_2)}}|f|_{H^{s+m_1+m_2-1}}.
	\end{equation}
\end{prop}

\begin{remark}
	If we let $\sigma^1= \Lambda^s$ and  $s=\frac{1}{2}$ then we have that
	\begin{equation}\label{Commutator Op}
		|[\Lambda^s,\mathrm{Op}(\sigma)]f|_{H^{\frac{1}{2}}} \lesssim |\sigma|_{H^{t_0+1}_{(m)}}  |f|_{H^{s+m-\frac{1}{2}}}.
	\end{equation}\\
\end{remark}

	\begin{lemma}\label{est Op L}
		Let $t_0>\frac{1}{2}$, $s\geq 0$, $\zeta \in H^{t_0+3}(\R)$ and define the symbol
		\begin{equation*}
				L(x,\xi, z) = e^{-z\big( \frac{	\sqrt{\mu}|\xi|}{1+\ve^2\mu(\partial_x \zeta)^2} - i \frac{\ve \mu \partial_x \zeta \xi}{1+\ve^2\mu(\partial_x \zeta)^2}\big)}.
		\end{equation*}
		Then for any $f \in \mathscr{S}(\R)$ there is $C>0$ nondecreasing function of its argument such that,
		\begin{align}\label{Est op L dx}
			\|\Lambda^s\mathrm{Op}(L)\partial_x f\|_{L^2(\mathcal{S}^-)}
			&  \leq \mu^{-\frac{1}{4}}C(|\zeta|_{H^{t_0+1}})  |f|_{\mathring{H}^{s+\frac{1}{2}}}
			\\ \label{Est op dzL}
			\|\Lambda^s\mathrm{Op}(\partial_z L) f\|_{L^2(\mathcal{S}^-)} 
			& \leq \mu^{\frac{1}{4}} C( |\zeta|_{H^{t_0+1}})  |f|_{\mathring{H}^{s+\frac{1}{2}}}
			\\ \label{Est op 1/z dxL}
			\|\Lambda^s\mathrm{Op}(\frac{1}{z}\partial_xL) f\|_{L^2(\mathcal{S}^-)} 
			&  \leq  \ve \mu^{\frac{3}{4}}C(|\zeta|_{H^{t_0+2}}) |f|_{\mathring{H}^{s+\frac{1}{2}}}
			\\ \label{Est op 1/z dxxL}
			\|\Lambda^s\mathrm{Op}(\frac{1}{z}\partial_x^2L) f\|_{L^2(\mathcal{S}^-)} 
			&  \leq  \ve \mu^{\frac{3}{4}}C(|\zeta|_{H^{t_0+3}}) |f|_{\mathring{H}^{s+\frac{1}{2}}}
			\\ \label{Est op dx dz L}
			\|\Lambda^s\mathrm{Op}(\partial_x \partial_z L)f \|_{L^2(\mathcal{S}^-)} 
			&  \leq  \ve \mu^{\frac{3}{4}}C(|\zeta|_{H^{t_0+2}}) |f|_{\mathring{H}^{s+\frac{1}{2}}}
			\\ \label{Est op dxL dx}
			\|\Lambda^s\mathrm{Op}(\partial_xL) \partial_x f\|_{L^2(\mathcal{S}^-)} 
			&  \leq  \ve \mu^{\frac{1}{4}}C(|\zeta|_{H^{t_0+2}}) |f|_{\mathring{H}^{s+\frac{1}{2}}}.
		\end{align}
	\end{lemma}

\begin{remark}\label{remark op L}
	Estimate \eqref{Est op dzL} also holds for the symbol $\mathrm{Op}\big{(}(\partial_x^2\zeta)\partial_zL\big{)}$, but with constant depending on $|\zeta|_{H^{t_0+2}}$. 
\end{remark}

	\begin{proof} 
		
		The inequalities are proved similarly, where we first consider the proof of \eqref{Est op L dx}. To employ Proposition \ref{PDO Est Lannes}, we observe by Sobolev embedding that there exist $C=C(|\zeta|_{t_0+1})>0$ nondecreasing such that
		\begin{equation*}
			-\frac{1}{1+\ve^2\mu(\partial_x \zeta)^2} \leq -\frac{1}{C(|\zeta|_{H^{t_0+1}})},
		\end{equation*}
		and we use it to define the symbol:
		\begin{equation*}
			\sigma_z(x,\xi) =  e^{-z( \frac{1}{1+\ve^2\mu(\partial_x \zeta)^2}-\frac{1}{2C})\sqrt{\mu}|\xi|} e^{ iz \frac{\ve \mu \partial_x \zeta \xi}{1+\ve^2\mu(\partial_x \zeta)^2}}.
		\end{equation*}
		Then for $z \in (0,\infty)$ and any  $|\xi|\geq \frac{1}{4}$ we have that
		\begin{align*}
			|\partial_{\xi}^{\beta}\sigma_z(x, \xi)|_{H^{t_0}} 
			& \lesssim
			C(|\partial_x \zeta |_{H^{t_0}})(z\sqrt{\mu})^{\beta}e^{-\frac{z}{2C}\sqrt{\mu}|\xi|} 
			\\ 
			& \leq
			C(|\partial_x \zeta |_{H^{t_0}})(z\sqrt{\mu}|\xi|)^{\beta}e^{-\frac{z}{2C}\sqrt{\mu}|\xi|} \langle \xi \rangle^{-\beta},
		\end{align*}
		by the algebra property of $H^{t_0}(\R)$, so that
		\begin{equation*}
			\mathcal{L}_{4,t_0}^0(\sigma_z) \lesssim C(|\zeta|_{H^{t_0+1}}).
		\end{equation*}
		While in low frequencies, the symbol is bounded, which implies
		\begin{equation*}
			l_{0,t_0}(\sigma_z)  \leq C(|\zeta|_{H^{t_0+1}}).
		\end{equation*}
		Since $\sigma \in \Gamma_{t_0}^0$ we may use Proposition \ref{PDO Est Lannes} with inequality \eqref{product est Op} and Plancherel's identity to say that
		\begin{align*}
			| \mathrm{Op}(L)f|_{H^s} 
			& 
			\leq
			C(|\zeta|_{H^{t_0+1}})|\mathrm{e}^{-\frac{z}{2C} \sqrt{\mu}|\mathrm{D}|}f|_{H^s}.
		\end{align*}
		To conclude, we use this inequality with Plancherel's identity and Fubini's Theorem to see that
		\begin{align*}
			\|\Lambda^s\mathrm{Op}(L)\partial_x f\|_{L^2(\mathcal{S}^-)}^2 
			& =
			\int_{0}^{\infty} |(\mathrm{Op}(L)\partial_xf)(\cdot,z)|_{H^s}^2 \: \mathrm{d}z
			\\ 
			& 
			\leq
			C(|\zeta|_{H^{t_0+1}})\int_{0}^{\infty} |\mathrm{e}^{-\frac{z}{2C} \sqrt{\mu}|\mathrm{D}|}\partial_x f|_{H^s}^2 \: \mathrm{d}z
			\\ 
			& 
			\leq
			C(|\zeta|_{H^{t_0+1}})
			\int_{\R} ||\xi|^{\frac{1}{2}}\langle \xi \rangle^{s}\hat{f}(\xi)|^2\int_{0}^{\infty} e^{-\frac{z}{C} \sqrt{\mu}|\xi|}|\xi| \: \mathrm{d}z \mathrm{d}\xi
			\\ 
			& 
			\leq \mu^{-\frac{1}{2}}
			C(|\zeta|_{H^{t_0+1}})
			|f|_{\mathring{H}^{s+\frac{1}{2}}}^2.
		\end{align*}

		The proof of the remaining inequalities essentially boils down to having one more polynomial power in $|\xi|$ when compared to  $z$. In particular, for the proof of \eqref{Est op dzL}, we see that there is a gain of $\sqrt{\mu}$ and a $|\xi|$ that appears after computing the derivative with respect to $z$. Therefore, the proof follows as it did for \eqref{Est op L dx}.

		For the proof of \eqref{Est op 1/z dxL}, there is also a gain in the small parameters. However, there is a polynomial dependence in $z$ and an additional $|\xi|$. Since we need $|\xi|$ for the integrability, we define 
		\begin{equation*}
			\sigma_z^{1}(x,\xi) = |\xi|^{-1}\partial_x \sigma_z(x,\xi),
		\end{equation*}
		and make the computation for $|\xi|>\frac{1}{4}$:
		\begin{align*}
			|\partial_{\xi}^{\beta}\sigma_z^{1}(x, \xi)|_{H^{t_0}} 
			& \leq
			\ve  C(| \zeta |_{H^{t_0+2}})(z\mu)(z\sqrt{\mu})^{\beta}e^{-\frac{z}{2C}\sqrt{\mu}|\xi|} 
			\\ 
			& \leq z\ve \mu
			C(|\partial_x \zeta |_{H^{t_0+2}})(z\sqrt{\mu}|\xi|)^{\beta}e^{-\frac{z}{2C}\sqrt{\mu}|\xi|} \langle \xi \rangle^{-\beta}.
		\end{align*}
		Then we use Proposition \ref{PDO Est Lannes} to find that
		\begin{align*}
			\|\Lambda^s\mathrm{Op}(\frac{1}{z}\partial_x L)f\|_{L^2(\mathcal{S}^-)}^2 
			&\leq (\ve \mu)^2
			C(|\zeta|_{H^{t_0+2}})\int_{0}^{\infty} |e^{-\frac{z}{2C} \sqrt{\mu}|\mathrm{D}|} |\mathrm{D}| f|_{H^s}^2 \: \mathrm{d}z
			\\ 
			& 
			\leq \ve^2  \mu^{\frac{3}{2}}
			C(|\zeta|_{H^{t_0+1}})
			|f|_{\mathring{H}^{s+\frac{1}{2}}}^2.
		\end{align*}

		 For the proof of \eqref{Est op 1/z dxxL},  we argue similarly and define 
		 \begin{equation*}
		 	\sigma_z^{2}(x,\xi) = |\xi|^{-1}\partial_x^2 \sigma_z(x,\xi),
		 \end{equation*}
	 	and find that
	 	\begin{align*}
	 		|\partial_{\xi}^{\beta}\sigma_z^{2}(x, \xi)|_{H^{t_0}} 
	 		& \leq z\ve \mu
	 		C(|\partial_x \zeta |_{H^{t_0+3}})
	 		\langle \xi \rangle^{-\beta}.
	 	\end{align*}
 		At this point, the proof is the same as for \eqref{Est op 1/z dxL}.

 		For the proof of \eqref{Est op dx dz L}, we consider 
 		\begin{equation*}
 			\sigma_z^{3}(x,\xi) = |\xi|^{-1}\partial_x \partial_z \sigma_z(x,\xi),
 		\end{equation*}
 		and finds that
 		\begin{align*}
 			|\partial_{\xi}^{\beta}\sigma_z^3(x, \xi)|_{H^{t_0}} 
 			& \lesssim
 			\ve C(|\partial_x \zeta |_{H^{t_0}})(z\mu)(\sqrt{\mu}|\xi|)(z\sqrt{\mu})^{\beta}e^{-\frac{z}{2C}\sqrt{\mu}|\xi|} 
 			\\ 
 			& \leq
 			\ve \mu  C(|\partial_x \zeta |_{H^{t_0}})(z\sqrt{\mu}|\xi|)^{\beta+1}e^{-\frac{z}{2C}\sqrt{\mu}|\xi|} \langle \xi \rangle^{-\beta},
 		\end{align*}
 		from which we deduce the final result as above.
 		
 		For the proof of \eqref{Est op dxL dx}, we will apply the operator to $\partial_x f$ so we consider 
 		\begin{equation*}
 			\sigma_z^{4}(x,\xi) = \partial_x \sigma_z(x,\xi),
 		\end{equation*}
 		and as for $\sigma_z^{1}$ we find that
 		\begin{align*}
 			|\partial_{\xi}^{\beta}\sigma_z^4(x, \xi)|_{H^{t_0}} 
 			& \lesssim
 			\ve C(|\partial_x \zeta |_{H^{t_0}})(z\mu |\xi|)(z\sqrt{\mu})^{\beta}e^{-\frac{z}{2C}\sqrt{\mu}|\xi|} 
 			\\ 
 			& \leq
 			\ve \sqrt{\mu}  C(|\partial_x \zeta |_{H^{t_0}})(z\sqrt{\mu}|\xi|)^{\beta+1}e^{-\frac{z}{2C}\sqrt{\mu}|\xi|} \langle \xi \rangle^{-\beta},
 		\end{align*}
 		where there is a loss of $\mu^{\frac{1}{4}}$ when computing the final integral as above.
	\end{proof}

	The next estimates are given on $H^s(\R)$ and are simpler to deal with. They are versions of estimates used in \cite{LannesTwoFluid13}, but are listed here for the sake of clarity.
	
	\begin{prop}
		Let $\ve, \mu, \gamma, h_{\min} \in (0,1)$, $t_0>\frac{1}{2}$, and $\zeta \in H^{t_0+1}(\R)$ such that 
		\begin{equation*}
			h: = 1+\varepsilon \zeta \geq h_{\min}, \quad \text{for all } x \in \R.
		\end{equation*}
		Then for $f, g\in \mathscr{S}(\R)$ and $s\geq t_0+1$ there exist   $C = C(h_{\min}^{-1},|\zeta|_{H^s}) >0$ nondecreasing function of its argument such that:
		\begin{itemize}
			\item [1.] Define the operator 	
			\begin{equation*}
				\mathrm{Op}\big{(} \frac{S^+}{S^-}\big{)}f(x) =-  \mathcal{F}^{-1}\Big{(} \tanh(\sqrt{\mu} t(x, \xi))\hat{f}(\xi)\Big{)}(x),
			\end{equation*}
			where 
			\begin{equation*}
				t(x,\xi) = (1+\ve \zeta)\frac{\arctan(\ve \sqrt{\mu}\partial_x \zeta)}{\ve 	\sqrt{\mu}\partial_x \zeta}|\xi|.
			\end{equation*}
			Then for $s\in (-t_0,t_0)$ there holds,  
			\begin{equation}\label{PDO est tanh}
				|\mathrm{Op}\big{(} \frac{S^+}{S^-}\big{)} f|_{H^{s}}  \leq C(h_{\min}^{-1},|\zeta|_{H^{t_0+1}}) |f|_{H^{s}},
			\end{equation}
			and for $ s\in (-t_0,t_0+1]$ there holds,
			\begin{align}
				\label{Commutator Lambda s S+/S-}
				|[\Lambda^s, \mathrm{Op}\big{(} \frac{S^+}{S^-}\big{)}]f|_{H^{\frac{1}{2}}} 
				& \leq 
				\ve  C(h_{\min}^{-1},|\zeta|_{H^{t_0+2}}) |f|_{H^{s-\frac{1}{2}}}
				\\ 
				\label{adjoint est on S+/D}
				|(\mathrm{Op}\big( \frac{S^+}{S^-}\big)^{\ast} - \mathrm{Op}\big( \frac{S^+}{S^-}\big))f|_{H^s} & \leq
				 \ve  C(h_{\min}^{-1},|\zeta|_{H^{t_0+2}})
				 |f|_{H^{s-1}}.
			\end{align}
			
			\item [2.] Define the operators
			\begin{equation*}
				\mathrm{Op}\big{(}S_J\big{)}f(x) =  \mathcal{F}^{-1}\Big{(}\big{(}1 +  \gamma\tanh(\sqrt{\mu} t(x, \xi) \big{)} \hat{f}(\xi)\Big{)}(x)
			\end{equation*}
			\begin{equation*}
					\mathrm{Op}\big{(} \frac{1}{S_J}\big{)}f(x)
					 =
					  \mathcal{F}^{-1}\Big{(}\frac{\hat{f}(\xi)}{1 +  \gamma\tanh(\sqrt{\mu} t(x, \xi))}\Big{)}(x).
			\end{equation*}
			Then for $s\in (-t_0,t_0)$ there holds,
			\begin{equation}\label{Op 1/SJ}
				|\mathrm{Op}\big{(}\frac{1}{S_J}\big{)} f|_{H^s}  \leq   C(h_{\min}^{-1},|\zeta|_{H^{t_0+1}}) |f|_{H^s},
			\end{equation}
			and for $ s\in (-t_0,t_0+1]$ there holds,
			\begin{align}
				\label{Op S_J 1/SJ} 
				|\big(1 -\mathrm{Op}(S_J) \mathrm{Op}\big{(}\frac{1}{S_J}\big{)}\big{)}f|_{H^{s}} & \leq  \ve  C(h_{\min}^{-1},|\zeta|_{H^{t_0+2}}) |f|_{H^{s-1}}
				\\ 
				\label{Op SJ 1}
				|\big{[}\mathrm{Op}\big{(}\frac{1}{S_J}\big{)}, \mathcal{H} \big{]}f|_{H^{s}}  \leq  & \ve  C(h_{\min}^{-1},|\zeta|_{H^{t_0+2}}) |f|_{H^{s-1}}
				\\ 
				\label{Op SJ 2}
				|\big{[}\mathrm{Op}\big{(}\frac{1}{S_J}\big{)}, \partial_x \big{]}f|_{H^{s}}  \leq  & \ve  C(h_{\min}^{-1},|\zeta|_{H^{t_0+2}}) |f|_{H^{s}}
				\\
				\label{Op SJ 3}
				|\big{(}\mathrm{Op}\big{(}\frac{1}{S_J}\big{)}^{\ast} - \mathrm{Op}\big{(}\frac{1}{S_J}\big{)}\big{)}f|_{H^{s}}  \leq  & \ve  C(h_{\min}^{-1},|\zeta|_{H^{t_0+2}}) |f|_{H^{s-1}}.
			\end{align}

			\item [3.] Define the operators
			\begin{align*}
				\mathrm{Op}\Big( \frac{S^+}{S^-S_J} \Big)f(x)
				& = 
				-\mathcal{F}^{-1}\Big{(} \frac{\tanh(\sqrt{\mu} t(x, \xi))}{ 1 +  \gamma \tanh(\sqrt{\mu} t(x, \xi))}\hat{f}(\xi)\Big{)}(x)
				\\ 
				\mathrm{Op}\Big(\partial_x \frac{S^+}{S^-S_J} \Big)f(x)
				& = 
				-
				\mathcal{F}^{-1}\Big{(}\partial_x  \frac{\tanh(\sqrt{\mu} t(x, \xi))}{ 1 +  \gamma \tanh(\sqrt{\mu} t(x, \xi))}\hat{f}(\xi)\Big{)}(x)
			\end{align*}
			Then for $s\in (-t_0,t_0)$ there holds,
			\begin{align}\label{est S+/S-SJ}
				|\mathrm{Op}\Big( \frac{S^+}{S^-S_J}\Big{)}f|_{H^s} & \leq    C(h_{\min}^{-1},|\zeta|_{H^{t_0+1}})  |f|_{H^s},
				\\ 
				\label{Op dx S+/S- 1/SJ}
				| \mathrm{Op}\Big( \partial_x\frac{S^+}{S^-S_J} \Big) f|_{H^s}
				& \leq 
				\ve C(h_{\min}^{-1},|\zeta|_{H^{t_0+2}}) |f|_{H^s}.
			\end{align} 
			and for $ s\in (-t_0,t_0+1]$ there holds,
			\begin{align}	
				\label{commutator est S+/S-SJ}
				|[\mathrm{Op}\Big( \frac{S^+}{S^-S_J}\Big{)},f]\partial_x g|_{H^s} &\leq   \ve  C(h_{\min}^{-1},|\zeta|_{H^{t_0+2}})    |f|_{H^s}|g|_{H^s}
				\\ 
				\label{adjoint est S+/S-SJ}
				|\Big{(}\mathrm{Op}\Big( \frac{S^+}{S^-S_J} \Big)^{\ast}-\mathrm{Op}\Big( \frac{S^+}{S^-S_J} \Big)\Big{)}f|_{H^s} & \leq  \ve  C(h_{\min}^{-1},|\zeta|_{H^{t_0+2}})   |f|_{H^{s-1}}
				\\
				\label{Op S+/S- 1/SJ}
				|\Big{(} \mathrm{Op}\Big( \frac{S^+}{S^-} \Big) \mathrm{Op}\Big{(} \frac{1}{S_J} \Big{)}  - \mathrm{Op}\Big( \frac{S^+}{S^-S_J} \Big)\Big{)} f|_{H^{s}} & \leq \ve   C(h_{\min}^{-1},|\zeta|_{H^{t_0+2}})    |f|_{H^{s-1}}
			\end{align}
			\item [4.] Define the operators 
			\begin{equation*}
				\mathrm{Op}\Big(\frac{\mathfrak{P}^2}{S_J S^-}\Big)f(x)
				= 
				-\frac{1}{\sqrt{\mu}}\mathcal{F}^{-1}\Big{(} \frac{|\xi|(1+\sqrt{\mu}|\xi|)^{-1}}{ 1 +  \gamma \tanh(\sqrt{\mu} t(x, \xi))}\hat{f}(\xi)\Big{)}(x)
			\end{equation*}
			\begin{equation*}
				\mathrm{Op}\Big(\frac{S_J S^-}{\mathfrak{P}^2}\Big)\mathfrak{P}f(x)
				= 
				-\sqrt{\mu}
				\mathcal{F}^{-1}\Big{(} \frac{1 +  \gamma \tanh(\sqrt{\mu} t(x, \xi))}{|\xi| (1+\sqrt{\mu}|\xi|)^{-1}}\widehat{\mathfrak{P}f}(\xi)\Big{)}(x).
			\end{equation*}
			Then for  $k=0,1,$ and  $s\in (-t_0,t_0)$ there holds,
			\begin{align}
				\label{S-SJ comp est}  
				|\big{(}	\mathrm{Op}(S^-S_J) -\mathrm{Op}(S^-)\mathrm{Op}(S_J)  \big{)}f|_{H^{s+\frac{k}{2}}} 
				& \leq \ve \mu^{1-\frac{k}{4}}C(h_{\min}^{-1},|\zeta|_{H^{t_0+2}})  |f|_{\dot{H}_{\mu}^{s+\frac{1}{2}}}.
				\\ 
				\label{1 - op b2...}
				|\big{(} 1- \mathrm{Op}(\frac{\mathfrak{P}^2}{S_J S^-})\mathrm{Op}(\frac{S_J S^-}{\mathfrak{P}^2}) \big{)}\mathfrak{P}f|_{H^{s+\frac{k}{2}}} 
				 & \leq
				 \ve \mu^{-\frac{k}{4}}C(|\zeta|_{H^{{t_0+2}}}) |f|_{H^s}.
			\end{align}
		\color{black}
		
	\end{itemize}	
	\end{prop}
	
	\begin{proof}
	
		We consider each point separately in individual steps. \\ 
		
		\noindent
		\underline{Step 1.}  For the proof of \eqref{PDO est tanh}, we will verify that 
		\begin{equation}\label{claim S+/S-}
			\sigma : = \frac{S^+}{S^-} = \tanh(\sqrt{\mu} t ) \in \Gamma_{t_0}^0.
		\end{equation}
		In this case, the proof follows by Theorem \ref{PDO Est Lannes} and estimate \eqref{product est Op}:
		\begin{align*}
			|\mathrm{Op}\big{(} \frac{S^+}{S^-}\big{)} f|_{H^{s}}  
			 \leq 	|\sigma|_{H^{t_0}_{(0)}}  |f|_{H^{s}}.
		\end{align*} 
		To verify \eqref{claim S+/S-}, we observe by Sobolev embedding and the non-cavitation condition that
		\begin{align*}
			h_{\min}\leq h = 1+ \varepsilon\zeta \leq C(|\zeta|_{H^{t_0}}),
		\end{align*}
		and 
		\begin{align*}
			0< \frac{1}{C(|\zeta|_{t_0+1})} \leq g := \frac{\arctan(\varepsilon \sqrt{\mu}\partial_x \zeta)}{\varepsilon \sqrt{\mu}\partial_x \zeta }\leq C(|\zeta|_{H^{t_0+1}}).
		\end{align*}
		Then since $|\tanh(r)|\leq |r|$		we can use it to  prove that $\sigma$ satisfies
		\begin{equation*}
			|\tanh(\sqrt{\mu} t(\cdot, \xi))|_{H^{t_0}} \leq C(|\zeta|_{H^{t_0+1}}),
		\end{equation*}
		for all $\xi \in \R$. On the other hand, for $|\xi| \geq \frac{1}{4}$ and  $\beta \geq 1$ we have that $\xi \mapsto  \sigma(\cdot, \xi)$ is smooth and since
		\begin{equation}\label{def tanh}
			r \mapsto \tanh(r) = 1 - \frac{2}{e^{2r} +1},
		\end{equation} 
		we have that
		\begin{align*}
			|\partial_{\xi}^{\beta}\sigma(\cdot, \xi)|_{H^{t_0}} 
			& \lesssim
			C(| \zeta |_{H^{t_0+1}})(\sqrt{\mu})^{\beta}e^{-\frac{1}{C}\sqrt{\mu}|\xi|} 
			\\ 
			& \leq
			C(| \zeta |_{H^{t_0+1}})(\sqrt{\mu}|\xi|)^{\beta}e^{-\frac{1}{C}\sqrt{\mu}|\xi|} \langle \xi \rangle^{-\beta},
		\end{align*}
		which implies
		\begin{equation*}
			|\sigma|_{H^{t_0}_{(0)}} \leq C(h_{\min}^{-1},|\zeta|_{t_0+1}).
		\end{equation*}

		For the proof of \eqref{Commutator Lambda s S+/S-}, we will use \eqref{Commutator Op}. However, to make the small parameters appear, we  decompose $\sigma$ into three pieces:
		\begin{align}
			\: \sigma(x,\xi) 
			& =\notag
			\tanh(\sqrt{\mu}|\xi| h(x) g(x))
			\\
			& = 
			\Big{(}\tanh(\sqrt{\mu}|\xi| h(x) g(x)) - \tanh(\sqrt{\mu}|\xi|  g(x)) \Big{)}
			+
			\Big{(}\tanh(\sqrt{\mu}|\xi|  g(x)) - \tanh(\sqrt{\mu}|\xi|  ) \Big{)} \notag
			\\ 
			& \hspace{0.5cm}\notag
			+ \tanh(\sqrt{\mu}|\xi| ) 
			\\ 
			& = 	\sigma^1(x,\xi) + 	\sigma^2(x,\xi) + 	\sigma^3(\xi),\label{decom sigma}
		\end{align}
		where $h$ and $g$ are as above. Then for $\sigma(x,\xi)  = \Sigma (h(x),g(x),\xi)$, we have $(h,g) \mapsto \Sigma (h,g,\xi)$ is smooth and for 
		$$\sigma^1(x,\xi) = \Sigma (h(x),g(x),\xi) - \Sigma (1,g(x),\xi),$$
		we observe by the Mean Value Theorem, the algebra property of $H^{t_0+1}(\R)$, and the identity
		$$|h-1|_{H^{t_0+1}} = \ve |\zeta|_{H^{t_0+1}}$$
		that
		\begin{align*}
				|\sigma^1(\cdot, \xi)|_{H^{t_0+1}} 
				\leq \ve C(h_{\min}^{-1}, |\zeta|_{H^{t_0+2}})
				\sqrt{\mu} |\xi| e^{-\frac{1}{C}\sqrt{\mu}|\xi|}, 
		\end{align*}
		for any $\xi \in \R$. While for $|\xi|\geq \frac{1}{4}$ we note that 
		$$|\partial_{\xi}^{\beta}\partial_h	\Sigma(h(\cdot), f(\cdot), \xi) |_{H^{t_0+1}}\lesssim C(h_{\min}^{-1}, |\zeta|_{H^{t_0+2}}) (\sqrt{\mu})^{\beta} \sqrt{\mu}|\xi|
		e^{-\frac{1}{C}\sqrt{\mu}|\xi|},$$
		to get 
			\begin{align*}
			|\partial_{\xi}^{\beta}\sigma^1(\cdot, \xi)|_{H^{t_0+1}} 
			\leq \ve C(h_{\min}^{-1}, |\zeta|_{H^{t_0+2}})
			(\sqrt{\mu} |\xi|)^{\beta+1} e^{-\frac{1}{C}\sqrt{\mu}|\xi|} \langle \xi \rangle^{-\beta},
		\end{align*}
		and we conclude that 
		\begin{equation*}
			|\sigma^1|_{H^{t_0}_{(0)}} \leq \ve C(h_{\min}^{-1},|\zeta|_{t_0+2}).
		\end{equation*}
		For the estimate on $\sigma^2$, we argue similarly but instead note that  $|\frac{\arctan(r)}{r}-1|\leq |r|$ and we can use it to obtain
		\begin{equation*}
			|g-1|_{H^{t_0+1}} \leq  \ve \sqrt{\mu}	|\partial_x\zeta|_{H^{t_0+1}},
		\end{equation*}
		from which we deduce the bound
		\begin{equation*}
			|\sigma^2|_{H^{t_0}_{(0)}} \leq \ve \sqrt{\mu}C(h_{\min}^{-1},|\zeta|_{H^{t_0+2}}).
		\end{equation*}
		Then since $[\Lambda^{\frac{1}{2}}, \sigma^3(\mathrm{D})]f = 0$, we may conclude by  \eqref{Commutator Op} that 
		\begin{align*}
			|[\Lambda^s, \mathrm{Op}\big{(} \frac{S^+}{S^-}\big{)}]f|_{H^{\frac{1}{2}}} 
			& \leq
			|[\Lambda^s, \mathrm{Op}(\sigma^1)]f|_{H^{\frac{1}{2}}} 
			+
			|[\Lambda^s, \mathrm{Op}(\sigma^2)]f|_{H^{\frac{1}{2}}} 
			\\ 
			& \leq 
			\ve  C(h_{\min}^{-1},|\zeta|_{H^{t_0+2}}) |f|_{H^{s-\frac{1}{2}}}.
		\end{align*}

		For the proof of \eqref{adjoint est on S+/D}, we again use the decomposition given by \eqref{decom sigma} and then conclude by \eqref{adjoint est} using that the symbol is real-valued and $(\sigma^3)^{\ast} = \sigma^3$. \\ 

		\noindent
		\underline{Step 2.} For the proof of \eqref{Op 1/SJ}, we  argue as in the previous step to find that
		\begin{equation}\label{est step 2}
		 	\frac{1}{1 +  \gamma\sigma} \in \Gamma_{t_0}^0
		 	\quad 
		 	\text{with} \quad 
		 	\Big{|}\frac{1}{1 +  \gamma\sigma}\Big{|}_{H^{t_0}_{(0)}} \leq C(h_{\min}^{-1},|\zeta|_{H^{t_0+1}}),
		\end{equation}
		Then for \eqref{Op 1/SJ} we use \eqref{product est Op}.

		For the proof of \eqref{Op S_J 1/SJ} we use the notation in Step 1. to make the following decomposition
		\begin{align*}
			1-\mathrm{Op}(S_J)\mathrm{Op}\big{(}\frac{1}{S_J}\big{)} 
			& =
			\gamma \mathrm{Op}\big{(} \frac{\sigma}{S_J}\big{)}-
			\gamma \mathrm{Op}(\sigma)\mathrm{Op}\big{(}\frac{1}{S_J}\big{)} 
			\\
			& = 
			\gamma \sum\limits_{j=1}^2 \Big{(}\mathrm{Op}\big{(} \frac{\sigma^j}{S_J}\big{)}-
			 \mathrm{Op}(\sigma^j)\mathrm{Op}\big{(}\frac{1}{S_J}\big{)}\Big{)} - \gamma[\sigma^{3}(\mathrm{D}), \mathrm{Op}\big(\frac{1}{S_J}\big{)}]
			 \\ 
			 & =  \mathrm{Op}(A_1) +  \mathrm{Op}(A_2).
		\end{align*}
		where we used that $\mathrm{Op}\big{(}\frac{1}{S_J}\big{)} \circ \sigma^3(\mathrm{D}) = \mathrm{Op}\big{(}\frac{\sigma^3}{S_J}\big{)}$. Then to estimate the contribution of $A_1$, we apply the composition estimate \eqref{Comp est}, the estimates in Step 1., and \eqref{est step 2}  to deduce that
		\begin{align*}
			|\mathrm{Op}(A_1)f|_{H^s} = 
			\Big|\Big{(}\mathrm{Op}\big{(} \frac{\sigma^j}{S_J}\big{)}-
			\mathrm{Op}(\sigma^j)\mathrm{Op}\big{(}\frac{1}{S_J}\big{)}\Big{)}f\Big|_{H^s} 
			& \lesssim
			|\sigma^j|_{H^{t_0+1}_{(0)}} \big|\frac{1}{S_J}\big|_{H^{t_0+1}_{(0)}}|f|_{H^{s-1}}
			\\ 
			& \leq 
			\ve  C(h_{\min}^{-1},|\zeta|_{H^{t_0+2}})  |f|_{H^{s-1}},
		\end{align*}
		for $j=1,2$. On the other hand, for the symbol associated to $A_2$ we observe that
		%
		%
		%
		%
		%
		%
		%
		%
		%
		\begin{align*}
			\Big|\frac{1}{S_J} -\frac{1}{1+\gamma\sigma^3(\mathrm{D})}\Big|_{H^{t_0+1}_{(0)}}
			&
			\leq \gamma
			\big{|}\frac{ \sigma^1}{(1+\gamma \sigma)(1+\gamma\tanh(\sqrt{\mu}|\mathrm{D}| g))} \big{|}_{H^{t_0+1}_{(0)}}
			\\ 
			&
			\hspace{0.5cm}+
			\gamma\Big{|}
			\frac{ \sigma^2}{(1+\gamma \sigma)(1+\gamma \tanh(\sqrt{\mu}|\mathrm{D}| ))} \Big{|}_{H^{t_0+1}_{(0)}}
			\\ 
			& \leq \ve \gamma C(h_{\min}^{-1},|\zeta|_{H^{t_0+2}}),
		\end{align*}
		and then conclude by the commutator estimate \eqref{Thm 6} that
		\begin{align*}
				|\mathrm{Op}(A_2)f|_{H^s} = |[\sigma^{3}(\mathrm{D}), \mathrm{Op}\big(\frac{1}{S_J}\big{)}]f|_{H^s} \leq \ve  C(h_{\min}^{-1},|\zeta|_{H^{t_0+2}}) |f|_{H^{s-1}}.
		\end{align*}
	
		For the proof of \eqref{Op SJ 1}, \eqref{Op SJ 2}, and  \eqref{Op SJ 3},  we argue as above where we use  \eqref{Thm 6} for the first two, and for \eqref{Op SJ 3} we use the adjoint estimate \eqref{adjoint est}.
		\\
		
		\noindent
		\underline{Step 3.} For the proof of estimates \eqref{est S+/S-SJ}, \eqref{Op dx S+/S- 1/SJ}, \eqref{commutator est S+/S-SJ}, \eqref{adjoint est S+/S-SJ}, and \eqref{Op S+/S- 1/SJ} we simply use the observation
		\begin{equation*}
			\gamma\frac{S^+}{S^-S_J} = \frac{1}{1+\gamma \sigma}-1,
		\end{equation*}
		and arguing similarly as in Step 2. to deduce the results.
		
		For the proof of \eqref{S-SJ comp est}, we use the notation in Step 1. to make the observation that
		\begin{align*}
			\mathrm{Op}(S^-S_J) -\mathrm{Op}(S^-)\mathrm{Op}(S_J) 
			& = \gamma \sqrt{\mu}\Big{(} \mathrm{Op}(|\mathrm{D}|\sigma)-|\mathrm{D}|\mathrm{Op}(\sigma)\Big{)}
			\\
			& = 
			\gamma \sqrt{\mu}
			\sum \limits_{j=1}^2
			\Big{(} \mathrm{Op}(|\mathrm{D}|\frac{\sigma^j}{\mathfrak{P}})-|\mathrm{D}|\mathrm{Op}(\frac{\sigma^j}{\mathfrak{P}})\Big{)}\mathfrak{P}
			\\ 
			& =
			-	
			\gamma \sqrt{\mu}
			\sum \limits_{j=1}^2[|\mathrm{D}|,\mathrm{Op}(\frac{\sigma^j}{\mathfrak{P}})\Big{)}]\mathfrak{P}.
		\end{align*}
		Then we would like to conclude by the commutator estimate \eqref{Thm 6}:
		\begin{align*}
			|\big{(}	\mathrm{Op}(S^-S_J) -\mathrm{Op}(S^-)\mathrm{Op}(S_J)  \big{)}f|_{H^{s+\frac{k}{2}}} 
			&\lesssim  \sqrt{\mu} \max_{j=1,2} 
			  \Big| \big[|\mathrm{D}|,\mathrm{Op}(\frac{\sigma^j}{\mathfrak{P}})\big]\mathfrak{P}f\Big|_{H^{s+\frac{k}{2}}}
			\\
			& \lesssim 
			\sqrt{\mu} \max_{j=1,2}|\frac{\sigma^j}{\mathfrak{P}}|_{H^{t_0+1}_{(-\frac{k}{2})}} |\mathfrak{P}f|_{H^{s}}.
		\end{align*}
		Then, arguing as in Step 1. we deduce that
		\begin{align*}
			|\mathfrak{P}^{-1}\sigma^j(\cdot, \xi)|_{H^{t_0+1}} 
			& 
			\leq 
			\ve\sqrt{\mu} C(h_{\min}^{-1}, |\zeta|_{H^{t_0+2}})
			 (1+\sqrt{\mu}|\xi|)^{\frac{1}{2}} e^{-\frac{1}{C}\sqrt{\mu}|\xi|}
			 \\ 
			 & 
			 \leq 
			 \ve\sqrt{\mu} C(h_{\min}^{-1}, |\zeta|_{H^{t_0+2}}),
		\end{align*}
		and for $|\xi|\geq \frac{1}{4}$ we get
		\begin{align*}
			|\partial_{\xi}^{\beta}(\mathfrak{P}^{-1}\sigma^j(\cdot, \xi))|_{H^{t_0+1}} 
			& \leq
			\ve \sqrt{\mu} C(h_{\min}^{-1}, |\zeta|_{H^{t_0+2}})
			(1+\sqrt{\mu} |\xi|)^{\frac{1}{2}} (\sqrt{\mu})^{\beta}e^{-\frac{1}{C}\sqrt{\mu}|\xi|},
			\\ 
			& 
			\leq 
			\ve\sqrt{\mu} C(h_{\min}^{-1}, |\zeta|_{H^{t_0+2}})	\langle \xi \rangle^{\frac{k}{2}}\langle\sqrt{\mu} |\xi|\rangle^{\beta + \frac{1}{2}}e^{-\frac{1}{C}
			\sqrt{\mu}|\xi|} \langle \xi\rangle^{-\beta-\frac{k}{2}} .  
		\end{align*}
		We conclude the final estimate from the following bound
		\begin{equation*}
			|\mathfrak{P}^{-1}\sigma^j|_{H^{t_0}_{(-\frac{k}{2})}} \leq \ve \mu^{\frac{1}{2}-\frac{k}{4}} C(h_{\min}^{-1},|\zeta|_{H^{t_0+2}}).
		\end{equation*}

		\noindent
		\underline{Step 4.} For the proof of \eqref{1 - op b2...},  we need to work on the domain of $\mathfrak{P}f$ for the composition to be well-defined in low frequency. With this in mind, we define the symbols 
		$$\sigma^4 = \frac{\mathfrak{P}^2}{S_J S^-} \quad \text{and} \quad \sigma^5 = \frac{S_J S^-}{\mathfrak{P}^2},$$
		then make the decomposition
		\begin{align*}
			 \mathrm{Op}(\sigma^4)\circ\mathrm{Op}(\sigma^5)\circ |\mathrm{D}|
			 & =
			 \mathrm{Op}(\frac{1}{S_J}) \circ \frac{|\mathrm{D}|}{(1+\sqrt{\mu}|\mathrm{D}|)} \circ \mathrm{Op}(S_J ) \circ (1+\sqrt{\mu}|\mathrm{D}|)
			 \\ 
			 & = 
			 \mathrm{Op}(\frac{1}{S_J}) \circ |\mathrm{D}| \circ \mathrm{Op}(S_J ) 
			 +
			 \mathrm{Op}(\frac{1}{S_J}) \circ \frac{|\mathrm{D}|}{(1+\sqrt{\mu}|\mathrm{D}|)} \circ [\mathrm{Op}(S_J ), \sqrt{\mu}|\mathrm{D}|]
			 \\ 
			 & =
			 \mathrm{Op}(\frac{1}{S_J}) \circ \mathrm{Op}(S_J )\circ |\mathrm{D}| 
			 +
			  \gamma \mathrm{Op}(\frac{1}{S_J})\sum \limits_{j=1}^2 [ |\mathrm{D}|, \mathrm{Op}(\sigma^j) ]
			 \\
			 & \hspace{0.5cm}
			 +
			 \gamma  \mathrm{Op}(\frac{1}{S_J}) \circ \frac{\sqrt{\mu}|\mathrm{D}|}{(1+\sqrt{\mu}|\mathrm{D}|)} \circ \sum \limits_{j=1}^2[\mathrm{Op}(\sigma^j), |\mathrm{D}|],
		\end{align*}
		and use it together with the boundedness of $\mathrm{Op}(\frac{1}{S_J})$, given by \eqref{Op 1/SJ}, to see that
		\begin{align*}
			\mathrm{LHS}_{\eqref{1 - op b2...}} 
			& =
			|\big{(} 1- \mathrm{Op}(\frac{\mathfrak{P}^2}{S_J S^-})\mathrm{Op}(\frac{S_J S^-}{\mathfrak{P}^2}) \big{)}\mathfrak{P}f|_{H^{s + \frac{k}{2}}} 
			\\ 
			& =
			\Big|\big{(} 1- \mathrm{Op}(\sigma^4)\mathrm{Op}(\sigma^5) \big{)}|\mathrm{D}| \frac{\mathfrak{P}}{|\mathrm{D}|}f\Big|_{H^{s+ \frac{k}{2}}} 
			\\ 
			& \leq 
			\Big|\big{(} 1-  \mathrm{Op}(\frac{1}{S_J}) \circ \mathrm{Op}(S_J ) \big{)} \mathfrak{P}f\Big|_{H^{s+ \frac{k}{2}}}
			+
			 \max\limits_{j=1,2} \Big|[ |\mathrm{D}|, \mathrm{Op}(\sigma^j) ] \frac{\mathfrak{P}}{|\mathrm{D}|}f\Big|_{H^{s+ \frac{k}{2}}}. 
		\end{align*}
		Then by \eqref{Op S_J 1/SJ} and the commutator estimate \eqref{Thm 6} we conclude that
		\begin{equation*}
			\mathrm{LHS}_{\eqref{1 - op b2...}} \leq  \ve  C(h_{\min}^{-1},|\zeta|_{H^{t_0+2}}) |(1+\sqrt{\mu} |\mathrm{D}|)^{-\frac{1}{2}} f|_{H^{s+ \frac{k}{2}}}.
		\end{equation*}

	\end{proof}

    \bibliographystyle{plain}
    \bibliography{Biblio.bib}

\end{document}